\documentclass[a4paper,11pt]{article}
\usepackage[utf8]{inputenc}
\usepackage{amsmath,amsfonts}
\usepackage{amsthm} 
\usepackage{amssymb}
\usepackage{multirow}
\usepackage{latexsym}
\usepackage{color}
\usepackage{graphicx} 
\usepackage{subcaption}

\newtheorem{lemma}{Lemma}[section]
\newtheorem{theorem}{Theorem}[section]

\newtheorem{proposition}[theorem]{Proposition}
\newtheorem{fact}[theorem]{Fact}

\allowdisplaybreaks

\usepackage[parfill]{parskip}

\usepackage{bbm}

\newcommand{\IR}{\mathbb{R}}
\newcommand{\IE}{\mathbb{E}}
\newcommand{\IP}{\mathbb{P}}
\newcommand{\Ind}{\mathbbm{1}}
\newcommand{\IN}{\mathbb{N}}

\renewcommand{\tilde}{\widetilde}

\newtheorem{remark}[theorem]{Remark}

 \numberwithin{equation}{section}
\begin{document}

\begin{titlepage}

\title{Nonparametric inference for continuous-time event counting and link-based dynamic network models}

\author{
Alexander Krei{\ss}\\ 
University of Mannheim \\
Department of Economics \\
L7, 3-5 \\
68161 Mannheim \\
Germany \\
kreiss@uni-mannheim.de
\and
Enno Mammen\\ 
Heidelberg University\\
Institute for Appplied Mathematics\\
Im Neuenheimer Feld 205 \\
69120 Heidelberg\\
Germany\\
mammen@math.uni-heidelberg.de
\and
Wolfgang Polonik\\ 
Department of Statistics\\
University of California, Davis\\
One Shields Ave.\\
Davis, CA 95616\\
USA\\
wpolonik@ucdavis.edu}

\date{\today}
\maketitle

\begin{abstract}
A flexible approach for modeling both dynamic event counting and dynamic link-based networks based on counting processes is proposed, and estimation in these models is studied. We consider nonparametric likelihood based estimation of parameter functions via  kernel smoothing. The asymptotic behavior of these estimators is rigorously analyzed in an asymptotic framework where the number of nodes tends to infinity. The finite sample performance of the estimators is illustrated through an empirical analysis of bike share data.  
\end{abstract}
\end{titlepage}

\setcounter{page}{1} 
\setcounter{section}{0}

\section{Introduction}
In this paper we present a modeling approach that can be applied to both dynamic interactions in networks as well as dynamic link deletion and addition in networks. The case of dynamic interactions considers a network as a collection of actors who can cause instantaneous interactions. Both directed and undirected interactions are considered. In the model we assume that the time at which an interaction happens, and the pair of actors involved in this interaction, are random. We call this model a dynamic network interaction model. The Enron e-mail data set provides a typical example for a data set that can be modeled in such a way. Here one person sending an e-mail to another person is interpreted as the interaction. While such interactions can be thought of as edges between two nodes, and while the nodes themselves persist over time, each such edge only exists for an infinitesimal time. In contrast to interaction events, we also consider models for connections between the actors of a network that persist over a longer  time period. In this case, connections comprise four quantities: sender, receiver (a pair of actors), and the lifetime of each connection (start and end). Here we speak of dynamic networks. Examples for such models are social networks with links indicating an ongoing friendship between two actors. The notions of dynamic network interactions and of dynamic networks are different but they are closely related. A dynamic network defines two network interaction models, one given by the starting time and one given by the ending time of a connection. Furthermore, a network interaction model of e-mails defines a dynamic network by aggregation, where the network shows a connection between two actors as long as they have exchanged e-mails over a certain time period in the past. 

In this paper we mainly consider dynamic network interaction models, because dynamic networks can be understood as two dynamic network interaction models. In the model the distribution of the next event conditional on the past is expressed as depending on two quantities: The covariates and the parameter function. The covariates are random but observed. They summarize the relevant history of past interactions. Naturally they are functions of the time. How the covariates influence the distribution is regulated by an unobserved, deterministic parameter function. Thus, the model allows for the effect of the covariates to be changing over time. It is this effect (the parameter function) that is our primary interest of estimation.

In this paper we will make use of counting process models from survival analysis that we adapt for our network models. We will develop asymptotic theory for a kernel-based estimator of parameter functions in the models. The estimator is based on a localized likelihood criterion.

\subsection{Literature Review and Related Work}
Random networks/graphs have been considered in various scientific branches since a long time, in particular in the social sciences (c.f. the textbook \cite{Wasserman:Faust1994}). The importance of the analysis of random networks within the more statistical and machine learning literature is more recent, but corresponding literature is significant by now (e.g. see \cite{Goldenberg2010,Kolaczyk2009}). The reason for this increase in significance is not least due to the development of modern technologies that lead to the ever increasing number of complex data sets that are encoding relational structures. Examples for real network data can be found at the Koblenz Network Collection KONECT, the European network data collection SocioPatterns, the MIT based collection Reality Commons, the data sets made available by the Max Planck Institute for Softwaresysteme (MPI-SWS), or the Stanford Large Network Dataset Collection (SNAP). In this paper, we will use the Capital Bikeshare Performance Data  (see {\tt http://www.capitalbikeshare.com/system-data}), which is a data set collected on the Washington, DC bikeshare system.

Modeling and analyzing dynamic random networks is challenging as networks can have a multitude of different topological properties. In the literature such topological properties are measured by various quantities, including the flow through the network, the degree distribution, centrality, the existence of hubs, sparsity etc. Time-varying or dynamic random networks appear quite naturally, even though the dynamic aspect significantly adds to the complexity of modeling and analyzing the networks. Early work on networks involving temporal structures can already be found in \cite{Katz:Proctor1959}, who consider a discrete time Markov process for friendships (links). Other relevant literature using discrete time settings include work on dynamic exponential random graph models \cite{Sarakr:Moore2006,Sarkar2007,Guo2007,AhmedXing2009,Hanneke2010,Krivitsky2012,Krivitsky:Handcock2014}, dynamic infinite relational models \cite{Ishiguro2010}, dynamic block models \cite{Ho2011,Xing2012,Xu:Hero2014,Xu2015}, dynamic nodal states models \cite{Kolar:Xing2009,Kolaretal2010}, various dynamic latent features model \cite{Foulds2011},  dynamic multi-group membership models \cite{Kim:Lescovec2013},  dynamic latent space models \cite{Durante2014,Sewell:Chen2015}, and dynamic Gaussian graphical models \cite{Zhou2008,Kolar:Xing2012}. Also time-continuous models have been discussed in the literature. They include link-based continuous-time Markov processes \cite{Wasserman1980,Leenders1995,Lee:Priebe2011}, actor-based continuous-time Markov processes \cite{Snijders1996,Snijders2001,Snijders2010}, and also models based on counting processes as in \cite{Perry:2013}, who are considering the modeling of network interaction data, and \cite{Butts:2008}, who applies such a model to radio communication data. Link Prediction, a problem related to the analysis of dynamic networks, received quite some attention in the computer science literature (e.g. see \cite{Liben-Nowell:Kleinberg2007,Backstrom:Leskovec2011}).

\subsection{Our Work}
\label{sec:ow}
In this paper, we study a network model under the asymptotics that the network size $n$ (the number of actors) is growing to infinity. So let $V_n:=\{1,...,n\}$ denote the set of $n$ actors (also called agents or nodes), and let $L_n \subseteq \{(i,j): i < j,\,i,j \in V_n\}$ denote the set of all possible links among them. For directed networks $L_n$ is the set of all ordered pairs. Furthermore, let $G_n := (V_n,L_n)$ be the corresponding graph. For each pair of actors $(i,j)\in L_n$ we denote by $N_{n,ij}\colon[0,T]\to\IN$ the number of interactions between these two
$$N_{n,ij}(t)=\#\{\textrm{interaction events between }i\textrm{ and }j\textrm{ up to and including time }t\}.$$
 We assume that for $(i,j) \in L_n$, the processes $N_{n,ij}$ are  one-dimensional counting processes with respect to an increasing, right continuous, complete filtration ${\cal F}_{t}$, $t \in [0, T]$, i.e. the filtration obeys the \emph{conditions habituelles}, see \cite{Anderson93}, pp. 60. The $\sigma$-field ${\cal F}_{t}$ contains all information available up to the time point $t$. For simplicity we formulate all results for undirected interactions only, i.e, we assume that $N_{n,ij}=N_{n,ji}$ for all pairs $(i,j)\in L_n$. All results can be formulated for the directed case as well (see also the discussion in Section \ref{sec:assumptions} after assumption (A2)).

In our approach, we model the intensities of the counting processes $N_{n,ij}$ at time $t$ only for a random subset of edges $\{(i,j)\in L_n: C_{n,ij}(t)=1\}$.  The  functions $C_{n,ij}(t) \in \{0,1\}$ are indicator functions assumed to be predictable with respect to the filtration ${\cal F}_t$, and they determine the {\it active part} of the network. Our aim is to model the active part. The definition of the active part depends on the application. For instance, modeling the edges between actors $i$ and $j$ might depend on whether $i$ and $j$ have had a low or a high interaction intensity in the past. We will come back to this point later.

For the set $\{(i,j)\in L_n: C_{n,ij}(t)=1\}$ the intensities of the counting processes $N_{n,ij}$ are modeled by
$$\lambda_{n,ij}(\theta,t):=\Phi(\theta(t);\left(X_{n,ij}(s)\right)_{i,j=1,...,n}: s\leq t),$$
where $X_{n,ij}\colon[0,T]\to\IR^q$ are $\mathcal{F}_t$-predictable covariates, $\Phi$ is a link function, and the parameter function $\theta\colon[0,T]\to\IR^q$ is the target of our estimation method. The presented above approach is quite flexible and general. In order to be more specific, and for modeling reasons explained in Section \ref{sec:basic_model}, we will, in the following, assume that $\Phi$ has the following Cox-type form
\begin{equation}
\label{eq:definitionG}
\Phi(\theta;\left(C_{n,ij}(s),X_{n,ij}(s)\right)_{i,j=1,...,n}: s\leq t)=C_{n,ij}(t)\exp(\theta^T(t) X_{n,ij}(t)).
\end{equation}
Butts suggests in \cite{Butts:2008} the same modeling framework with a constant parameter in an empirical analysis to radio communication data from responders to the World Trade Center Disaster. In this paper we extend the work of \cite{Butts:2008} to time varying parameter functions and provide asymptotic theory. Another related model can be found in \cite{Leenders1995}. Perry and Wolfe studied in \cite{Perry:2013} a model similar to \eqref{eq:definitionG} with constant parameters. In their specification the intensity was equal to $\lambda(t) \exp(\theta ^T X_{n,ij}(t))$ where $\lambda(t) $ is an unknown baseline hazard. They developed asymptotic theory for maximum partial likelihood estimators of $\theta$ in an asymptotic framework where the time horizon $T$ converges to infinity. Our work was motivated by their research but it differs in several respects. First of all we allow the parameters to change over time. Our estimates of the parameter functions can be used for statistical inference on time changes in the effects of covariates. Furthermore, by choosing the first component $X_{n,ij}^{(1)} = 1$ our model includes the time-varying baseline intensity $e^{\theta^{(1)}(t)}$. Thus in contrast to \cite{Perry:2013} we propose a fit of the full specification of the intensity function including all parameters and the baseline intensity. Furthermore, our aim is to model large networks whereas \cite{Perry:2013} considered relatively small networks over a long time period $T$. Thus in our asymptotics we let the number of actors converge to infinity instead of $T$. We will argue below (at the end of Section \ref{sec:basic_model} and after Assumption (A6) in Section \ref{sec:assumptions}) that appropriate choices of the censoring factor $C_{n,ij}(t)$ allow for modeling large networks with degrees of the nodes/actors being relatively small compared to the size of the network.

Despite the strong interest in dynamic models for networks, rigorous statistical analysis of corresponding estimators (asymptotic distribution theory) are relatively sparse, in particular in the case of time-varying parameters, as considered here. The temporal models in the literature are usually Markovian in nature. In contrast to that, our continuous-time model based on counting processes allows for non-Markovian structures (i.e. dependence on the infinite past). This increases flexibility in the modeling of the temporal dynamics. Our model also allows for a change of the network size over time without the networks degenerating in the limit. Moreover, we are presenting a rigorous analysis of distributional asymptotic properties of the corresponding maximum likelihood estimators. To the best of our knowledge, no such analysis can be found in the literature, even for the simpler models indicated above.   

In Section~\ref{model}, we discuss our model (Section~\ref{sec:basic_model}), define our likelihood-based estimators, and present our main result on the point-wise asymptotic normality of our estimators in Section \ref{est:time-varying}. In Section~\ref{data-analysis}, we demonstrate the finite sample behavior and the flexibility of our approach by presenting an analysis of the Capital Bike-share Data. The proof of our main result is deferred to Section~\ref{proofs}. The appendix contains additional simulations where we compare network characteristics as degree distributions, cluster coefficients and diameters of the observed network with networks distributed according to the fitted model. Moreover, we discuss data adaptive bandwidth choices in the appendix.

\section{Link-based dynamic models}
\label{model}
\subsection{Link-based dynamic models with constant parameters}
\label{sec:basic_model}
We will first discuss the model described in Section \ref{sec:ow} with general link function $\Phi$ and {\it constant }parameter function $\theta\equiv\theta_0$. The following form of the log-likelihood for the parameter $\theta$ is shown in \cite{Anderson93}:
\begin{equation}
\label{eq:likelihood}
\ell_T(\theta)=\sum_{0<t\leq T}\sum_{(i,j)\in L_n}\Delta N_{n,ij}(t)\log \lambda_{n,ij}(\theta,t)-\int_0^T\sum_{(i,j)\in L_n}\lambda_{n,ij}(\theta,t)dt,
\end{equation}
where $\Delta N_{n,ij}(t):=N_{n,ij}(t)-N_{n,ij}(t-)$ (and $N_{n,ij}(t-):=\lim_{\delta\to0,\delta>0}N_{n,ij}(t-\delta)$) is the jump height (either 0 or 1) of $N_{n,ij}$ at $t$. And hence, we obtain the maximum likelihood estimator as
$$\hat{\theta}:=\arg \max_{\theta\in\Theta}\ell_T(\theta),$$
where $\Theta$ denotes the range in which the true parameter is located. The choice of $\Phi$ as in \eqref{eq:definitionG} allows for an easy interpretation of the parameters: The intensity has the form $\prod_{k=1}^qe^{\theta_kX_{n,ij}^{(k)}(t)},$ where $X_{n,ij}^{(k)}(t)$ denotes the $k$-th component of $X_{n,ij}(t)$. Hence, $\theta_k$ quantifies the impact of $X_{n,ij}^{(k)}(t)$ on the intensity, given that the remaining covariate vector stays the same.

The presence of the function $C_{n,ij}(t)$ enhances the modeling flexibility significantly. By choosing $C_{n,ij}$ the researcher who applies the model is able to fit the model only to a sub-network. This becomes necessary when it is natural to assume that certain pairs of actors behave fundamentally different from others. For instance, consider a social media network, and contrast pairs impacting each others activities in the network by exchanging messages regularly, with pairs consisting of actors from different social communities hardly interacting with each other. It is intuitive that these two pairs cannot be modeled accurately by the same model. In this situation, it would instead be advantageous to restrict to those pairs who have recently interacted, say, and fit the model only to interactions among them. On the other hand, the interaction intensity of course is dynamic, and thus different pairs might be included over time. This is achieved by the presence of the selector variables $C_{n,ij}(t).$ Also note that $C_{n,ij}(t)=0$ for $t\in[a,b]$ does not necessarily mean that there will be no interactions between $i$ and $j$ in $[a,b]$, it rather means that interactions which happen between $i$ and $j$ in $[a,b]$ are not fitted by our model.

Two things should be noted about these selectors: Firstly, choosing $(C_{n,ij})_{i,j}$ too conservatively is not a problem in the sense that we still estimate the `correct' parameters. For instance, suppose that $C_{n,ij}(t)$ and $\theta$ are the correct quantities to be used in the model  \eqref{eq:definitionG}. Assume now that $C^*_{n,ij}$ is a predictable selector that is more conservative than $C_{n,ij}$, i.e., $C_{n,ij}(t)=0$ implies $C_{n,ij}^*(t)=0$. Then, the observations are given by $t\mapsto N^*_{n,ij}(t):=\int_0^tC^*_{n,ij}(s)dN_{n,ij}(s)$ for $t\in[0,T]$. Clearly, $N^*_{n,ij}$ is a counting process comprising those jumps of $N_{n,ij}$ at which $C_{n,ij}^*$ equals 1. By assumption, $C_{nij}^*(t)C_{n,ij}(t)=C_{n,ij}^*(t)$ and hence the compensator of $N_{n,ij}^*$ is given by $C_{n,ij}^*(t)\exp(\theta^T(t)X_{n,ij}(t))$. Thus, the processes $N_{n,ij}^*$ can also be used to estimate $\theta$. On the other hand, using fewer data of course leads to a loss of information, and this might effect the efficiency of the parameter estimator (cf. Theorem \ref{thm:asymptotic_normality}). We do not attempt here to determine the best $C_{n,ij}$ in a data driven way.  Instead,  in real data applications, and motivated by this discussion, we attempt to choose $C_{n,ij}$ in a way that is not too liberal. This is illustrated in Section~\ref{data-analysis}, where we set $C_{n,ij}(t)$ equal to zero, if there was no event between $i$ and $j$ for a certain period $\Delta t = (t - \delta, t),$ for some $\delta > 0$, so that our model is only fitted to ``active" pairs. For pairs with low activity one may look for a different model. Thus a proper choice of $C_{n,ij}(t)$ allows to split up the analysis into different regimes.

Secondly, consider the social media example again. Suppose that the probability that $C_{n,ij}(t)=1$ is the same for all pairs $(i,j)$. We will assume that, with positive probability, links can form at any time $0 < t < T$, i.e., $\IP(C_{n,ij}(t) = 1)>0$. However, it is intuitive that even when more and more people connect to the platform, one particular actor will not acquire an unbounded number of friends. Instead it seems reasonable to assume that its number of friends is bounded. In such a case, the fraction of pairs $(i,j)$ for which $C_{n,ij}(t)=1$ should be of order $\frac{1}{n}$. Hence, it is natural to assume that $\IP(C_{n,ij}(t)=1)\to0$ as $n\to\infty$. This way sparsity in the observations can be captured in the model. For the asymptotic result (Theorem \ref{thm:asymptotic_normality}) to hold, the network cannot be too sparse (essentially an increase in the number of actors must lead to an increase in the number of active pairs). This will be made precise in the assumptions given in Section \ref{sec:assumptions}.

\subsection{Estimation in time-varying coefficient models}
\label{est:time-varying}
In time series applications, it turns out that powerful fits can be achieved by letting the time series parameters depend on time, and this is what we consider here as well. We will use the above model with $\theta$ in (\ref{eq:definitionG}) depending on $t$, or in other words, $\theta= \theta(t)$ is now a parameter function. \\

An estimator of this parameter function at a given point $t_0$ can be obtained by maximizing the following local likelihood function in $\mu$ which is obtained by localizing the likelihood \eqref{eq:likelihood} for a constant parameter at time $t_0$ by means of a kernel $K$
\begin{align} \label{eq:likeloc1}
\ell_T(\mu, t_0) &= \sum_{0 < t \le T}\frac{1}{h} K\left ( \frac{ t  - t_0}{h}\right ) \sum_{(i,j) \in L_n} \Delta N_{n,ij}(t)\,\log \lambda_{n,ij}(\mu,t)\, \\ \nonumber & \qquad - \int_0^T \sum_{(i,j) \in  L_n}\frac{1}{h} K\left ( \frac{t - t_0}{h}\right )  \lambda_{n,ij}(\mu,t) dt,
\end{align}
where $K$ is a kernel function (positive and integrating to one), and $h = h_n$ is the bandwidth. The corresponding local MLE is defined as 
\begin{align} \label{eq:MLloc}
\hat\theta(t_0)  = \arg\max_{\theta \in \Theta} \ell_T(\theta, t_0),
\end{align}
with $\Theta$ being the allowed range of the parameter function $\theta$. Recall that we use the following Cox-type form of the intensity:
\begin{align}\label{CoxHazard}
\lambda_{n,ij}(\theta,t) = C_{n,ij}(t) \exp\big\{\theta(t)^T X_{n,ij}(t)\big\}.
\end{align}

With this choice, the local log-likelihood can be written as (up to a term not depending on $\theta$):
\begin{align}\label{eq:loglik}
&\ell_T(\theta,t_0)  =\sum_{i,j=1}^n\int_0^T\frac{1}{h}K\left(\frac{t-t_0}{h}\right)\theta^TX_{n,ij}(t)dN_{n,ij}(t) \nonumber \\
 &  \qquad  -\sum_{i,j=1}^n\int_0^T\frac{1}{h}K\left(\frac{t-t_0}{h}\right)C_{n,ij}(t)\exp(\theta^TX_{n,ij}(t))dt.
\end{align}
The maximum likelihood estimator $\hat{\theta}_n(t_0)$ studied in this paper is defined as the maximizer of \eqref{eq:loglik} over $\theta\in\Theta$, where $\Theta\subseteq\IR^q$ is an appropriate parameter space. Denote by $L_n(t_0)$ the set of active edges, i.e., the set of all pairs $(i,j)$, such that, $C_{n,ij}(t_0)=1$. We denote by $|L_n(t_0)|$ the size of the set $L_n(t_0)$. Our main theoretical result, given below, says that for a given $t_0\in(0,T)$, the maximum likelihood estimator $\hat{\theta}_n(t_0)$ exists, is asymptotically consistent, and is asymptotically normal. 

To formulate our main result, the following technical assumptions are needed.

\subsection{Assumptions}
\label{sec:assumptions}

Our assumptions do not specify the dynamics of the covariates $X_{n,ij}(t)$ and of the censoring variable $C_{n,ij}(t)$. Instead, we assume that the stochastic behavior of these variables stabilizes for $n\to \infty$. Assumption (A1) is specific to our setting and it states our general understanding of the dynamics, while assumptions (A2), (A3) and (A5) are standard. Assumption (A4) can be found similarly in \cite{Perry:2013}. It guarantees that the covariates are well behaved. Finally, (A6) and (A7) specifically describe the dependence relations in our context. They quantify the idea that while the network grows the actors get further and further apart and hence influence each other less and less. In what follows, we state our assumptions and briefly discuss their meaning and the intuition behind them.

We denote derivatives by $\partial$. In particular, $\partial_t$ and $\partial_{t^2}$ refer to the first and second derivative with respect to $t$ respectively. If we derive with respect to a vector $\theta$, $\partial_{\theta}$ refers to the gradient and $\partial_{\theta^2}$ refers to the Hessian matrix.

{\bf (A1) Exchangeability } \\
{\em Assume that for every $n$ and any $s,t \in [t_0 - h, t_0 +h],$
\begin{enumerate}
\item the joint distribution of $(C_{n,ij}(t),X_{n,ij}(t))$ is identical for all pairs $(i,j)$,
\item the conditional distribution of the $q$-dimensional covariate $X_{n,ij}(t)$ given that $C_{n,ij}(s)=1$, has a density $f_{s,t}(y)$ with respect to a measure $\mu$ on $\IR^q$. This conditional distribution does not depend on $(i,j)$ and $n$. We use the shorthand notation $f_s$ for $f_{s,s}$.
\end{enumerate}}

The most restrictive part of (A1) is that the conditional distribution of $X_{n,ij}(t),$ given $C_{n,ij}(s)=1,$ does not depend on $i$, $j$. Observe that this holds if the array of $(C_{n,ij},X_{n,ij})_{i,j}$ is jointly exchangeable in $(i,j)$ for any fixed $n$. The additional assumption that the conditional distribution of $X_{n,ij}(t),$ given $C_{n,ij}(s)=1,$ does not change with $n$ is not very restrictive, because it is natural to assume that the distribution depends only on the local structure of the network (Recall the discussion in Section \ref{sec:basic_model} in which we assumed that a fixed vertex $i$ has only a bounded number of close interaction partners $j$ while the network grows). We make this additional assumption mainly to avoid stating lengthy technical assumptions allowing to interchange the order of differentiation and integration at several places in the proof.\\

We add some standard assumptions on the kernel. 

{\bf (A2) Kernel and Bandwidth}\\
{\em Suppose that the Kernel $K$ and the bandwidth $h$ fulfil the following conditions.
\begin{enumerate}
\item $K$ is  positive and supported on $[-1,1]$.
\item $\int_{-1} ^1 K(u)\mathrm{d}u=1$, $\int_{-1}^1 K(u)u\mathrm{d}u=0$ and $\max _{-1 \leq u \leq 1}  K(u) < \infty$.
\item As $n\to\infty$, $h=o(1)$, $l_n:=\frac{n(n-1)}{2}\IP(C_{n,12}(t_0)=1)\to\infty$ with $l_nh\to\infty$, and $l_nh^5=O(1)$.
\end{enumerate}}

Note that, $l_n$ is the effective sample size at time $t_0$, because $\frac{n(n-1)}{2}$ is the number of possible links between vertices, of which, in the average, we observe the fraction $\IP(C_{n,12}(t_0)=1)$. (For directed networks, one simply has to replace $\frac{n(n-1)}{2}$ by $n (n-1)$ in the definition of $l_n$.) With this in mind, the assumptions on the bandwidth are standard.

The next assumption states smoothness conditions on the parameter curve  $\theta_0$.

{\bf (A3) Smoothness of Parameter}\\
{\em Let $\Theta$ be the convex parameter space and $\theta_0:[0,T]\to\Theta$ the parameter function.
\begin{enumerate}
\item $\theta_0$ is twice continuously differentiable in a neighborhood of $t_0$.
\item The value $\theta_0(t_0)$ lies in the interior of $\Theta$.
\end{enumerate}}

We continue with some tail conditions on $f_s(y)$ and its derivatives. They are fulfilled if, e.g., the covariates are bounded. 

{\bf (A4) Moment Conditions} \\
{\em For $\mu$-almost all $y$ (where $\mu$ is as in (A1)), $s\mapsto f_s(y)$ is twice continuously differentiable. Let $U_h:=[t_0-h,t_0+h]$. There are bounded, open and convex neighborhoods $U$ of $t_0$ and $V\subseteq\Theta$ of $\theta_0(t_0)$ such that for all pairs $(i,j)$ and $(k,l)$ and $\tau:=\sup_{\theta\in V}\|\theta\|$,
\begin{align} 
&\int\sup_{s\in U}\left\{\left(1+\|y\|+\|y\|^2+\|y\|^3\right)\left|f_s(y)\right|+\left(1+\|y\|+\|y\|^2\right)\left|\partial_s f_s(y)\right|\right.\nonumber \\ 
&\quad\left.+\left(1+\|y\|\right)\left|\partial{s^2} f_s(y)\right|+\|y\|^2\cdot f_{s,t_0}(y)\right\}\cdot\exp(\tau\cdot\|y\|)\mathrm{d}\mu(y)<\infty, \label{eqa41} \\ 
&\sup_{s,t\in U_h}\IE\bigg(\|X_{n,ij}(s)\|^2\cdot\|X_{n,kl}(t)\|^2  \nonumber \\
&\quad\quad\quad\cdot \left.\left.e^{\tau(\|X_{n,ij}(s)\|+\|X_{n,kl}(t)\|)}\right|C_{n,ij}(t_0)=1,\,C_{n,kl}(t_0)=1\right)=O(1).  \label{eqa42}
\end{align}
For $k\in\{2,3\}$:
\begin{align} 
&\sup_{s\in U_h}\IE\left(\|X_{n,12}(s)\|^ke^{\tau\|X_{n,12}(s)\|}\bigg|C_{n,12}(s)=1,C_{n,12}(t_0)=0\right)=O(1), \label{eqa43} \\
& \IE\bigg(\sup_{s\in U_h}\left[\|X_{n,12}(s)\|+\|X_{n,12}(s)\|^2+\|X_{n,12}(s)\|^3+\|X_{n,12}(s)\|^4\right] \nonumber \\ 
&\hspace*{5cm}\cdot e^{\tau\|X_{n,12}(s)\|}\bigg|C_{n,12}(s)=1\bigg)<+\infty. \label{eqa44}
\end{align}
}

{\bf (A5) Identifiability}\\ {\em $\theta^TX_{n,12}(t_0)=0$ a.s. (w.r.t. $f_{t_0}$) implies that $\theta=0$.}\\

The following assumption addresses the asymptotic behavior of the distributions of the processes $C_{n,ij}(t)$. In particular, for $t$ in a neighborhood of $t_0$, we postulate asymptotic stability of the marginal distributions of these processes, and also a certain kind of asymptotic independence of $C_{n,ij}$ and $C_{n,kl}$ for $ |\{i,j\}\cap\{k,l\}|=0$. 

{\bf (A6) Asymptotic Uncorrelatedness I} \\ {\em For $w(u)= K(u)$ and $w(u) = K^2(u) / \int K^2(v) \mathrm{d}v $ it holds that 
\begin{equation} \label{eqa71} \int_{-1}^1w(u)\frac{\IP(C_{n,12}(t_0+uh)=1)}{\IP(C_{n,12}(t_0)=1)}\mathrm{d}u\to1\textrm{ as }n\to\infty.
\end{equation}
For
\begin{eqnarray*}
&&A_{n, ij, kl} \\
&:=&\int_{-1}^1\int_{-1}^1w(u)w(v)\frac{\IP(C_{n,ij}(t_0+uh)=1,C_{n,kl}(t_0+vh)=1)}{\IP(C_{n,12}(t_0)=1)^2}\mathrm{d}u\mathrm{d}v, 
\end{eqnarray*}
we assume that 
\begin{equation} \label{eqa72}
A_{n, ij, kl}=\left\{\begin{array}{ccc}
o(n^2) &\text{for}& |\{i,j\}\cap\{k,l\}|=2, \\
o(n) &\text{for}& |\{i,j\}\cap\{k,l\}|=1, \\
1+o(1) &\text{for}& |\{i,j\}\cap\{k,l\}|=0.
\end{array}\right.
\end{equation}
Furthermore, it holds that, as $n\to\infty$
\begin{equation} \label{eqa74}\int_0^T\frac{1}{h}K\left(\frac{s-t_0}{h}\right)\frac{\IP(C_{n,12}(t_0)=0,\,C_{n,12}(s)=1)}{\IP(C_{n,12}(t_0)=1)}ds=O(h),\end{equation}
and, for edges with $|\{i,j\}\cap\{k,l\}|\leq1$,
\begin{equation} \label{eqa73}
\frac{\IP(C_{n,ij}(t_0)=1,C_{n,kl}(t_0)=1)}{\big(\IP(C_{n,12}(t_0)=1)\big)^2} = O(1).\\[10pt]
\end{equation}
}
Note firstly that, due to the localization of our likelihood function, time dependence is present only locally around the target time $t_0$. Condition \eqref{eqa71} appears reasonable in our asymptotics where the size of the network increases: Consider, for instance a dynamic social media network, and assume, for example, that we consider data from a certain geographic region. One might assume that at night the number of active pairs, i.e. the pairs with $C_{n,ij}=1,$ is lower than during the day, and that there is a gradual decrease between 8pm and 11pm, say. This time window does not get narrower when $n$ increases and hence a slow change of the distribution over time seems to be a reasonable assumption. Assumption \eqref{eqa74} holds, for instance, in the following model: Assume that in the previous example communications between pairs end at $\delta_0:=8pm$ plus a certain random time $\delta_{n,ij}$, i.e., $C_{n,ij}(t)=\Ind(t\leq\delta_0+\delta_{n,ij})$. In this case, the ratio of probabilities in \eqref{eqa74} becomes
$$\frac{\IP(C_{n,ij}(t_0)=0,C_{n,kl}(s)=1)}{\IP(C_{n,12}(t_0)=1)}=\frac{\IP(\delta_{n,ij}\in[s-\delta_0,t_0-\delta_0))}{\IP(\delta_{n,ij}\geq t_0-\delta_0)}.$$
Since we are using a localizing kernel, the length of the interval $[s-\delta_0,t_0-\delta_0)$ is of the order $h$, and if $\delta_{n,ij}$ has a density, then \eqref{eqa74} holds.

If we assume that relabeling the vertices does not change the joint distribution of the whole process (i.e. if we assume exchangeability), then, the joint distribution of two pairs $(i,j)$ and $(k,l)$ depends only on $|\{i,j\}\cap\{k,l\}|$. It is thus natural to distinguish the three regimes $|\{i,j\}\cap\{k,l\}|\in\{0,1,2\}$. This pattern will appear again in the next Assumption (A7). Let us for the moment consider $C_{n,ij}$ that are constant over time. Then, in \eqref{eqa72}, the case $|\{i,j\}\cap\{k,l\}|=2$ is true because $\frac{\IP(C_{n,ij}=1,C_{n,kl}=1)}{\IP(C_{n,ij}=1)^2}=\IP(C_{n,12}=1)^{-1}= o(n^2)$ according to Assumption (A2).

We discuss the remaining cases for the uniform configuration model. In this model all vertices have (approximately) the same pre-defined degree $\kappa,$ and we assume that the $C_{n,ij}$ are created as follows: Equip each vertex $i=1,...,n$ with $\kappa$ edge stubs, and create edges by randomly pairing the stubs. After that, discard multiple edges and self-loops. If two vertices $i$ and $j$ are connected after this process, set $C_{n,ij}=1$. We use the same heuristics as e.g. in \cite{N10}, Chapter 13.2, to compute the probability of edges. Fix $i$ and $j$, then, for any fixed edge stub of $i$, there are $\kappa n-1$ stubs left to pair with, $\kappa$ of which belonging to vertex $j.$ Hence, the probability of connecting to $j$ is given by $\frac{\kappa^2}{\kappa n-1}$ as there are $\kappa$ edge stubs from $i$ as well. Thus, for large $n$, we obtain the following probabilities:
\begin{align*}
\IP(C_{n,12}=1)&\approx\frac{\kappa}{n} \\
\IP(C_{n,12}=1,\,C_{n,23}=1)&=\IP(C_{n,12}=1|C_{n,23}=1)\cdot\IP(C_{n,23}=1)\approx \frac{\kappa(\kappa-1)}{n^2} \\
\IP(C_{n,12}=1,\,C_{n,34}=1)&=\IP(C_{n,12}=1|C_{n,34}=1)\cdot\IP(C_{n,34}=1)\approx \frac{\kappa^2}{n^2}.
\end{align*}
We see now that also for $|\{i,j\}\cap\{k,l\}|\leq1$, the assumptions \eqref{eqa72} and \eqref{eqa73} hold.\\

The next assumption involves $\theta_{0,n},$ defined as the maximizer of
\begin{equation}
\label{eq:thetanulln}
\theta\mapsto\int_0^T\frac{1}{h}K\left(\frac{s-t_0}{h}\right)g(\theta,s)ds,
\end{equation}
where $g$ is defined in (A7). We show later that $\theta_{0,n}$ is uniquely defined, and that $\theta_{0,n}$ is close to $\theta_0(t_0)$ (see Lemma~ \ref{lem:theta0ntothetat0t0} and Proposition~\ref{prop:thetanulln}, respectively). Define furthermore
\begin{eqnarray} \tau_{n,ij}(\theta,s)&:=&X_{n,ij}(s)X_{n,ij}(s)^T\exp(\theta^TX_{n,ij}(s)), \label{eq:taunij} \\ \label{eq:gn_integral1}
g(\theta,t)&:=&\IE\left[\theta^TX_{n,ij}(t)\exp(\theta_0(t)^TX_{n,ij}(t))\right. \\ \nonumber
&& \hspace*{1.5cm}\qquad \left.-\exp(\theta^TX_{n,ij}(t))|C_{n,ij}(t)=1\right]
\\
\label{eq:gn_integral}
&=&\int_{\IR^q}\left(\theta^Tye^{\theta_0(t)^Ty}-e^{\theta^Ty}\right)f_t(y)\mathrm{d}\mu(y),
\\ \nonumber
 f_{n,1}(\theta,s,t|(i,j),(k,l))&:=&\IE(\tau_{n,ij}(\theta,s)\tau_{n,kl}(\theta,t)|C_{n,ij}(s)=1, C_{n,kl}(t)=1),  \\[5pt] \nonumber
f_{2}(\theta,t)&:=&\IE(\tau_{n,ij}(\theta,t)|C_{n,ij}(t)=1) = - \partial_{\theta^2} g(\theta,t), 
\\ \nonumber
r_{n,ij}^{(a)}(s)&:=&C_{n,ij}(s)X_{n,ij}^{(a)}(s)\left(e^{\theta_0(s)^TX_{n,ij}(s)}-e^{\theta_{0,n}^TX_{n,ij}(s)}\right) \\
&&\quad\quad\quad\quad\quad\quad-\partial_{\theta}g(\theta_{0,n},s), \nonumber
\end{eqnarray}
where $X_{n,ij}^{(a)}$ is the $a$-th entry of the vector $X_{n,ij}(s)\in\IR^q$. Note that, by Assumption (A1), $f_{2}$ and $g$ do not depend on $(i,j)$ and $n$. Keep also in mind that the fact that the covariates $X_{n,ij}(s)$ are vectors implies that $\tau_{n,ij}$, $f_{n,1}$ and $f_2$ are matrices and the expectations are to be understood element-wise.

{\bf (A7) Asymptotic Uncorrelatedness II} \\ {\em We assume that $f_{n,1}$ depends on $(i,j)$ and $(k,l)$ only through $|\{i,j\}\cap\{k,l\}|$. Moreover, we assume that, for all sequences $(\theta_n)_{n\in\IN}$ with $\theta_n  \to \theta_0(t_0)$ as $n\to\infty$ and all $u,v \in [-1,1],$ it holds that  $f_{n,1}(\theta_n, t_0 + u h, t_0 + vh, (i,j),(k,l))$ converges to a matrix that depends only on $ |\{i,j\}\cap\{k,l\}|$. We denote this limit by  $f_1(\theta_0(t_0), |\{i,j\}\cap\{k,l\}|)$, and assume that
\begin{equation} \label{eqa81} 
f_1(\theta_0(t_0), 0) =f_2(\theta_0(t_0),t_0)^2.
\end{equation}
For $r_{n,ij}^{(a)}(s),$ we assume that, with $\rho_{n,ijkl}^{(a)}(u,v):=r_{n,ij}^{(a)}(t_0+uh)r_{n,kl}^{(a)}(t_0+vh)$ and for $|\{i,j\}\cap\{k,l\}|=0$,
\begin{equation} 
\iint\limits_{[-1,1]^2}K(u)K(v)\IE\left(\rho_{n,ijkl}^{(a)}(u,v)|C_{n,ij}(t_0)=1,C_{n,kl}(t_0)=1\right)\mathrm{d}u\mathrm{d}v=o\big(\big(l_nh\big)^{-1}\big).\label{eqa82}
\end{equation}}
%
%
Assumption (A7) specifies in which sense the covariates are asymptotically uncorrelated. For motivating these assumptions build a graph $\mathcal{G}$ with vertices $1,...,n$ and $(i,j)$ being an edge if $C_{n,ij}(t_0)=1$. Denote by $d_{\mathcal{G}}$ the distance function between edges on $\mathcal{G}$ (that is, the number of edges on a shortest path, i.e., adjacent edges have distance 0). In the same heuristic as given after Assumption (A6), this graph becomes very large (asymptotics over the number of vertices) and sparse ($n$ vertices and of order $n$ edges), because every vertex is incident to at most $\kappa$ edges. In this scenario, the number of pairs of edges $e_1$ and $e_2$ for which $d_{\mathcal{G}}(e_1,e_2)=d$ is of order $(\kappa-1)^d\cdot n$, and there are of order $n^2$ many pairs of edges in total. Let now $A_{i,j}$ be arbitrary, centered random variables indexed by the edges of $\mathcal{G}$. We make the assumption that $A_{i,j}$ is influenced equally by all $A_{k,l}$ with $(k,l)$ being adjacent to $(i,j)$. In mathematical terms, we formulate this assumption as $\IE(A_{i,j}A_{k,l}|d_{\mathcal{G}}((i,j),(k,l))=d)\approx C\cdot\kappa^{-d}$. Then, we obtain for non-adjacent edges $(i,j)$ and $(k,l)$,
\begin{align*}
\IE(A_{i,j}A_{k,l})&=\sum_{d=1}^{\infty}\IP(d_{\mathcal{G}}((i,j),(k,l))=d)\cdot\IE(A_{i,j}A_{k,l}|d_{\mathcal{G}}((i,j),(k,l))=d) \\
&\approx\sum_{d=1}^{\infty}\frac{n(\kappa-1)^d}{n^2}C\cdot\kappa^{-d} \\
&=\frac{C}{n}(\kappa-1),
\end{align*}
which converges to zero after being multiplied with $l_nh\approx nh$ (in this case). Because, in \eqref{eqa81} and \eqref{eqa82}, we consider only expectations conditional on $C_{n,ij}(t)=1$, we can think of $A_{n,ij}$ being the random variables $\tau_{n,ij }$ (a centered version of it) or $r_{n,ij}^{(a)}$ and the expectations in the above heuristic are conditional expectations conditionally the respective conditions in \eqref{eqa81} and \eqref{eqa82}. This serves as motivation for these two assumptions. Moreover, unconditionally, $\tau_{n,ij}$ and $\tau_{n,kl}$ (and $r_{n,ij}$ and $r_{n,kl}$) do not need to be uncorrelated.

\subsection{The main asymptotic result}

\begin{theorem}
\label{thm:asymptotic_normality}
Suppose that Assumptions (A1)--(A7) hold for a point $t_0 \in (0,T)$. Then, with probability tending to one, the derivative of the local log-likelihood function $\ell_T(\theta,t_0)$ has a root $\hat{\theta}_n(t_0)$, satisfying, as $n\to\infty$
\begin{eqnarray} \label{eqtheo21} \sqrt{l_nh}\left(\hat{\theta}_n(t_0)-\theta_0(t_0)+ \frac 1 2 h^2\Sigma^{-1}v-h^2 B_n\right)\overset{\mathcal{D}}{\to} N\left(0,\int_{-1}^1K(u)^2\mathrm{d}u\ \Sigma^{-1}\right)\end{eqnarray}
with (denote by $\partial_t$ and $\partial_{t^2}$ the first and second derivative with respect to $t$ respectively, $\partial_{\theta}$ and $\partial_{\theta^2}$ are defined analogously, mind that $\theta$ is a vector)%
\begin{eqnarray*}
v&:=&\int_{-1}^1K(u)u^2\mathrm{d}u\cdot\partial_{\theta}\partial_{t^2}g(\theta_0(t_0),t_0), \\
\Sigma&:=&-\partial_{\theta^2}g(\theta_0(t_0),t_0), \\
\gamma_{n,ij}(s)&:=&(1-C_{n,ij}(t_0))C_{n,ij}(s), \\
B_n&:=&\frac{1}{l_n}\sum_{i,j=1}^n\int_0^T\frac{1}{h}K\left(\frac{s-t_0}{h}\right)\frac{\gamma_{n,ij}(s)}{h}\tau_{n,ij}(\theta_0(s),s)\theta_0'(t_0)\frac{t_0-s}{h}\mathrm{d}s,
\end{eqnarray*}
and $\tau_{n,ij}(\theta,s)=X_{n,ij}(s)X_{n,ij}(s)^T\exp(\theta^TX_{n,ij}(s))$ was defined in \eqref{eq:taunij}. If, in addition, $\frac{|L_n(t_0)|}{l_n}\overset{\IP}{\to}1$, then $l_n$ can be replaced by $|L_n(t_0)|$.
\end{theorem}

Recalling that (in the case of undirected networks) $l_n=\frac{n(n-1)}{2}\IP(C_{n,ij}(t_0)=1)$ is the effective sample size, i.e., the expected number of pairs relevant for estimation of $\theta_0(t_0)$, we see that Theorem \ref{thm:asymptotic_normality} is a classical asymptotic normality result up to the additional bias term $B_n,$ which we will discuss next. It holds that
\begin{eqnarray*}
&&\IE(|B_n|) \\
&\leq&\frac{1}{l_n}\sum_{i,j=1}^n\int_0^T\frac{1}{h}K\left(\frac{s-t_0}{h}\right)\IE\left(\frac{\gamma_{n,ij}(s)}{h} \|\tau_{n,ij}(\theta_0(s),s)\| \right)\|\theta_0'(t_0)\|\frac{|t_0-s|}{h}ds \\
&=&\int_0^T \frac 1 h K\left(\frac{s-t_0}{h}\right)\frac{\IP(C_{n,12}(t_0)=0,C_{n,12}(s)=1)}{h\IP(C_{n,12}(t_0)=1)} \frac{|t_0-s|}{h}ds\cdot\|\theta_0'(t_0)\|\\
&&\quad\quad\times\sup_{s\in U_h}\IE\left[\|\tau_{n,12}(\theta_0(s),s)\|\Big| C_{n,12}(s)=1, C_{n,12}(t_0)=1\right].
\end{eqnarray*}
This is of order O(1) by \eqref{eqa74} and \eqref{eqa43}. Hence, we get that $B_n=O_P(1)$. In general, the expectation does not converge to 0. Thus, in general we will have an additional bias term of order $h^2$. Let us suppose that  one can show $B_n - \IE(B_n)= o(1)$ by using some additional assumptions that bound the second moment of this term. We have that $ \IE(B_n)= o(1)$ if $\frac{\IP(C_{n,12}(t_0)=0,C_{n,12}(s)=1)}{h\IP(C_{n,12}(t_0)=1)}=o(h)$. This assumption can only hold  if only for a negligible minority of edges  the membership to the active set changes. In particular, for the extreme case of $C_{n,ij}$ being constant, we have $\gamma_{n,ij}\equiv0$ and $B_n=0$. Hence, the bias term $B_n$ is induced by a change in the sparsity of the active set.

\begin{remark}
If one is just interested in consistency, the assumptions can be weakened. In the proof of Theorem \ref{thm:asymptotic_normality} we need to prove the convergence of a certain quantity to a normal distribution. In order to establish consistency it is sufficient that this quantity converges to zero when being multiplied with a certain null-sequence. In order to show this weaker requirement we do not need the assumptions which impose rates on certain quantities. More precisely we do not need Assumptions (A6), \eqref{eqa74} and \eqref{eqa73} and (A7), \eqref{eqa82}. Moreover, the Assumptions (A4), \eqref{eqa42} and \eqref{eqa43} may be dropped.
\end{remark}

\subsection{Direct network modeling}
\label{sec:direct_network_modelling}
We consider the following general model for the link-based dynamics of a random network, using a multivariate continuous-time counting process approach allowing for arbitrary dependence structure between the links by applying the model for dynamic interactions twice: Once for the formation of new links and once for the deletion of existing links (this separation can also be found in \cite{Krivitsky:Handcock2014}). As before, let $V_n=\{1,...,n\}$, be the set of vertices and $L_n$ be the set of edges. Note that here we are considering undirected networks. But directed networks can be handled similarly. For a given link $(i,j)$, we let 
$$Z_{n,ij}(t) = 
\begin{cases}
1 & \text{if link from $i$ to $j$ is present at time $t$}\\
0 & \text{otherwise.}
\end{cases}
$$
Then
$$Z_n(t) = \big( Z_{n,ij}(t)\big)_{(i,j) \in L_n}$$

describes the random network, or, equivalently, the (upper half of the) adjacency matrix at time $t$. To describe the dynamics of the links over time we introduce two processes, $N^+_{n,ij}(t)$ and $N^-_{n,ij}(t),$ counting how often a link $(i,j)$ was added or deleted, respectively,  until time $t$. Formally, 
\begin{align*}
N^+_{n,ij}(t) &= \#\{s \le t:\, Z_{n,ij}(s)  - Z_{n,ij}(s-) = 1\},\\ N^-_{n,ij}(t) &= \#\{s\le t:\, Z_{n,ij}(s)  - Z_{n,ij}(s-) = -1\}.
\end{align*}

With these definitions, we can write, for $(i,j) \in L_n$,
\begin{align*}
Z_{n,ij}(t) = Z_{n,ij}(0) +  N^+_{n,ij}(t) - N^-_{n,ij}(t).
\end{align*}
For $v\in\{+,-\}$, the intensities of the counting processes $N^v_{n,ij}(t)$  are here defined as
\begin{align} \label{eq:Gdef}
\lambda^v_{n,ij}(\theta,t) &= \Phi^{v}_{n,ij}(\theta^v;(Z_n(s),X^v_{n,ij}(s)): s < t)
\end{align}
with 
\begin{align} 
\Phi^{+}_{n,ij}&(\theta^+;(Z_n(s),X^+_{n,ij}(s)): s < t) \nonumber \\
&= \gamma^{+}(\theta^+;(Z_n(s),X^+_{n,ij}(s)): s < t)\,\big(1 - Z_{n,ij}(t-)\big), \label{eq:Gdef1} \\
\Phi^{-}_{n,ij}&(\theta^-;(Z_n(s),X^-_{n,ij}(s)): s < t) \nonumber \\
&= \gamma^{-}(\theta^-;(Z_n(s),X^-_{n,ij}(s)): s < t)\,Z_{ij}(t-) \label{eq:Gdef2}
\end{align}
for some functions $\gamma^+$ and $\gamma^-$ respectively, where $\theta^+$ and $\theta^-$ are two different parameters, determining the addition and the deletion processes, respectively.  The vectors $X^v_{n,ij}(t)$ for $v\in\{+,-\}$ denote covariates that are assumed to be ${\cal F}_t$-predictable.  Note that this definition of the intensities makes sure that, as it should be, a link can only be added if it was not present immediately before, and similarly for the removal for a link.

These definitions of the intensities fit into the framework of Section \ref{est:time-varying} with  intensity function \eqref{CoxHazard}, when choosing $ \lambda_{n,ij}^v(\theta^v,t)=C^v_{n,ij}(t)\cdot\exp(\theta^v(t)^TX^v_{n,ij}(t))$ with $C_{n,ij}^v(t)$ being predictable $\{0,1\}$-valued processes that fulfill $C_{n,ij}^+(t)= 0$ if $Z_{n,ij}(t-)=1$, and $C_{n,ij}^-(t)= 0$ if $Z_{n,ij}(t-)=0$.
Again, as in Section \ref{est:time-varying}, we allow that the parameter is a function of time. To sum it up: The processes $N_{n,ij}^+$ are modeled with intensity $\lambda_{n,ij}^+(\theta_0^+,t)$ and the processes $N_{n,ij}^-$ are modeled with intensity function $\lambda_{n,ij}^-(\theta_0^-,t)$. Our model allows the covariates $X_{n,ij}^v$ and the true parameter functions $\theta_0^v$ to be different for $v='+'$ and $v='-'$. For estimating the parameters, we consider observations of the same type only, i.e., we will compute two maximum likelihood estimators: the estimator of $\theta_0^+(t)$ based on the processes $N_{n,ij}^+$, and the estimator for $\theta_0^-$ based on the processes $N_{n,ij}^-$. Both estimators can be treated as coming from an interaction based model and hence the theory from Section \ref{est:time-varying} can be applied.

\section{Application to Bike Data}
\label{data-analysis}
Here we illustrate the finite sample performance of our estimation procedure described above, by considering the Capital Bikeshare (CB) Performance Data, publicly available at {\tt http://www.capitalbikeshare.com/system-data}. The available data describes the usage of the CB-system at  Washington D.C. from 2010 to 2018. However, for computational reasons, and in order to keep the presentation concise, we will present two analyses of sub data sets. In the first analysis, we study the bike data from Jan 2012 to March 2016. In order to reduce computational complexity we aggregate the data over days. In this first analysis it is our main interest to predict the activity of an edge based on the past. In the second analysis we focus on a short period, April and May 2018, and we keep the time-continuous scale of bike events.

The available data set does not contain bike rentals over several days or below 60 seconds, and service rides are excluded. While the last two seem not restrictive, for bike rentals for more than a day, we note the pricing structure of CB: No matter which pass you buy (single ride, day pass, 3-day pass, 30-day pass or annual membership) the basic fee always includes only bike rides for less than 30min. If you keep a bike for longer than 30min extra fees apply (if you buy e.g. a day pass for \$8 you have an unlimited number of bike rides up to 30min, but if you keep the bike for 10h, you will have to pay \$142). We assume therefore that such long rentals are not the companies primary business and that they occur only very rarely. It is however possible to return a bike to a bike station and immediately re-rent it. Such that, in practice, if you want to make a way on the bike which takes you more than 30min, you can make an intermediate stop to avoid cost. Lastly, we have no information about the status of the stations themselves. In particular, we do not know if a station is empty or full.

It should be noted that, while we believe that this example serves as a serious and interesting illustration of our proposed method, it is not meant to be a full-fledged analysis of bike sharing performance. We would rather like to make the case for the potential of the model along with the estimation strategy presented in this paper by arguing that intuitively convincing results for the bike sharing data indicate that the model might also be beneficially used in more complex situations (i.e., without a strong a priori intuition).

Generally, in both analyses we consider the bike stations to be vertices in the network. Whenever somebody rents a bike at station $i$ and returns it at station $j$, we consider this an event from $i$ to $j$ and in this case we say that $(i,j)$ has been used. In May 2018 the CB network comprises 527 bike stations which were used at least once in April or May 2018. This results in a total of 277,202 possible directed combinations. Of these 277,202 directed combinations, only 39,722 connections have been used at least once in April 2018, and only 9,131 combinations have been used ten times or more. We conclude that the network is very sparse and that it is very challenging for a model to capture the entire biking behavior among all 527 bike stations. Thus, we restrict our analysis to some subset of pairs of bike station which (we assume) can be reasonably modeled by the same model. Consider an example: Bike stations in Alexandria and Derwood are 50km apart, and, on the other hand, some bike stations in downtown Washington are just separated by one block. Certainly, every now and then, somebody might take such a bike ride, but we cannot expect that our model will capture all these special cases. This restriction is realized by appropriate choice of the indicator functions $C_{n,ij}$ (see below for more details). Note here that, on a general level, $C_{n,ij}(t)=1$ means that the pair $(i,j)$ is, at time $t$, regarded as being part of the model in the sense that events from $i$ to $j$ can be captured by our model. On the other hand $C_{n,ij}(t)=0$ does simply mean that the pair $(i,j)$ does not belong to those edges of interest to us. While there might still be bike rides from $i$ to $j$, we do not attempt to model them by using our model.

\subsection{Analysis 1: January 2012 till March 2016}
In this part of the analysis we intend to model the biking activity on one weekday (Friday) based on the past. By biking activity on an edge $(i,j)$ we mean the number of bike rides between bike stations $i$ and $j$. Direction does not matter to us in this part of the study. The decision to only model one weekday was made mainly to reduce the computational burden. Comparing the results for different weekdays might be instructive, in particular comparing a regular working day and a day on the weekend. Note also that in the period of four and a quarter years, which we consider, eight Fridays were actually public holidays (thus being possibly more like a weekend than a weekday). They were Independence Day (2013, 2014), the Friday after Thanksgiving (2012, 2013) and Fridays during Christmas and New Year's Holidays. Notice, however, that the parameter function is allowed to change over time. Thus, we assume that the influence of public holidays is not causing any problems.

\begin{figure}[h]
\begin{subfigure}{0.5\textwidth}
\includegraphics[width=\linewidth]{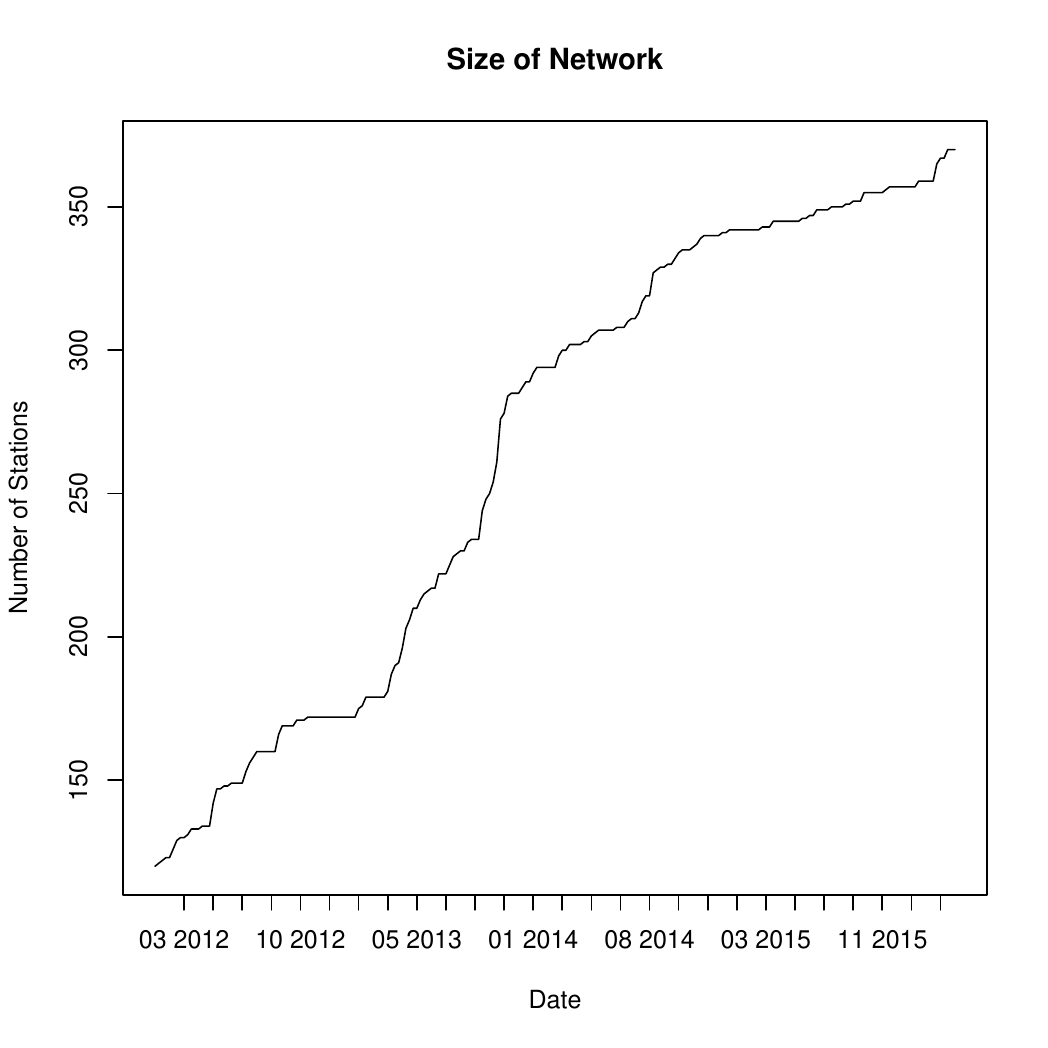}
\caption{Shows the number of available bike stations}
\label{fig:size_of_betwork}
\end{subfigure}%
\begin{subfigure}{0.5\textwidth}
\includegraphics[width=\linewidth]{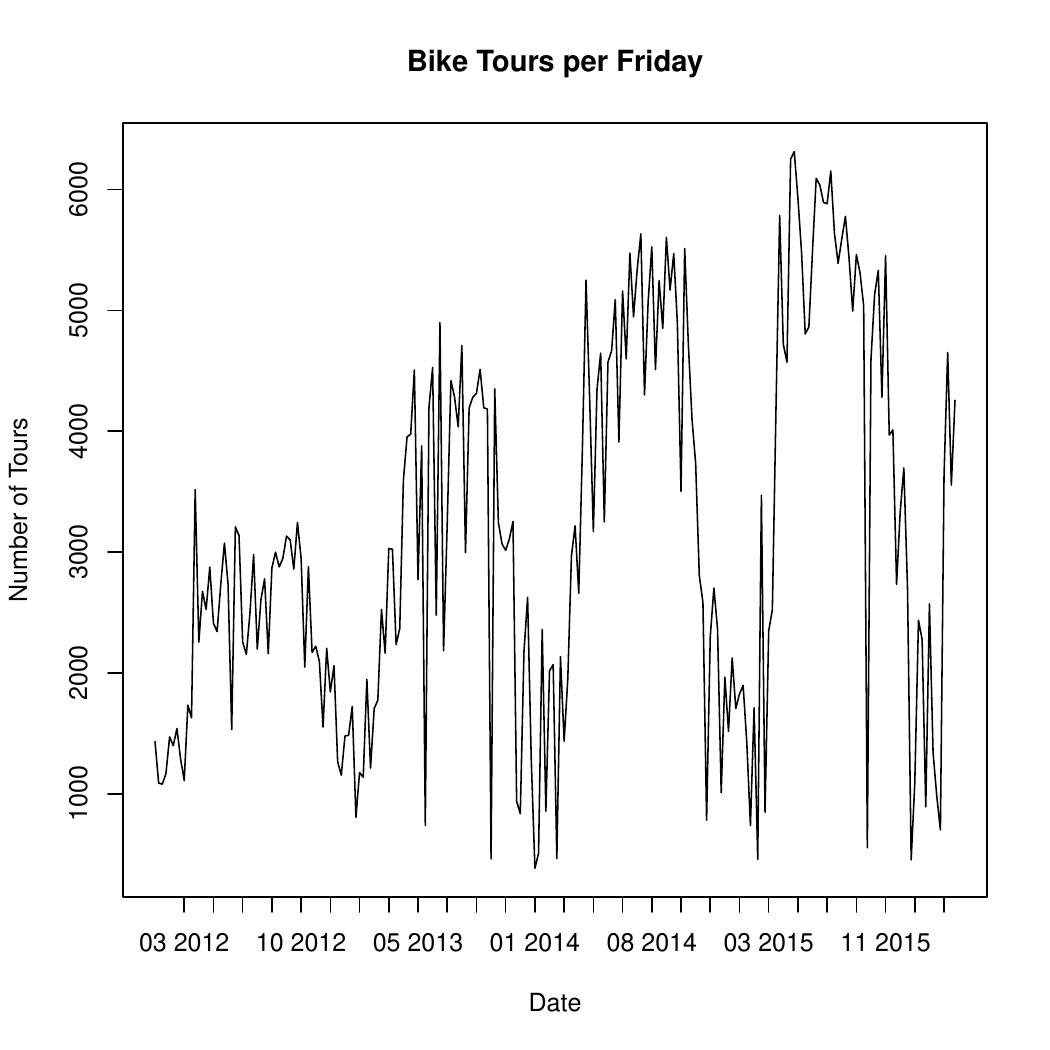}
\caption{Number of tours undertaken per Friday}
\label{fig:number_of_tours}
\end{subfigure}
\caption{Simple descriptive statistics of the bike data}
\label{fig:destat}
\end{figure}
   
Figure~\ref{fig:destat} shows some summary statistics of the data. In Figure \ref{fig:size_of_betwork}, we see the number of available bike stations, which is strongly increasing. Figure \ref{fig:number_of_tours} shows the number of bike tours on Fridays. An obvious seasonal periodicity is visible with low activity in winter. In order to reduce computational complexity to a minimum (fitting the model takes several minutes on standard laptop), we assume that the covariates change only at midnight and stay constant over the day. Furthermore, we estimate the time-varying parameter function $\theta$ only for one time point per day, namely 12pm noon. The next paragraph contains more details.

Since we do not consider any asymptotics here, we omit the index $n$. Time $t$ is measured in hours of consecutive Fridays. So, if $k$ is the current week, and $r$ is the time on Friday (in 24h), then $t:=(k-1)\cdot24+r$. Thus, with $r_t:=(t \mod 24),$ the quantity $k_t:=\frac{t-r_t}{24}+1$ gives the week the time point $t$ falls into. The processes $N_{i,j}(t)$, counting the number of tours between $i$ and $j$ on Fridays, are modeled as counting processes with intensities $\lambda_{i,j}(\theta(t),t):=\alpha(t)\exp(\theta^TX_{i,j}(k_t))\cdot C_{i,j}(k_t)$. The covariate vector $X_{i,j}(k_t)$ and the censoring indicator $C_{i,j}(k_t)$ will be defined later. Note that they both only depend on $k_t$, i.e. on the current week, and not on the actual time on the Friday under consideration. The function $\alpha$ is 24 periodic and integrates to  one over a period, i.e., $\alpha(t)=\alpha(t+24)$ and $\int_{t}^{t+24}\alpha(s)ds=1$. The role of the (unobservable) function $\alpha$ is to argue that discretizing the biking activity is not introducing a bias even when the biking activity varies over the day. Suppose now, that our target is the estimation of the parameter vector $\theta(t_0)$ with $t_0 = (k_{t_0}-1)24 + r_0$ and $r_0= 12$, say. We choose a piecewise constant kernel $K$ with $K((24k +x) /h) =K(24k/h),$ for all $k\in \mathbb{N}$ and $0\leq x<24$. Substituting in these choices of the intensity and the kernel to the log-likelihood \eqref{eq:likeloc1}, we see that our maximum likelihood estimator maximizes the function
\begin{eqnarray*}
&&\theta\mapsto\sum_{k=0}^{k_T}K_\kappa(k-k_{t_0})\theta^TX_{i,j}(k)\int_{k\cdot24}^{(k+1)\cdot24}dN_{i,j}(t) \\
&&\quad\quad\quad\quad-\sum_{k=0}^{k_T}K_\kappa(k-k_{t_0})\exp(\theta^TX_{i,j}(k))C_{i,j}(k),
\end{eqnarray*}
where $\int_{k\cdot24}^{(k+1)\cdot24}dN_{i,j}(t)$ gives the number of tours between $i$ and $j$ on the Friday in week $k,$ and where $K_\kappa(k)= K(k/\kappa)$ with $\kappa= h/24$. In our empirical analysis, we chose $K_\kappa(k) $ as triangle weights with support $\{-\kappa,...,\kappa\}$ and considered only integer choices of the bandwidth $\kappa$. The bandwidth choice is discussed at the end of this section. Note that due to this discretization we essentially obtain a sequence of generalized linear Poisson models with time varying parameters. In the second analysis, in Section \ref{subsec:second_analysis}, we use the full time-continuous potential of the model for dynamic interaction networks.

We explain now the choice of our covariate vector $X_{i,j}$. Denote by $\Delta_{i,j}(k,d)$ the number of tours between $i$ and $j$ on day $d$ in week $k$, where $d=4$ means Monday and $d=7$ refers to Thursday (for us the week starts on Fridays, i.e. Friday is $d=1$). For $r\in(0,1)$, we encode the activity between $i$ and $j$ in week $k$ as  $A_{i,j,k}=(1-r)\sum_{d=4}^7r^{7-d}\Delta_{i,j}(k,d)$ (mind the limits of the summation - Fridays are not included).
In our simulations, we chose $r=0.8$ (this choice is somewhat arbitrary, and a full study of the data would include investigating the sensitivity of the parameter estimate on the choice of $r$ as well as a data driven choice. We do not attempt to do this here). We construct a network $G(k),$  for every week $k,$ by connecting $i$ and $j$, if and only if, there was at least one tour on the Friday in that week. We denote by $I_{i,j,k}$ the number of common neighbors of $i$ and $j$ in the graph $G(k)$. We let $d_{i,k}$ be the degree of node $i$ in $G(k)$, $T_{i,j,k}$ the number of tours between $i$ and $j$ on the Friday in the $k$-th week, and $T_{i,j,k,k-1} = (T_{i,j,k}+T_{i,j,k-1})/2$ the average number of tours on the two Fridays in weeks $k$ and $k-1$.
Finally we collect everything in the covariate vector:
\begin{eqnarray*}
X_{i,j}(k)&:=&\Bigg(1,A_{i,j,k-1},I_{i,j,k-1},\max(d_{i,k-1},d_{j,k-1}), \\
&&\quad\quad\quad\quad\quad T_{i,j,k-1,k-2},\Ind(T_{i,j,k-1,k-2}=0)\Bigg)^T.
\end{eqnarray*}
The censoring indicator function $C_{i,j}$ is defined to be equal to zero, if there was no tour between stations $i$ and $j$ in the last four weeks. In other words, we attempt to model only those connections which are used regularly in the considered time frame. In summary, we estimate a total of six parameter curves, corresponding to the effects of six covariates in our model:
\begin{itemize}
\item $\theta_1(t)\quad\triangleq\quad$ baseline 
\item $\theta_2(t)\quad\triangleq\quad$ activity between stations on previous week-days
\item $\theta_3(t)\quad\triangleq\quad$ common neighbors of stations 
\item $\theta_4(t)\quad\triangleq\quad$ popularity of station, measured by degrees
\item $\theta_5(t)\quad\triangleq\quad$ activity between stations on two previous Fridays
\item $\theta_6(t)\quad\triangleq\quad$ inactivity between stations on two previous Fridays
\end{itemize}

Figure \ref{fig:avg_covariates} gives some impression of the distribution of the covariates over time. We only consider covariates 2-5 (that is the entries 2-5 of the covariate vector $X_{i,j}(k)$). The first covariate is always equal to one, while the last one is an indicator and thus either zero or one, and so we do not give plots for them. Each panel in Figure \ref{fig:avg_covariates} shows the 50\%, 80\%, 90\% and 99\% quantiles of the respective covariate. We see that the quantiles mainly stay on a moderate level with some larger values in between. This effect is more pronounced for the activity based covariates 2 and 5.

\begin{figure}
\begin{center}
\includegraphics[width=0.9\textheight,angle=-90]{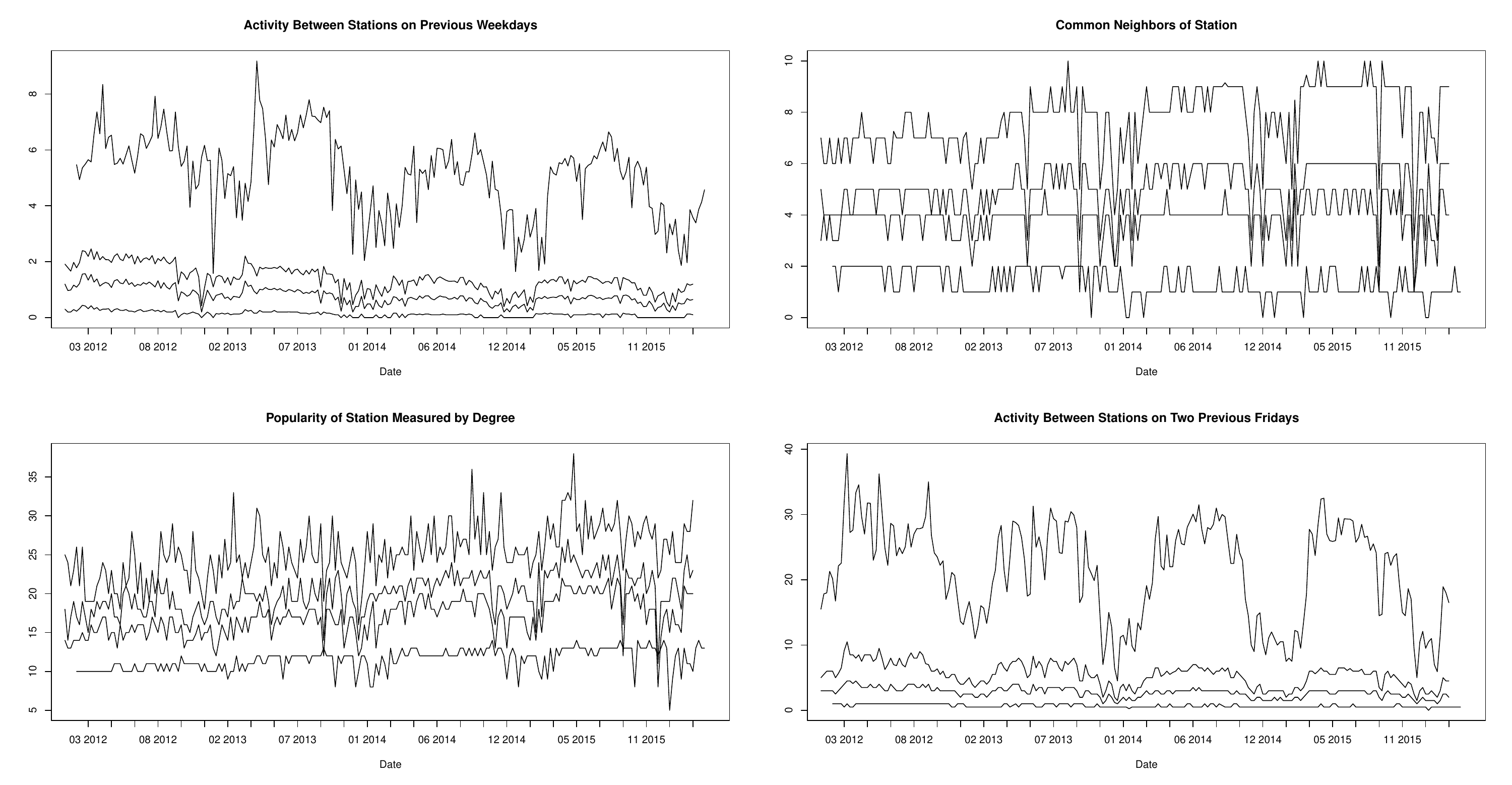}
\end{center}
\caption{Shown are the per day quantiles of four covariates. The curves correspond to the 50\%, 80\%, 90\%, 99\% quantiles (from bottom to top).}
\label{fig:avg_covariates}
\end{figure}

\begin{figure}
\begin{subfigure}{\textwidth}
\centering
\includegraphics[width=0.7\linewidth]{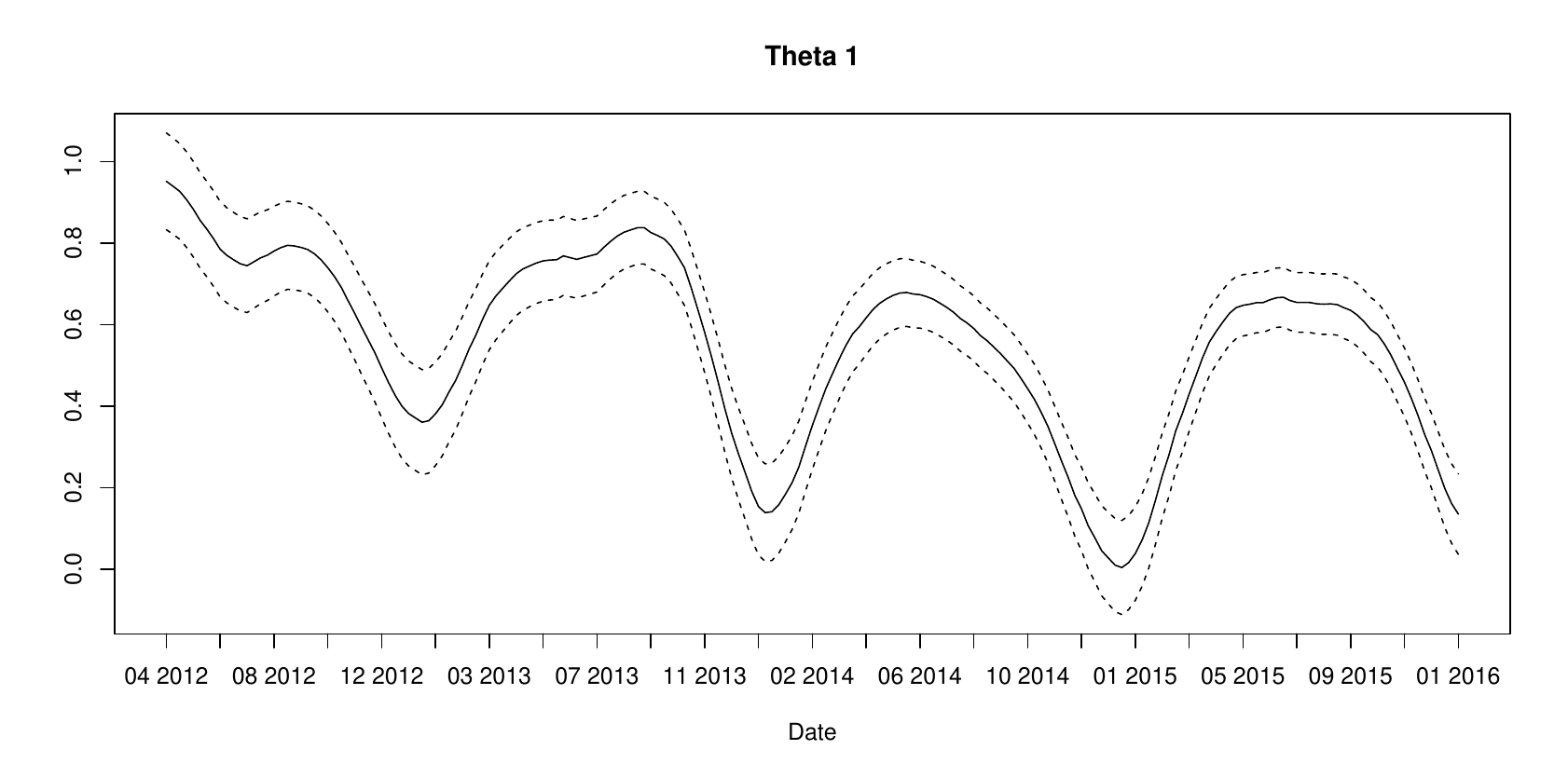}
\caption{Baseline weight}
\label{fig:theta1}
\end{subfigure}
\begin{subfigure}{\textwidth}
\centering
\includegraphics[width=0.7\linewidth]{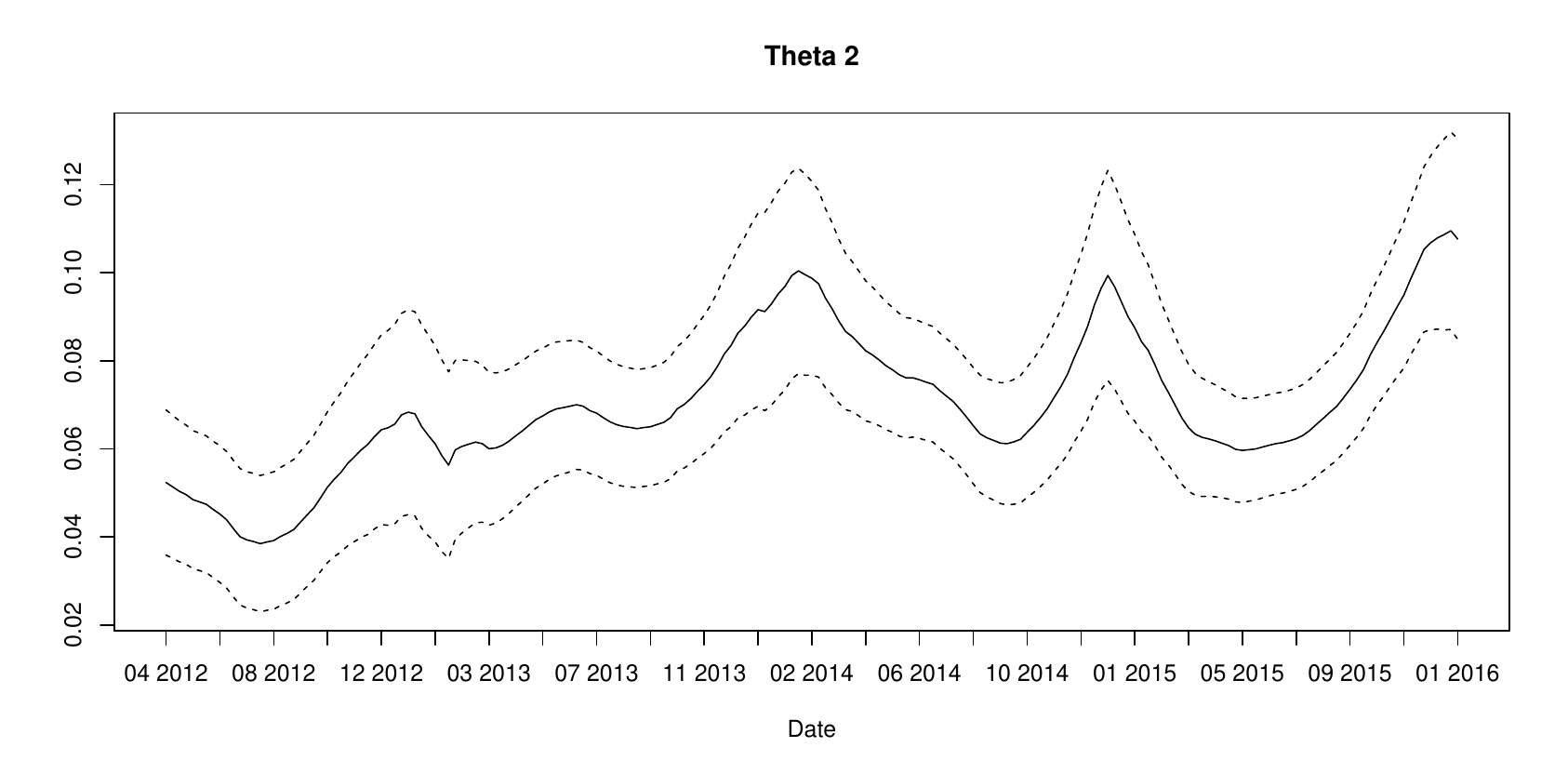}
\caption{Previous week-day activity between stations weight}
\label{fig:theta2}
\end{subfigure}
\begin{subfigure}{\textwidth}
\centering
\includegraphics[width=0.7\linewidth]{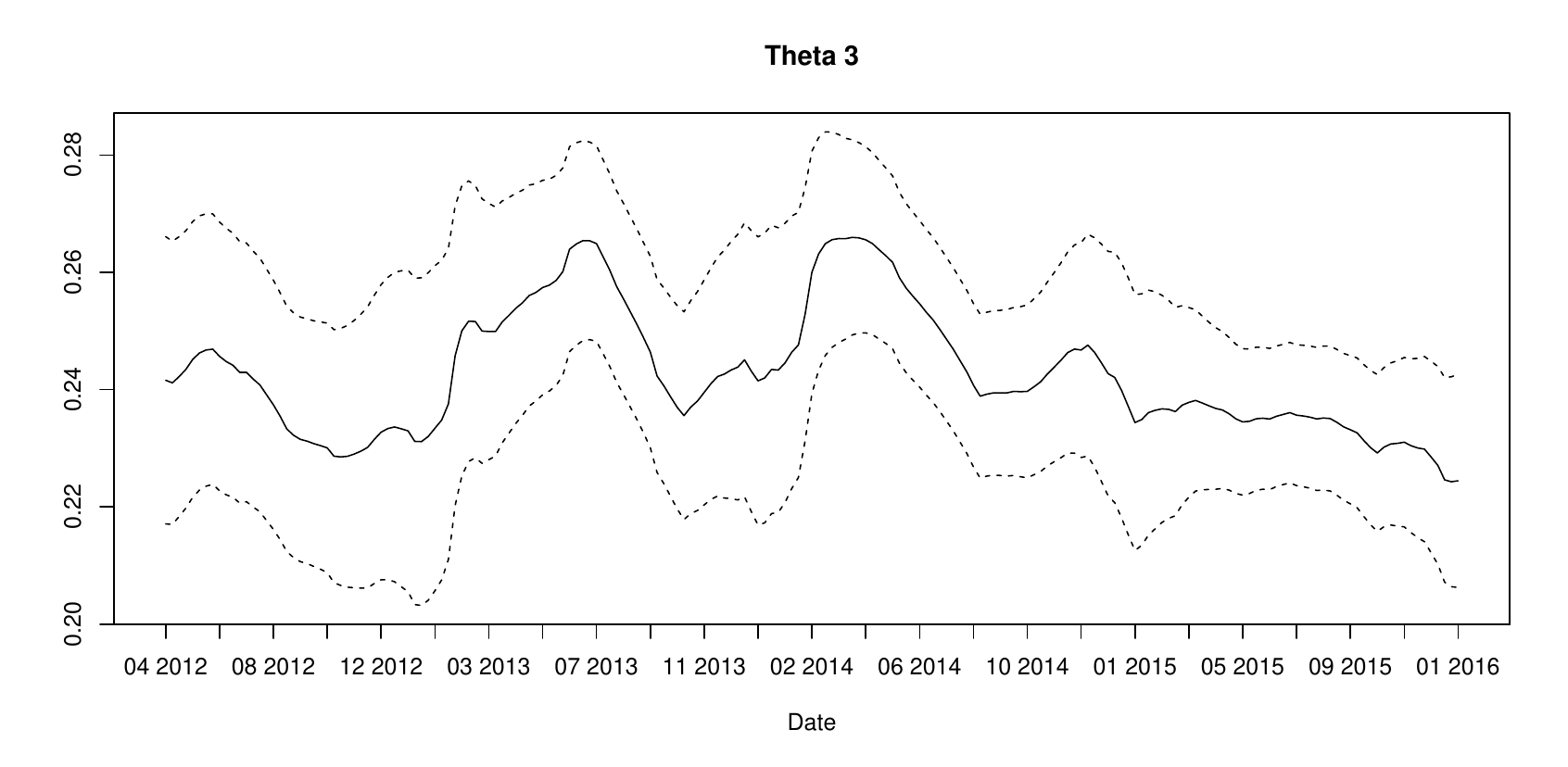}
\caption{Common neighbors weight}
\label{fig:theta3}
\end{subfigure}
\caption{Estimates of $\theta_1(t)$, $\theta_2(t)$ and $\theta_3(t)$ (solid curves). The dotted curves indicate 99\% pointwise confidence regions (plus minus 2.58 times the asymptotic standard deviation). }
\label{fig:theta123}
\end{figure}

\begin{figure}
\begin{subfigure}{\textwidth}
\centering
\includegraphics[width=0.7\linewidth]{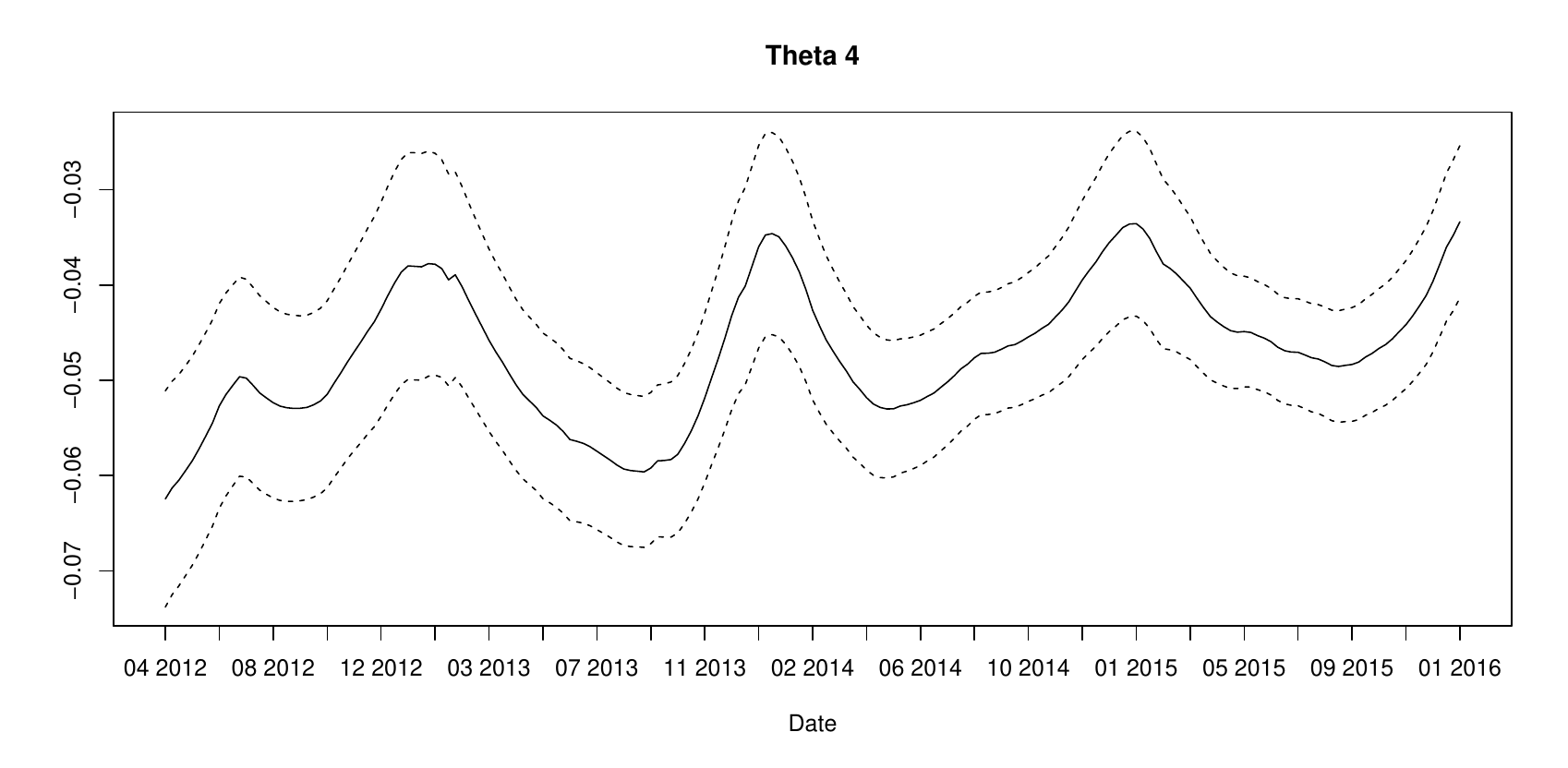}
\caption{Maximal degree weight}
\label{fig:theta4}
\end{subfigure}
\begin{subfigure}{\textwidth}
\centering
\includegraphics[width=0.7\linewidth]{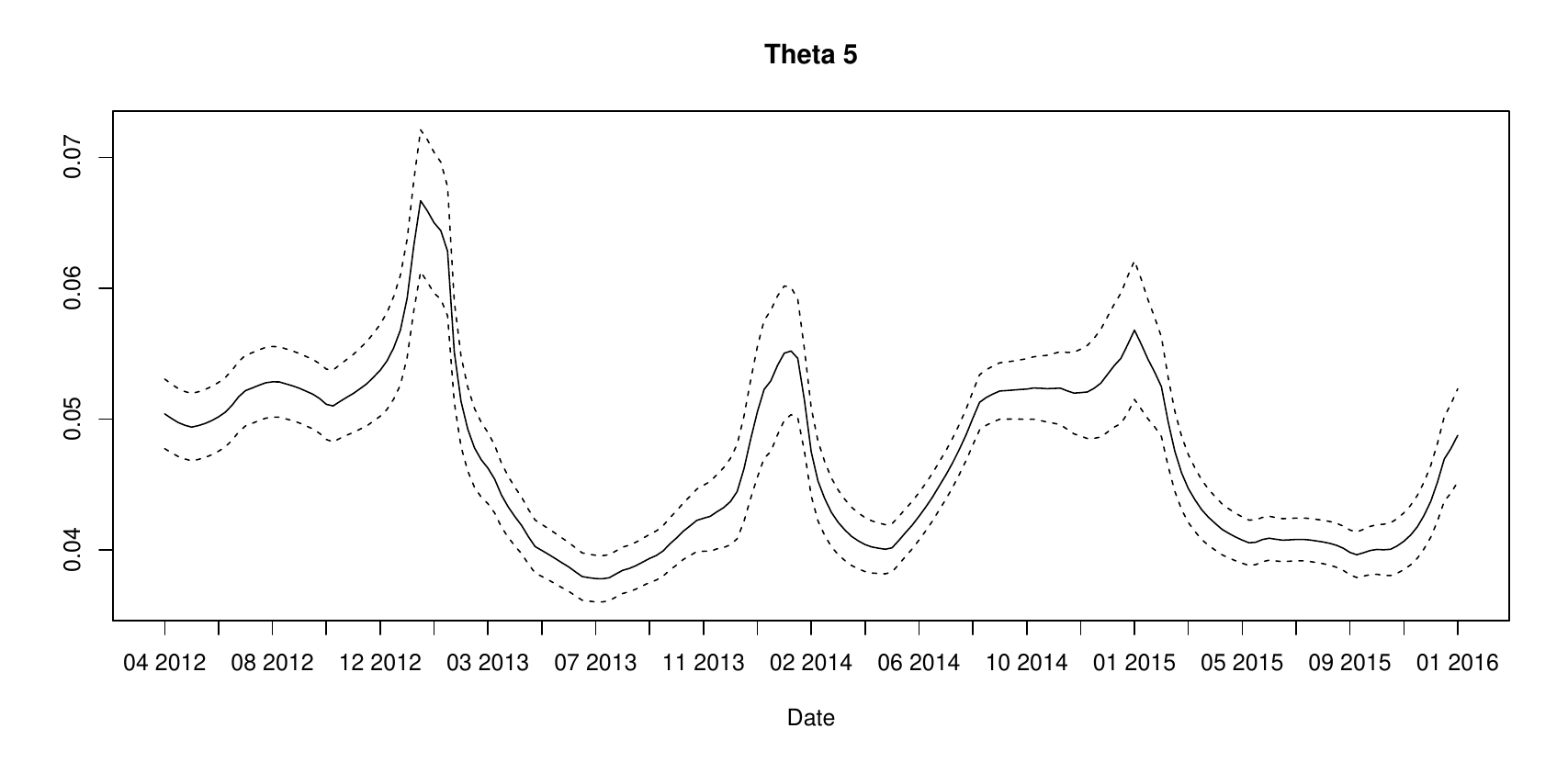}
\caption{Previous Friday activity between stations weight}
\label{fig:theta5}
\end{subfigure}
\begin{subfigure}{\textwidth}
\centering
\includegraphics[width=0.7\linewidth]{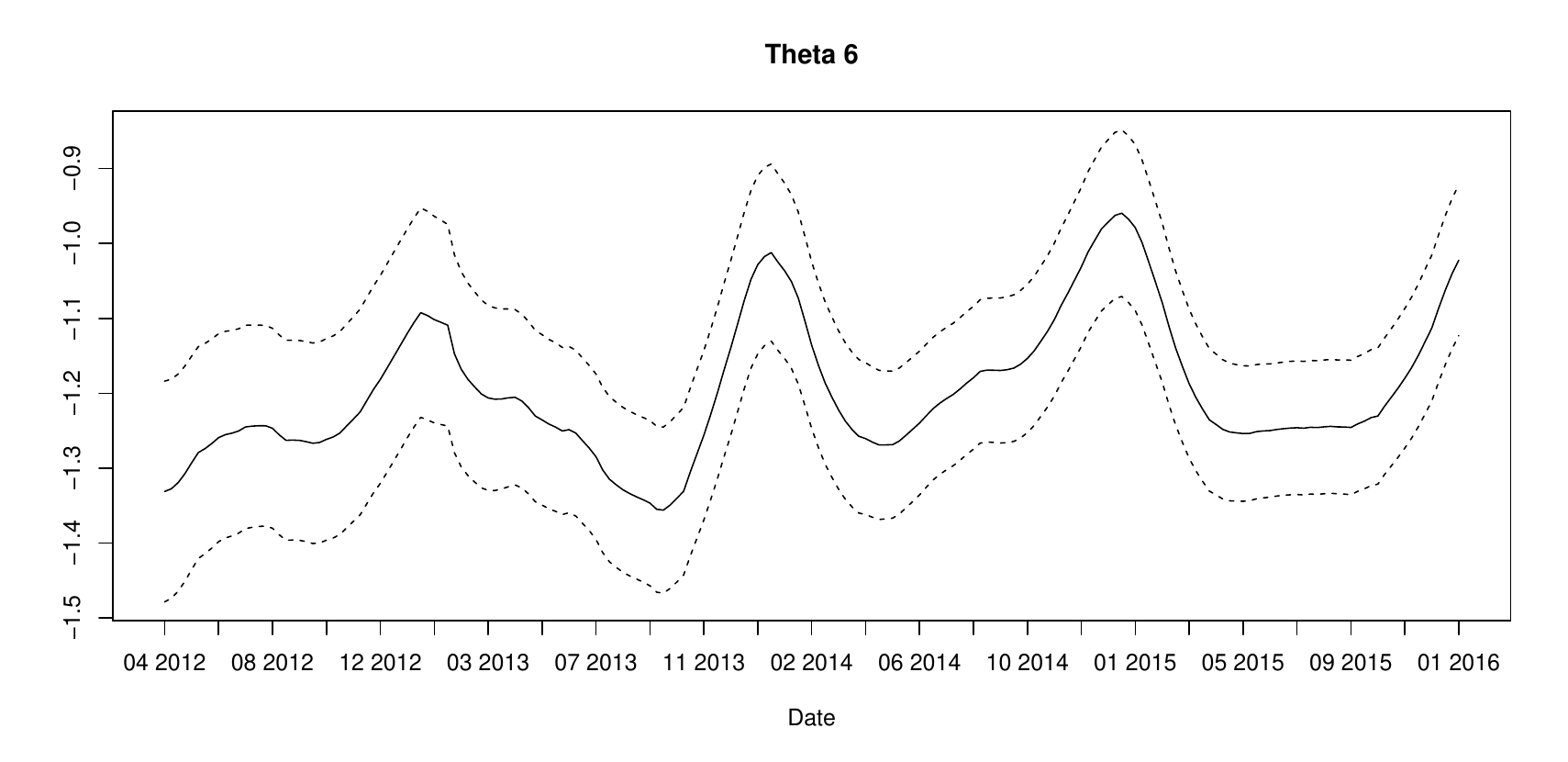}
\caption{Inactivity on previous Fridays weight}
\label{fig:theta6}
\end{subfigure}
\caption{Estimates of $\theta_4(t)$, $\theta_5(t)$ and $\theta_6(t)$ (solid curves). The dotted curves indicate 99\% pointwise confidence regions (plus minus 2.58 times the asymptotic standard deviation).}
\label{fig:theta456}
\end{figure}

The resulting estimated parameter curves are shown in Figures \ref{fig:theta123} and \ref{fig:theta456}. All calculations have been executed on the BwForCluster (cf.\ Acknowledgement). Since we expect the parameter function to vary slowly, we used the last estimated value as initial value for the estimation at the next point in time. In all six parameter curves in Figures \ref{fig:theta123} and \ref{fig:theta456}, the solid curves show the estimated parameter curve and the dotted curves indicate approximative 95\% point-wise confidence sets, which we obtained by omitting the bias in Theorem \ref{thm:asymptotic_normality} and approximating $\Sigma$ at time $t_0$ by $\frac{1}{|L_n(t_0)|}\partial_{\theta^2}\ell_T(\hat{\theta}_n(t_0),t_0)$, where $|L_n(t_0)|$ is the number of active edges at time $t_0$. In all plots we observe a clearly visible seasonality. Looking at Figure \ref{fig:theta2}, we see that activity during the week (Monday to Thursday) is more important during the winter months than in the summer. A plausible interpretation for this might be that the opportunist cyclists might be less active in winter because of the colder weather. So only those keep using a bike, who ride the same tour every day regardless of the weather. This makes the activity during the week a better predictor.

Figure \ref{fig:theta3} shows that the number of common neighbors always has a significant positive effect on the hazard. This reflects the empirical finding that observed networks cluster more than totally random networks (e.g.\ see  \cite{J08}).

The influence of the popularity of the involved bike stations is investigated in Figure \ref{fig:theta4} (measured by the degree of the bike station). Interestingly, it always has a significant negative impact. The size of the impact is higher in the summer months, which again supports the hypothesis that in summer the behavior of the network as a whole appears \emph{more random} than in winter. But still, the negative impact is a bit unforeseen. This finding can be interpreted as the observed network having no hubs. Another reason for this effect might be, that stations can only host a fixed number of bikes: If a station $i$ is empty, no new neighbors can be formed. A similar saturation effect happens if a lot of bikes arrive at station $i$.
Moreover, it is plausible that effects caused by the degrees are already included in \ref{fig:theta2}, as well as in Figure \ref{fig:theta5}. They show the effect of the bike rides on the days immediately preceding the current Friday, and the effect of the average number of bike tours on the last two Fridays, respectively. In Figure \ref{fig:theta5}, we observe a similar behavior as in Figure \ref{fig:theta2} (even more pronounced): In summer the predictive power of the tours on the last two Fridays is significantly lower than in winter, underpinning the theory that the destinations in summer tend to be based on more spontaneous decisions.
Finally, in Figure \ref{fig:theta6}, we observe that no bike tours on the last two Fridays between a given pair of stations always has a significant negative impact on the hazard. Again a very plausible finding.

We are currently working on testing whether the parameter functions depend on time, i.e., on testing for constancy of the parameter functions. For a complete data analysis it would then be interesting to add time as a covariate (or time dependent covariates), and to see if the parameter functions show always a significant time-dependency.

{\sc Modeling other network characteristics.} In stochastic network analysis, a central strand of research is concerned with the question of whether characteristics observed in real networks can be adequately mimicked by stochastic network models. Important characteristics are degree distribution, clustering coefficient and  diameter (these and other characteristics can be found in \cite{J08} Chapter 2.2, we define them also in the appendix). As in \cite{ZAAL14}, Chapter 4, we compare these three characteristics with a typical network produced by our model. In order to see how much our fitted model is able to capture these characteristics, we have simulated 3840\footnote{We chose to simulate 3840 networks, because we had $32$ cores available, and on each of the cores we ran 120 predictions, which could be done in reasonable time.} networks corresponding to three randomly chosen days, by using the network model with the fitted parameters of the corresponding day. We then compared the simulated three characteristics on these three days to the ones observed in the networks (this way of assessing the goodness of fit is also used in \cite{Hunter2008}). Here, we present the results for the degree distribution on 7th December 2012. The other results are reported in the appendix. 

In our analysis, we consider fitting sub-networks defined by the popularity of their edges: For given values  $0 \le l_1 < l_2 \le \infty$, the network is constructed by placing an edge between a pair of nodes $(i,j),$ if the number of tours between $i$ and $j$ falls between $l_1$ and $l_2$. Different ranges of $l_1$ and $l_2$ are considered. The idea is to consider the network of low frequented tours (for $l_1=1$ and $l_2=3$) up to the network of highly frequented tours (for $l_1=10$ and $l_2=\infty$).

Figure \ref{fig:degrees49} shows the simulated degree distributions for six different choices of $l_1$ and $l_2$. The dotted lines indicate 10\% and 90\% quantiles of the simulated graphs, and the solid line shows the true degree distribution. We see that, in all six cases, the approximation is reasonable accurate, in particular if one takes into account  that we did not specifically aim at reproducing the degree distributions.  The plots show that the largest degree of the simulated networks and the observed network lie not too far from each other, and the overall shape of the degree distribution is captured well. It should also be noted that we used only six covariates, whereas in other related empirical work much higher dimensional models have been used, see e.g.\ the discussions in \cite{Perry:2013}.

\begin{figure}
\centering
\begin{subfigure}{0.35\textwidth}
\centering
\includegraphics[height=0.25\textheight,width=\linewidth]{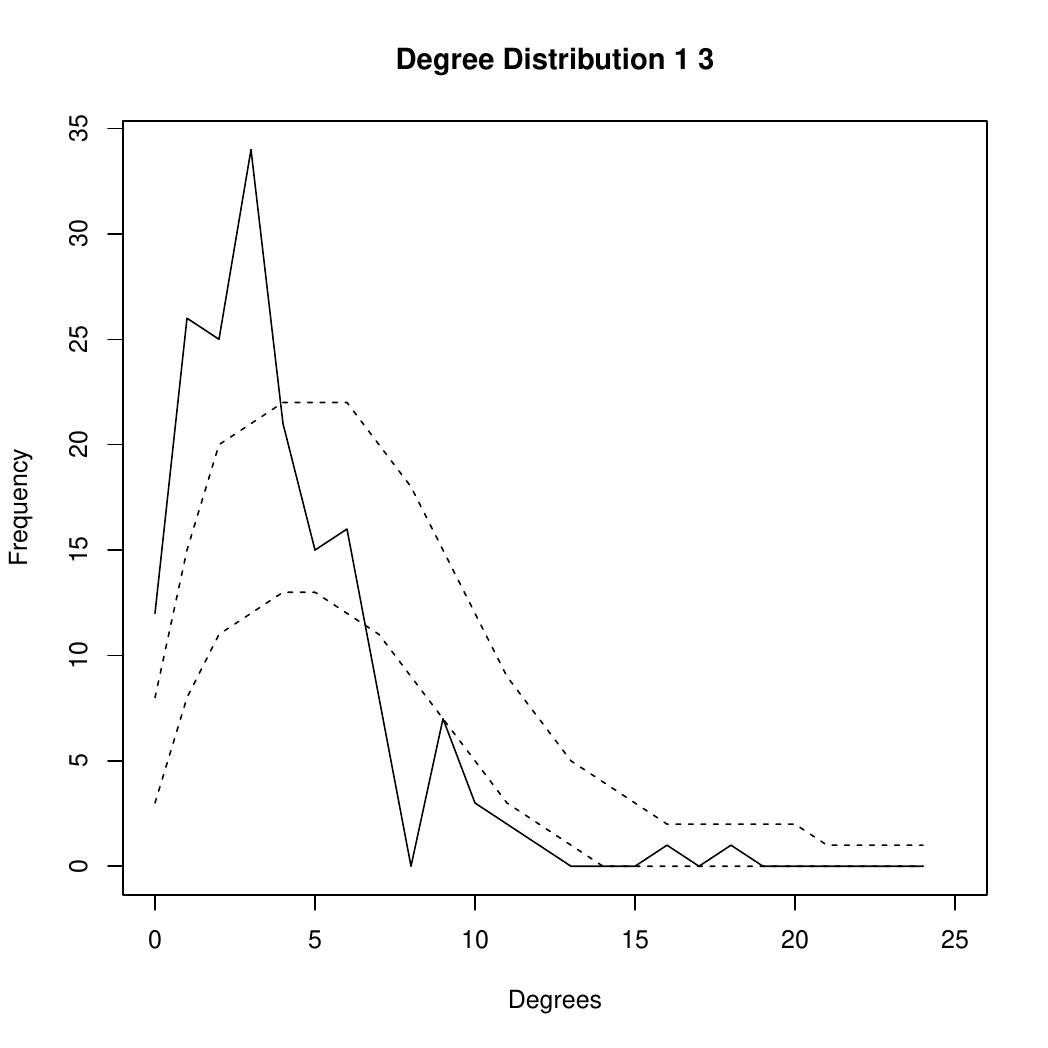}
\begin{minipage}{1\textwidth}
\caption{Tour frequency between one and three}
\end{minipage}
\label{fig:degrees49_13}
\end{subfigure}%
\hspace*{0.5cm}
\begin{subfigure}{0.35\textwidth}
\centering
\includegraphics[height=0.25\textheight,width=\linewidth]{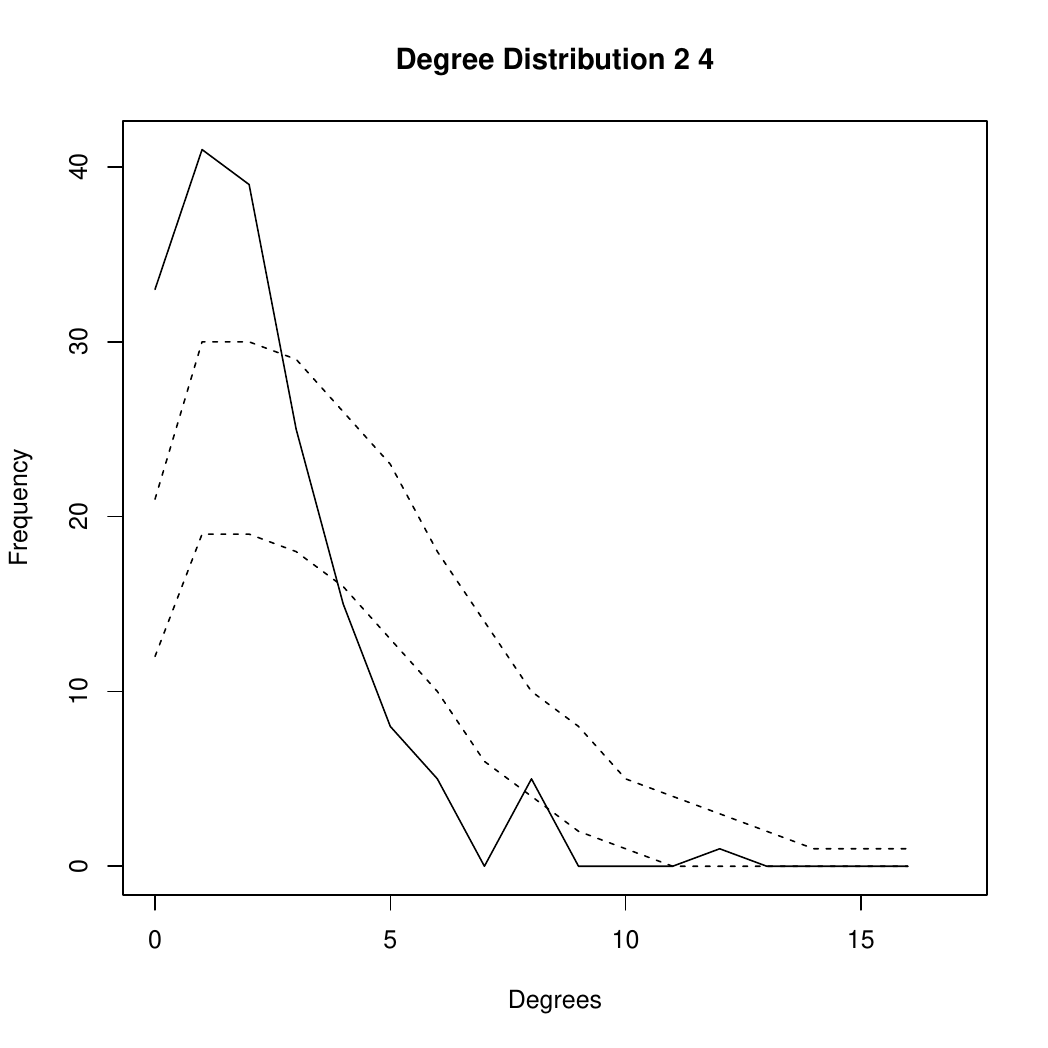}
\begin{minipage}{1\textwidth}
\caption{Tour frequency between two and four}
\end{minipage}
\label{fig:degrees49_24}
\end{subfigure}

\begin{subfigure}{0.35\textwidth}
\centering
\includegraphics[height=0.25\textheight,width=\linewidth]{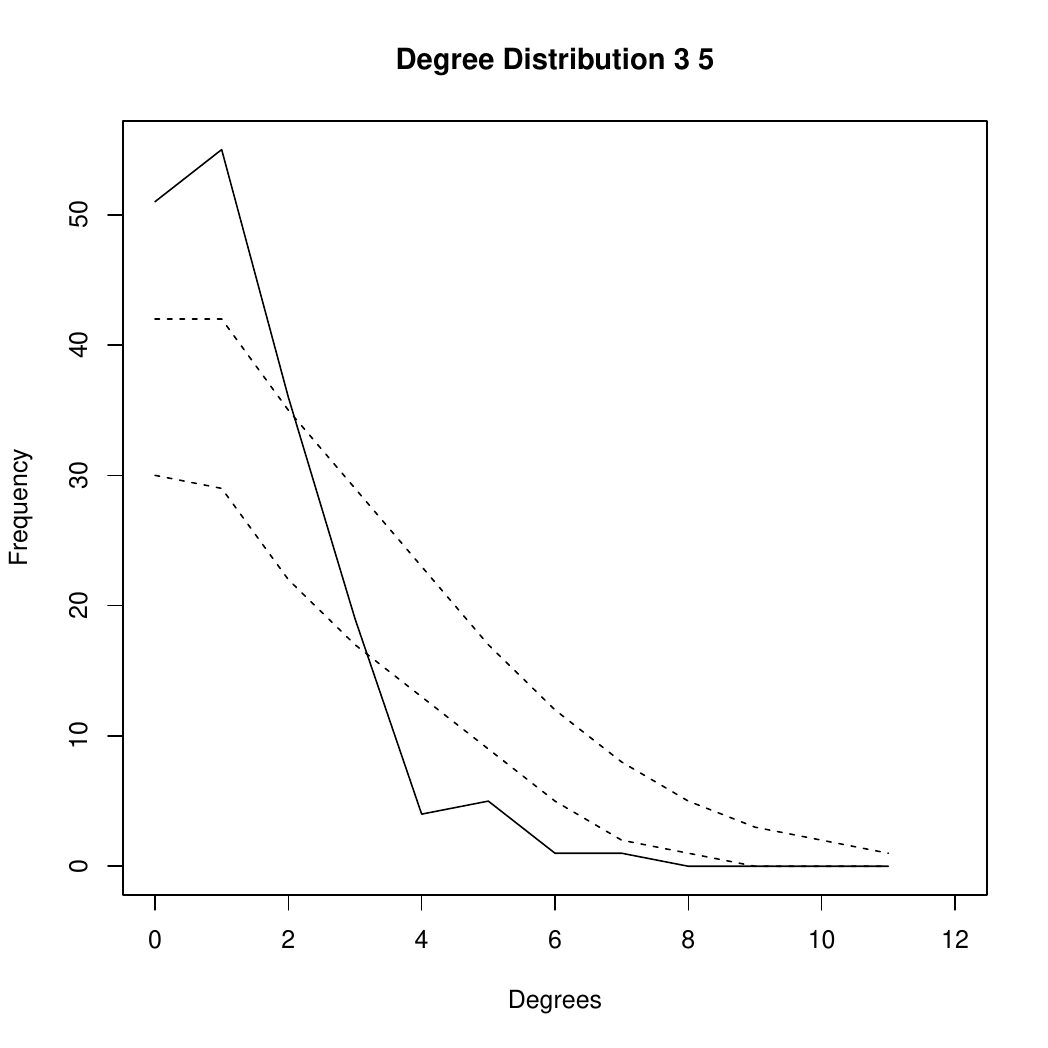}
\begin{minipage}{1\textwidth}
\caption{Tour frequency between three and five}
\end{minipage}
\label{fig:degrees49_35}
\end{subfigure}%
\hspace*{0.5cm}
\begin{subfigure}{0.35\textwidth}
\centering
\includegraphics[height=0.25\textheight,width=\linewidth]{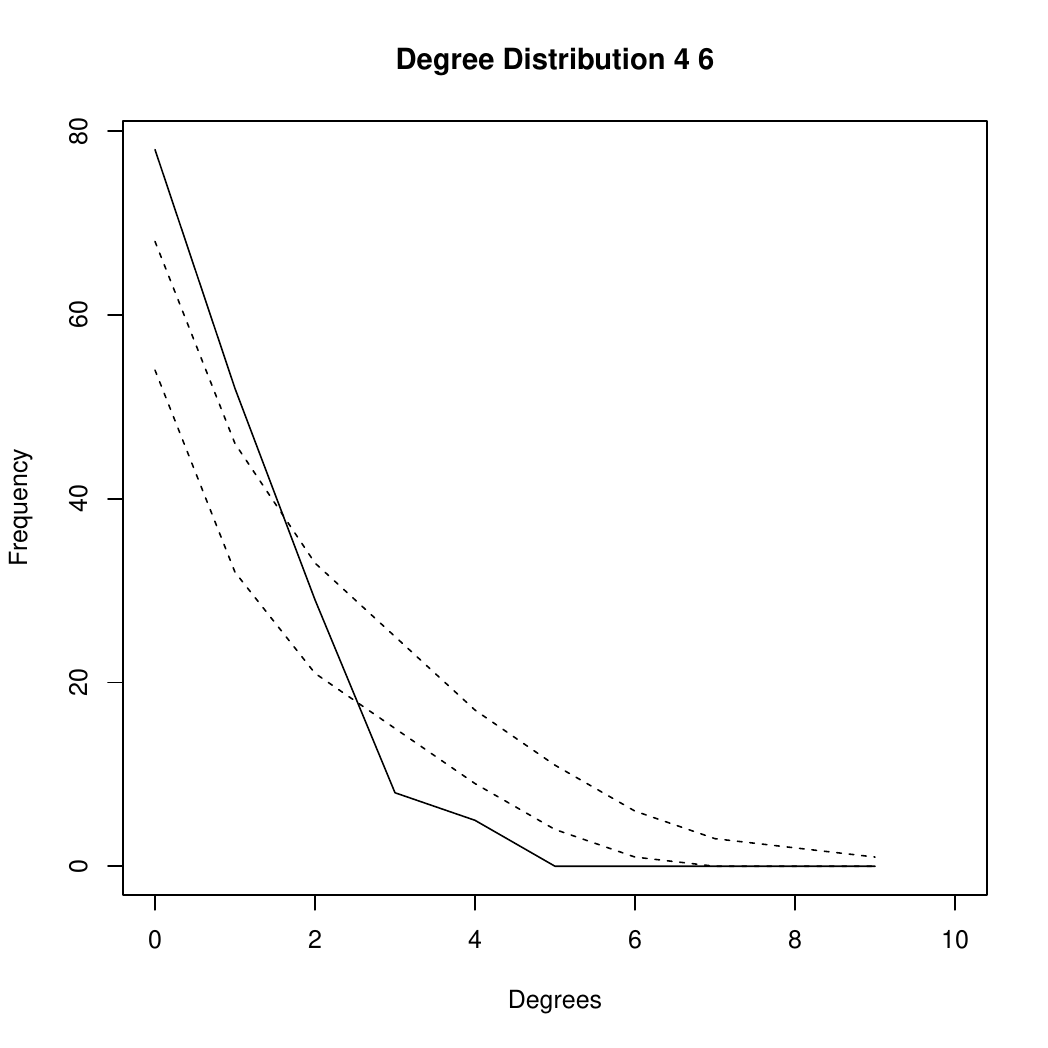}
\begin{minipage}{1\textwidth}
\caption{Tour frequency between four and six}
\end{minipage}
\label{fig:degrees49_46}
\end{subfigure}

\begin{subfigure}{0.35\textwidth}
\centering
\includegraphics[height=0.25\textheight,width=\linewidth]{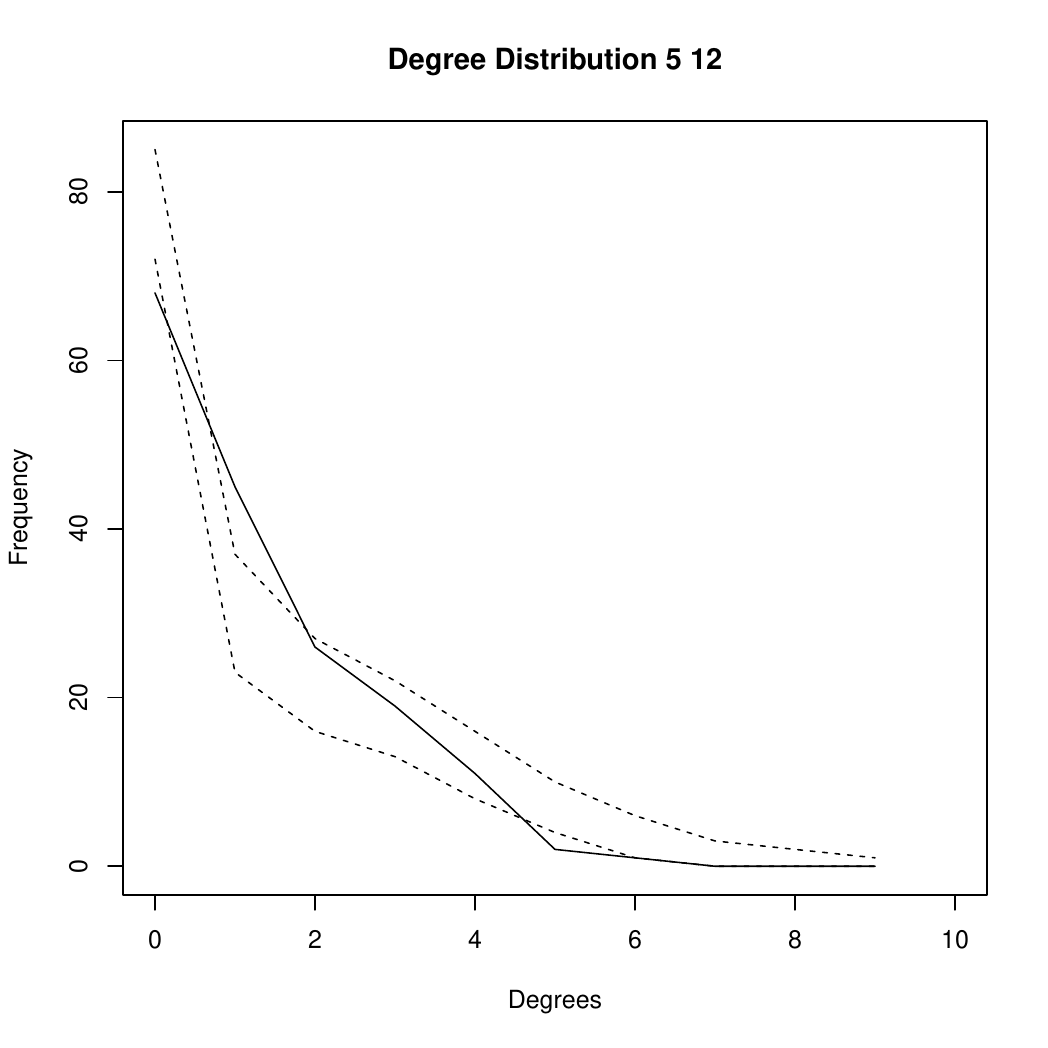}
\begin{minipage}{1\textwidth}
\caption{Tour frequency between five and twelve}
\label{fig:degrees49_512}
\end{minipage}
\end{subfigure}%
\hspace*{0.5cm}
\begin{subfigure}{0.35\textwidth}
\centering
\includegraphics[height=0.25\textheight,width=\linewidth]{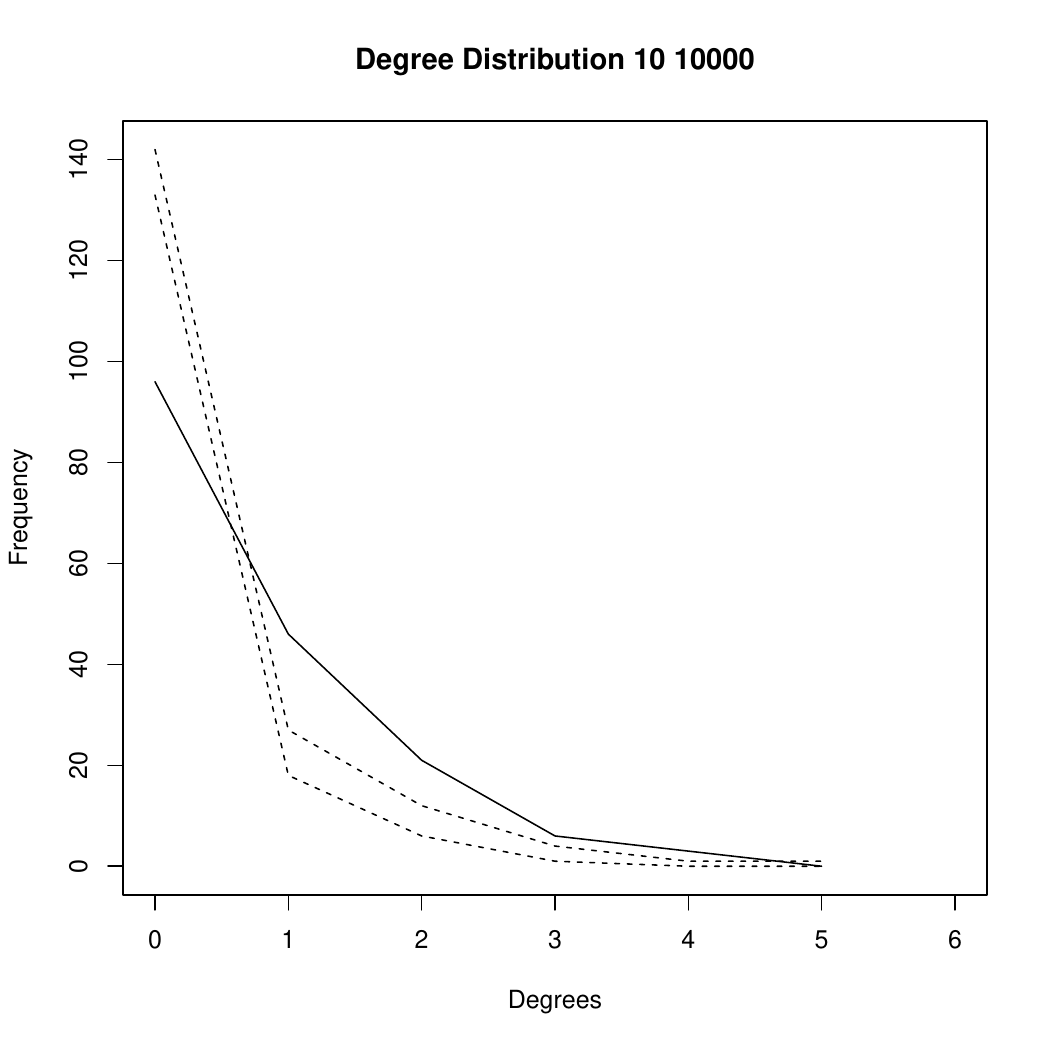}
\begin{minipage}{1\textwidth}
\caption{Tour frequency larger than ten}
\end{minipage}
\label{fig:degrees49_10inf}
\end{subfigure}
\caption{Simulated degree distributions of sub-networks with different tour frequencies (see individual caption) for 7th December 2012. Dotted lines show 10\% and 90\% quantiles of simulations and solid line shows true distributions.}
\label{fig:degrees49}
\end{figure}

{\sc Brief remark on choice of bandwidth via one-side cross validation.} To choose the bandwidth, we calculate a  local linear estimate with a one-sided kernel $K_{+,\kappa}(k) =K_\kappa(k) \Ind(k< 0)$. For all values of  $\kappa$, the fitted value of the conditional expectation of $X_{i,j}(k_{t_0}),$ given the past,  is compared with the outcome of $X_{i,j}(k_{t_0})$. This is done for all  non-censored edges. The results for different bandwidths are shown in Figure \ref{fig:bandwidth_selection}.  We see that the prediction error of the model decreases, until we reach the bandwidth $\kappa= 23$. In one-sided cross-validation, one now makes use of the fact that the ratio of asymptotically optimal bandwidths of two kernel estimators with different kernels, $K$ and $L$ is equal to $\rho= [\int K^2(u) \mathrm{d} u (\int u^2L(u) \mathrm{d} u)^2 (\int L^2(u) \mathrm{d} u)^{-1} (\int u^2K(u) \mathrm{d} u)^{-2}]^{1/5}$. For a triangular kernel, and its one-sided version, we get $\rho \approx 1.82$.
 The one-sided CV bandwidths is given by dividing 23 by $\rho$ which yields bandwidth roughly twelve (here we also only consider integer bandwidths). More details on the one-sided cross-validation approach are presented in the appendix.
 
\begin{figure}[h]
\centering
\includegraphics[width=0.6\textwidth]{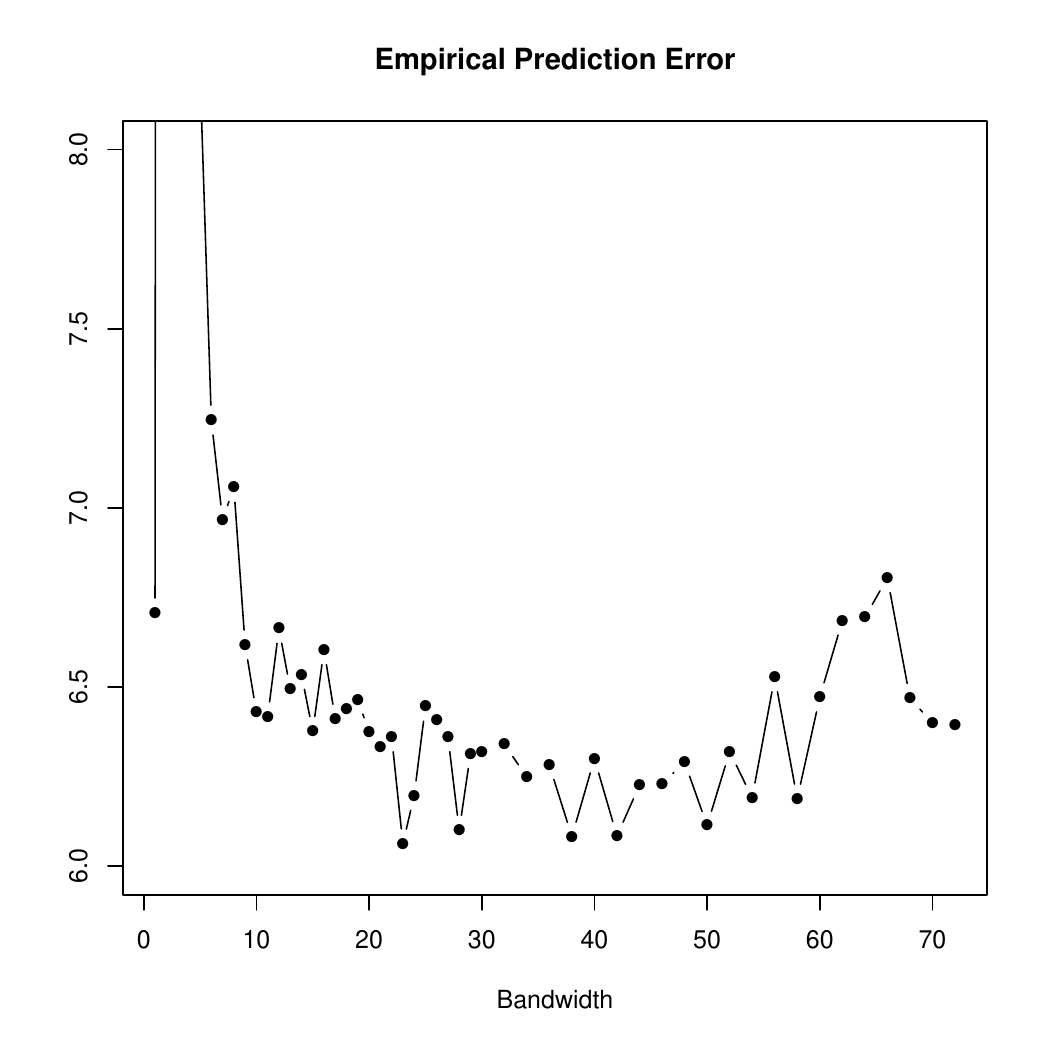}
\caption{Mean Squared Prediction Error for different bandwidths.}
\label{fig:bandwidth_selection}
\end{figure}

\subsection{Analysis 2: May 2018}
\label{subsec:second_analysis}
In this analysis, we study the biking behavior in May 2018 in more detail. In particular, we can (in contrast to before) assume that the number of bike stations remains constant over the observation period. Our main interest in this part lies in illustrating how the model can be used to understand how the system would change if another bike station were built. Let us firstly look at the distribution of bike rides over four weeks in April: this is shown in Figure \ref{fig:april}. We see a clear daily pattern: During weekdays the number of bike rides spikes in the morning and in the afternoon while it shows deep valleys (going almost down to zero) at night and not so deep valleys around midday. The weekends show a clearly different pattern by not exhibiting the morning/afternoon spikes so visibly. The only weekdays which depart from these pattern are April 24 and 25. These were both rainy days (we use weather data from the weather station at Washington D.C. Dulles Airport, as reported on Weather Underground). However, we should say that there are other rainy days which do not show such a visible effect. Interestingly, April 16th (Emancipation day that year, a public holiday) is showing similar behavior as the other weekdays but with a smaller number of bike rides (maybe one would have rather expected that public holidays behave like weekends).

\begin{figure}
\includegraphics[width=\textwidth]{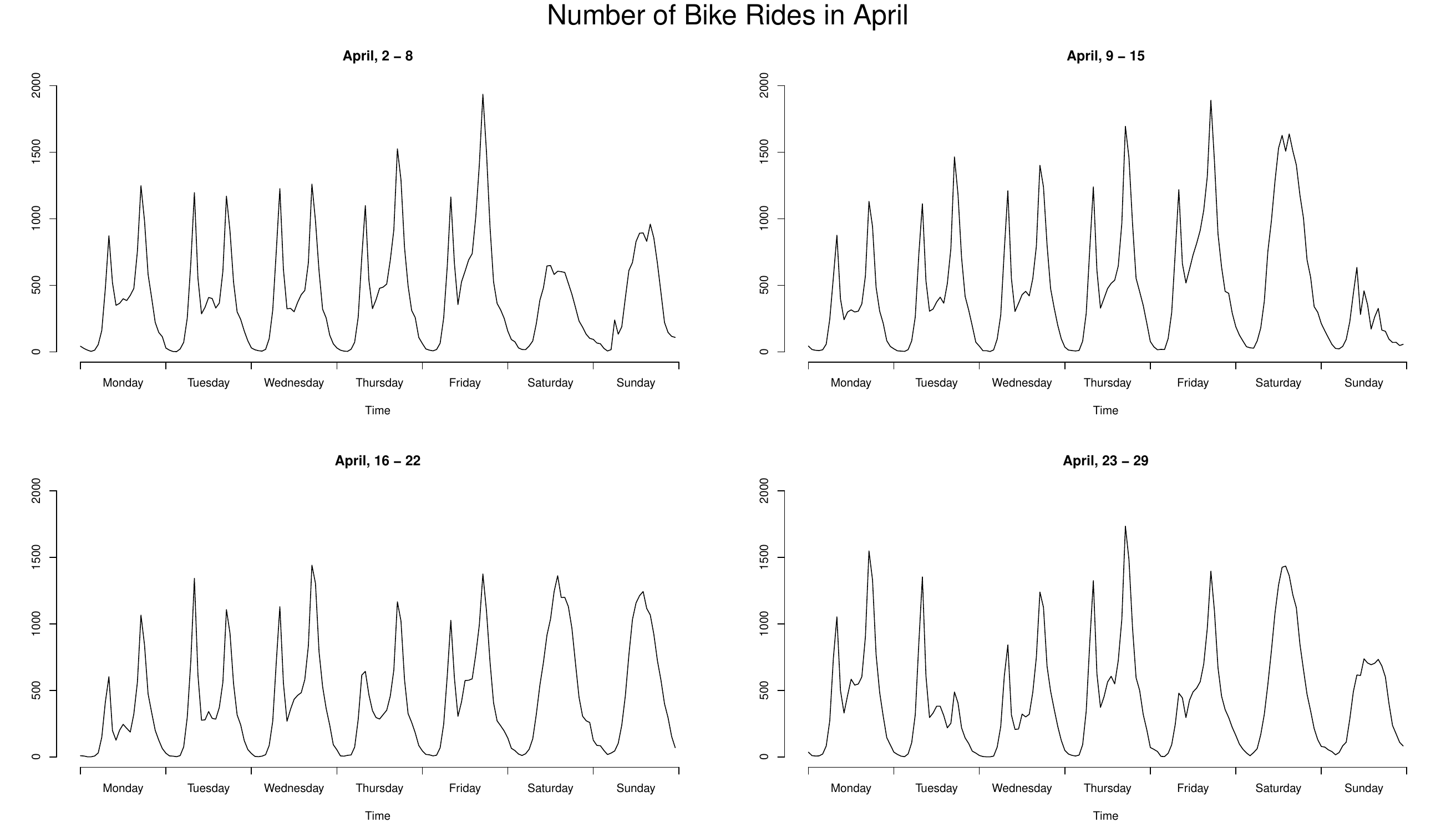}
\caption{Number of bike rides from one station to another station (i.e., returns at the same station are not included) in April 2018.}
\label{fig:april}
\end{figure}

In this analysis we choose to restrict to those 9,131 connections which have been used ten times or more in April. However, for this part, we keep the directions of the bike rides. Thus, for any two bike stations, in our model, we let $C_{n,ij}$ be the indicator that the directed connection $(i,j)$ has had ten or more rides in April. The covariates which we use here are based on the distances between two bike stations and their densities. The distance from station $i$ to $j$ is given by the time it takes to go from station $i$ to station $j$ on a bike. These times were computed by using Google maps on a weekday afternoon. Note that the travel time from $i$ to $j$ can be different than that from $j$ to $i$ because of different one-way structures of the streets and possible ascends. Sometimes people on bikes do not quite follow traffic regulations but nevertheless we assume that the travel time is a reasonable measure of distance. We will denote the distance from station $i$ to station $j$ by $d_{i,j}$. The density of a bike station is measured in terms of the number of neighboring stations. We denote by $n(i)$ the average of: 1) the number of bike stations which can be reached from $i$ in less than three minutes and 2) the number of bike stations from which $i$ can be reached in less than three minutes. Our intuition was that if a bike station is full or empty, then people would have to go to another bike station instead. However, we assume that people would not accept an arbitrarily long detour. Therefore, we chose the limit of three minutes bike riding time (not walking time) for neighboring stations. Of course, this threshold is somewhat arbitrary and a full analysis should consider the sensitivity of the results with respect to this threshold. Note lastly, that people can see the availability of bikes and empty docks in advance. Thus, in our model we assume that the bias introduced by people arriving at a full station and being forced to go to another station (no matter how far away) is negligible.

Lastly, we mention that we did not include the precipitation as a covariate for two reasons: Firstly, we wanted to have an hourly analysis but we only had daily data for precipitation. Thus, we could not determine the times of actual rain and, probably, prospective rain in the evening is not going to impact the biking activity in the morning: In a CB member survey from 2016 a bit more than half of the respondents said that one of their main reasons for joining CB is to have access to one-way trips, thus we assume that the possibility of rain in the evening would not stop people from using a bike in the morning. Secondly, and possibly more severely, it is difficult to include covariates which are constant across all connections. If such a covariate were zero it would mean that its corresponding parameter has no influence at all on the intensities. While in theory identification is possibly still valid, the practical computation will break down.

Our aim is to use the model in order to quantify the possible impact of a new bike station on the system. We let $h\approx1.1$hours (this bandwidth was chosen by the same procedure as outlined at the end of the precious subsection). With such a short bandwidth we will not smooth out differences between morning and afternoon. The covariate vector $X_{n,ij}$ is given by
\begin{equation*}
X_{n,ij}:=\begin{pmatrix}
1 \\ \log(d_{i,j}\vee1) \\ \log(d_{i,j}\vee1)^2 \\ \log(n(i)\vee1) \\ \log(n(j)\vee1)
\end{pmatrix}.
\end{equation*}
Note that, in order to avoid taking the logarithm of zero, all quantities have been bounded from below by 1. In Figure \ref{fig:w1_neighbours} we show the estimated parameter values for the second week of May. The solid lines show the estimates while the dotted lines show approximative 99\% point-wise asymptotic confidence regions as provided in the theory above (they were approximated in the same way as in the previous section by the Hessian of the likelihood at the estimate). We assume that the bias is negligible. The results for the other weeks look similarly. Therefore, we consider the results for the entire month only for the intercept and the two covariates indicating the number of neighbors of starting and ending station, cf. Figures \ref{fig:intercept}-\ref{fig:slogn}.

\begin{figure}
\centering
\includegraphics[width=0.95\textheight,angle=-90]{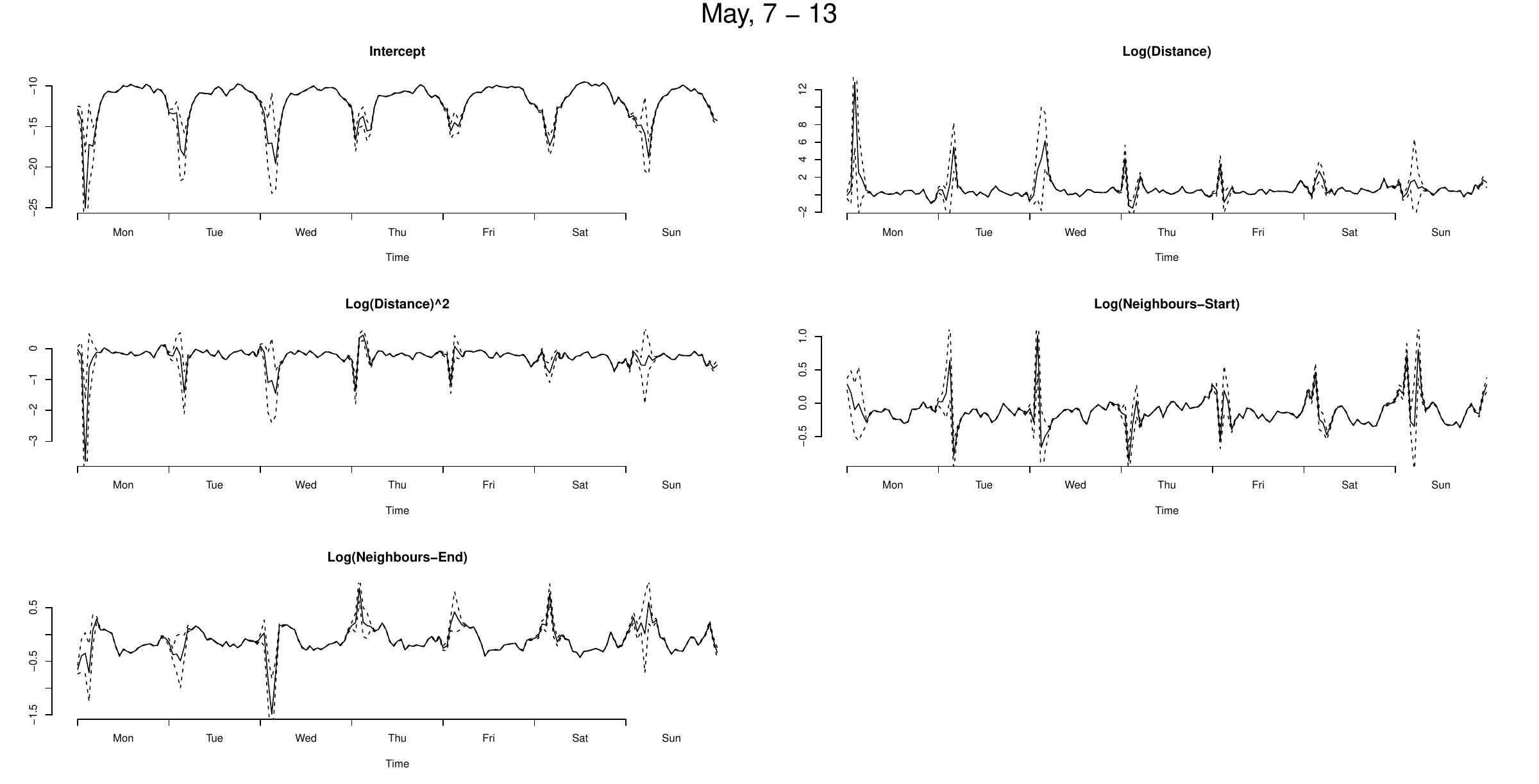}
\caption{Estimates (solid lines) of the parameters. Dotted lines show 99\% pointwise confidence sets.}
\label{fig:w1_neighbours}
\end{figure}

\begin{figure}
\includegraphics[width=\textwidth]{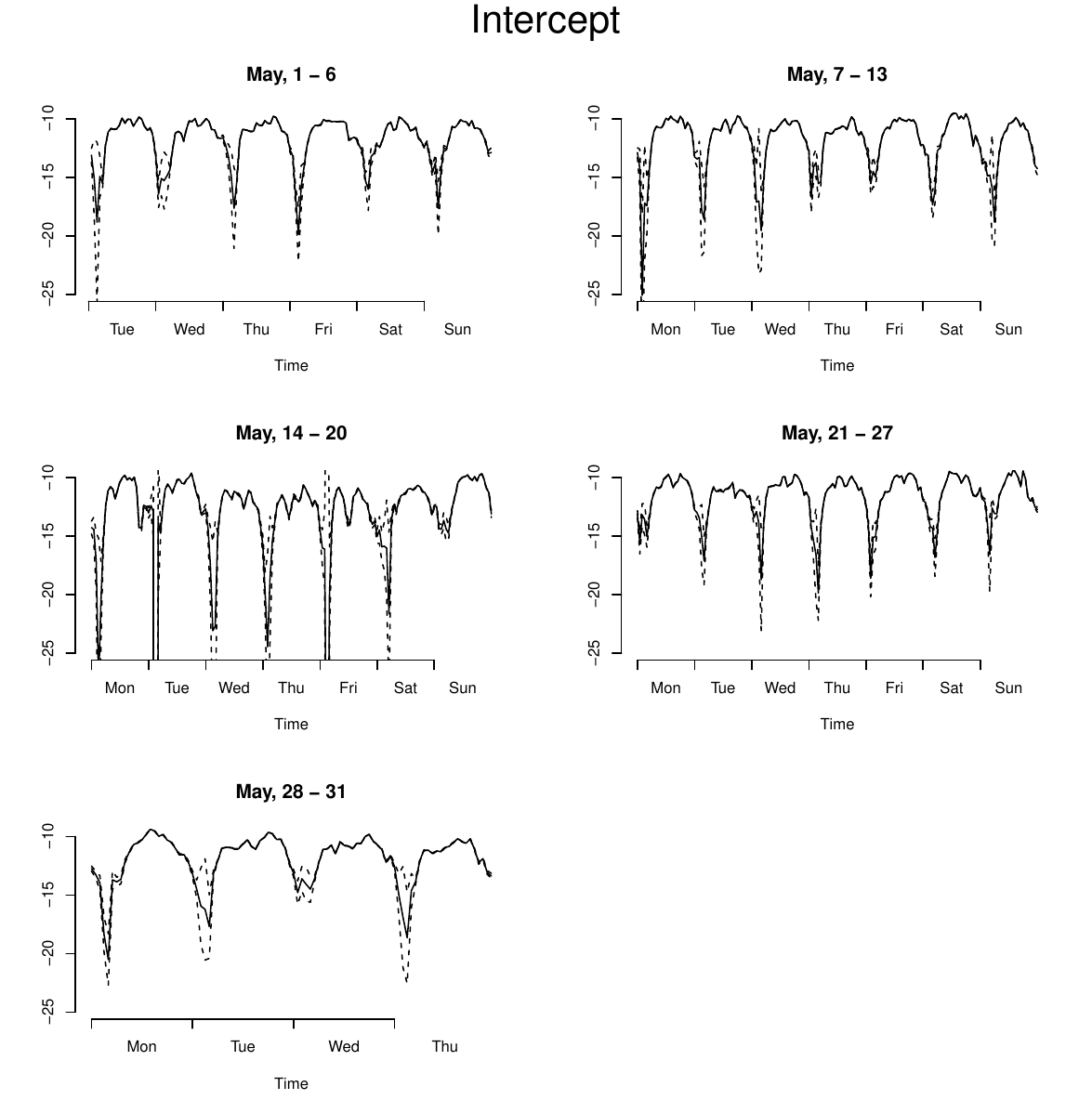}
\caption{Estimate (solid lines) of the intercept parameter. Dotted lines show 99\% confidence confidence sets.}
\label{fig:intercept}
\end{figure}

\begin{figure}
\includegraphics[width=\textwidth]{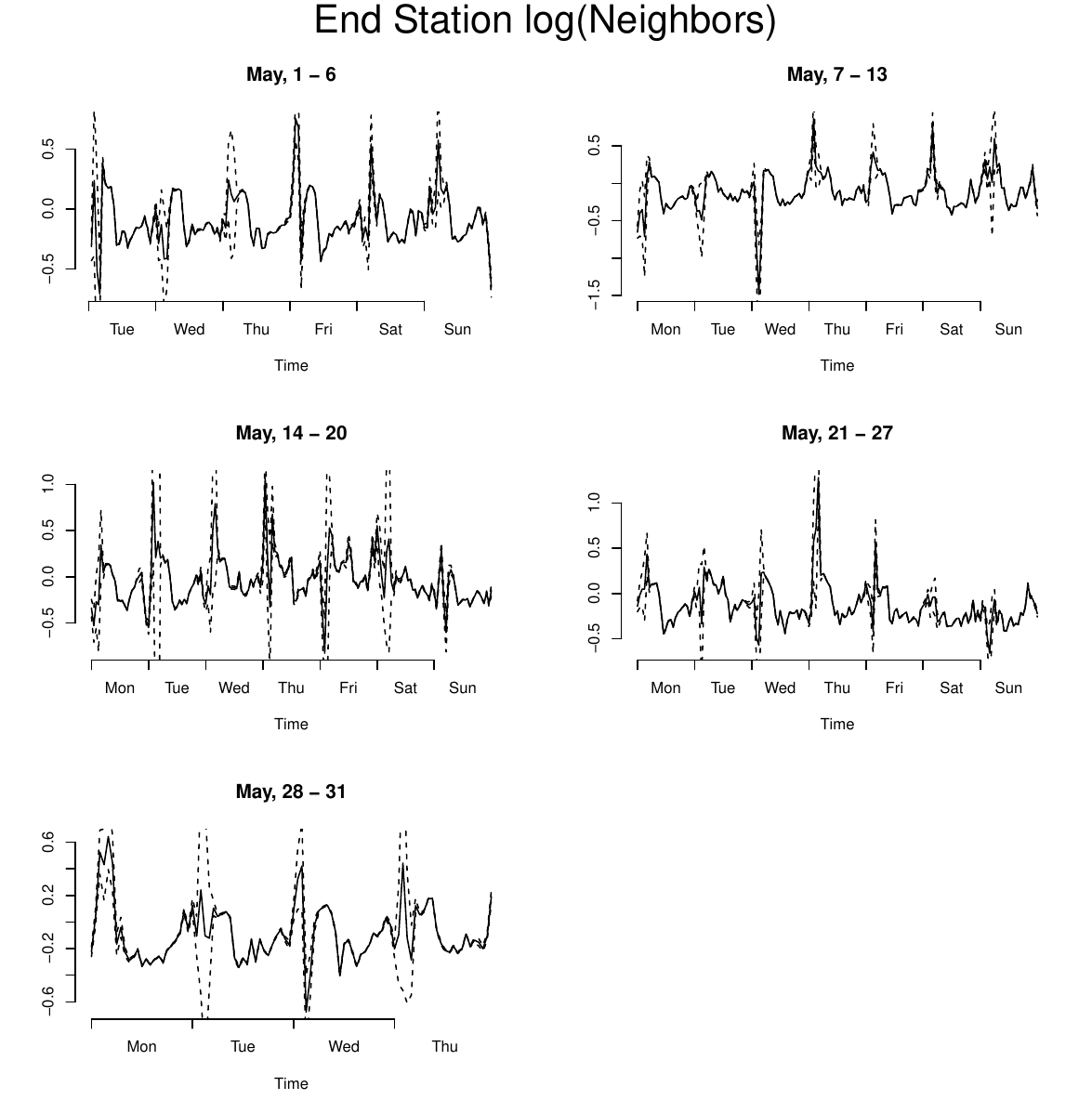}
\caption{Estimates (solid lines) of the parameter corresponding to $\log(n(j))$. Dotted lines show 99\% confidence confidence sets.}
\label{fig:elogn}
\end{figure}

\begin{figure}
\includegraphics[width=\textwidth]{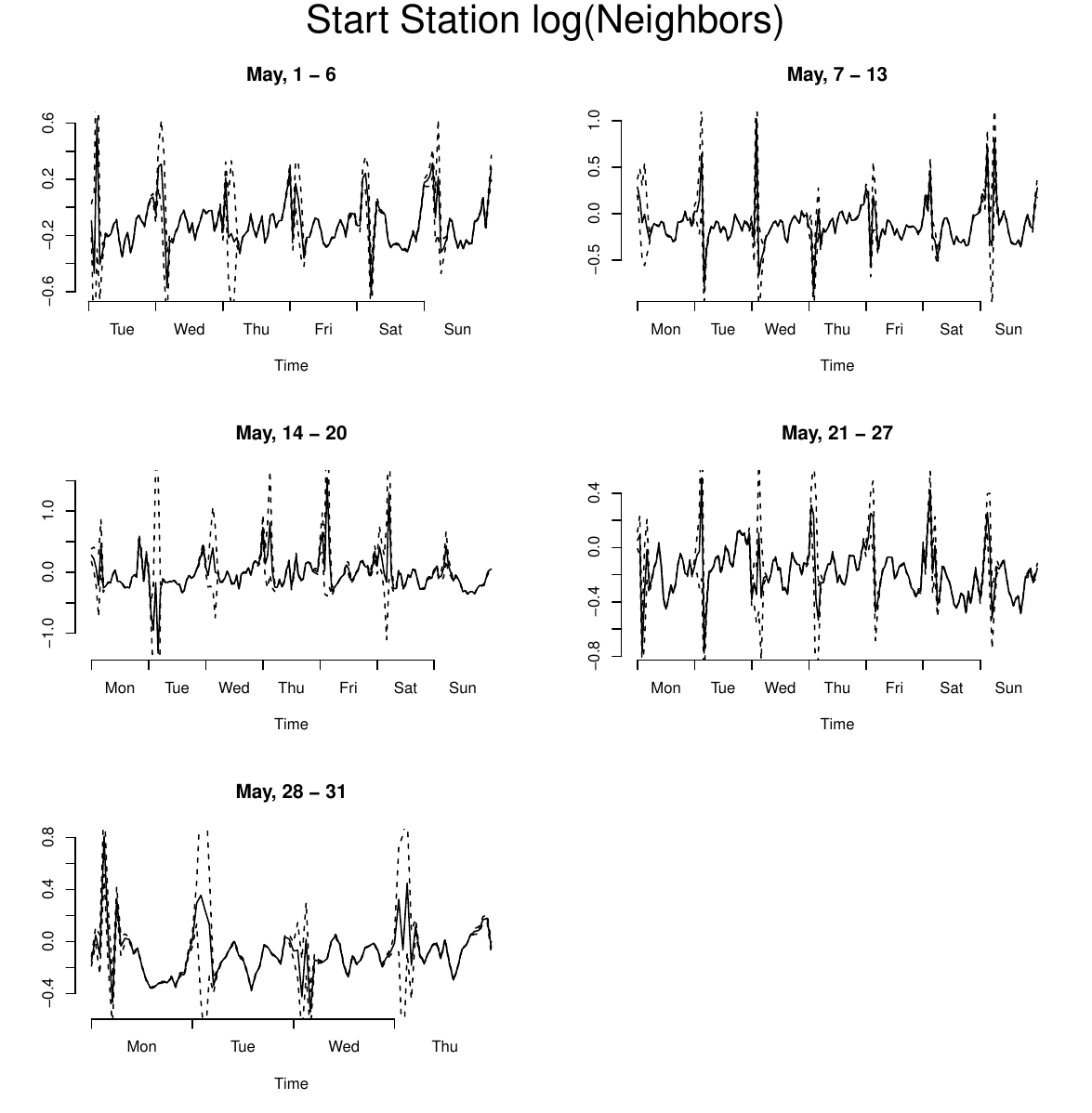}
\caption{Estimates (solid lines) of the parameter corresponding to $\log(n(i))$. Dotted lines show 99\% confidence confidence sets.}
\label{fig:slogn}
\end{figure}

We summarize some observations about these estimates:
\begin{enumerate}
\item The estimates of the intercept shows valleys over night and summits during the day. On May 16th-18th and 22nd the intercept is lower than usual. Most weekdays show a valley around midday between two peeks in the morning and in the afternoon (very pronounced on Friday, 18th, and Wednesday  23rd to Friday 25th). Monday, 28th, doesn't show this behavior. On weekends these peeks are sometimes visible, sometimes not.
\item On weekdays: Estimates of the parameter corresponding to the log-number of bike stations in the neighborhood around the ending station show negative valleys towards the afternoon/evening and become positive during night time.
\item On weekdays the log-number of neighboring stations around the start station show the opposite behavior: Estimates show negative valleys towards the morning. Sometimes they become positive during night time.
\end{enumerate}

Before discussing the interpretation of these findings we mention findings about the second and third covariate: The first three covariates model the influence of the log-time ($\log(d_{i,j})$) on the intensity of bike rides as a quadratic function when the densities of start and end station remain fixed. The estimates from Figure \ref{fig:w1_neighbours} suggest that the third covariate, that is the factor in front of $\log(d_{i,j})^2$, is often negative. That means that the parabola is open to the bottom and has hence a maximum. As the parameters change over time, the location of this extremum is also moving. In Figure \ref{fig:min} the location of the extremum of this parabola is shown. We observe the following:

\begin{enumerate}
\item[4.] The locations of the extrema lie almost always well in the region between 1 and 5 min.
\item[5.] There is no clear pattern visible.
\item[6.] The location of the maximum is not changing much.
\end{enumerate}

\begin{figure}
\includegraphics[width=\textwidth]{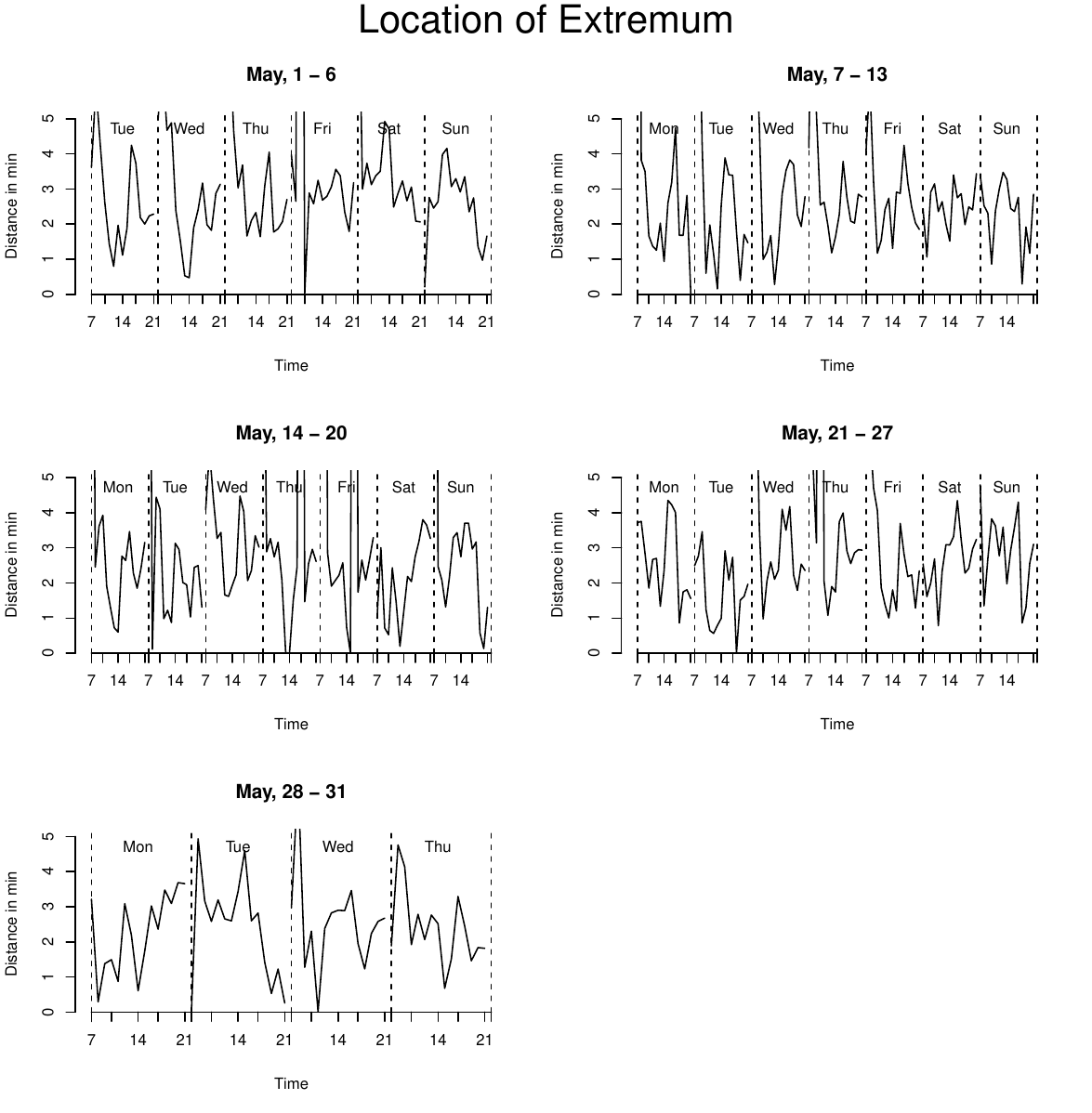}
\caption{Location of extremums of intensity function when the distance between stations (y-axis) is varying while the densities of start and end station remain fixed. Shown are day times (x-axis) between 7am and 9pm.}
\label{fig:min}
\end{figure}

A possible interpretation of these findings is as follows: Point one is a very plausible observation as people use bikes less during night. The exceptions mentioned in point one are rainy days (cf. Figure \ref{fig:rain}). It is interesting to note that other rainy days (like May 6th) are not so much visible in the estimates. We should note that we only have average precipitation information available per day. So it might be possible that the rain came during night when there was not much biking anyway. The peeks in the morning and in the afternoon correspond to the increased activity in the morning and in the afternoon which we observed earlier. Note that May 28th was Memorial Day in 2018. Thus it is not surprising that this day does not show these peeks which might correspond to people commuting to work. The second and third observations mean that in the evening a larger number of bike stations close to the destination yields a lower intensity while in the morning a larger number of bike stations close to the origin yields a lower intensity (however, this effect is not so pronounced). We explain this observation by using commuters: Commuters start their way to work in Washington D.C. possibly from central locations (like bike stations close to Union Station or Metro Stations) and disperse from there through the city to their respective work places. Thus, such stations might empty in the morning and people have to walk to other bike stations. This would yield an effect as observed: Stations with many near by stations in the morning share their traffic with the remaining stations. In the evening the reverse effect is happening: People return from their non-central work places to central bike stations causing these bike stations to fill up and hence forcing people to go to other empty bike stations near by.

\begin{figure}
\includegraphics[width=\textwidth]{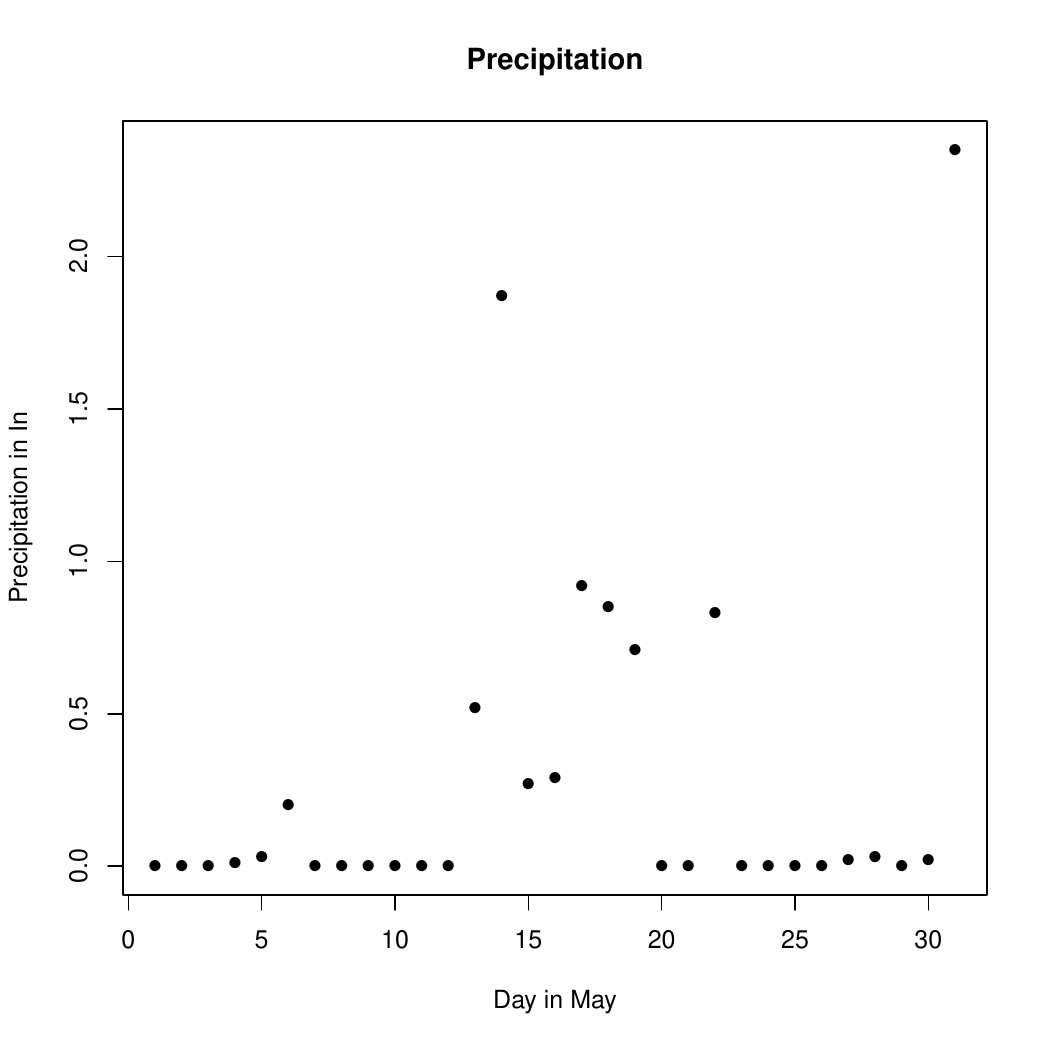}
\caption{Amount of precipitation at the weather station Washington D.C. Dulles Airport (Source: Weather Underground)}
\label{fig:rain}
\end{figure}

During night time we observed a positive effect of the number of neighbors of the destination station. A possible explanation for this might be that the stations were most likely built in a way such that the density of stations is large if the demand for bikes in this neighborhood is large. Thus, the number of neighboring bike stations serves as a proxy for popularity of a bike station. Keeping this in mind, we interpret the positive parameter during night time as indicator for the hypothesis that, at night, people prefer to go to central locations (like metro stations) by bike. The absence of positivity of the number of neighbors of the origin station indicates that the reverse is not always happening: People do not leave from central locations (this is plausible as we would expect more people to go home at night).

The findings four to six are mostly interesting because they indicate that an actual parabola is fit indicating a strict convex behavior of the intensity. It is important to note that by choosing the parameters two and three the model fit could result in a strictly monotone fit of the parabola (when the extremum of the parabola is located outside the interval $[0\textrm{min},30\textrm{min}]$). That this is not the case indicates that there is a non-proportional change in the activity as the distance between bike stations changes. Generally, the location of the maximum in the short distance range indicates that people prefer to take shorter routes rather than longer routes (cf. also the discussion of the CB pricing system).

If CB was to built a new bike station we would assume that they do it because in that region they suspect high biking activity (e.g. because bike stations are constantly empty or full). In that sense we assume that the idea that a high density of biking stations indicates a region interesting for biking is not violated. However, adding a new station changes the covariates in the network and thus we can use the parameter estimates from above in order to predict how the biking activity at other bike stations might change. The introduction of a new station changes for example the number of neighboring bike stations of its neighbors. Thus we could predict how much traffic is diverted from the existing stations to a new station. Moreover, we can also use it to predict the number of bike rides to the station and from the station. This can help to find an accurate size for the new bike station.

In this example we saw that a time varying parameter choice is useful in order to be able to distinguish morning and afternoon as well as rainy and non-rainy periods. Moreover, we illustrated how the covariates could be chosen in order to assess the effect of adding a new bike station.

\section{Proof of Theorem~\ref{thm:asymptotic_normality}}
\label{proofs}
In the proof, we do not distinguish explicitly directed and undirected networks: in the undirected case, we always assume $i<j$, moreover we will need $l_n=O(n^2\IP(C_{n,12}(t_0)=1)$, which is true in both cases. The processes $N_{n,ij}$ are counting processes with intensity given by $\lambda_{n,ij}(\theta_0(t),t)$. We can decompose these counting processes as (Doob-Meyer Decomposition, see e.g.\ \cite{Anderson93} Chapter II.4)
\begin{eqnarray} \label{eqyalea}N_{n,ij}(t)=M_{n,ij}(t)+\int_0^t\lambda_{n,ij}(\theta_0(s),s)ds,\end{eqnarray}
where $M_{n,ij}$ is a local, square integrable martingale. We use this decomposition of the counting processes in order to decompose the likelihood and its derivatives. Let $P_n(\theta)$ be defined as 
\begin{align}
P_n(\theta)&:=\frac{1}{l_nh}\sum_{i,j=1}^n\int_0^TK\left(\frac{s-t_0}{h}\right)C_{n,ij}(s) \big[\theta^TX_{n,ij}(s)\exp(\theta_0(s)^TX_{n,ij}(s)) \nonumber \\ & \qquad-\exp(\theta^TX_{n,ij}(s))\big]ds. \label{eq:thetatilde}
\end{align}
Note that we do not make the dependence of $P_n(\theta)$ on $t_0$ explicit in the notation. In order to reduce notation, we write for the derivative of a function $\psi(\theta)$ of one variable $\theta$ (which might be a vector) simply $\psi':=\partial_{\theta}\psi$ and $\psi'':=\partial_{\theta^2}\psi$. Using $P_n(\theta),$ we can write 
\begin{align}
\frac{1}{l_n}\ell(\theta,t_0)&=\frac{1}{l_nh}\sum_{i,j=1}^n\int_0^TK\left(\frac{t-t_0}{h}\right)\theta^TX_{n,ij}(t)dM_{n,ij}(t)\,+\,P_n(\theta), \label{eq:decomlikelihood} \\
\frac{1}{l_n}\cdot\partial_{\theta}\ell(\theta,t_0)&=\frac{1}{l_nh}\sum_{i,j=1}^n\int_0^TK\left(\frac{t-t_0}{h}\right)X_{n,ij}(t)dM_{n,ij}(t)+P'_n(\theta), \label{eq:decomscore} \\
\frac{1}{l_n}\cdot\partial_{\theta^2}\ell(\theta,t_0)&=P''_n(\theta). \label{eq:decomfisher}
\end{align}

Recall that $\theta_{0,n}$ is defined as the maximizer of $\theta\mapsto\int_0^T\frac{1}{h}K\left(\frac{s-t_0}{h}\right)g(\theta,s)ds$, 
where $g$ is defined in (A7). Note that the function $g$ does not depend on $n$, see Assumption (A1). Lemma \ref{lem:theta0ntothetat0t0} shows that $\theta_{0,n}$ is uniquely defined. The value $\theta_{0,n}$ is the deterministic counterpart of  the random quantity $\tilde{\theta}_n(t_0)$ that is defined as the solution of  $P^\prime_n(\tilde{\theta}_n(t_0))=0$. The existence of the latter is considered in Proposition \ref{prop:thetatilde}.  \\

In all lemmas and propositions of this section, we assume that Assumptions (A1), (A2) and (A5) hold as they are permanently used (also implicitly in other assumptions). The other assumptions will be mentioned in those places where they are needed. \medskip

\begin{lemma}
\label{lem:gmax}
We have
\begin{eqnarray*}
&&\theta^Ty\exp(\theta_0(s)^Ty)-\exp(\theta^Ty) \\
&\leq&\theta_0(s)^Ty\exp(\theta_0(s)^Ty)-\exp(\theta_0(s)^Ty).
\end{eqnarray*}
Equality holds, if and only if, $\theta_0(s)^Ty=\theta^Ty$. In particular, $\theta_0(s)$  is the unique maximizer of  $\theta\mapsto g(\theta,s)$.
\end{lemma}
\begin{proof}
Note that, for arbitrary $y\in\IR,$
$$\frac{d}{dx}(xe^y-e^x)=e^y-e^x$$
implies that the differentiable function $x\mapsto xe^y-e^x$ has the unique maximizer $x=y$. This also implies the second statement of the lemma by (A5).
\end{proof}

\begin{fact}
\label{lemA} Assume (A4), \eqref{eqa41} and (A3) hold. For $j\in\{0,1,2\},k \in\{0,1,2,3\},$ with $j+k \leq 3,$ the partial derivatives of order $j$ of the function $g(\theta,s)$ with respect to $s,$ and of order $k$ with respect to $\theta,$ exist, for $(t,\theta)\in U\times V$ (cf. Assumption (A4) for a definition of $U$ and $V$). The partial derivatives can be calculated by interchanging the order of integration and differentiation in \eqref {eq:gn_integral}. All these partial derivatives of $g(\theta,s)$ are absolutely bounded on $U\times V$. For the calculation of the first two derivatives of $g$ with respect to $\theta,$ differentiation and application of the expectation operator can be interchanged in \eqref {eq:gn_integral1}. The matrix $\Sigma $ is invertible.
 \end{fact}
\begin{proof}  The statement of this fact follows immediately from  \eqref{eqa41} of Condition (A4). Note that  the functions $\theta_0$,  $\theta_0'$ and $\theta_0''$ are absolutely bounded in a neighborhood of $t_0$. This holds because these functions are continuous in a neighborhood of $t_0$, see (A3). Invertibility of $\Sigma$ is a consequence of (A5).
\end{proof}
 
 \begin{lemma}
\label{lem:theta0ntothetat0t0} Assume Fact \ref{lemA} holds and that $\Theta$ is convex. For $n$ large enough, $\theta_{0,n}$ (the maximizer of \eqref{eq:thetanulln}) is well defined and unique. It holds that $\theta_{0,n}\to\theta_0(t_0)$ as $n \to \infty$. In particular, $\theta_{0,n}\in V$ for $n$ large enough.
\end{lemma}
 
\begin{proof}[Proof of Lemma \ref{lem:theta0ntothetat0t0}]
The function $\theta\mapsto g(\theta,t_0)$ is strictly concave and hence $\theta_0(t_0)$ is its unique maximizer (cf. Lemma \ref{lem:gmax}). Thus
$$\phi_n(\theta):=\int_0^T\frac{1}{h}K\left(\frac{s-t_0}{h}\right)g(\theta,s)ds$$
is strictly concave too. Moreover, we know that $\partial_t g(\theta,t) $ is absolutely bounded on $U\times V$. This implies that $\phi_n$ converges to $g(\theta,t_0)$, uniformly on $V$. Hence, $\phi_n$ has a local maximizer $\theta_{0,n}$ in the open set $V$. By strict convexity, $\theta_{0,n}$ is the unique global maximum. The convergence $\theta_{0,n}$ to $\theta_0(t_0)$ follows by uniform convergence of $\phi_n$ to $g$.
\end{proof}

\begin{lemma}
\label{lemC} Assume Fact \ref{lemA} holds.  With $\Sigma_n =-\int_{-1}^1K(u)\int_0^1\partial_{\theta^2}g(\theta_0(t_0)+\alpha(\theta_{0,n}-\theta_0(t_0)),t_0+uh)\mathrm{d} \alpha \mathrm{d}u$, we have  
$$\Sigma_n \to \Sigma\textrm{ as }n\to\infty.$$
Moreover, the sequence
$$v_n= 2 \int_{-1}^1K(u)\int_0^1(1-\alpha)\partial_{t^2}\partial_{\theta}g(\theta_0(t_0),t_0+(1-\alpha)uh)u^2\mathrm{d} \alpha \mathrm{d}u$$
is bounded, and it holds that $v_n \to v$, as $n \to \infty$.\end{lemma}

\begin{proof}
Using Lemmas \ref{lem:theta0ntothetat0t0} and Fact~\ref{lemA}, we conclude that  the integrand 
$$\partial_{\theta^2}g(\theta_0(t_0)+\alpha(\theta_{0,n}-\theta_0(t_0)),t_0+uh) \to \partial_{\theta^2}g(\theta_0(t_0),t_0)$$
(note that $u\in[-1,1]$ and $\alpha\in[0,1]$). The first statement of the lemma follows by an application of Lebesgue's Dominated Convergence Theorem, and the fact that $\partial_{\theta^2}g$ is bounded as a continuous function on a compact set.  The second statement of the lemma follows similarly.
\end{proof} 

\begin{proposition}
\label{prop:thetanulln}
Assume Fact \ref{lemA} holds. We have, for $t_0\in(0,T),$
$$\theta_{0,n}=\theta_0(t_0)+h^2\Sigma^{-1}v+o(h^2).$$
\end{proposition}
\begin{proof}[Proof of Proposition \ref{prop:thetanulln}]
Since $\theta_0(s)$ maximizes $\theta\mapsto g(\theta,s)$ (cf. Lemma \ref{lem:gmax}), we have $\partial_{\theta}g(\theta_0(s),s)=0$. Furthermore, by definition of $\theta_{0,n},$ we have
$$\int_0^TK\left(\frac{s-t_0}{h}\right)\partial_{\theta}g(\theta_{0,n},s)ds=0.$$
Having observed that, we compute, for $h$ small enough,
\begin{eqnarray}
0&=&\frac{1}{h}\int_0^TK\left(\frac{s-t_0}{h}\right)\partial_{\theta}g(\theta_{0,n},s)ds \nonumber \\
&=&\int_{-1}^1K(u)\partial_{\theta}g(\theta_{0,n},t_0+uh)\mathrm{d}u \nonumber \\
&=&\int_{-1}^1K(u)\bigg[\partial_{\theta}g(\theta_0(t_0),t_0+uh)\nonumber \\
&& \qquad +\int_{0}^1\partial_{\theta^2}g(\theta_0(t_0)+\alpha(\theta_{0,n}-\theta_0(t_0)),t_0+uh)\mathrm{d} \alpha(\theta_{0,n}-\theta_0(t_0))\bigg]\mathrm{d}u \nonumber \\
&=&\int_{-1}^1K(u)\partial_{\theta}g(\theta_0(t_0),t_0+uh)\mathrm{d}u+ \Sigma_n(\theta_{0,n}-\theta_0(t_0)) .\label{eq:thetanulln_expansion}
\end{eqnarray}
$\Sigma_n$ converges to the invertible matrix $\Sigma$ by Lemma \ref{lemC}. The first integral is of order $h^2$. This follows by a Taylor expansion in the time parameter:
 \begin{eqnarray*}
&&\int_{-1}^1K(u)\partial_{\theta}g(\theta_0(t_0),t_0+uh)\mathrm{d}u \\
&=&\int_{-1}^1K(u)\left[{\partial_{\theta}g(\theta_0(t_0),t_0)}+\frac{d}{dt}g_{\theta}(\theta_0(t_0),t_0)uh+\right. \\
&&\quad\quad\quad\quad\left.\int_0^1(1-\alpha)\frac{ \mathrm{d}^2}{ \mathrm{d}t^2} \partial_{\theta}g(\theta_0(t_0),t_0+(1-\alpha)uh)\mathrm{d} \alpha u^2h^2\right]\mathrm{d}u \\
&=& \frac 1 2 h^2 v_n.
\end{eqnarray*}
By Lemma \ref{lemC}, $v_n$ is bounded. Thus, together with \eqref{eq:thetanulln_expansion}, we have established 
$$\theta_{0,n}=\theta_0(t_0)-(\Sigma_n^{-1}-\Sigma^{-1}+\Sigma^{-1})\tfrac{1}{2}h^2v_n=\theta_0(t_0)-\tfrac{1}{2}h^2\Sigma^{-1}v_n-\tfrac{1}{2}h^2(\Sigma_n^{-1}-\Sigma^{-1})v_n.$$

The statement of the proposition now follows from $v_n \to v$.
\end{proof}

\begin{lemma}
\label{lemD} Assume Fact \ref{lemA}, (A4) \eqref{eqa41}, \eqref{eqa44}, (A6) \eqref{eqa71}, \eqref{eqa72} and (A7), \eqref{eqa81} hold. We have
\begin{equation} \label{lemDclaim1} 
P'_n(\theta_{0,n})\overset{\IP}{\to}0.
\end{equation}
For any $k,l\in\{1,...,q\},$ it holds that
\begin{equation} \label{lemDclaim2} 
P''_n(\theta_{0,n})\overset{\IP}{\to}-\Sigma.
\end{equation}
Moreover,
\begin{equation} \label{lemDclaim3}
\sup_{k,l,r,\theta}\left|\partial_{\theta_k}\partial_{\theta_l}P_n^{'(r)}(\theta)\right|=O_P(1),
\end{equation}
where $P_n^{'(r)}$ denotes the $r$-th component of $P'_n$, the supremum runs over $k,l,r\in\{1,...,q\},$ and $\theta\in V$.\end{lemma}

\begin{proof}
We start by showing that $P_n'(\theta_{0,n})=o_P(1)$. This holds, if $\IE(\|P_n'(\theta_{0,n})\|)=o(1)$. Define $\rho_{n,ij}(\theta,s):=\|X_{n,ij}(s)\|\cdot\left|\exp(\theta_0(s)^TX_{n,ij}(s))-\exp(\theta^TX_{n,ij}(s))\right|$. By positivity of $\rho_{n,ij}(\theta,s),$ we may apply Fubini's Theorem, and thus we compute
\begin{align*}
&\IE(\|P_n'(\theta_{0,n})\|) \\
&\leq\frac{1}{l_n}\sum_{i,j=1}^n\int_{-1}^1K(u)\IE\left(C_{n,ij}(t_0+uh)\rho_{n,ij}(\theta_{0,n},t_0+uh)\right)\mathrm{d}u \\
&=\frac{1}{l_n}\sum_{i,j=1}^n\int_{-1}^1K(u)\IP(C_{n,ij}(t_0+uh)=1) \\
&\quad\quad\quad\quad\quad\times\IE\left(\rho_{n,ij}(\theta_{0,n},t_0+uh)|C_{n,ij}(t_0+uh)=1\right)\mathrm{d}u.
\end{align*}
The expectation in the integral expression can be bounded by applying a Taylor expansion: 
\begin{align*}
&\IE\left(\rho_{n,ij}(\theta_{0,n},s_u)|C_{n,ij}(s_u)=1\right) \\
&\leq\IE\Bigg(\int_0^1\exp\left(\left[\theta_0(s_u)-\alpha\cdot(\theta_0(s_u)-\theta_{0,n})\right]^TX_{n,ij}(s_u)\right)\mathrm{d} \alpha \\
&\quad\quad\quad\times\|X_{n,ij}(s_u)\|^2\Bigg|C_{n,ij}(s_u)=1\Bigg)\cdot\|\theta_0(s_u)-\theta_{0,n}\|,
\end{align*}
where  $s_u=t_0+uh$.
Now, by \eqref{eqa44} in Assumption (A4), the expectation in the last upper bound is bounded by a constant $C,$ uniformly in $u\in [-1,1]$. Using $\sup_{u\in[-1,1]}\|\theta_0(t_0+uh)-\theta_{n,0}\|= o(1),$ we obtain
\begin{eqnarray*}
&&\IE(\|P_n'(\theta_{0,n})\|) \\
&\leq&\frac{1}{l_n}\sum_{i,j=1}^n\int_{-1}^1K(u)\IP(C_{n,ij}(t_0+uh)=1)\mathrm{d}u\cdot C\cdot\sup_{v\in[-1,1]}\left\|\theta_0(t_0+vh)-\theta_{0,n}\right\| \\
&=&C\cdot\IP(C_{n,ij}(t_0)=1)^{-1} \cdot\int_0^1K(u)\IP(C_{n,ij}(t_0+uh)=1)\mathrm{d}u\cdot o(1) \\
&=&o(1),
\end{eqnarray*}
where the last equality is a consequence of \eqref{eqa71}. This shows \eqref{lemDclaim1}.

We now show \eqref{lemDclaim2}. With $\partial_{\theta^2}g(\theta_{0,n},s) = - \IE(\tau_{n,ij}(\theta_{0,n},s)|C_{n,ij}(s)=1),$ Fact~\ref{lemA} gives
$$\IE(P_n''(\theta_{0,n}))=-\frac{1}{l_nh}\sum_{i,j=1}^n\int_0^TK\left(\frac{s-t_0}{h}\right)\IP(C_{n,ij}(s)=1){\IE(\tau_{n,ij}(\theta_{0,n},s)|C_{n,ij}(s)=1)}ds.$$
For \eqref{lemDclaim2}, it suffices to show:
\begin{eqnarray}
P_n''(\theta_{0,n})-\IE(P_n''(\theta_{0,n}))&=&o_P(1), \label{eq:Pnza1} \\
\IE(P_n''(\theta_{0,n}))+\Sigma&=&o(1). \label{eq:Pnza2} 
\end{eqnarray}
For the proof of  \eqref{eq:Pnza2}, we note that with $ a_n(u) = \frac{\IP(C_{n,12}(t_0+uh)=1)}{\IP(C_{n,12}(t_0)=1)},$
\begin{eqnarray*}
&&\IE(P_n''(\theta_{0,n}))+\Sigma \\
&=&\int_{-1}^1K(u)\left[a_n(u)\partial_{\theta^2}g(\theta_{0,n},t_0+uh)-\partial_{\theta^2}g(\theta_0(t_0),t_0)\right]\mathrm{d}u \\
&=&\int_1^1K(u)a_n(u)\left[\partial_{\theta^2}g(\theta_{0,n},t_0+uh)-\partial_{\theta^2}(\theta_0(t_0),t_0)\right]\mathrm{d}u\\ && \qquad +\partial_{\theta^2}g(\theta_0(t_0),t_0)\int_{-1}^1K(u)(a_n(u)-1)\mathrm{d}u \\
&=&o(1).
\end{eqnarray*}
Here we use \eqref{eqa71}, and $\theta_{0,n}-\theta_0(t_0)= o(1)$ (see Proposition~\ref{prop:thetanulln}). 

For the proof of \eqref{eq:Pnza1}, we write $K_{h,t_0}(s):=K\left(\frac{s-t_0}{h}\right)$ and
\begin{align*}
&P_n''(\theta_{0,n})-\IE(P_n''(\theta_{0,n})) \\
=&\tfrac{1}{l_nh}\sum_{i,j = 1}^n\int_0^TK_{h,t_0}(s)\left[-C_{n,ij}(s)\tau_{n,ij}(\theta_{0,n},s)+\IP(C_{n,ij}(s)=1)\partial_{\theta^2}g(\theta_{0,n},s)\right]ds.
\end{align*}
We  will apply Markov's inequality to show that this term converges to zero. When squaring the above sum, we can split the resulting double sum into three parts, depending on whether $|\{i,j\} \cap \{k,l\}| = 0, 1$ or $2$. Thus we have to show that  the following three sequences converge to zero:
\begin{eqnarray}
&&\IE\Bigg(\frac{1}{l_n^2h^2}\sum_{(i,j)}\bar \kappa_{n,ij}(\theta_{0,n})^2\Bigg) =o(1), \label{eq:tv} \\
&&\IE\Bigg(\frac{1}{l_n^2h^2}\underset{\textrm{sharing one vertex}}{\sum_{(i,j),(k,l)}}\bar \kappa_{n,ij}(\theta_{0,n})
\bar \kappa_{n,kl}(\theta_{0,n})\Bigg)=o(1), \label{eq:ov} \\
&&\IE\Bigg(\frac{1}{l_n^2h^2}\underset{\textrm{sharing no vertex}}{\sum_{(i,j),(k,l)}}\bar \kappa_{n,ij}(\theta_{0,n})
\bar \kappa_{n,kl}(\theta_{0,n})\Bigg)=o(1),\label{eq:nv}
\end{eqnarray}
where $\kappa_{n,ij}(\theta_{0,n},s):=-C_{n,ij}(s)\tau_{n,ij}(\theta_{0,n},s)+\IP(C_{n,ij}(s)=1)\partial_{\theta^2}g(\theta_{0,n},s),$ and $ \bar \kappa_{n,ij}(\theta_{0,n}):= \int_0^TK\left(\frac{s-t_0}{h}\right)\kappa_{n,ij}(\theta_{0,n},s)ds$.
Now note that 
\begin{eqnarray*}
&&\IE\big(\bar \kappa_{n,ij}(\theta_{0,n})
\bar \kappa_{n,kl}(\theta_{0,n})\big) \\
&=&\int_{-1}^1\int_{-1}^1K(u)K(v)\IE\big(\kappa_{n,ij}(\theta_{0,n},t_0+uh)\kappa_{n,kl}(\theta_{0,n},t_0+vh)\big)\mathrm{d}u\mathrm{d}v,
\end{eqnarray*}
and that the sum in \eqref{eq:tv} has  $O(n^2)$ terms, \eqref{eq:ov} comprises  $O(n^3)$ terms, and finally \eqref{eq:nv} has $O(n^4)$ terms (these orders are true for both: directed and undirected networks). Thus, it is sufficient to show that
\begin{align}
\int_{-1}^1\int_{-1}^1&K(u)K(v)\frac{\IE\big(\kappa_{n,ij}(\theta_{0,n},t_0+uh)\kappa_{n,kl}(\theta_{0,n},t_0+vh)\big)}{\IP(C_{n,12}(t_0)=1)^2}\mathrm{d}u\mathrm{d}v \nonumber \\
&=\left\{\begin{array}{ccc}
o(n^2) &\text {for }& |\{i,j\}\cap\{k,l\}|=2 \\
o(n)   &\text {for }& |\{i,j\}\cap\{k,l\}|=1 \\
o(1)   &\text {for }& |\{i,j\}\cap\{k,l\}|=0. \\
\end{array}\right. \label{t1-t2}
\end{align}
For the proof of \eqref{t1-t2}, we note that 
$$\IE\big(\kappa_{n,ij}(\theta_{0,n},t_0+uh)\kappa_{n,kl}(\theta_{0,n},t_0+vh)\big)  = T_{n,1}(u,v) - T_{n,2}(u,v),$$
where
\begin{eqnarray*}
T_{n,1}(u,v)
&=&\IP(C_{n,ij}(t_0+uh)=1,C_{n,kl}(t_0+vh)=1)\\
&& \qquad \times f_{n,1}(\theta_{0,n},t_0+uh,t_0+vh|(i,j),(k,l)),\\
T_{n,2}(u,v)
&=&\IP(C_{n,ij}(t_0+uh)=1)\IP(C_{n,kl}(t_0+vh)=1) \\
&& \qquad \times f_{2}(\theta_{0,n},t_0+uh)f_{2}(\theta_{0,n},t_0+vh). \label{eq:part2}
\end{eqnarray*}
We get
\begin{eqnarray} \nonumber
&&\IP(C_{n,12}(t_0)=1)^{-2}
\int_{-1}^1\int_{-1}^1K(u)K(v)T_{n,2}(u,v) \mathrm{d}u\mathrm{d}v\\
 &&\qquad =\left [\int_{-1}^1K(u)a_n(u)f_{2}(\theta_{0,n},t_0+uh)\mathrm{d}u\right ]^2
 \to f_2(\theta_0(t_0),t_0)^2 , \label{yale1}
\end{eqnarray}
where, again, \eqref{eqa71} and continuity of $f_2(\theta, t) = - \partial_{\theta^2} g(\theta,t)$ has been used.
Furthermore,  we have that
\begin{align} \nonumber
&\IP(C_{n,12}(t_0)=1)^{-2}
\int_{-1}^1\int_{-1}^1K(u)K(v)T_{n,1}(u,v) \mathrm{d}u\mathrm{d}v \\ \nonumber
=&\int_{-1}^1\int_{-1}^1K(u)K(v)\frac{\IP(C_{n,ij}(t_0+uh)=1,C_{n,kl}(t_0+vh)=1)}{\IP(C_{n,12}(t_0)=1)^2} \\ \nonumber
&\times \left(f_{n,1}(\theta_{0,n},t_0+uh,t_0+vh|(i,j),(k,l))-f_1(\theta_0(t_0),|\{i,j\}\cap\{k,l\}|)\right)\mathrm{d}u\mathrm{d}v \\ \nonumber
&\qquad +\int_{-1}^1\int_{-1}^1K(u)K(v)\frac{\IP(C_{n,ij}(t_0+uh)=1,C_{n,kl}(t_0+vh)=1)}{\IP(C_{n,12}(t_0)=1)^2} \\ \nonumber
&\qquad \qquad \times
f_1(\theta_0(t_0),|\{i,j\}\cap\{k,l\}|)\mathrm{d}u\mathrm{d}v \\[10pt]& 
\left\{\begin{array}{lcc} =o(n^2) &\text{for}& |\{i,j\}\cap\{k,l\}|=2\\
=o(n) &\text{for}& |\{i,j\}\cap\{k,l\}|=1 \\
\to f_{1}(\theta_0(t_0),0)= f_2(\theta_0(t_0),t_0)^2 &\text{for}& |\{i,j\}\cap\{k,l\}|=0
\end{array}\right.  \label{yale2}
\end{align}
by Assumptions \eqref{eqa72} and \eqref{eqa81}. From \eqref{yale1} and \eqref{yale2}, we obtain \eqref{t1-t2}. This shows\eqref{lemDclaim2}.

For the proof of \eqref{lemDclaim3}, we calculate a  bound for the expectation of the absolute value of  the third derivative of $P_n$. With $s=t_0+uh$, it holds (recall that $\tau:=\sup_{\theta\in V}\|\theta\|$)
\begin{eqnarray*}
&&\IE\left(\sup_{k,l,r,\theta}\left|\partial_{\theta_k}\partial_{\theta_l}P_n^{'(r)}(\theta)\right|\right) \\
&\leq&\frac{1}{\IP(C_{n,12}(t_0)=1)}\int_{-1}^1K(u)\IP(C_{n,12}(s)=1)\\
&&\hspace*{3cm}\times \IE\left(\left.\|X_{n,12}(s)\|^3e^{\tau\>X_{n,12}(s)\|}\right| C_{n,12}(s)=1\right)\mathrm{d}u,
\end{eqnarray*}
where \eqref{eqa41} has been used to get that the order of differentiation and integration can be interchanged and where Fubini could be used because   all involved terms are non-negative. The upper bound for the expectation in the integral expression is bounded by Assumptions \eqref{eqa41} and \eqref{eqa71}. This shows \eqref{lemDclaim3}.
\end{proof}

\begin{lemma}
\label{lemE} Assume that Fact \ref{lemA}, (A3), (A4) \eqref{eqa42}, \eqref{eqa43},  (A6) \eqref{eqa74}, \eqref{eqa73}, (A7) \eqref{eqa82} hold.  It holds that
\begin{eqnarray} \nonumber
&&\frac{1}{l_nh}\sum_{i,j=1}^n\int_0^TK\left(\frac{s-t_0}{h}\right)C_{n,ij}(t_0)\\
&&\hspace*{1cm} \times\bigg[C_{n,ij}(s)X_{n,ij}(s)\left(e^{\theta_0(s)^TX_{n,ij}(s)}-e^{\theta_{0,n}^TX_{n,ij}(s)}\right) -\partial_{\theta}g(\theta_{0,n},s)\bigg]ds\nonumber \\
&&\hspace*{8cm} =o_P\left(\frac{1}{\sqrt{l_nh}}\right). \label{lemEclaim1}
\end{eqnarray}
With $B_n$ from Theorem \ref{thm:asymptotic_normality}, we have
\begin{eqnarray} \nonumber
&&\frac{1}{l_nh}\sum_{i,j=1}^n\int_0^TK\left(\frac{s-t_0}{h}\right)(1-C_{n,ij}(t_0)\\
&&\hspace*{2cm} \times C_{n,ij}(s)X_{n,ij}(s)\left(e^{\theta_0(s)^TX_{n,ij}(s)}-e^{\theta_{0,n}^TX_{n,ij}(s)}\right)ds)\nonumber\\
&&\hspace*{6cm}=h^2\cdot B_n+o_P(h^2). \label{lemEclaim2}\end{eqnarray}
\end{lemma}
\begin{proof}
In this proof, we use the shorthand notation $ K_{h,t_0}(s) = \frac{1}{h}K\big(\frac{z - t_0}{h}\big)$. We begin with proving \eqref{lemEclaim2}. Denote for vectors $a,b\in\IR^q$ by $[a,b]$ the connecting line between $a$ and $b$. Note firstly that by a Taylor series application for a random (depending on $X_{n,ij}(s)$) intermediate value $\theta^*(s)\in[\theta_0(s),\theta_{0,n}]$
\begin{align}
&e^{\theta_0(s)^TX_{n,ij}(s)}-e^{\theta_{0,n}^TX_{n,ij}(s)} \nonumber \\
&\hspace*{1cm}=X_{n,ij}(s)^Te^{\theta^*(s)^TX_{n,ij}(s)}\cdot(\theta_0(s)-\theta_{0,n}). \label{eq:bigOnessofr}
\end{align}
Hence, we obtain
\begin{eqnarray}
&&\frac{1}{l_n}\sum_{i,j=1}^n\int_0^TK_{h,t_0}(s)(1-C_{n,ij}(t_0))C_{n,ij}(s)\nonumber \\
&&\hspace*{3cm}\times X_{n,ij}(s)\left(e^{\theta_0(s)^TX_{n,ij}(s)}-e^{\theta_{0,n}^TX_{nij}(s)}\right)ds \nonumber \\
&=&\frac{1}{l_n}\sum_{i,j=1}^n\int_0^TK_{h,t_0}(s)(1-C_{n,ij}(t_0))C_{n,ij}(s)X_{n,ij}(s)X_{n,ij}(s)^T \nonumber \\[5pt]
&&\hspace*{1cm}\times\;\; e^{\theta^*(s)^TX_{n,ij}(s)}\cdot(\theta_0(s)-\theta_0(t_0)+\theta_0(t_0)-\theta_{0,n})\mathrm{d}s \label{eq:intermed}
\end{eqnarray}
We decompose \eqref{eq:intermed} into two terms by splitting $\theta_0(s)-\theta_0(t_0)+\theta_0(t_0)-\theta_{0,n}=(\theta_0(s)-\theta_0(t_0))+(\theta_0(t_0)-\theta_{0,n})$. For the second part we obtain, by using that $\|\theta^*(s)\|$ is bounded by $\tau$ (because $\theta_{0,n},\theta_0(t_0)\in V$ by Lemma \ref{lem:theta0ntothetat0t0} and $\theta_0$ is continuous), use also Fubini in the second line and rewrite as a conditional expectation in the last line
\begin{align}
&\IE\Bigg(\Bigg\|\frac{1}{l_n}\sum_{i,j=1}^n\int_0^TK_{h,t_0}(s)(1-C_{n,ij}(t_0))C_{n,ij}(s)X_{n,ij}(s)X_{n,ij}(s)^T \nonumber \\
&\hspace*{3cm}\times\; e^{\theta^*(s)^TX_{n,ij}(s)}\cdot(\theta_0(t_0)-\theta_{0,n})\mathrm{d}s\Bigg\|\Bigg) \nonumber \\
&\leq\int_0^TK_{h,t_0}(s)\IE\left(\tfrac{(1-C_{n,12}(t_0))C_{n,12}(s)}{\IP(C_{n,12}(t_0)=1)}\left\|X_{n,12}(s)\right\|^2e^{\tau\|X_{n,12}(s)\|}\right)ds\\[5pt]
 &\hspace*{3cm}  \times\;\|\theta_0(t_0)-\theta_{0,n}\| \nonumber \\[2pt]
&=\int_0^TK_{h,t_0}(s)\tfrac{\IP(C_{n,12}(t_0)=0,\,C_{n,12}(s)=1)}{\IP(C_{n,12}(t_0)=1)} \nonumber \\[3pt]
&\hspace*{1cm} \times \IE\left(\|X_{n,12}(s)\|^2e^{\tau\|X_{n,12}(s)\|}\Big|C_{n,12}(s)=1,\, C_{n,12}(t_0)=0\right)ds\\[3pt] 
&\hspace*{3cm} \times\|\theta_0(t_0)-\theta_{0,n}\|, \nonumber \\
&=O(h^3) \label{eq:archeresult}
\end{align}
where the last equality holds, because by assumption \eqref{eqa74} the first factor is $O(h)$, the second factor is uniformly bounded by \eqref{eqa43} and  $\|\theta_{0,n} - \theta_0(t_0)\| = O(h^2)$ by Proposition \ref{prop:thetanulln}. We now discuss the second term of the split of \eqref{eq:intermed}. Recall therefore the definitions of $\gamma_{n,ij}(s)$ and $\tau_{n,ij}(\theta,s)$ from Theorem \ref{thm:asymptotic_normality} and \eqref{eq:taunij}, respectively. Applying the above and using that $\theta_0(s)-\theta_0(t_0)=\theta'_0(t^*)(s-t_0)$ for an appropriate point $t^*\in[t_0,s]$, we obtain
\begin{align}
\eqref{eq:intermed} 
&=h^2\Big(\tfrac{1}{l_n}\sum_{i,j=1}^n\int_0^TK_{h,t_0}(s)\tfrac{\gamma_{n,ij}(s)}{h}X_{n,ij}(s)X_{n,ij}(s)^T \nonumber \\
&\quad\quad\times\;\; e^{\theta^*(s)^TX_{n,ij}(s)}\tfrac{\theta_0'(t^*)(t_0-s)}{h}\mathrm{d}s\Big)+o_P(h^2) \nonumber \\
&=h^2\Big(\tfrac{1}{l_n}\sum_{i,j=1}^n\int_0^TK_{h,t_0}(s)\tfrac{\gamma_{n,ij}(s)}{h}\tau_{n,ij}(\theta_0(s),s)\tfrac{\theta_0'(t_0)(t_0-s)}{h}\mathrm{d}s \nonumber \\
&+\tfrac{1}{l_n}\sum_{i,j=1}^n\int_0^TK_{h,t_0}(s)\tfrac{\gamma_{n,ij}(s)}{h}\tau_{n,ij}(\theta_0(s),s)\tfrac{(\theta_0'(t^*)-\theta_0'(t_0))(t_0-s)}{h}\mathrm{d}s \label{eq:teil1} \\
&\quad+\tfrac{1}{l_n}\sum_{i,j=1}^n\int_0^TK_{h,t_0}(s)\tfrac{\gamma_{n,ij}(s)}{h}X_{n,ij}(s)X_{n,ij}(s)^T \nonumber \\
&\quad\quad\quad\times\left(e^{\theta^*(s)^TX_{n,ij}(s)}-e^{\theta_0(s)^TX_{n,ij}(s)}\right)\tfrac{\theta_0'(t^*)(t_0-s)}{h}\mathrm{d}s\Big) \label{eq:teil2} \\
&+o_P(h^2). \nonumber
\end{align}
Hence, we need to prove that \eqref{eq:teil1} and \eqref{eq:teil2} are $o_P(1)$ (these lines individually without the leading $h^2$ from the first line) and we are done with the proof. $K$ is supported on $[-1,1]$ and hence $s\in U_h:=[t_0-h,t_0+h]$. Moreover, continuity of $\theta_0'$ yields $\sup_{s\in U_h}\frac{(\theta_0(t^*)-\theta_0'(t_0)(t_0-s)}{h}\to0$. Hence, we can show $\eqref{eq:teil1}=o_P(1)$ by similar arguments which lead to \eqref{eq:archeresult}. For \eqref{eq:teil2} we apply apply Taylor again to get for another intermediate point $\theta^{**}(s)\in[\theta_0(s),\theta^*(s)]$
$$e^{\theta^*(s)^TX_{n,ij}(s)}-e^{\theta_0(s)^TX_{n,ij}(s)}=X_{n,ij}(s)^Te^{\theta^{**}(s)^TX_{n,ij}(s)}(\theta^*(s)-\theta_0(s)).$$
Now arguments are again similar to the ones leading to \eqref{eq:archeresult}, we just have to use the power three part in \eqref{eqa43} and the fact that $\sup_{s\in U_h}\|\theta^*(s)-\theta_0(s)\|\leq\sup_{s\in U_h}\|\theta_0(s)-\theta_{0,n}\|$ which converges to zero by continuity of $\theta$ and Proposition \ref{prop:thetanulln}. This concludes the proof of \eqref{lemEclaim2}.

To prove  \eqref{lemEclaim1}, we have to show that
\begin{equation}
\label{eq:abc}
\frac{1}{\sqrt{l_nh}}\sum_{i,j}\int_0^TK\left(\frac{s-t_0}{h}\right)C_{n,ij}(t_0)r_{n,ij}(s)ds=o_P(1),
\end{equation}
where $r_{n,ij}(s)$ was defined before Assumption (A7). We do this by showing that every component of the left hand side of \eqref{eq:abc} is $o_P(1)$, i.e., we replace $r_{n,ij}(s)$ by $r_{n,ij}^{(a)}$ for $a\in\{1,...,q\}$. By an application of Markov's inequality, this holds if
\begin{eqnarray*}
&&\frac{h}{l_n}\sum_{(i,j),(k,l)}\int_{-1}^1\int_{-1}^1K(u)K(v)\IP(C_{n,ij}(t_0)=1,C_{n,kl}(t_0)=1) \\
&&\;\;\times\;\IE(r_{n,ij}^{(a)}(t_0+uh)r_{n,kl}^{(a)}(t_0+vh)|C_{n,ij}(t_0)=1, C_{n,kl}(t_0)=1)\mathrm{d}u\mathrm{d}v=o(1).
\end{eqnarray*}
We show this similarly as in the proof of Lemma \ref{lemD} by splitting the sum in three sums corresponding to $|\{i,j\} \cap \{k,l\}| = 2,1,$ or $0$. The corresponding sums have $O(n^2), O(n^3)$ and $O(n^4)$ terms, respectively. Before going through these three cases, we note that equations \eqref{eq:bigOnessofr} and \eqref{eqa42} imply that  $\sup_{u,v\in[-1,1]}\IE(r_{n,ij}^{(a)}(t_0+uh)r_{n,kl}^{(a)}(t_0+vh)|C_{n,ij}(t_0)=1, C_{n,kl}(t_0)=1)=O(h^2)$ for all $a\in\{1,...,q\}$ and for all $(i,j)$ and $(k,l)$. Now we get for the sum over edges with $|\{i,j\}\cap\{k,l\}|=2$ the bound
$$h\frac{\IP(C_{n,12}(t_0)=1)}{\IP(C_{n,12}(t_0)=1)}\int_{-1}^1\int_{-1}^1K(u)K(v) O(h^2)\mathrm{d}u\mathrm{d}v=o(1).$$
For the sum over edges with $|\{i,j\}\cap\{k,l\}|=1$, we get the following bound from \eqref{eqa73}
$$nh\IP(C_{n,12}(t_0)=1)\frac{\IP(C_{n,12}(t_0)=1,C_{n,23}(t_0)=1)}{\IP(C_{n,12}(t_0)=1)^2}O(h^2)=O(1)\cdot\frac{l_n}{n}O(h^3).$$
Observing that $\frac{l_n h^3}{n} = \frac{l_n^{3/5} (h^5)^{3/5} l_n^{2/5}}{n}  = O(\frac{ l_n^{2/5}}{n}) = O( n^{-1/5} P(C_{n,12}(t_0) =1)^{2/5}) = o(1)$, the bound is of order $o(1).$ 

By using  \eqref{eqa73} and \eqref{eqa82}, we get the following bound for the sum over edges with $|\{i,j\}\cap\{k,l\}|=0$:
 \begin{align*}
&l_nh\frac{\IP(C_{n,12}(t_0)=1,C_{n,34}(t_0)=1)}{\IP(C_{n,12}(t_0)=1)^2} \\
&\iint\limits_{[-1,1]^2}K(u)K(v)\IE\left(r_{n,12}^{(a)}(t_0+uh)r_{n,34}^{(a)}(t_0+vh)|C_{n,12}(t_0)=1, C_{n,34}(t_0)=1\right)\mathrm{d}u\mathrm{d}v \\
&\hspace*{1.5cm}=o(1).
\end{align*}
This concludes the proof of \eqref{lemEclaim1}.
\end{proof}

\begin{proposition}
\label{prop:thetatilde}
Assume that the assumptions of the Lemmas \ref{lemD} and \ref{lemE} hold. With probability tending to one, the equation $P'_n(\theta)=0$ (cf. equation \eqref{eq:thetatilde}) has a solution $\tilde{\theta}_n(t_0)$, which has the property $$\tilde{\theta}_n(t_0)=\theta_{0,n}+h^2\cdot B_n+o_P\left(\frac{1}{\sqrt{l_nh}}\right)+o_P(h^2).$$
\end{proposition}

\vspace*{8pt}

To prove this proposition,  we will make use of the following theorem, see \cite{Deimling1985}:\medskip

\begin{theorem}
\label{thm:kantorovich} (Newton-Kantorovich Theorem)
Let $R(x)=0$ be a system of equations where $R:D_0\subseteq\IR^q\to\IR$ is a function defined on $D_0$. Let $R$ be differentiable and denote by $R'$ its first derivative. Assume that there is  an $x_0$ such that all expressions in the following statements exist and such that the following statements are true
\begin{enumerate}
\item $||R'(x_0)^{-1}||\leq B$,
\item $||R'(x_0)^{-1}R(x_0)||\leq\eta$,
\item $||R'(x)-R'(y)||\leq K||x-y||$ for all $x,y\in D_0$,
\item $r:=BK\eta\leq\frac{1}{2}$ and $\Omega_*:=\{x:||x-x_0||<2\eta\}\subseteq D_0$.
\end{enumerate}
Then there is $x^*\in\Omega_*$ with $R(x^*)=0$ and
$$||x^*-x_0||\leq2\eta\textrm{ and }||x^*-(x_0-R'(x_0)^{-1}R(x_0))||\leq2r\eta.$$
\end{theorem}
%
\begin{proof}[Proof of Proposition \ref{prop:thetatilde}]
We show that
$P'_n(\theta)$ 
has a root by using Theorem \ref{thm:kantorovich} with $D_0=V$ and $x_0=\theta_{0,n}$. Lemma \ref{lemD} gives  that $P'_n(\theta_{0,n})\overset{\IP}\to0$ and $P''_n(\theta_{0,n})\overset{\IP}\to-\Sigma$. Since  $\Sigma$ is invertible we also have that the sequence of random variables $B_n:=||P_n''(\theta_{0,n})^{-1}||$ is well-defined (for large $n$) and that it is of order $O_P(1)$. Thus we also have $\eta_n:=||P''_n(\theta_{0,n})^{-1}P'_n(\theta_{0,n})||=o_P(1)$. For the Lipschitz continuity of $P_n''$ we bound the partial derivatives of $P_n''$ by Lemma \ref{lemD}. Hence we conclude that every realization of $P_n''$ is Lipschitz continuous with (random) Lipschitz constant $K_n=O_P(1)$. Combining everything, we get that $r_n:=B_nK_n\eta_n=o_P(1)$. Thus with probability tending to one we have $r_n\leq\frac{1}{2}$, and hence the Newton-Kantorovich Theorem tells us that with probability tending to one the equation $P'_n(\theta)=0$ has a solution $\tilde{\theta}_n(t_0)\in D_0=V$ with the property that
$$\|\tilde{\theta}_n(t_0)-\theta_{0,n}\|\leq2\eta_n=o_P(1).$$

To prove the asserted rate, we have to investigate $\eta_n$ further. We note first that since $P_n''(\theta_{0,n})^{-1}$ is stochastically bounded, the rate of $\eta_n$ is determined by the rate with which $P'_n(\theta_{0,n})$ converges to zero. To find this rate we observe that every summand of $P_n'(\theta_{0,n})$ has expectation zero conditionally on $C_{n,ij}(s)=1$:
\begin{align*}
&\int_0^TK\big(\tfrac{s-t_0}{h}\big)\IE\big[C_{n,ij}(s)X_{n,ij}(s)\big(e^{\theta_0(s)^TX_{n,ij}(s)}-e^{\theta_{0,n}^TX_{n,ij}(s)}\big)\big|C_{n,ij}(s)=1\big]ds \\
&\hspace*{2cm}=\int_0^TK\big(\tfrac{s-t_0}{h}\big)\partial_{\theta}g(\theta_{0,n},s)ds =0
\end{align*}
by the assumption that $\theta_{0,n}$ maximizes $\theta\mapsto\int_0^TK\left(\frac{s-t_0}{h}\right)g(\theta,s)ds$. So, in $P'_n(\theta_{0,n}),$ we can subtract $C_{n,ij}(t_0)\int_0^TK\left(\frac{s-t_0}{h}\right)\partial_{\theta}g(\theta_{0,n},s)ds$ from every summand without changing anything, i.e.,
\begin{align*}
&P'_n(\theta_{0,n}) \\
=&\frac{1}{l_nh}\sum_{i,j=1}^n\int_0^TK\left(\frac{s-t_0}{h}\right)\left[C_{n,ij}(s)X_{n,ij}(s)\left(e^{\theta_0(s)^TX_{n,ij}(s)}-e^{\theta_{0,n}^TX_{n,ij}(s)}\right)\right . \\
& \hspace*{6cm}  -C_{n,ij}(t_0)\partial_{\theta}g(\theta_{0,n},s)\Big]ds \\
=&\frac{1}{l_nh}\sum_{i,j=1}^n\int_0^TK\left(\frac{s-t_0}{h}\right)C_{n,ij}(t_0)\left[C_{n,ij}(s)X_{n,ij}(s)\left(e^{\theta_0(s)^TX_{n,ij}(s)} \right . \right . \\
& \hspace*{6cm}   \left . \left . -e^{\theta_{0,n}^TX_{n,ij}(s)}\right)-\partial_{\theta}g(\theta_{0,n},s)\right]ds \\
&+\frac{1}{l_nh}\sum_{i,j=1}^n\int_0^TK\left(\frac{s-t_0}{h}\right)(1-C_{n,ij}(t_0))C_{n,ij}(s)X_{n,ij}(s)  \\
& \qquad \times \left(e^{\theta_0(s)^TX_{n,ij}(s)}-e^{\theta_{0,n}^TX_{n,ij}(s)}\right)ds.
\end{align*}
By Lemma \ref{lemE},  this term is equal to $h^2\cdot B_n+o_P\left(\frac{1}{\sqrt{l_nh}}\right)+o_P(h^2)$, which concludes the proof of Proposition \ref{prop:thetatilde}.
\end{proof}

\begin{lemma}
\label{lemF}
Assume that the assumptions of Lemmas \ref{lemD} and \ref{lemE} hold. For $k,l\in\{1,...,q\},$ we have that
\begin{align} &\frac{1}{l_nh}\sum_{i,j=1}^n\int_0^TK\left(\frac{s-t_0}{h}\right)^2X_{n,ij}^{(l)}(s)X_{n,ij}^{(k)}(s)C_{n,ij}(s)\exp(\theta_0(s)^TX_{n,ij}(s))ds \nonumber \\
\overset{\IP}{\to}&\int_{-1}^1K(u)^2\mathrm{d}u\ \Sigma_{k,l} \label{eqlemF1}
\end{align}
and
\begin{align}
\frac{1}{l_n h}\sum_{i,j=1}^n&\int_0^TK\left(\frac{s-t_0}{h}\right)^2\|X_{n,ij}(s)\|^2 \Ind\left(\frac{1}{\sqrt{l_nh}}\left\|K\left(\frac{s-t_0}{h}\right)X_{n,ij}(s)\right\|>\epsilon\right) \nonumber \\
&\times C_{n,ij}(s)\exp(\theta_0(s)^TX_{n,ij}(s))ds\overset{\IP}{\to}0. \label{eqlemF2}
\end{align}
Moreover, it holds  that
\begin{equation} \label{eqlemF3}\frac{1}{l_n}\partial_{\theta}^2\ell(\tilde{\theta}_n(t_0),t_0)=P''_n(\tilde{\theta}_n(t_0))\overset{\IP}{\to}-\Sigma.\end{equation}
\end{lemma}

\begin{proof} The proof of \eqref{eqlemF1} follows by using similar arguments as in the proof of Lemma \ref{lemD},  with $\theta_{0,n}$ replaced by $\theta_0(s)$, and with $K$ replaced  by $K^2$.

For the proof of claim \eqref{eqlemF2}, we calculate the expectation of the left hand side of \eqref{eqlemF2}. Because the integrand is positive, we can apply Fubini,  and we get that the expectation is equal to 
\begin{eqnarray*}
&& \int_0^T  
\IE \left [ \Ind\left(\frac{1}{\sqrt{l_nh}}\left\|K\left( \frac{s-t_0}{h}\right)X_{n,12}(s)\right\|>\epsilon\right)
\|X_{n,12}(s)\|^2 \right . \\
&&\times \left.\exp\Big(\theta_0(s)^TX_{n,12}(s)\Big)\right|C_{n,12}(s)=1\Bigg] 
 \frac 1 h K\left(\frac{s-t_0}{h}\right)^2 \frac{\IP(C_{n,12}(s)=1)}{\IP(C_{n,12}(t_0)=1)} \mathrm{ds}\\
&\leq&\frac{1}{\varepsilon}\cdot \frac{1}{\sqrt{l_nh}} \int_{-1}^1K^3(u)\frac{\IP(C_{n,12}(t_0+uh)=1)}{\IP(C_{n,12}(t_0)=1)} \\
&&\IE\left(\left.\|X_{n,12}(t_0+uh)\|^3e^{\tau\|X_{n,12}(t_0+uh)\|}\right|C_{n,12}(t_0+uh)=1\right) \mathrm{d}u\\
&=&O\left (\frac{1}{\sqrt{l_nh}}\right ) = o(1).
\end{eqnarray*}
Here we use \eqref{eqa71}, $\max _{-1 \leq u \leq 1}  K(u) < \infty$ and \eqref{eqa44}. This shows \eqref{eqlemF2}.

To see \eqref{eqlemF3}, we show that 
\begin{equation} \label{eqlemF4}P_n''(\theta_{0,n})-P_n''(\tilde{\theta}_n(t_0))=o_P(1).\end{equation}
 This then implies  \eqref{eqlemF3} because of \eqref{lemDclaim2}.

By using exactly the same arguments as in the proof of Lemma \ref{lemE}, we obtain
$$e^{\theta_{0,n}^TX_{n,ij}(s)}-e^{\tilde{\theta}_n(t_0)_n^TX_{n,ij}(s)}\leq \|X_{n,ij}(s)\|e^{\tau\|X_{n,ij}(s)\|}\cdot {\|\theta_{0,n}-\tilde{\theta}_n(t_0)\|}.$$
This gives
\begin{align*}
&P_n''(\theta_{0,n})-P_n''(\tilde{\theta}_n(t_0)) \\
=&\frac{1}{l_nh}\sum_{i,j=1}^n\int_0^TK\left(\frac{s-t_0}{h}\right)C_{n,ij}(s)X_{n,ij}(s)X_{n,ij}(s)^T \\
&\quad\quad\quad\times\left(e^{\tilde{\theta}_n(t_0)^TX_{n,ij}(s)}-e^{\theta_{0,n}^TX_{n,ij}(s)}\right)ds \\
\leq&\frac{1}{l_nh}\sum_{i,j=1}^n\int_0^TK\left(\frac{s-t_0}{h}\right)C_{n,ij}(s)\|X_{n,ij}(s)\|^3e^{\tau\|X_{n,ij}(s)\|}ds\,\times {\|\theta_{0,n}-\tilde{\theta}_n(t_0)\|}.
\end{align*}
The expectation of the first factor is bounded because of assumptions  \eqref{eqa71} and \eqref{eqa44}. Furthermore, the second term is of order $o_P(1)$ by Proposition \ref{prop:thetatilde}. Thus, the product is of order $o_P(1)$. This shows \eqref{eqlemF4} and concludes the proof of \eqref{eqlemF3}.
\end{proof}

\begin{proposition}
\label{prop:asn}
Assume the assumptions of Proposition \ref{prop:thetatilde} and Lemma \ref{lemF} hold. With probability tending to one, $\partial_{\theta}\ell_T(\theta,t_0)=0$ has a solution $\hat{\theta}_n(t_0)$, and
$$\sqrt{l_nh}\cdot(\hat{\theta}_n(t_0)-\tilde{\theta}_n(t_0))\overset{d}\to N\Big(0,\int_{-1}^1 K^2(u) \mathrm{d} u \ \Sigma^{-1}\Big).$$
\end{proposition}

\vspace*{8pt}

\begin{proof}[Proof of Proposition \ref{prop:asn}]
The proof is based on modifications of arguments used in the asymptotic analysis of parametric counting process models, see e.g.\ the proof of Theorem VI.1.1 on p.\  422 in \cite{Anderson93}.
Define
\begin{eqnarray*}
U^l(\theta)&:=&h\partial_{\theta_l}\ell_T(\theta,t_0),\quad l = 1,\ldots,q,
\end{eqnarray*}
and let $U_t^l(\theta)$ be defined as $U^l(\theta),$ but with $t$ being the upper limit of the integral in \eqref{eq:loglik}, (i.e., $U^l(\theta)=U_T^l(\theta)$). Furthermore, we write $U(\theta)=(U^1(\theta),...,U^q(\theta)),$ and the vector $U_t(\theta)$ is defined analogously.  In the first step of the proof, we will show that
\begin{equation}
\label{eq:normality}
\frac{1}{\sqrt{l_nh}}U_{T}(\tilde{\theta}_n(t_0))\overset{d}\to N\Big(0,\int_{-1}^1 K^2(u) \mathrm{d} u\ \Sigma\Big).
\end{equation}
For the local, square integrable martingale $M_{n,ij}$ defined in \eqref{eqyalea}, it holds that $M_{n,ij}$ and $M_{n,i'j'}$ are orthogonal, meaning that $<M_{n,ij},M_{{n,i'j'}}>_t=0$ if $(i,j)\neq (i',j')$, i.e. the predictable covariation process is equal to zero. For the predictable variation process of $M_{n,ij},$ we have 
\begin{equation}
\label{eq:pvm}
<M_{n,ij}>_t=\int_0^tC_{n,ij}(s)\exp(\theta_0(s)^TX_{n,ij}(s))ds.
\end{equation}
By definition of $\tilde{\theta}_n(t_0),$ see the statement of Proposition \ref{prop:thetatilde}, we have that (write $K_{h,t_0}(s):=K\left(\frac{s-t_0}{h}\right)$)
\begin{align}
&U^l_t(\tilde{\theta}_n(t_0)) \nonumber \\
=&\sum_{i,j=1}^n\int_0^tK_{h,t_0}(s)X^{(l)}_{n,ij}(s)dN_{n,ij}(s)\\
\nonumber & \qquad -\int_0^tK_{h,t_0}(s)C_{n,ij}(s)X^{(l)}_{n,ij}(s)\exp(\tilde{\theta}_{n}(t_0)^TX_{n,ij}(s))ds \nonumber \\
=&\sum_{i,j=1}^n\int_0^tK_{h,t_0}(s)X_{n,ij}^{(l)}(s)dM_{n,ij}(s) \nonumber \\
+&\int_0^tK_{h,t_0}(s)C_{n,ij}(s)X_{n,ij}^{(l)}(s)\left(\exp(\theta_0(s)^TX_{n,ij}(s))-\exp(\tilde{\theta}_n(t_0)^TX_{n,ij}(s))\right)ds \nonumber \\
=&\sum_{i,j=1}^n\int_0^tK_{h,t_0}(s)X_{n,ij}^{(l)}(s)dM_{n,ij}(s). \nonumber
\end{align}
So $\tilde{\theta}_n(t_0)$ was chosen such that the non-martingale part of $\partial_{\theta}\ell(\tilde{\theta}_n(t_0),t_0)$ vanishes. Now, we want to apply Rebolledo's Martingale Convergence Theorem, see e.g.\ Theorem II.5.1 in \cite{Anderson93}. This theorem implies \eqref{eq:normality}, provided a Lindeberg condition \eqref {eqlemF2} holds, and  
\begin{eqnarray} \label{rbolcond}
&&\Big\langle\frac{1}{\sqrt{l_nh}}U_{t}^{k}(\tilde{\theta}_n(t_0)),\frac{1}{\sqrt{l_nh}}U_{t}^{l}(\tilde{\theta}_n(t_0))\Big\rangle_T\overset{\IP}\to \int_{-1}^1 K^2(u) \mathrm{d} u\ \Sigma_{kl}(t_0).
\end{eqnarray}
To verify \eqref{rbolcond}, first note that  \eqref{eq:pvm} and \eqref {eqlemF1} imply finiteness of
$$\frac{1}{l_nh}\sum_{i,j=1}^n\int_0^tK_{h,t_0}(s)^2\left(X^{(l)}_{n,ij}(s)\right)^2d\langle M_{n,ij}\rangle_s,$$
with probability tending to one. Note that Lemma \ref{lemF}  is formulated with $t=T$, but the integral is finite also for $t<T$ simply because the integrand is non-negative. From now on we assume the above integral is finite. The process
$$\frac{1}{\sqrt{l_nh}}\sum_{i,j=1}^n\int_0^tK_{h,t_0}(s)X_{n,ij}^{(l)}(s)dM_{n,ij}(s)$$
is a local square integrable martingale, see e.g.\ Theorem II.3.1 on p.71 in \cite{Anderson93}. Since the martingales $M_{n,ij}$ are orthogonal, and by using Lemma \ref{lemF},  the predictable covariation satisfies
\begin{eqnarray*}
&&\Big\langle\frac{1}{\sqrt{l_nh}}U_{t}^{k}(\tilde{\theta}_n(t_0)),\frac{1}{\sqrt{l_nh}}U_{t}^{l}(\tilde{\theta}_n(t_0)) \Big\rangle_T \\
&=&\frac{1}{l_nh}\sum_{i,j=1}^n\int_0^TK_{h,t_0}(s)^2X_{n,ij}^{(k)}(s)X_{n,ij}^{(l)}(s)C_{n,ij}(s)\exp(\theta_0(s)^TX_{n,ij}(s))ds \\
&\overset{\IP}\to& \int_{-1}^1 K^2(u) \mathrm{d} u\ \Sigma_{kl}(t_0).
\end{eqnarray*}
This shows \eqref{rbolcond}, and concludes the proof of \eqref{eq:normality}.

We now show that
\begin{equation} \label{approxreb} ||\sqrt{l_nh}\cdot(\tilde{\theta}_n(t_0)-\hat{\theta}_n(t_0))-\sqrt{l_nh}Z_n||\overset{\IP}\to0, \end{equation} where 
$$Z_n ={P''_n(\tilde{\theta}_n(t_0))^{-1}\frac{1}{l_nh}U(\tilde{\theta}_n(t_0))}. $$
We want to apply the Newton-Kantorovich Theorem \ref{thm:kantorovich} with $R(\theta):= R_n(\theta):=\frac{1}{l_nh}U(\theta)$, $D_0=V$ and $x_0:=\tilde{\theta}_n(t_0)$. To this end, define 
$$B_n:=\|R_n'(\tilde{\theta}_n(t_0))^{-1}\|=\left\|P''_n\left(\tilde{\theta}_n(t_0)\right)^{-1}\right\|.$$
From Lemma \ref{lemF},  we know that $P''_n(\tilde{\theta}_n(t_0))$ converges and is invertible for $n$ large enough, and thus $B_n=O_P(1)$. Now let
$$\eta_n:=\|R_n'(\tilde{\theta}_n(t_0))^{-1}R_n(\tilde{\theta}_n(t_0))\|=\left\|Z_n\right\|.$$
Results \eqref{eqlemF3}  and \eqref{eq:normality} imply that $\eta_n=o_P(1)$. Next, notice that $P''_n$ has a Lipschitz constant  $K_n$ that is bounded  by the maximum of the third derivative of $P_n$. According to \eqref{lemDclaim3}, this maximum is bounded, and we obtain $K_n=O_P(1)$. Hence, $r_n=B_nK_n\eta_n=o_P(1)$. Now, Theorem \ref{thm:kantorovich} implies that, with probability converging to one, there is $\hat{\theta}_n(t_0)$ such that $U(\hat{\theta}_n(t_0))=0$ and
$$||\hat{\theta}_n(t_0)-\tilde{\theta}_n(t_0)||\leq2\eta_n\overset{\IP}\to0.$$

To obtain the asymptotic distribution of $\hat{\theta}_n(t_0),$ we note that, by \eqref{eq:normality} and \eqref{eqlemF3},
\begin{equation} \label{limitreb} \sqrt{l_nh}\cdot Z_n\overset{d}\to N(0,\int_{-1}^1 K^2(u) \mathrm{d} u\ \Sigma^{-1}). \end{equation}
Thus it holds that  $\sqrt{l_nh}\cdot Z_n=O_P(1),$ and as a consequence we get $\sqrt{l_nh} \cdot\eta_n=O_P(1)$. Using the second statement of the Newton-Kantorovich Theorem \ref{thm:kantorovich}, we obtain
$$||\sqrt{l_nh}\cdot(\tilde{\theta}_n(t_0)-\hat{\theta}_n(t_0))-\sqrt{l_nh}Z_n||\leq \sqrt{l_nh}\cdot 2r_n\eta_n=o_P(1).$$
Thus $\sqrt{l_nh}\cdot (\tilde{\theta}_n(t_0)-\hat{\theta}_n(t_0))$ and $\sqrt{l_nh}\cdot Z_n$ have the same limit  distribution. Because of \eqref{limitreb} this implies the statement of the proposition.
\end{proof}

{\sc Proof of Theorem~\ref{thm:asymptotic_normality}}
Combining Propositions~\ref{prop:thetatilde} and \ref{prop:asn}, and applying Slutzky's Lemma, we obtain by the assumptions on the bandwidth $h$ in (A2)
$$\sqrt{l_nh}\left(\hat{\theta}_n(t_0)- \theta_{0,n}(t_0)-h^2\cdot B_n\right)\to N\left(0,\int_{-1}^1K^2(u)du\,\Sigma^{-1}A\Sigma^{-1}\right).$$
With Proposition \ref{prop:thetanulln} this gives \eqref{eqtheo21}.

\section{Acknowledgement}
The authors acknowledge support by the state of Baden-W\"{u}rttemberg through bwHPC and the German Research Foundation (DFG) through grant INST 35/1134-1 FUGG. Research of Alexander Krei{\ss} and Enno Mammen was supported by Deutsche Forschungsgemeinschaft through the Research Training Group RTG 1953. Research of W. Polonik has been supported by the National Science Foundation under Grant No. DMS 1713108.

We would like to thank Michael Gertz and Andreas Spitz for many helpful discussions and suggestions. In particular, we are very grateful to Andreas Spitz for spending a lot of time in helping us processing and preparing the data.

\newpage

\section{Appendix}

\subsection{Simulations of degree distributions, cluster coefficients and diameters.}

Here we report additional simulations of degree distributions, cluster coefficients and diameters. In Section  \ref{data-analysis}, we have presented results for the degree distribution of networks based on the Washington DC bikeshare activity on 7th December 2012. In this section, we will consider the days 18th April 2014 and 10th July 2015, and also compare diameters and
clustering coefficients of the simulated and observed networks. As above, using the corresponding estimated parameter value for each of these days, we compute 3840 predictions and compared them with the observed values. The \emph{diameter} of a network is the longest among the shortest path between two vertices in the network. Typically, in observed networks the diameter is much smaller than the number of vertices (cf.\ \cite{J08}).
The \emph{clustering coefficient} is the number of complete triangles (triples of vertices which are completely connected) divided by the number of incomplete triangles (triples of vertices with at least two edges). Note that every complete triangle is also incomplete, hence the clustering coefficient is between zero and one. The clustering coefficient can be understood as the empirical probability that vertices are connected given that there is a third vertex to which both are connected. It has been reported  (cf.\ \cite{J08}), that in observed networks this number is usually significantly higher than in an Erd\"os-R\'enyi  network, where the presence of edges are i.i.d.\ random variables. 

Our question here is, \emph{Does a network which was simulated by our model look like the observed network?} or in other words \emph{Could one believe that the observed network is a realization of our model?}. To answer this, we consider the three network characteristics mentioned above, and empirically and visually compare the simulated results to the observed data. The heuristic justification underlying this approach is, that, if considered jointly, these three characteristics are able to discriminate between a range of different types of networks (see also \cite{J08,ZAAL14})

We start by presenting results for diameter and clustering coefficient on 7th of December 2012. As described in Section \ref{data-analysis}, where the degree distribution was discussed, we divide the edges between bike stations in six regimes by considering tour frequencies between the stations on the day. Figure \ref{fig:diameters49} shows the histograms of the simulated diameter in the different regimes. We see that, in \ref{fig:diameters49_512} (as before in Figure \ref{fig:degrees49_512}), the simulation and the reality appear to coincide nicely. In other words, for a moderate number of tours our model seems to fit well. It is interesting to note that our model performs differently in the different regimes suggesting that edges with different activity have to be modeled differently. Finally, in Figure \ref{fig:cluster49}, we see the histograms of the simulated clustering coefficients. The true value in the corresponding regime is shown in the titles of the plots. Overall, the performance appears reasonable. In particular, in Figure \ref{fig:cluster49_46} the histogram is nicely centered around the true value. Interestingly, the performance in the fifth regime ($l_1=5$ and $l_2=12$), shown in Figure \ref{fig:cluster49_512}, is not as good as the others. One explanation for this might be that here different covariates are needed. 

\begin{figure}
\centering
\begin{subfigure}{0.35\textwidth}
\centering
\includegraphics[height=0.25\textheight,width=\linewidth]{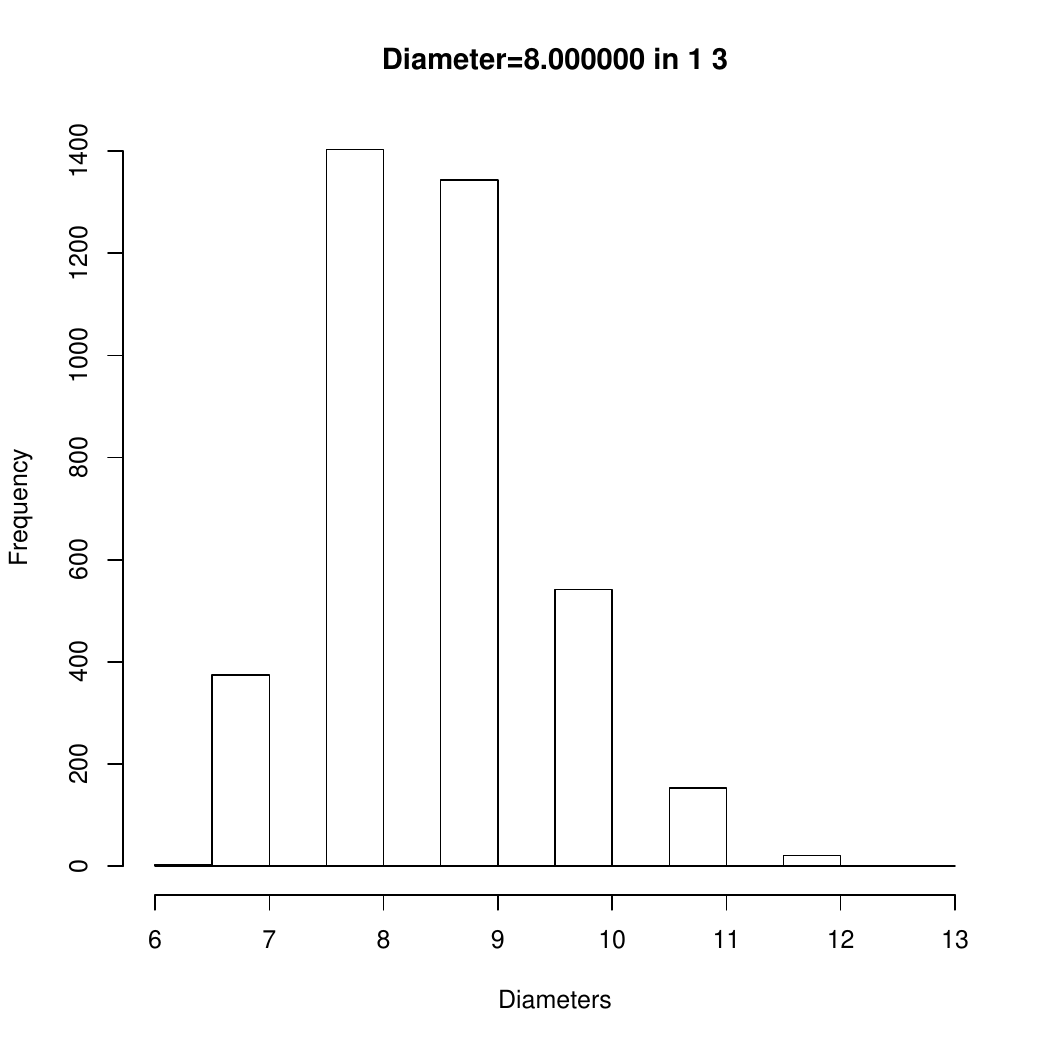}
\begin{minipage}{\textwidth}
\caption{Only edges with tour frequency between one and three}
\end{minipage}
\label{fig:diameters49_13}
\end{subfigure}%
\hspace*{1cm}
\begin{subfigure}{0.35\textwidth}
\centering
\includegraphics[height=0.25\textheight,width=\linewidth]{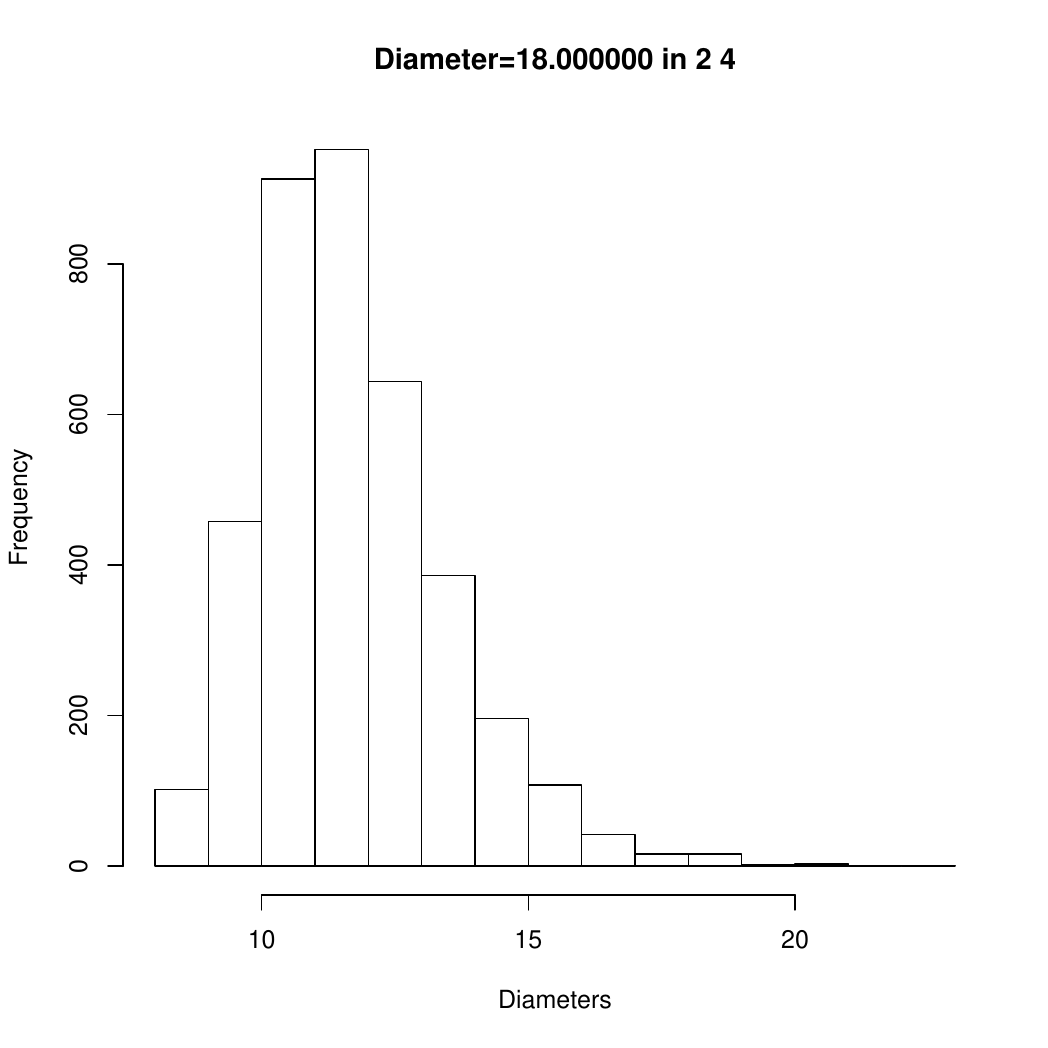}
\begin{minipage}{\textwidth}
\caption{Only edges with tour frequency between two and four}
\end{minipage}
\label{fig:diameters49_24}
\end{subfigure}

\begin{subfigure}{0.35\textwidth}
\centering
\includegraphics[height=0.25\textheight,width=\linewidth]{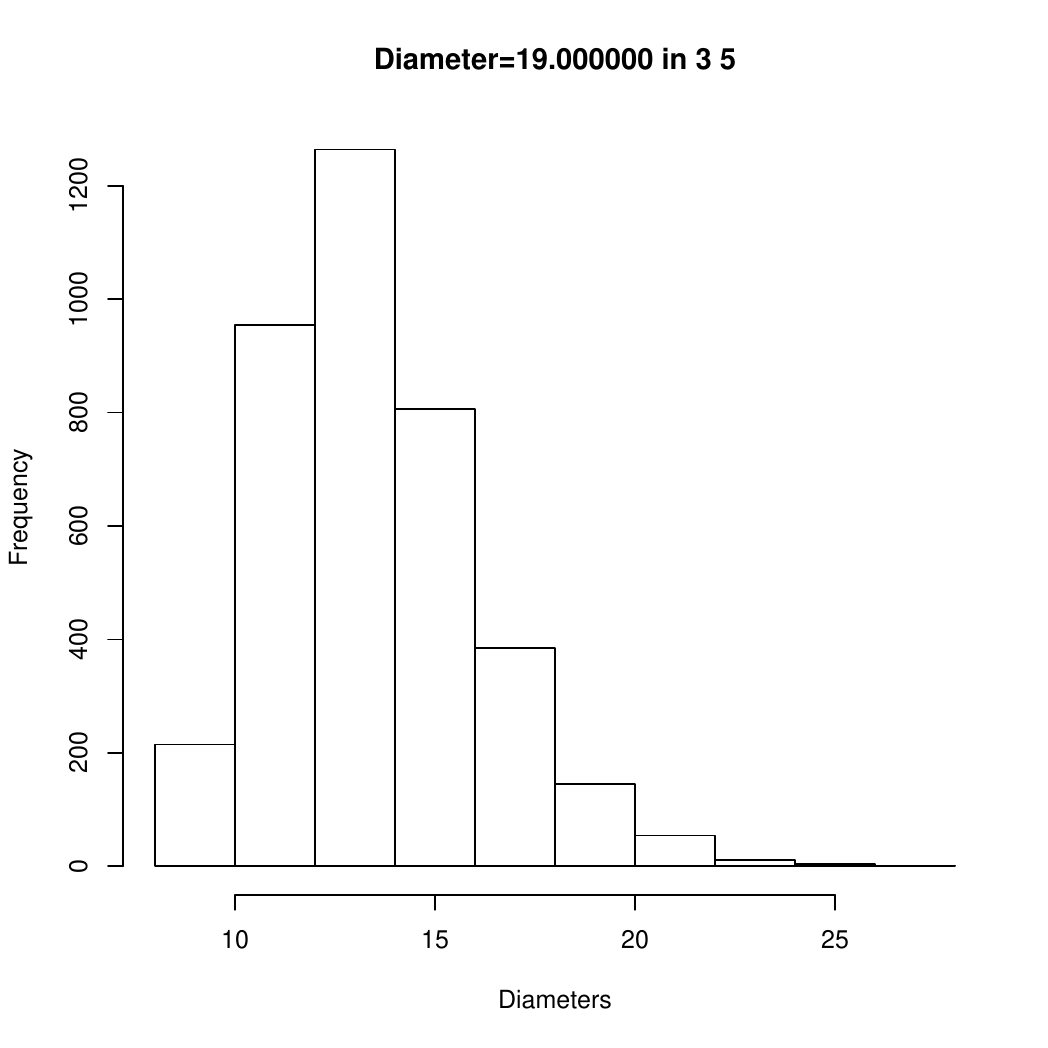}
\begin{minipage}{\textwidth}
\caption{Only edges with tour frequency between three and five}
\end{minipage}
\label{fig:diameters49_35}
\end{subfigure}%
\hspace*{1cm}
\begin{subfigure}{0.35\textwidth}
\centering
\includegraphics[height=0.25\textheight,width=\linewidth]{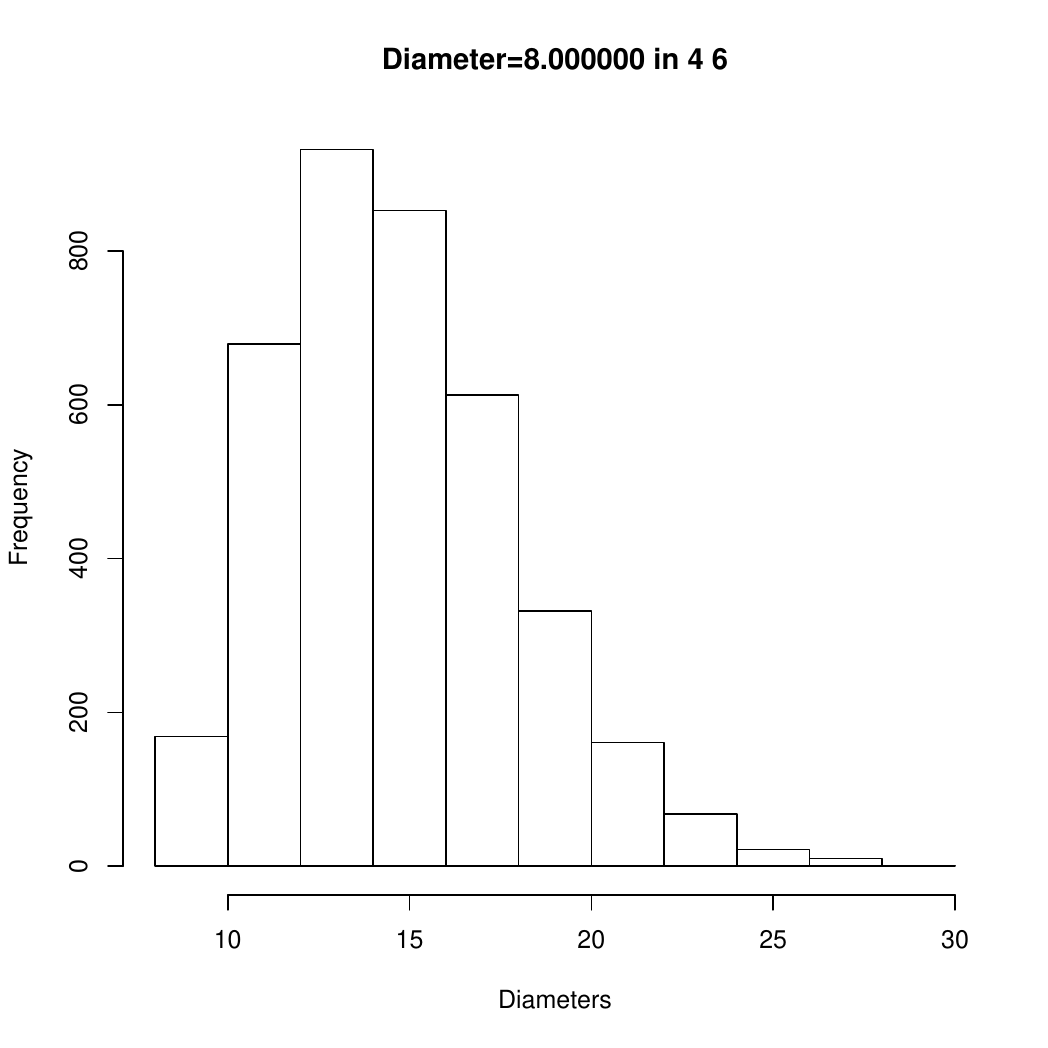}
\begin{minipage}{\textwidth}
\caption{Only edges with tour frequency between four and six}
\end{minipage}
\label{fig:diameters49_46}
\end{subfigure}

\begin{subfigure}{0.35\textwidth}
\centering
\includegraphics[height=0.25\textheight,width=\linewidth]{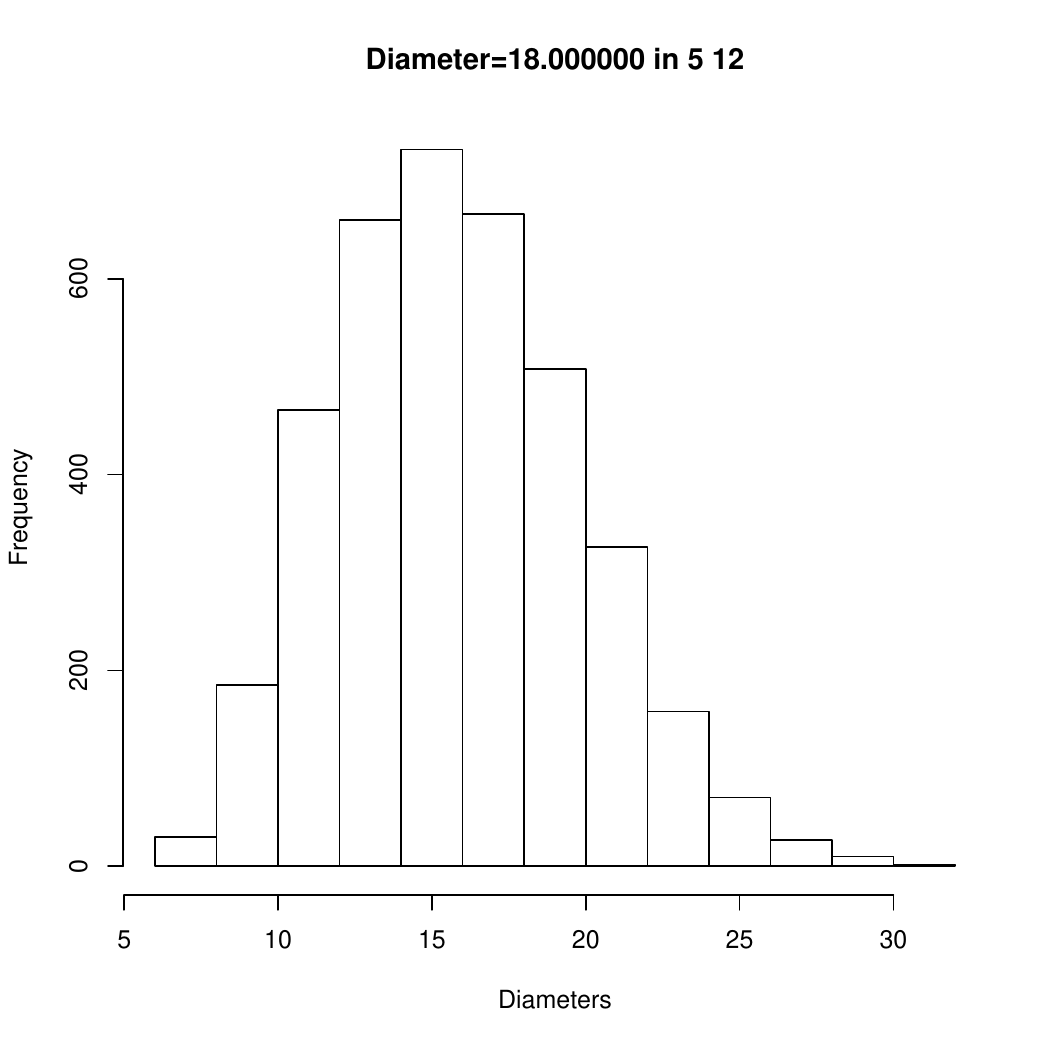}
\caption{Only edges with tour frequency between five and twelve}
\label{fig:diameters49_512}
\end{subfigure}%
\hspace*{1cm}
\begin{subfigure}{0.35\textwidth}
\centering
\includegraphics[height=0.25\textheight,width=\linewidth]{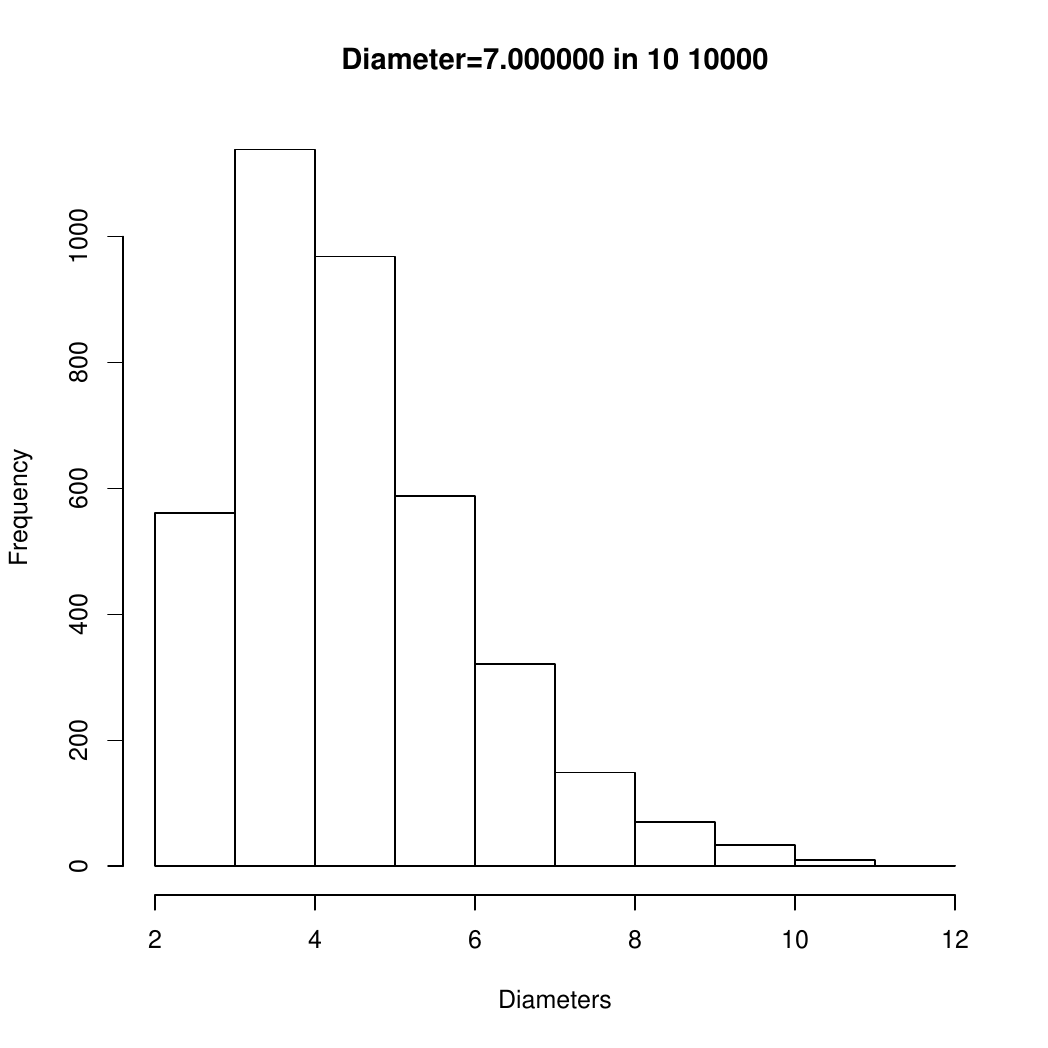}
\caption{Only edges with tour frequency larger than ten}
\label{fig:diameters49_10inf}
\end{subfigure}
\caption{Histograms of diameters of the graphs which arise by taking different edges into account (see individual caption) from simulations for 7th December 2012. In the title of the plot the observed value is shown.}
\label{fig:diameters49}
\end{figure}

\begin{figure}
\centering
\begin{subfigure}{0.35\textwidth}
\centering
\includegraphics[width=\linewidth]{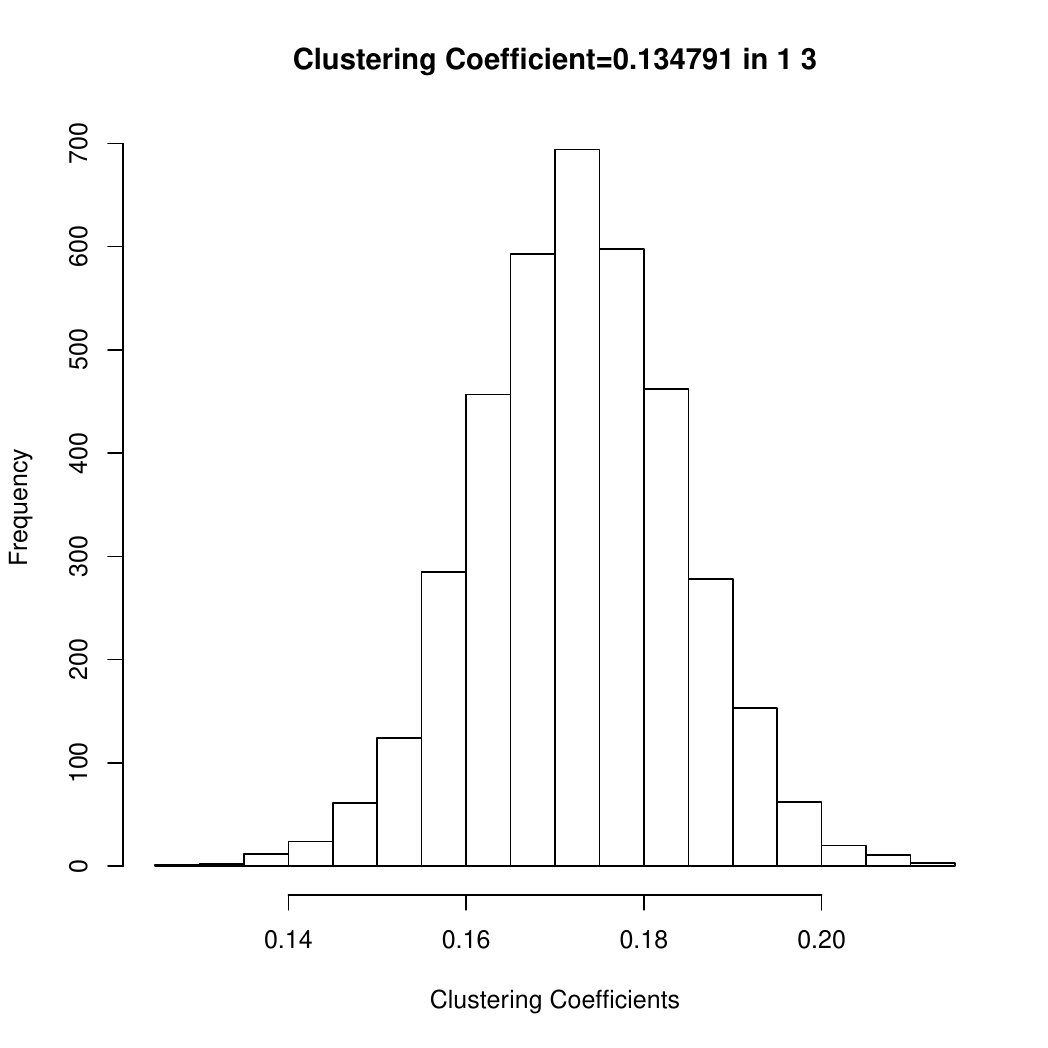}
\begin{minipage}{\textwidth}
\caption{Only edges with tour frequency between one and three}
\end{minipage}
\label{fig:cluster49_13}
\end{subfigure}%
\hspace*{1cm}
\begin{subfigure}{0.35\textwidth}
\centering
\includegraphics[width=\linewidth]{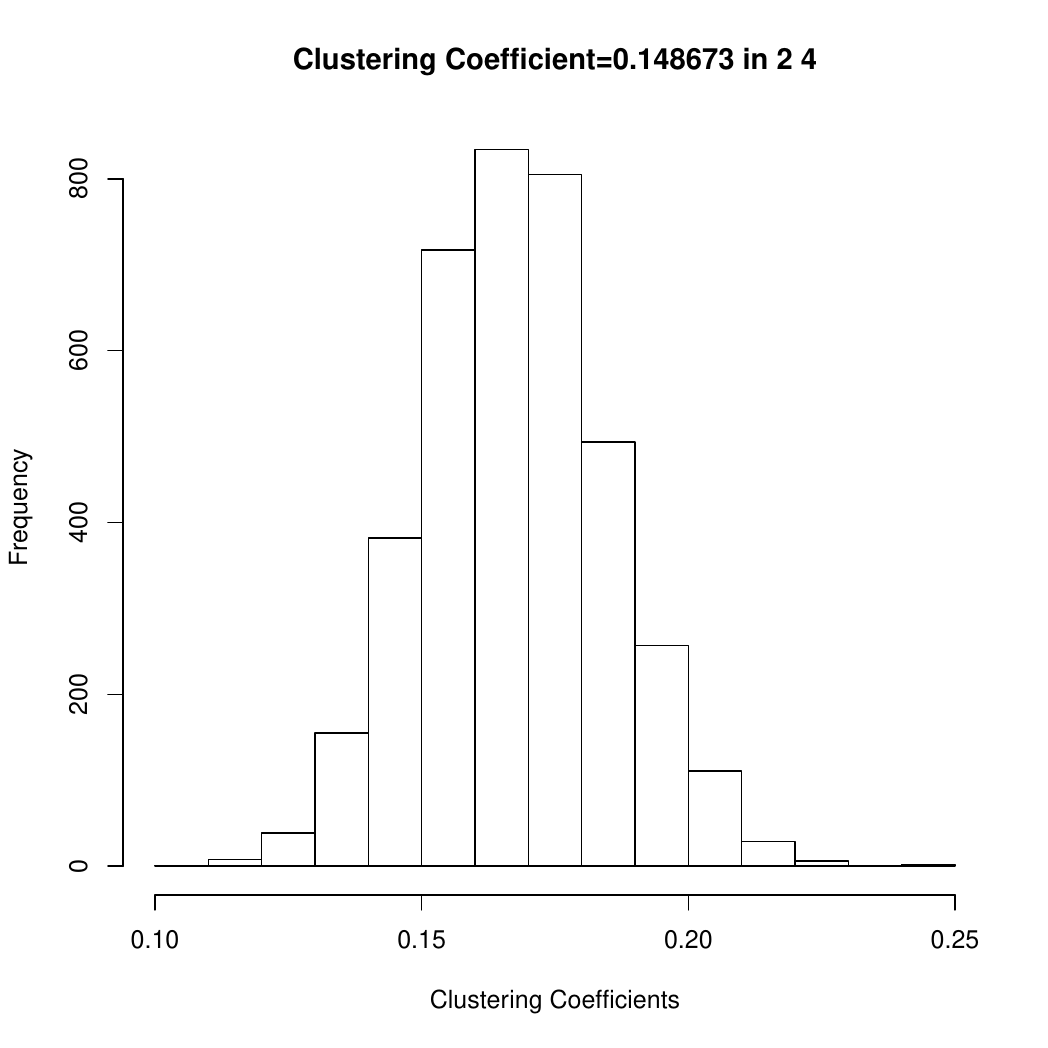}
\caption{Only edges with tour frequency between two and four}
\label{fig:cluster49_24}
\end{subfigure}

\begin{subfigure}{0.35\textwidth}
\centering
\includegraphics[width=\linewidth]{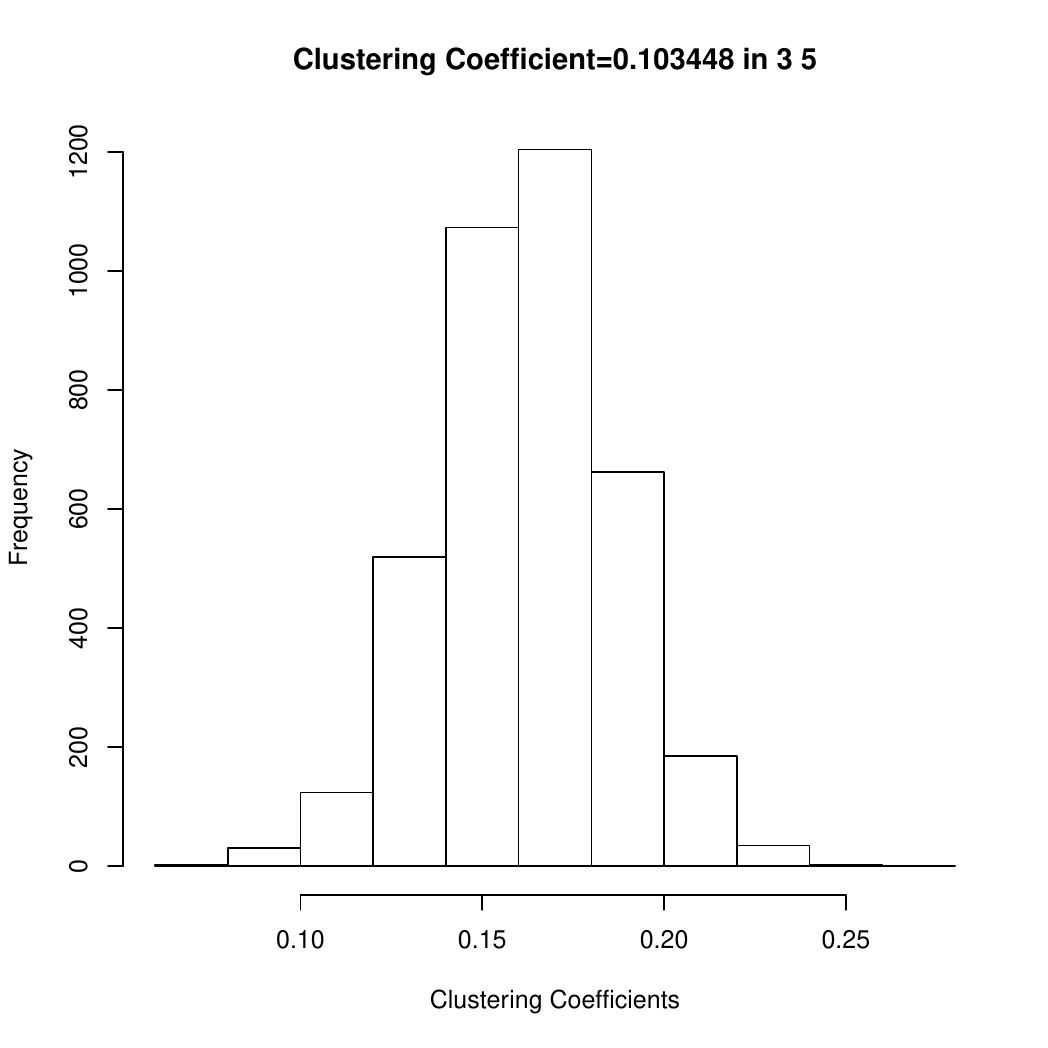}
\caption{Only edges with tour frequency between three and five}
\label{fig:cluster49_35}
\end{subfigure}%
\hspace*{1cm}
\begin{subfigure}{0.35\textwidth}
\centering
\includegraphics[width=\linewidth]{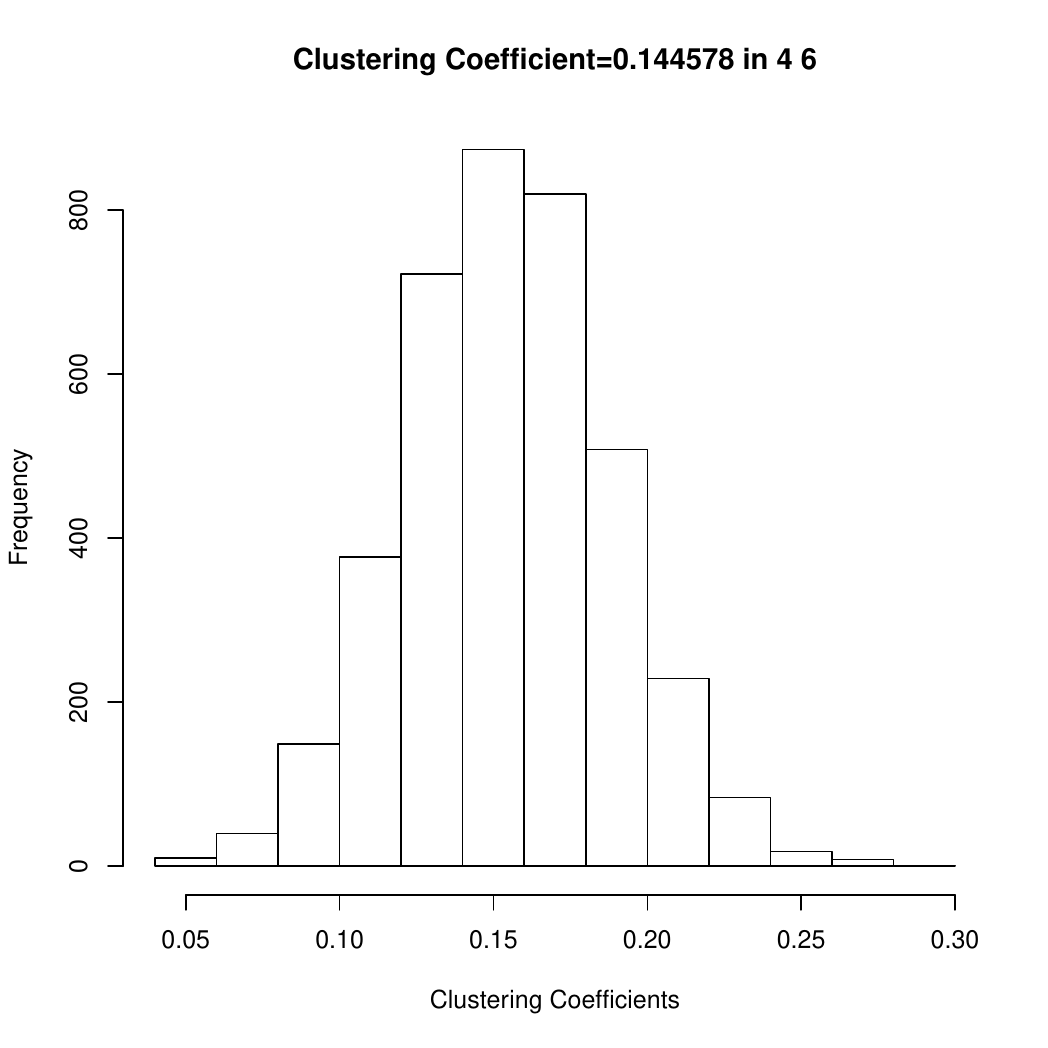}
\caption{Only edges with tour frequency between four and six}
\label{fig:cluster49_46}
\end{subfigure}

\begin{subfigure}{0.35\textwidth}
\centering
\includegraphics[width=\linewidth]{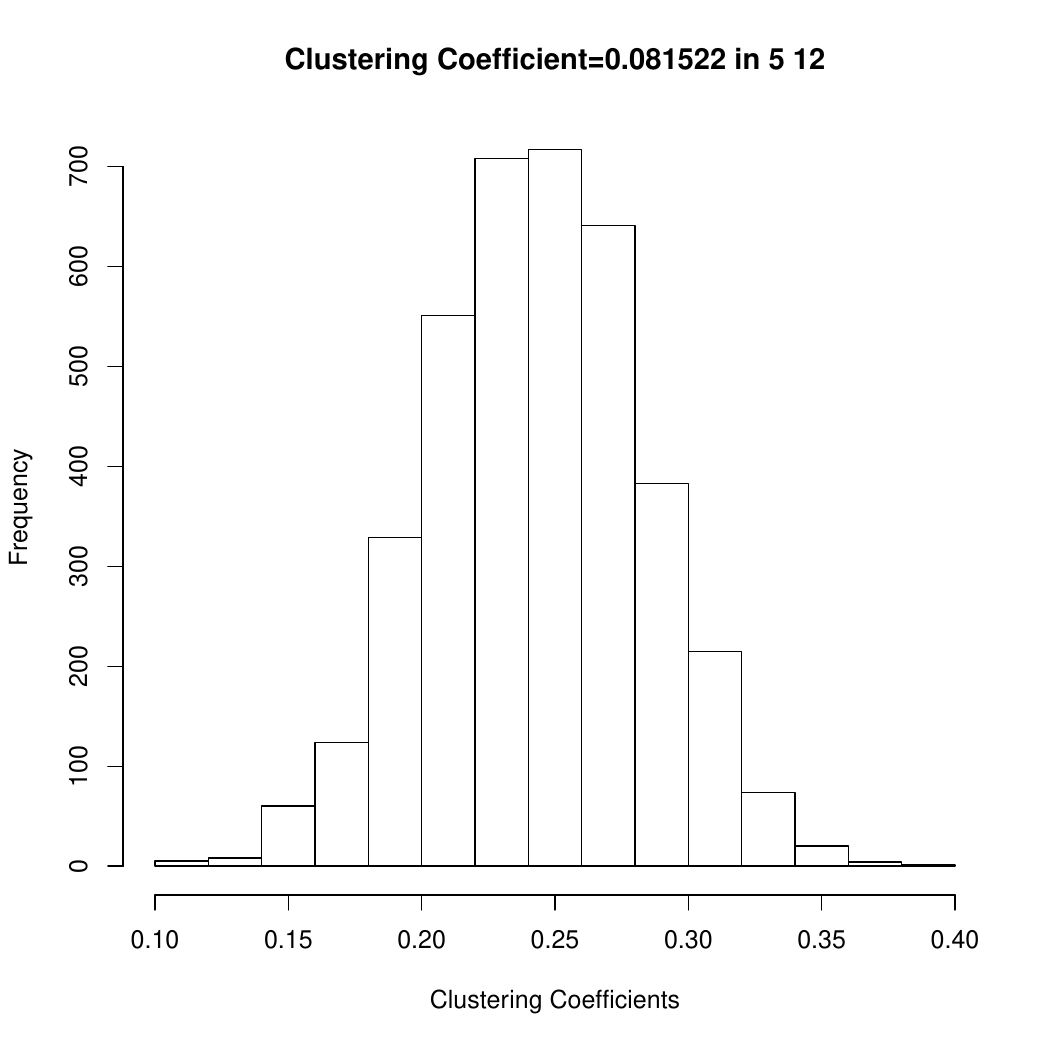}
\caption{Only edges with tour frequency between five and twelve}
\label{fig:cluster49_512}
\end{subfigure}%
\hspace*{1cm}
\begin{subfigure}{0.35\textwidth}
\centering
\includegraphics[width=\linewidth]{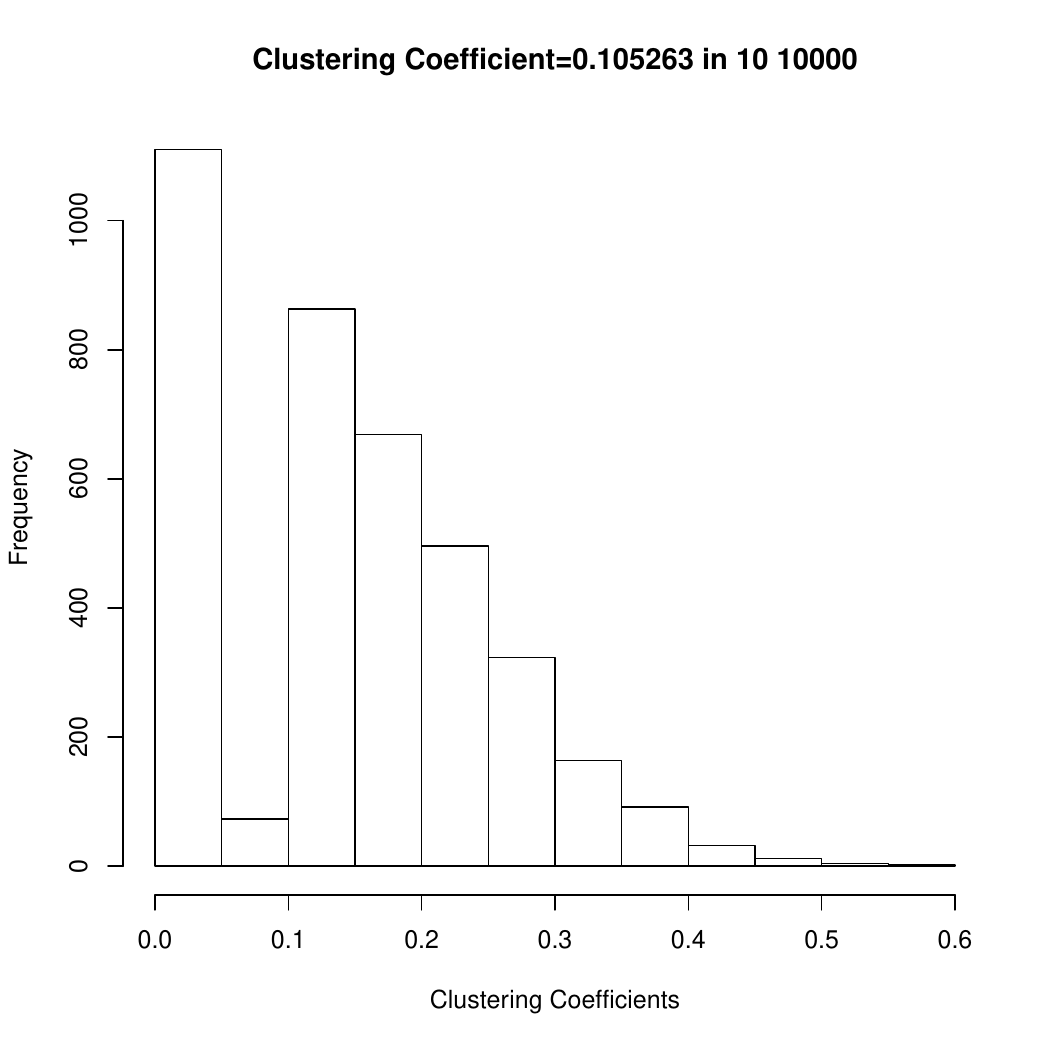}
\caption{Only edges with tour frequency larger than ten}
\label{fig:cluster49_10inf}
\end{subfigure}
\caption{Histograms of clustering coefficients of the graphs which arise by taking different edges into account (see individual caption) from simulations for 7th December 2012. In the title of the plot the observed value is shown.}
\label{fig:cluster49}
\end{figure}

In Figure \ref{fig:graphDec} we see one simulated graph compared to the true graph. The color of the edges determine how many tours happened relative the the other edges: The lowest 25\% of the edges are colored green, the next 25\% yellow, then orange and the highest 25\% of edges are colored red. Due to the integral value of the activity it is not the case that exactly 25\% of the edges are green and so on. The size of the vertices is relative to their degree. We see that the model is able to find the \emph{important} (i.e. high degree) vertices. For the edges we see that some red edges are at wrong places. But generally the vertices with high profile edges are recognized. The remaining graphs in Figure \ref{fig:graphs} show the same comparison for the two other dates under consideration. And we see that the results are similar.

Figures \ref{fig:degrees120} till \ref{fig:cluster184} show the results of the corresponding simulations for the other two dates. Overall the results are similar. It should be pointed out that even though the model is not able to reproduce every feature perfectly accurate, the simulated networks are still very close to the true observation. This becomes more obvious if we remind ourselves that only six parameters were used.
\begin{figure}
\centering
\begin{subfigure}{0.35\textwidth}
\centering
\includegraphics[width=\linewidth]{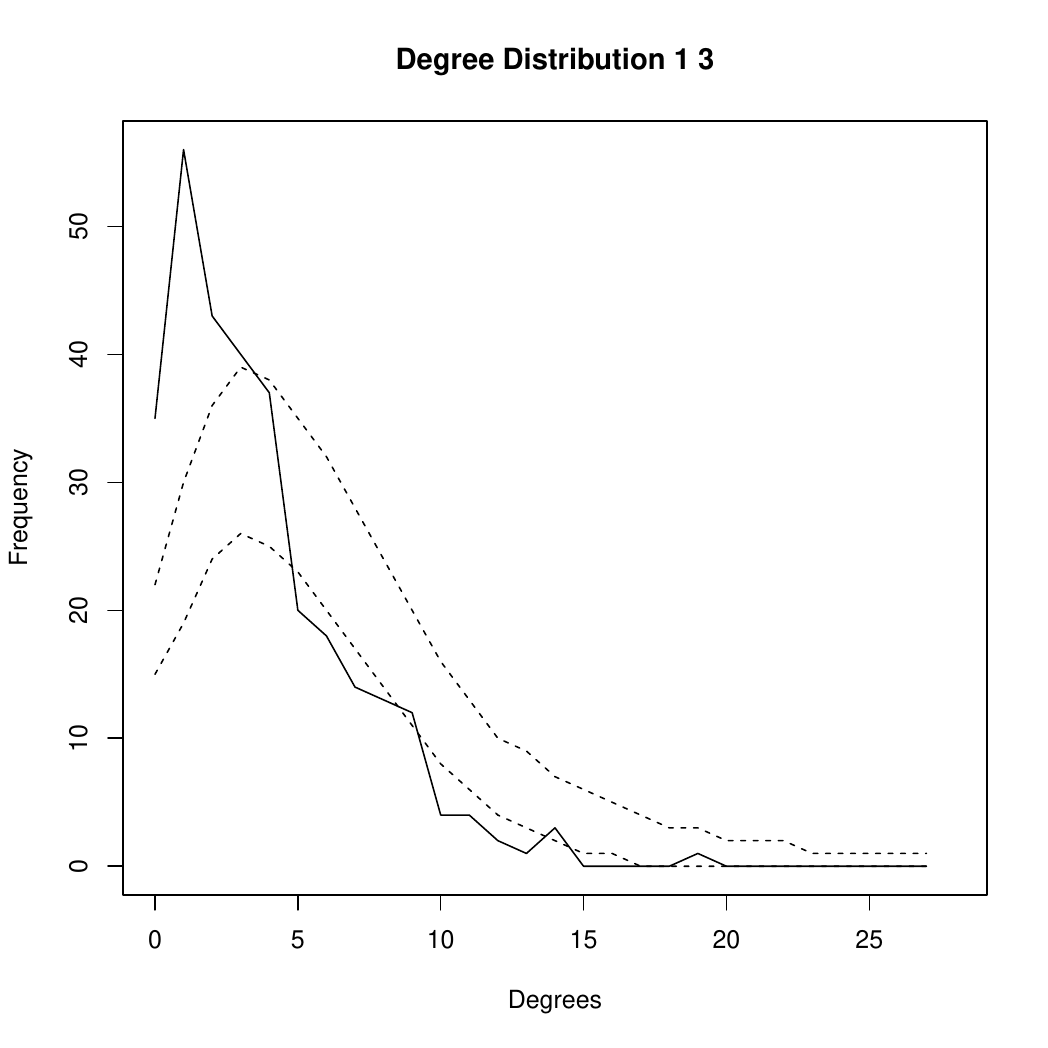}
\caption{Only edges with tour frequency between one and three}
\label{fig:degrees120_13}
\end{subfigure}%
\hspace*{1cm}
\begin{subfigure}{0.35\textwidth}
\centering
\includegraphics[width=\linewidth]{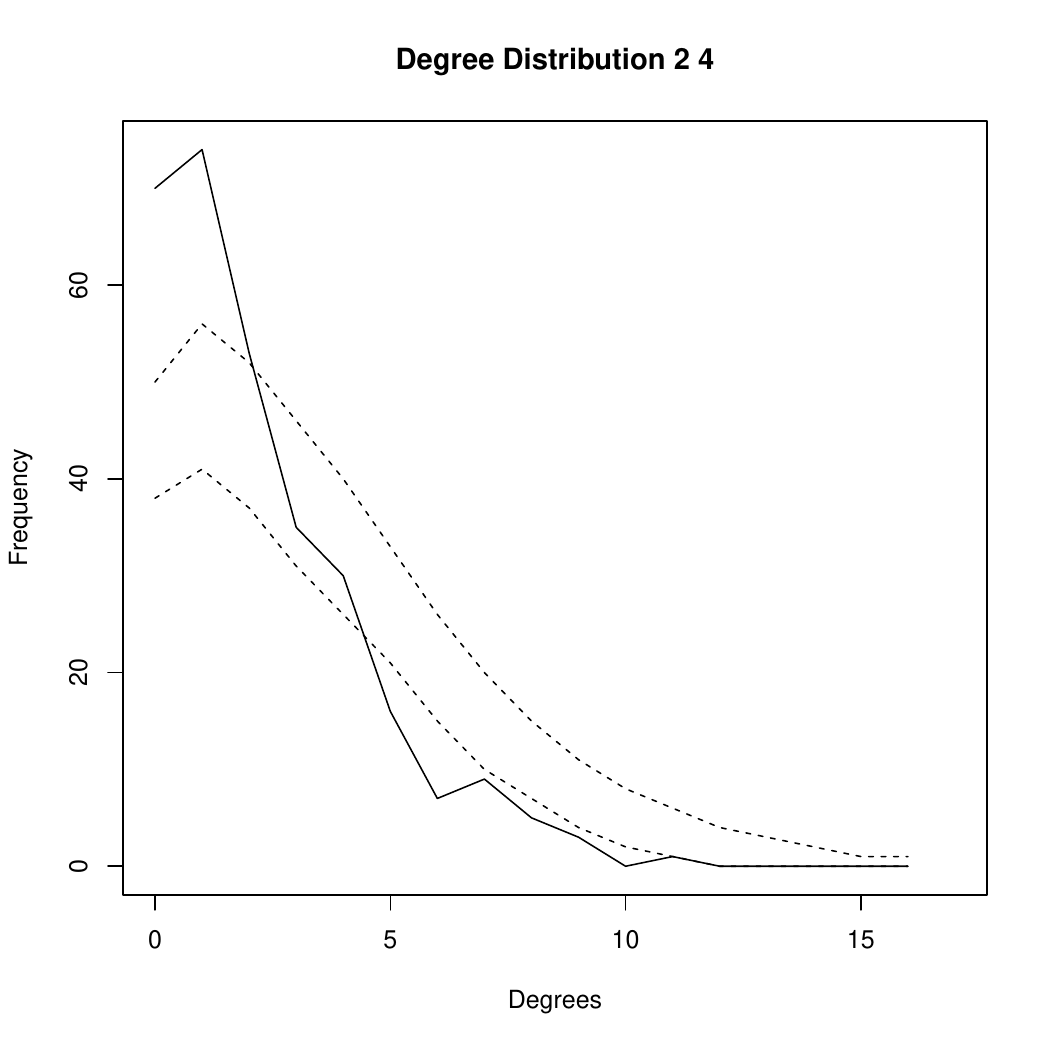}
\caption{Only edges with tour frequency between two and four}
\label{fig:degrees120_24}
\end{subfigure}

\begin{subfigure}{0.35\textwidth}
\centering
\includegraphics[width=\linewidth]{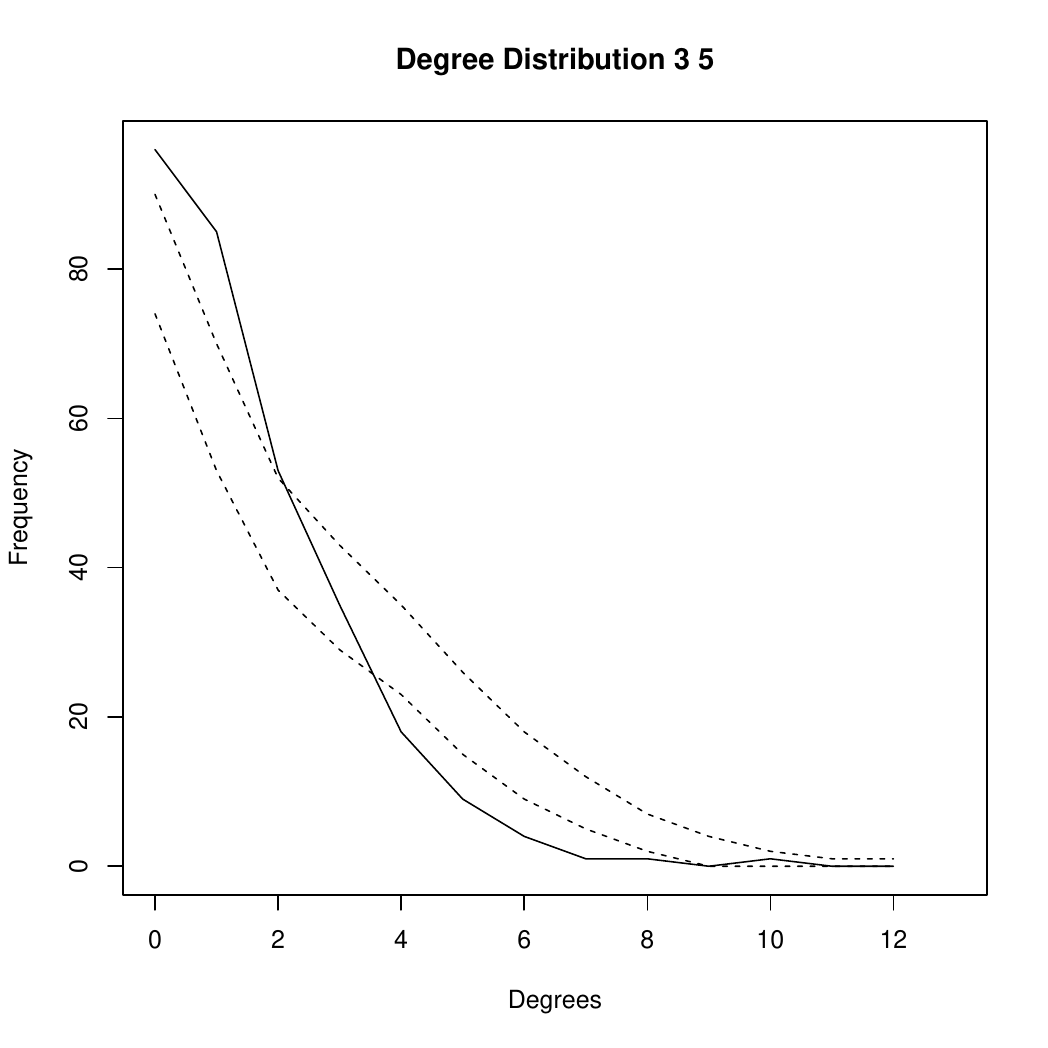}
\caption{Only edges with tour frequency between three and five}
\label{fig:degrees120_35}
\end{subfigure}%
\hspace*{1cm}
\begin{subfigure}{0.35\textwidth}
\centering
\includegraphics[width=\linewidth]{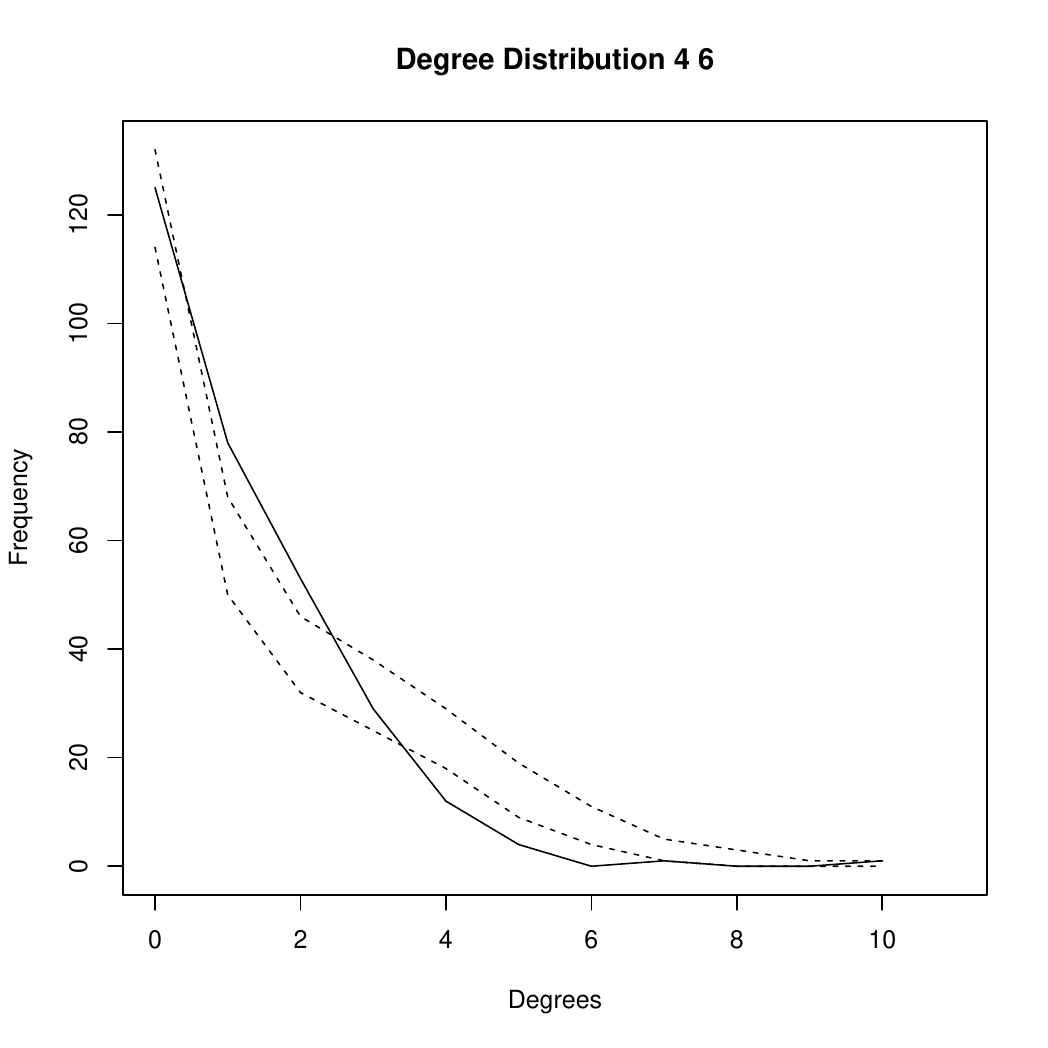}
\caption{Only edges with tour frequency between four and six}
\label{fig:degrees120_46}
\end{subfigure}

\begin{subfigure}{0.35\textwidth}
\centering
\includegraphics[width=\linewidth]{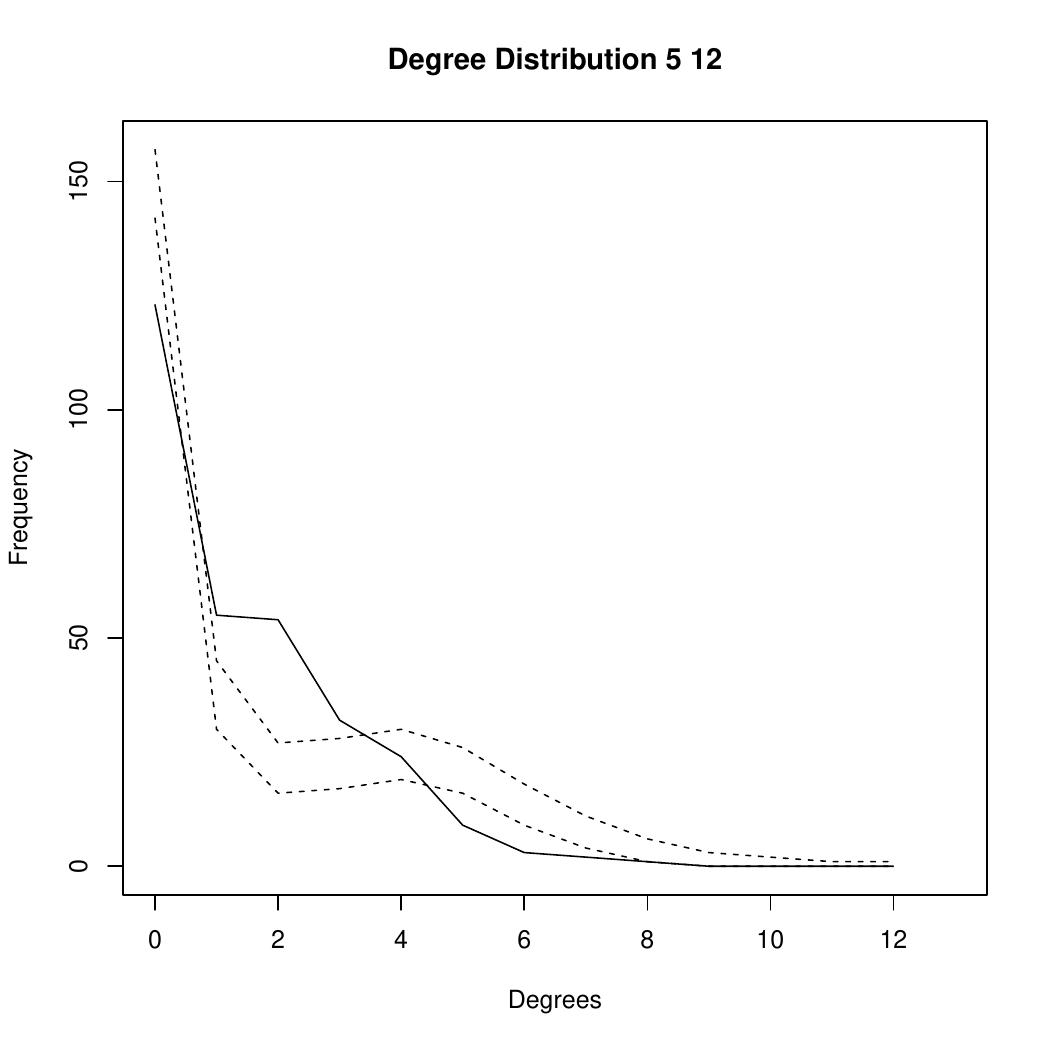}
\caption{Only edges with tour frequency between five and twelve}
\label{fig:degrees120_512}
\end{subfigure}%
\hspace*{1cm}
\begin{subfigure}{0.35\textwidth}
\centering
\includegraphics[width=\linewidth]{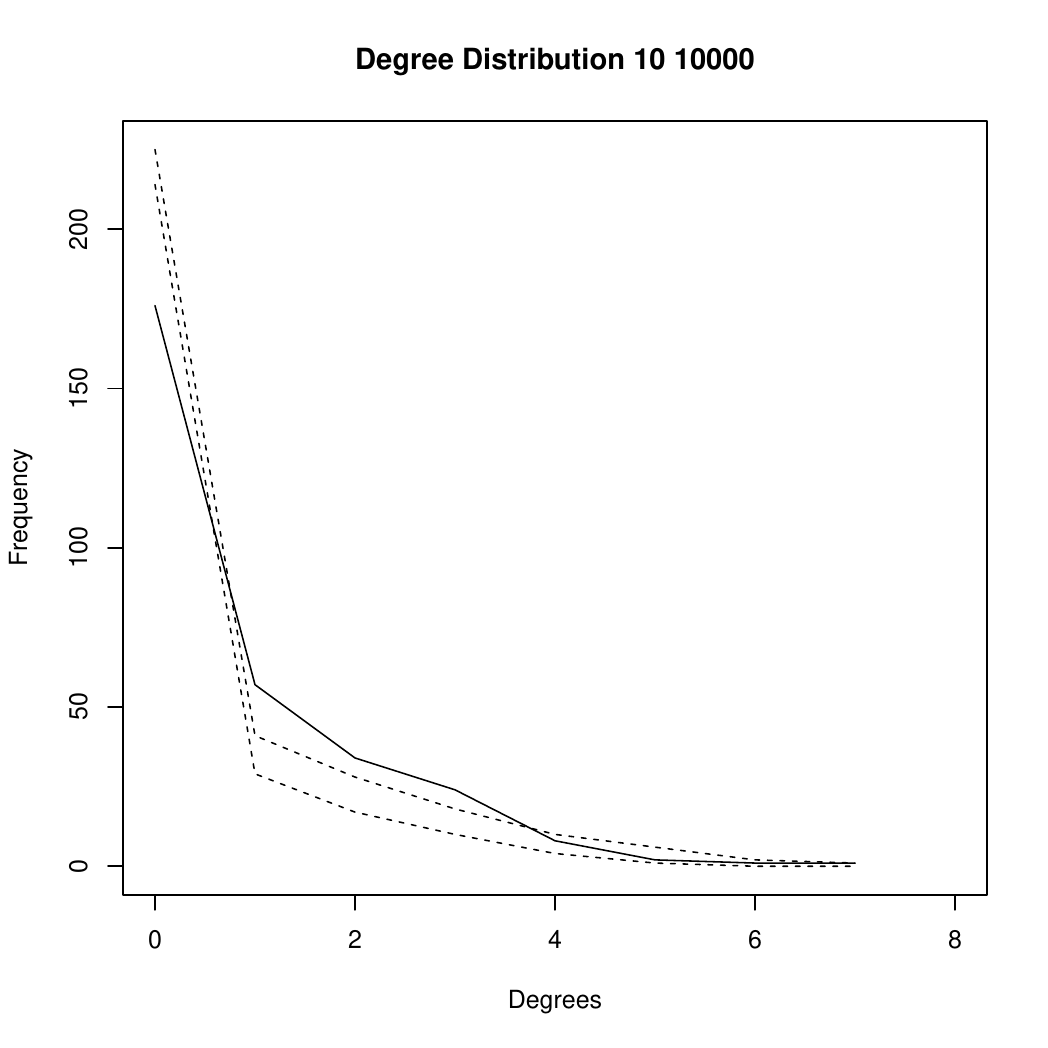}
\caption{Only edges with tour frequency between larger than ten}
\label{fig:degrees120_10inf}
\end{subfigure}
\caption{Degree distributions of the graphs which arise by taking different edges into account (see individual caption) from simulations for 18th April 2014. Dotted lines show 10\% and 90\% quantiles of simulations and solid line shows true distributions.}
\label{fig:degrees120}
\end{figure}

\begin{figure}
\centering
\begin{subfigure}{0.35\textwidth}
\centering
\includegraphics[width=\linewidth]{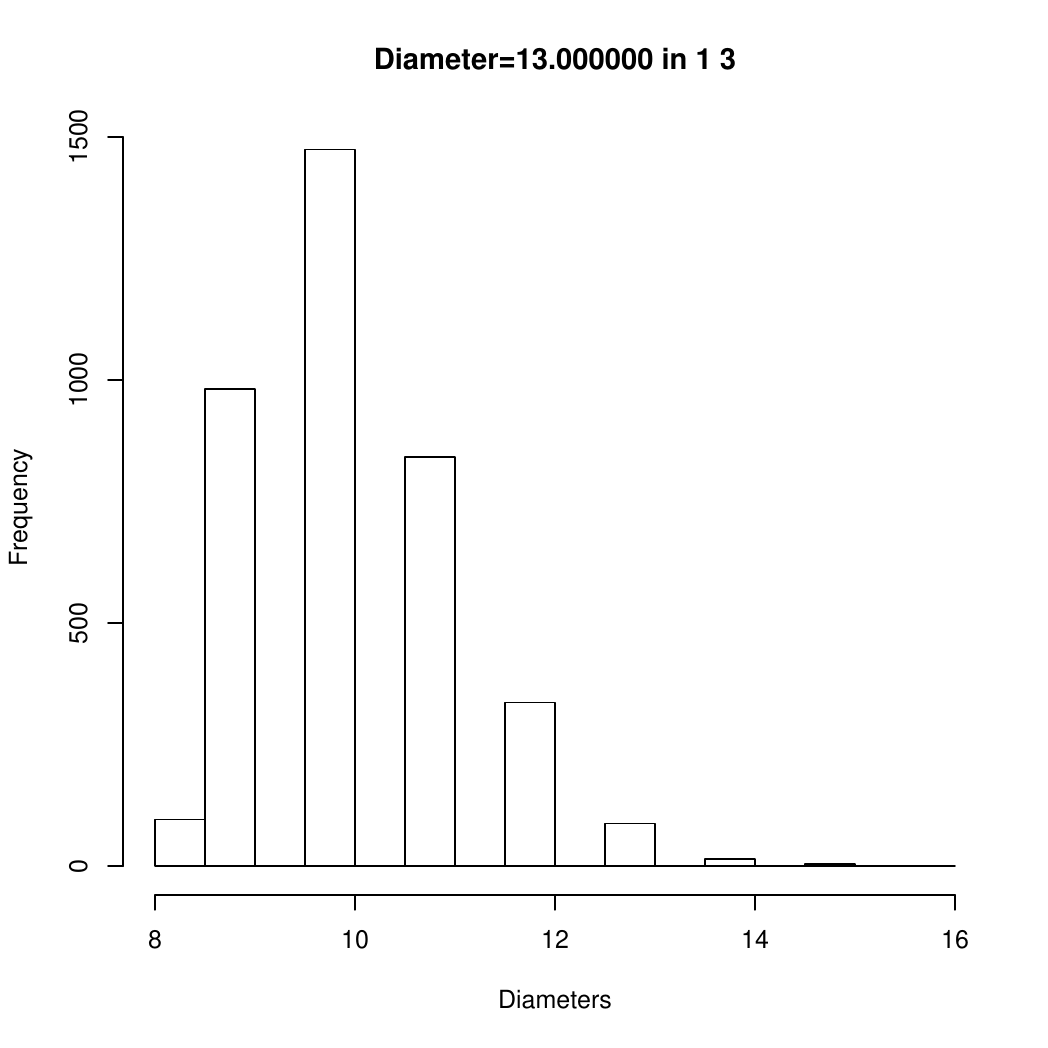}
\caption{Only edges with tour frequency between one to three}
\label{fig:diameters120_13}
\end{subfigure}%
\hspace{1cm}
\begin{subfigure}{0.35\textwidth}
\centering
\includegraphics[width=\linewidth]{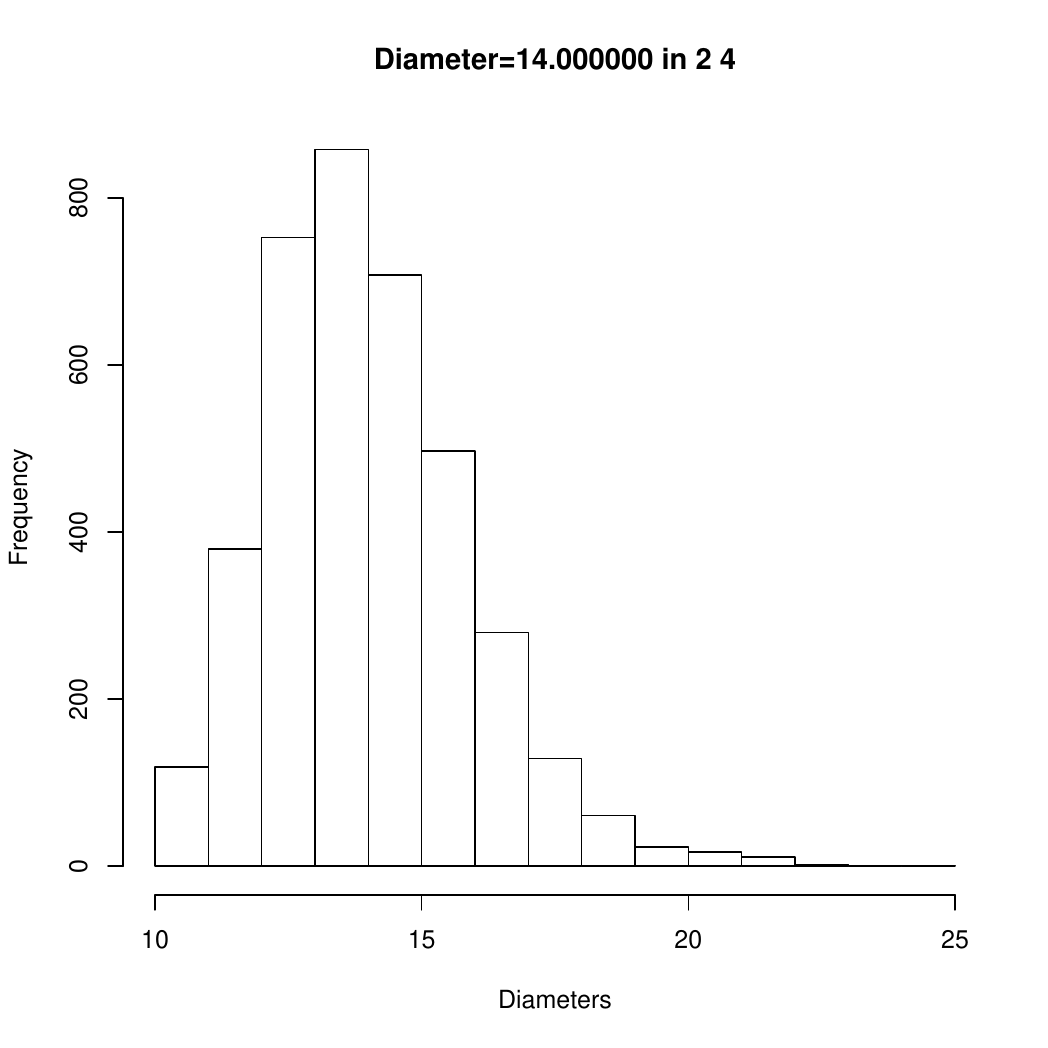}
\caption{Only edges with tour frequency between two to four}
\label{fig:diameters120_24}
\end{subfigure}

\begin{subfigure}{0.35\textwidth}
\centering
\includegraphics[width=\linewidth]{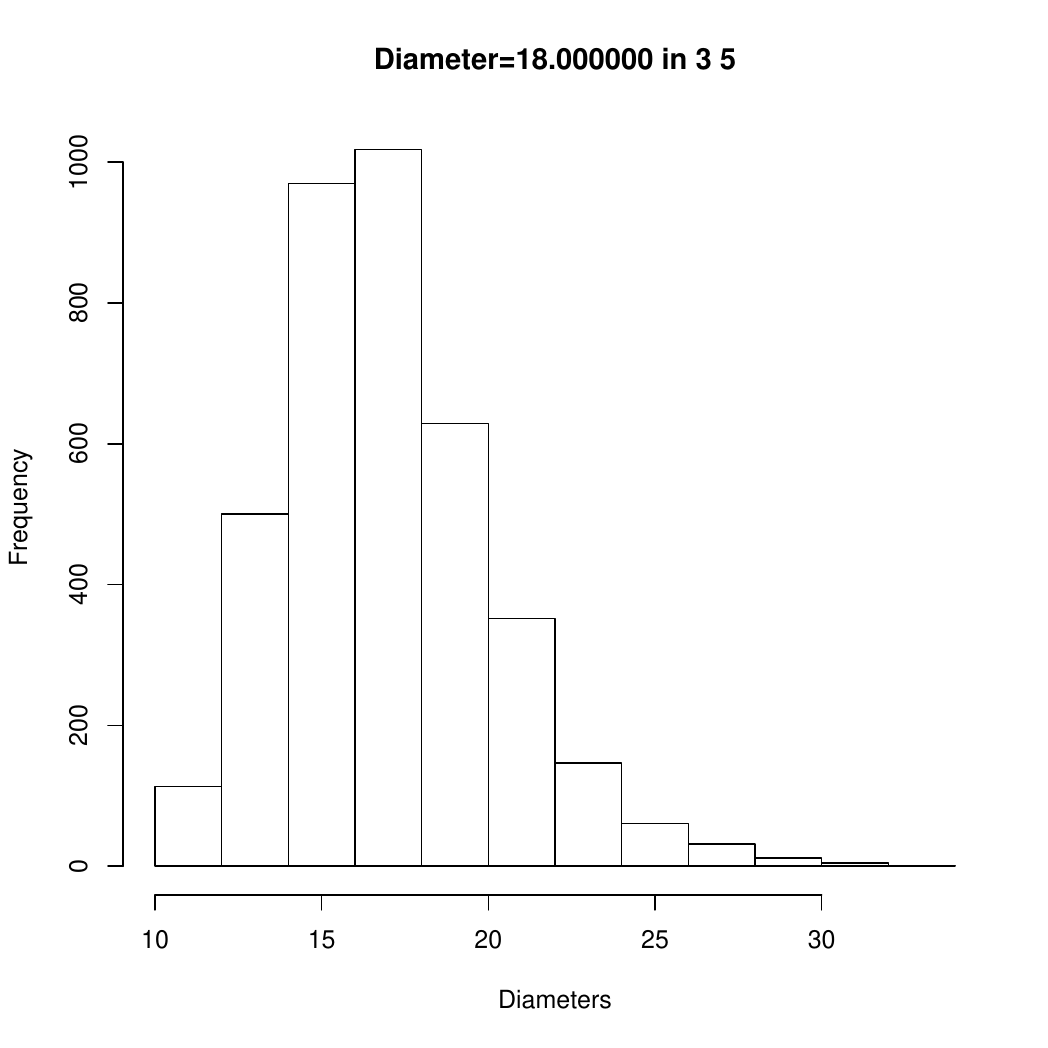}
\caption{Only edges with tour frequency between three to five}
\label{fig:diameters120_35}
\end{subfigure}%
\hspace*{1cm}
\begin{subfigure}{0.35\textwidth}
\centering
\includegraphics[width=\linewidth]{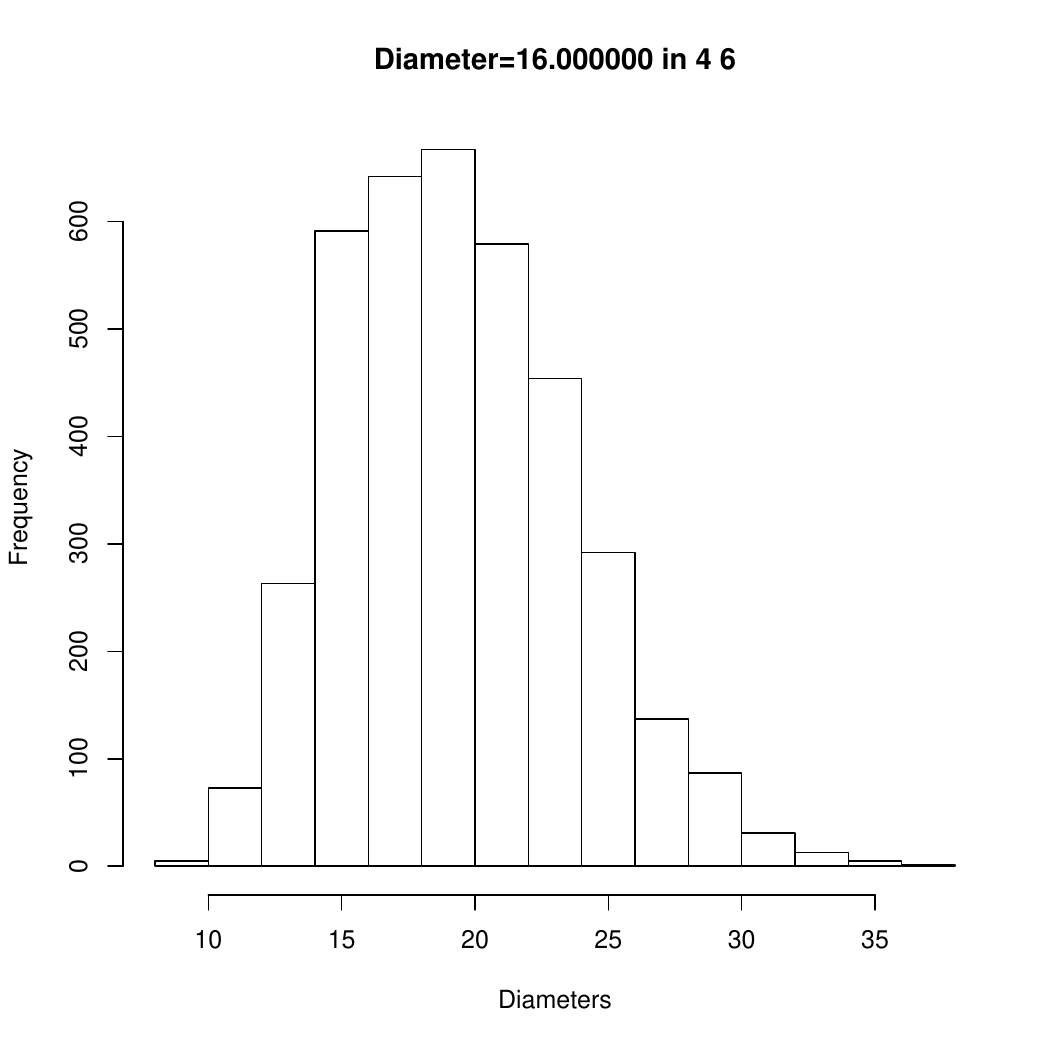}
\caption{Only edges with tour frequency between four to six}
\label{fig:diameters120_46}
\end{subfigure}

\begin{subfigure}{0.35\textwidth}
\centering
\includegraphics[width=\linewidth]{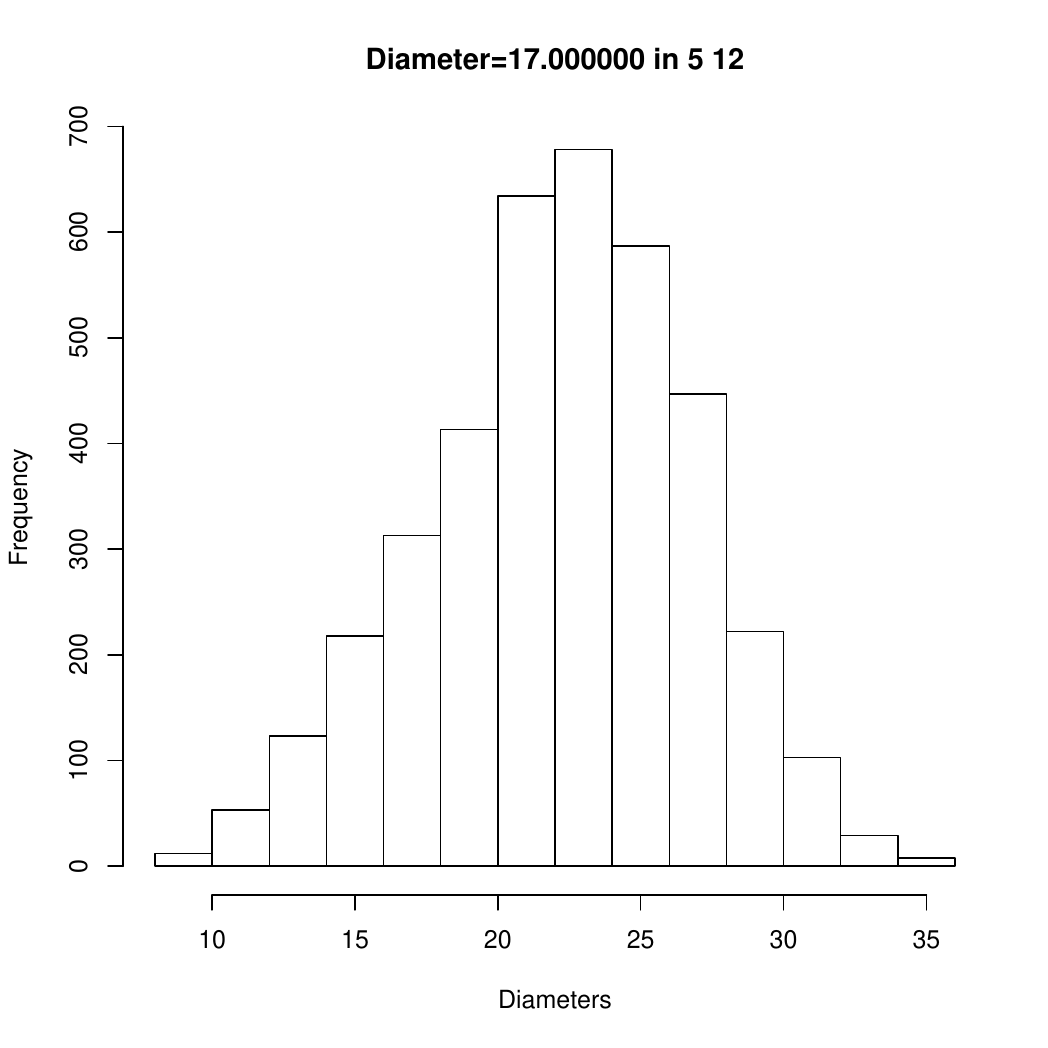}
\caption{Only edges with tour frequency between five to twelve}
\label{fig:diameters120_512}
\end{subfigure}%
\hspace*{1cm}
\begin{subfigure}{0.35\textwidth}
\centering
\includegraphics[width=\linewidth]{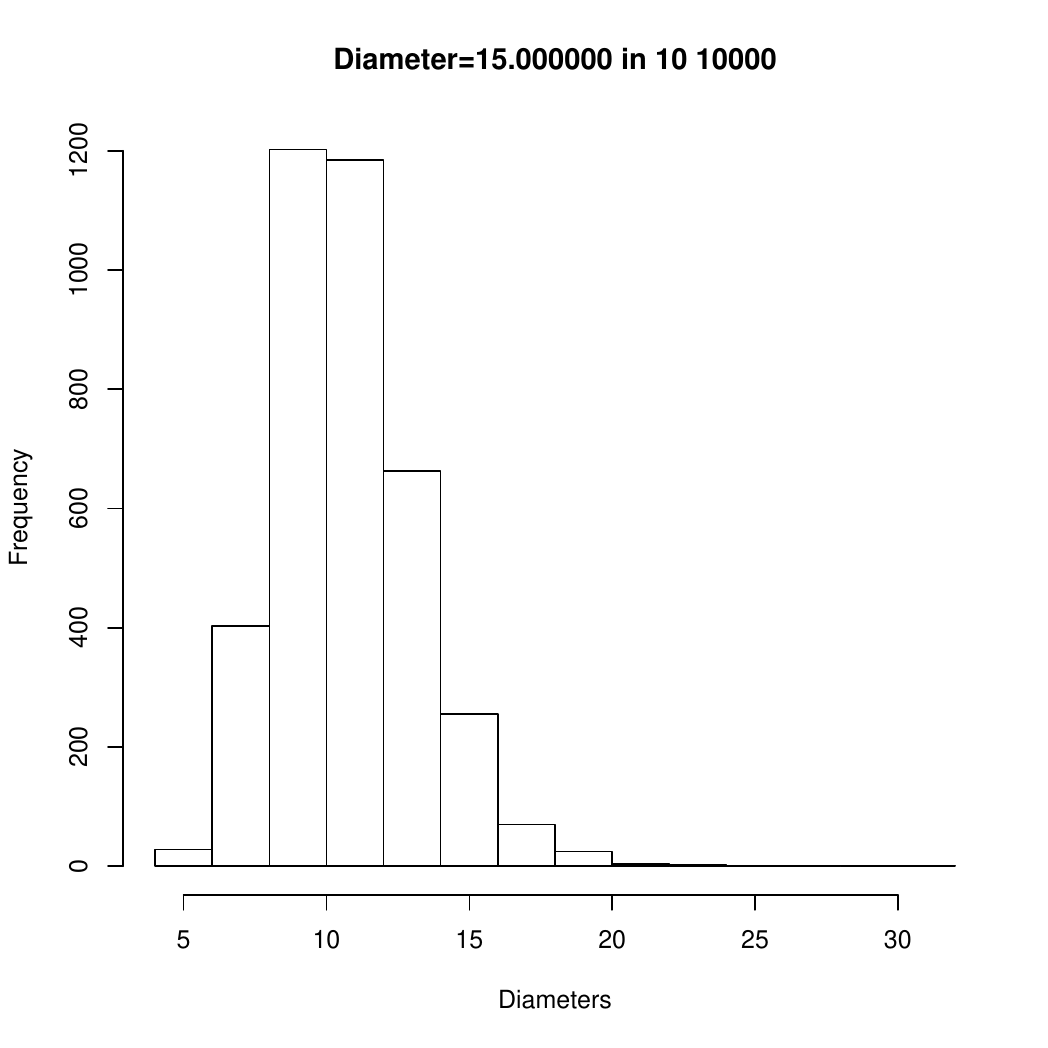}
\caption{Only edges with tour frequency more than ten}
\label{fig:diameters120_10inf}
\end{subfigure}
\caption{Histograms of diameters of the graphs which arise by taking different edges into account (see individual caption) from simulations for 18th April 2014. In the title of the plot the observed value is shown.}
\label{fig:diameters120}
\end{figure}

\begin{figure}
\centering
\begin{subfigure}{0.35\textwidth}
\centering
\includegraphics[width=\linewidth]{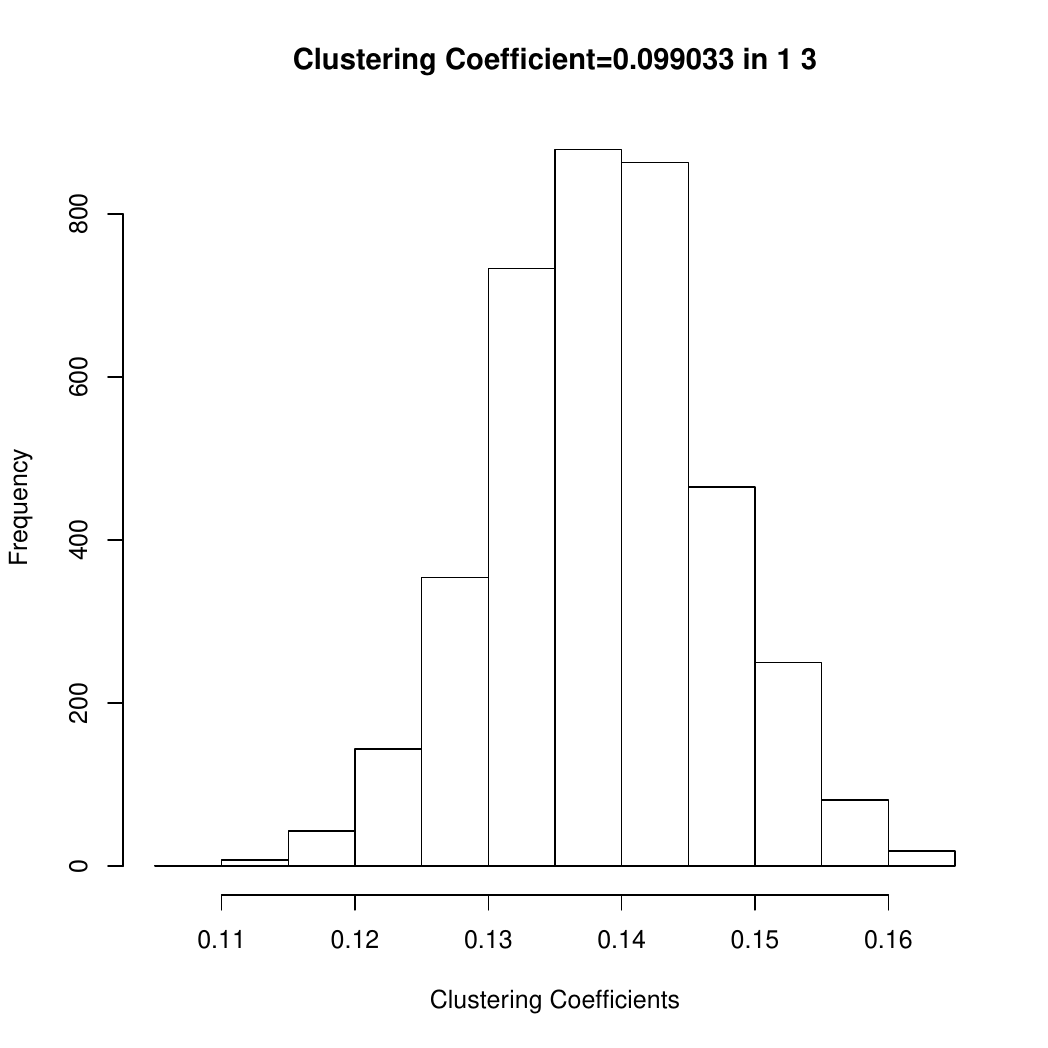}
\caption{Only edges with tour frequency between one to three}
\label{fig:cluster120_13}
\end{subfigure}%
\hspace*{1cm}
\begin{subfigure}{0.35\textwidth}
\centering
\includegraphics[width=\linewidth]{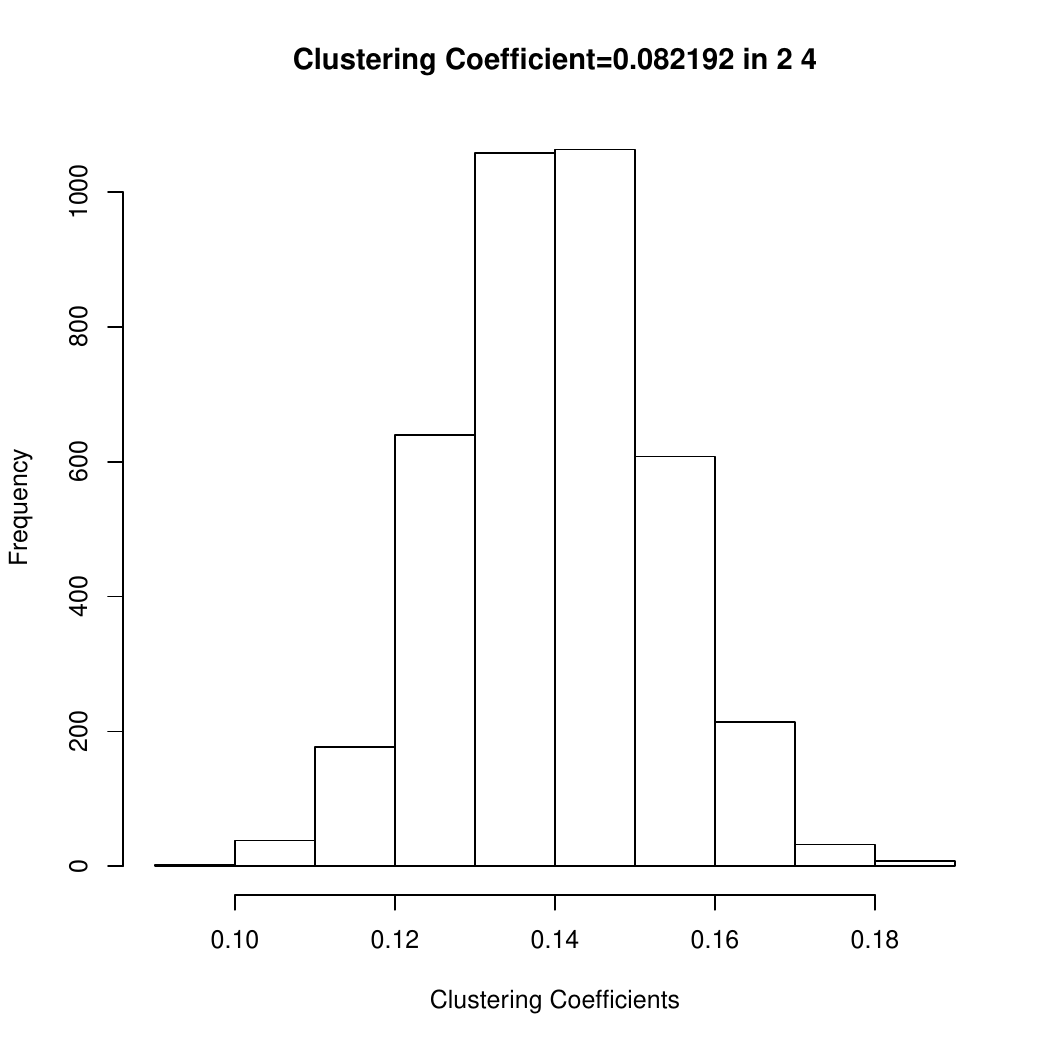}
\caption{Only edges with tour frequency between two to four}
\label{fig:cluster120_24}
\end{subfigure}

\begin{subfigure}{0.35\textwidth}
\centering
\includegraphics[width=\linewidth]{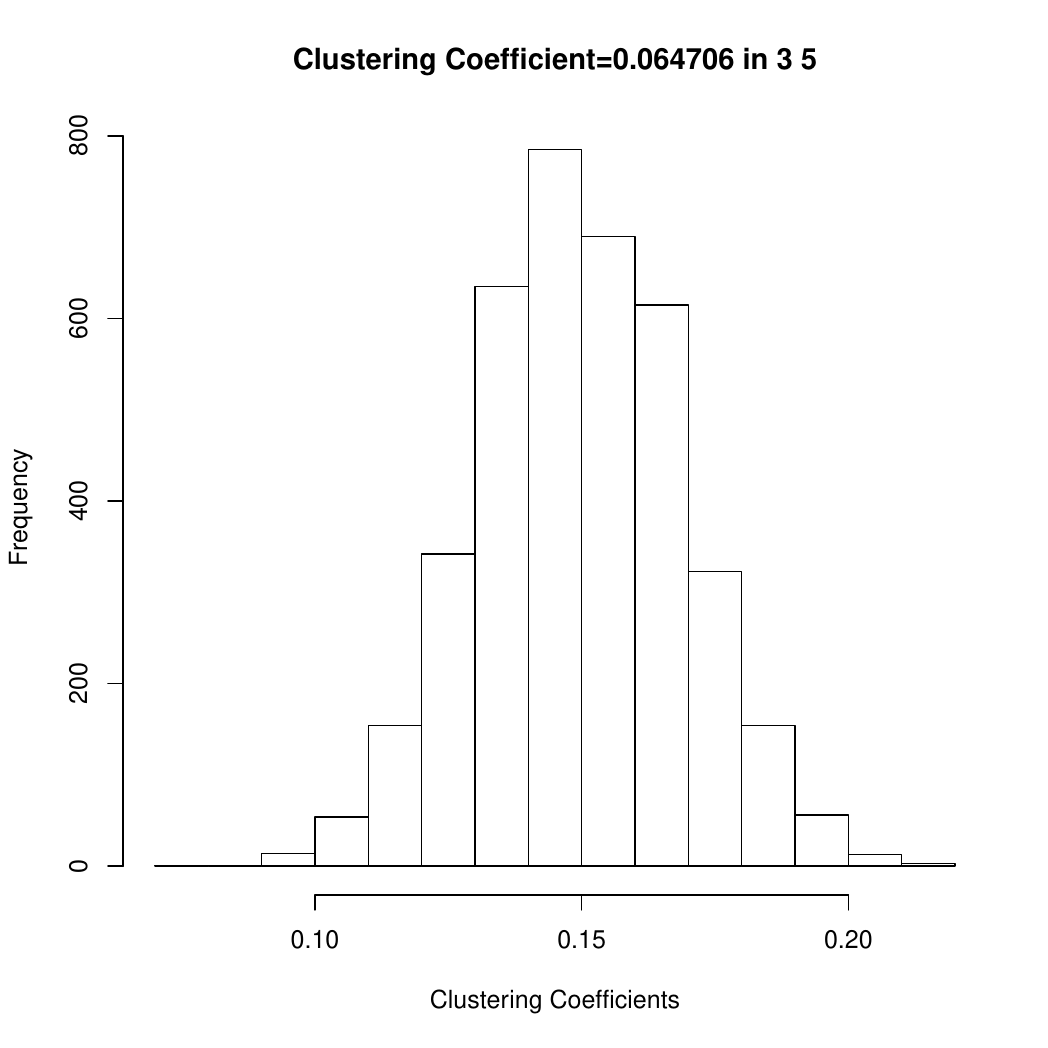}
\caption{Only edges with tour frequency between three to five}
\label{fig:cluster120_35}
\end{subfigure}%
\hspace*{1cm}
\begin{subfigure}{0.35\textwidth}
\centering
\includegraphics[width=\linewidth]{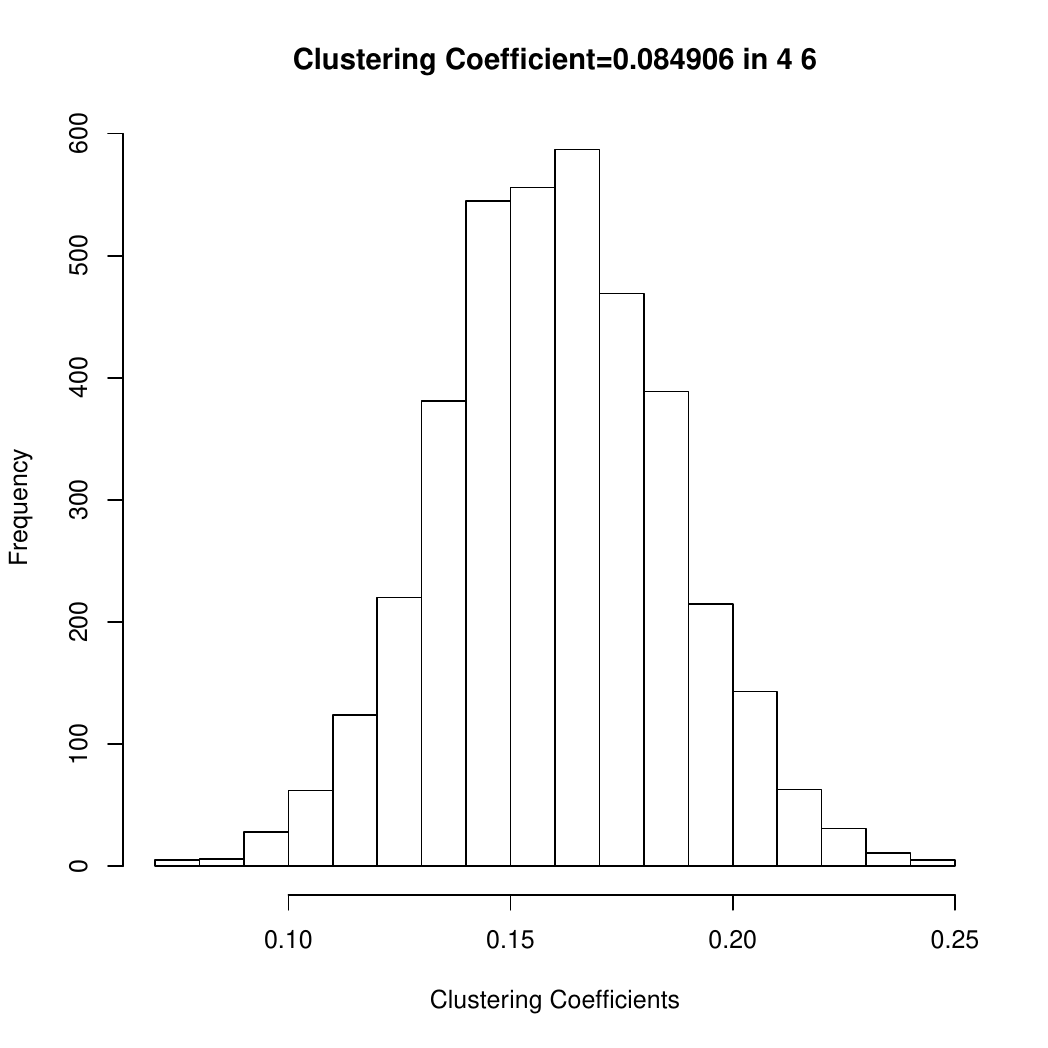}
\caption{Only edges with tour frequency between four to six}
\label{fig:cluster120_46}
\end{subfigure}

\begin{subfigure}{0.35\textwidth}
\centering
\includegraphics[width=\linewidth]{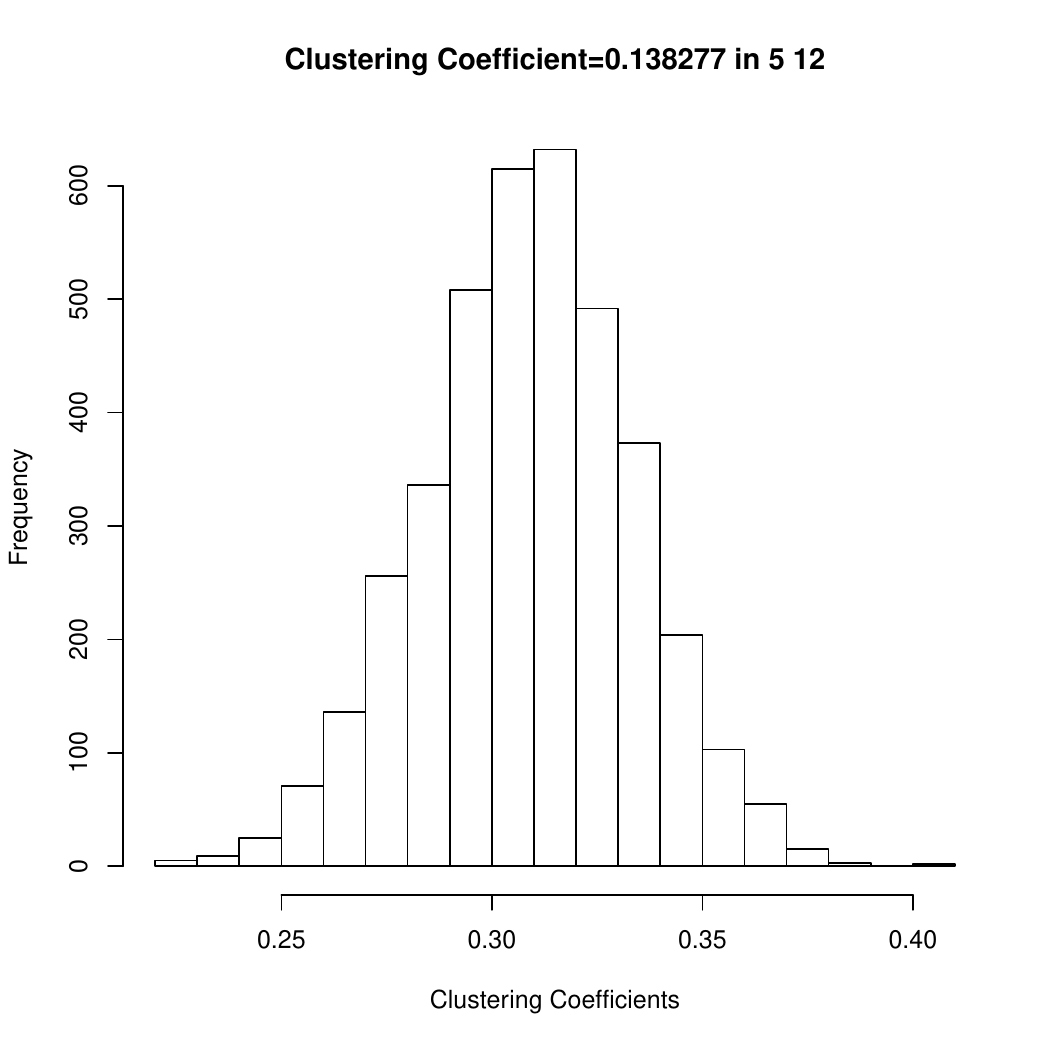}
\caption{Only edges with tour frequency between five to twelve}
\label{fig:cluster120_512}
\end{subfigure}%
\hspace*{1cm}
\begin{subfigure}{0.35\textwidth}
\centering
\includegraphics[width=\linewidth]{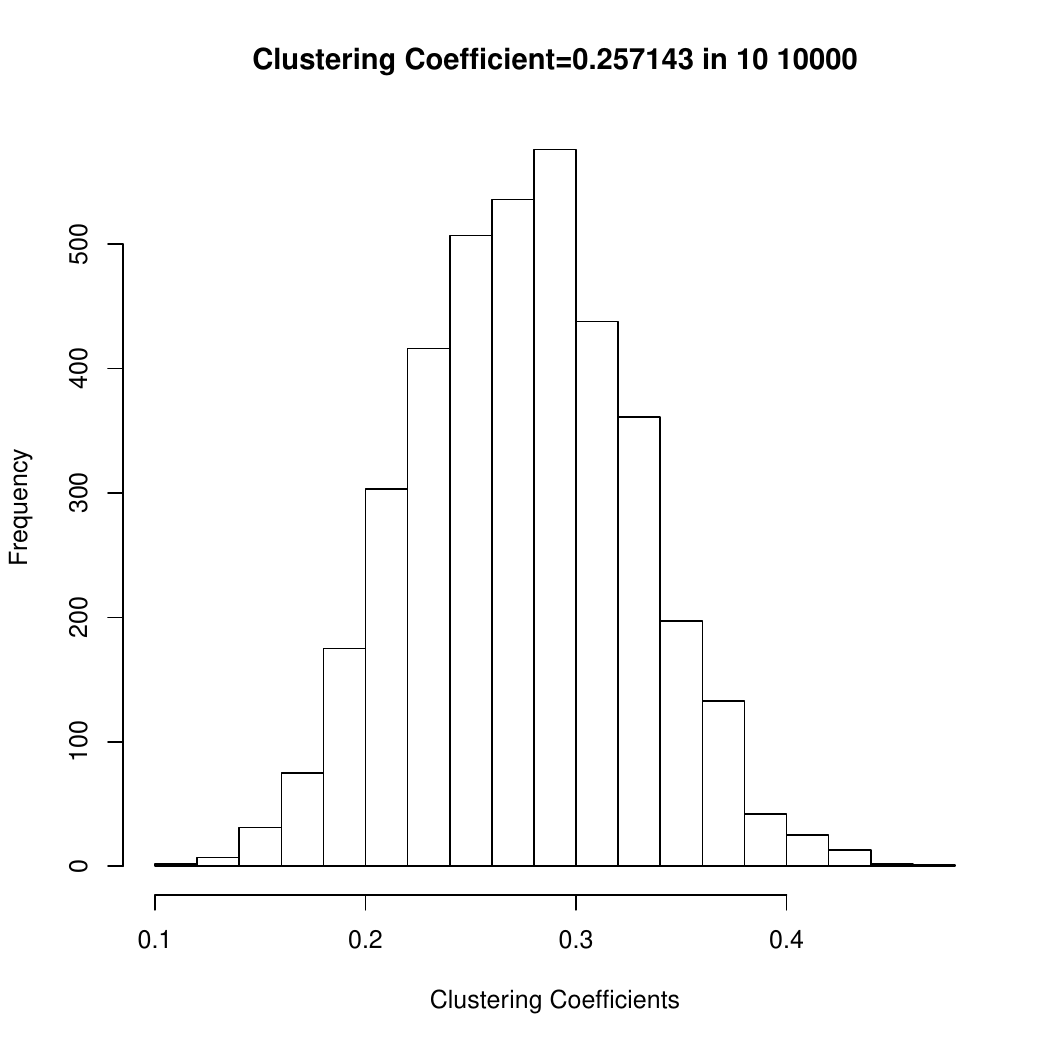}
\caption{Only edges with tour frequency more than ten}
\label{fig:cluster120_10inf}
\end{subfigure}
\caption{Histograms of clustering coefficients of the graphs which arise by taking different edges into account (see individual caption) from simulations for 18th April 2014. In the title of the plot the observed value is shown.}
\label{fig:cluster120}
\end{figure}

\begin{figure}
\centering
\begin{subfigure}{0.35\textwidth}
\centering
\includegraphics[width=\linewidth]{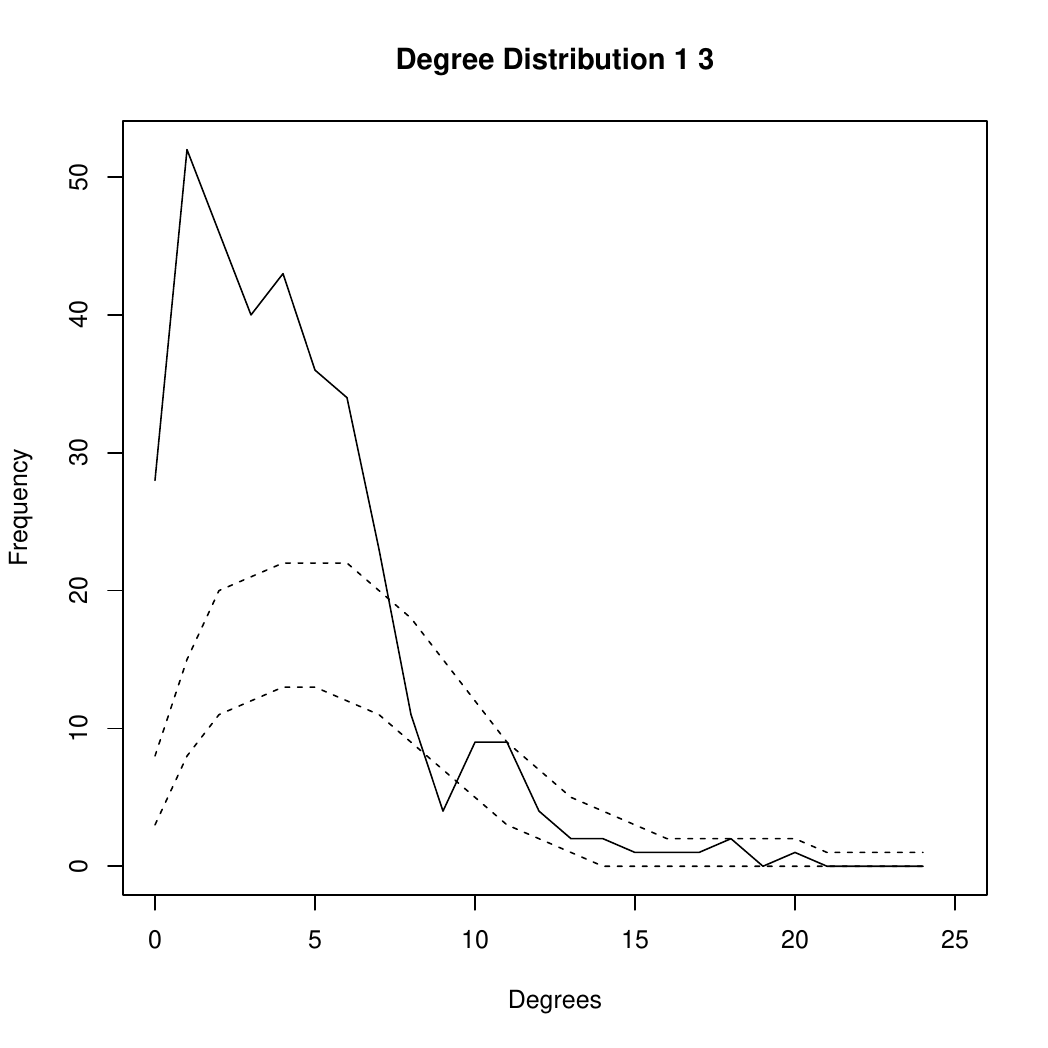}
\caption{Only edges with tour frequency between one to three}
\label{fig:degrees184_13}
\end{subfigure}%
\hspace*{1cm}
\begin{subfigure}{0.35\textwidth}
\centering
\includegraphics[width=\linewidth]{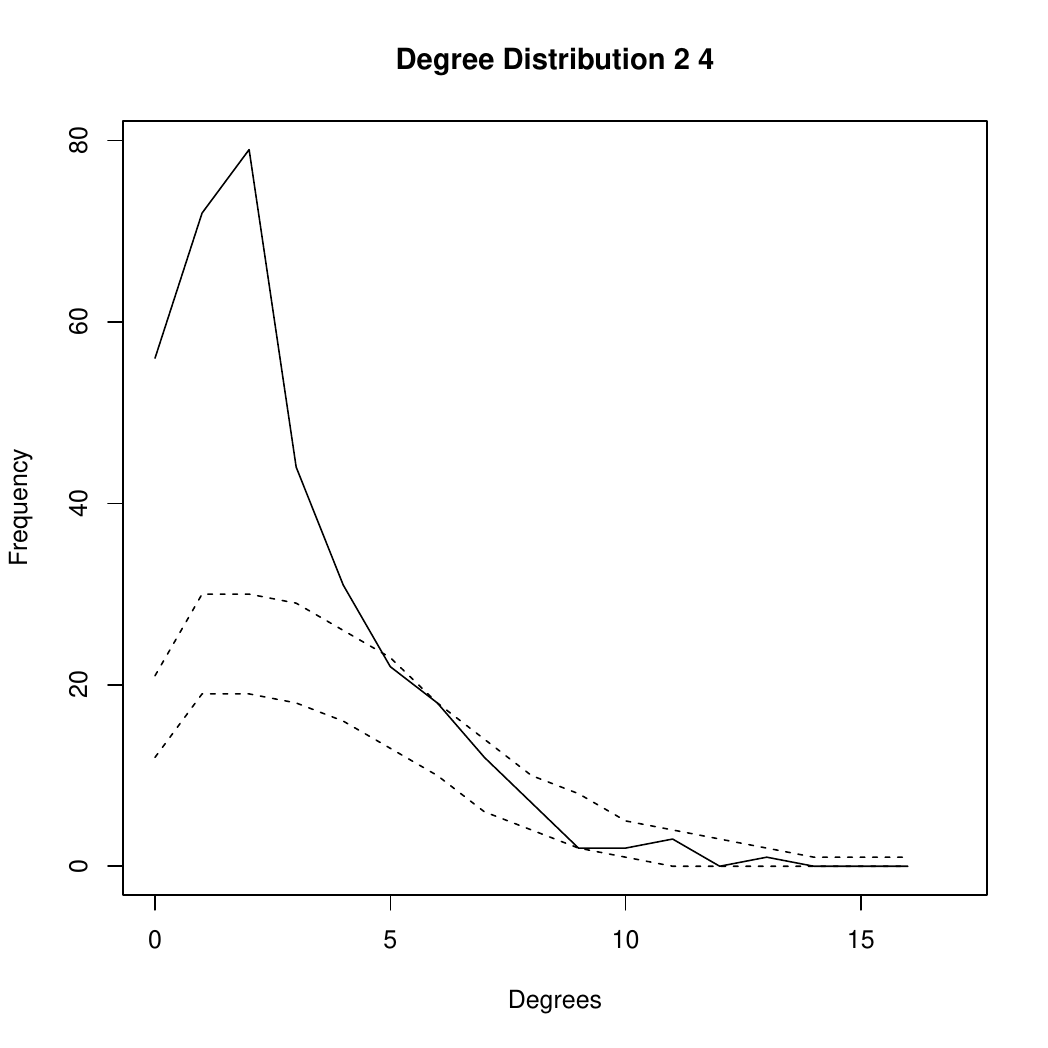}
\caption{Only edges with tour frequency between two to four}
\label{fig:degrees184_24}
\end{subfigure}

\begin{subfigure}{0.35\textwidth}
\centering
\includegraphics[width=\linewidth]{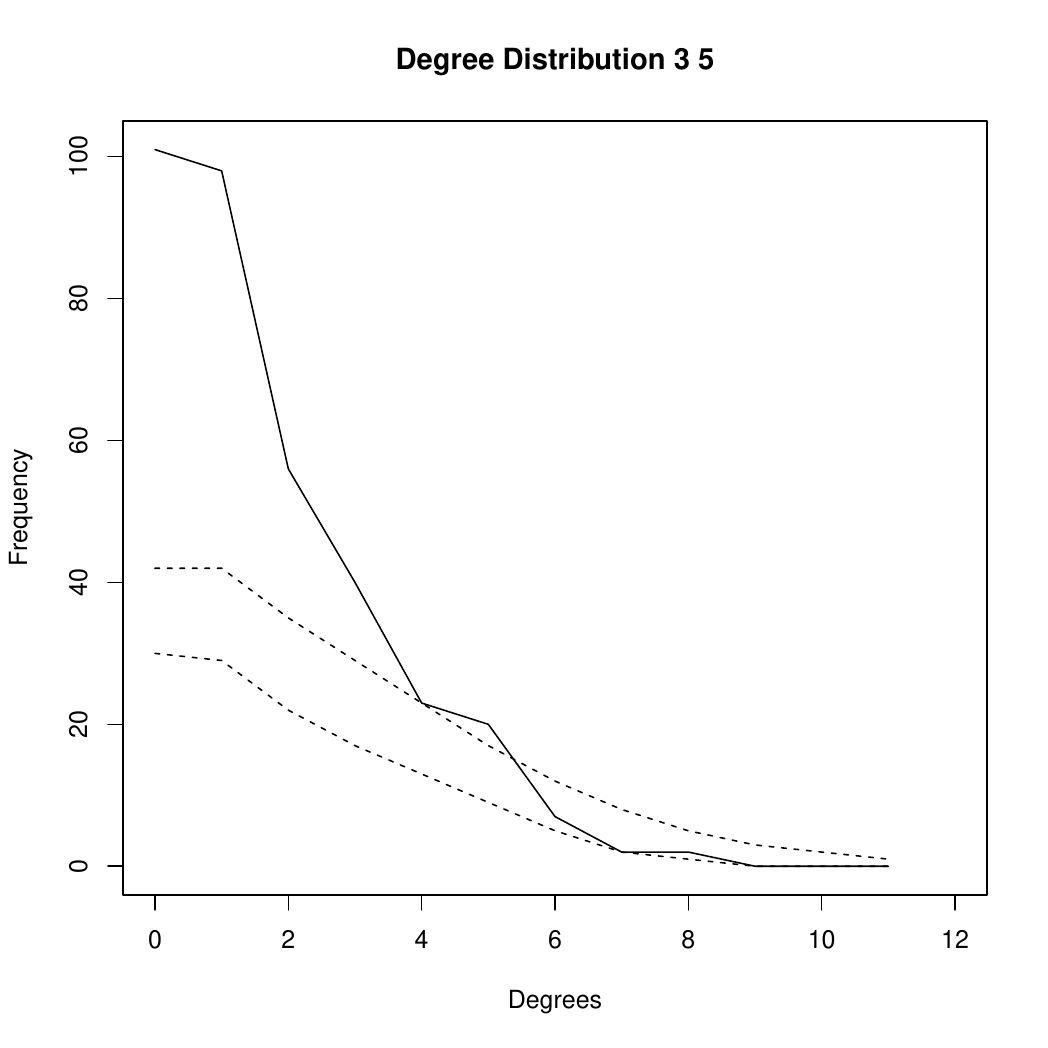}
\caption{Only edges with tour frequency between three to five}
\label{fig:degrees184_35}
\end{subfigure}%
\hspace*{1cm}
\begin{subfigure}{0.35\textwidth}
\centering
\includegraphics[width=\linewidth]{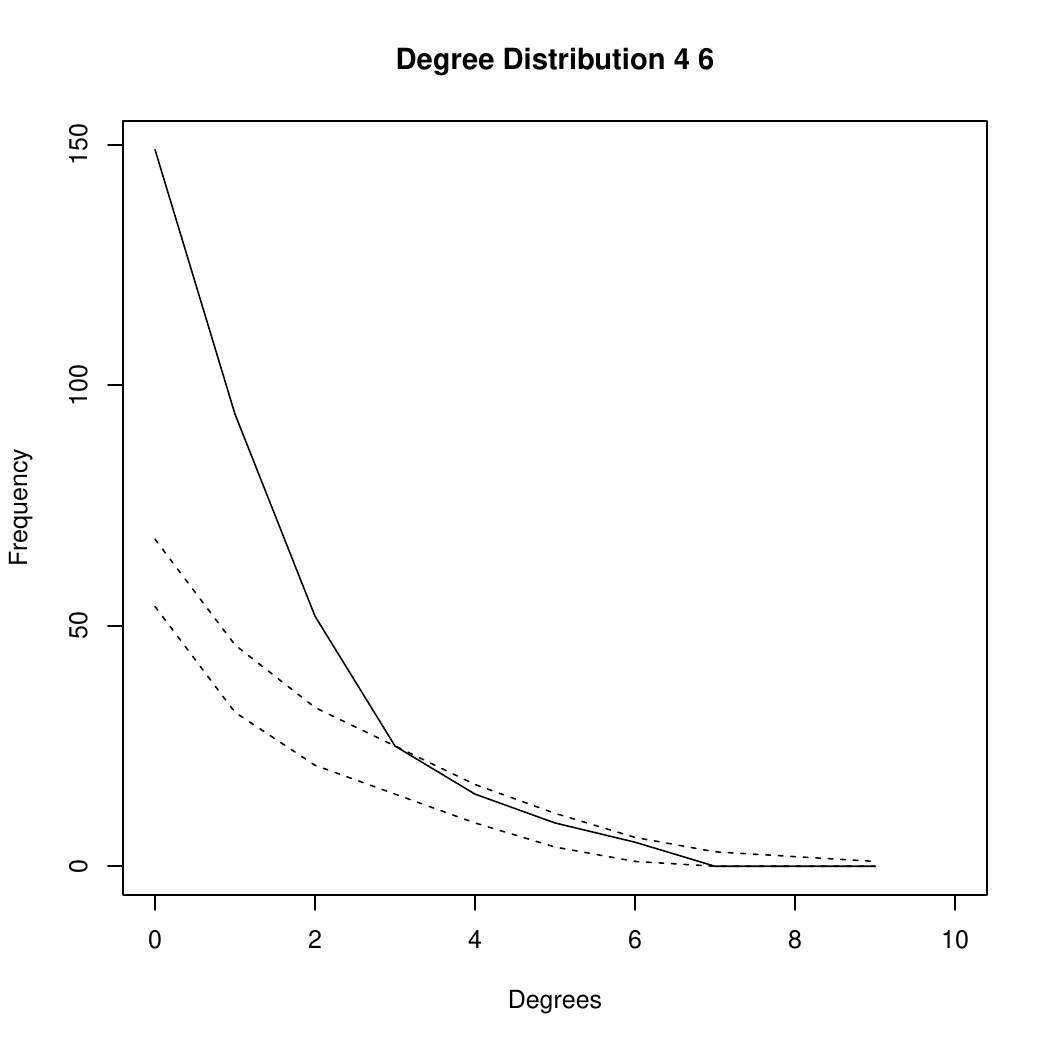}
\caption{Only edges with tour frequency between four to six}
\label{fig:degrees184_46}
\end{subfigure}

\begin{subfigure}{0.35\textwidth}
\centering
\includegraphics[width=\linewidth]{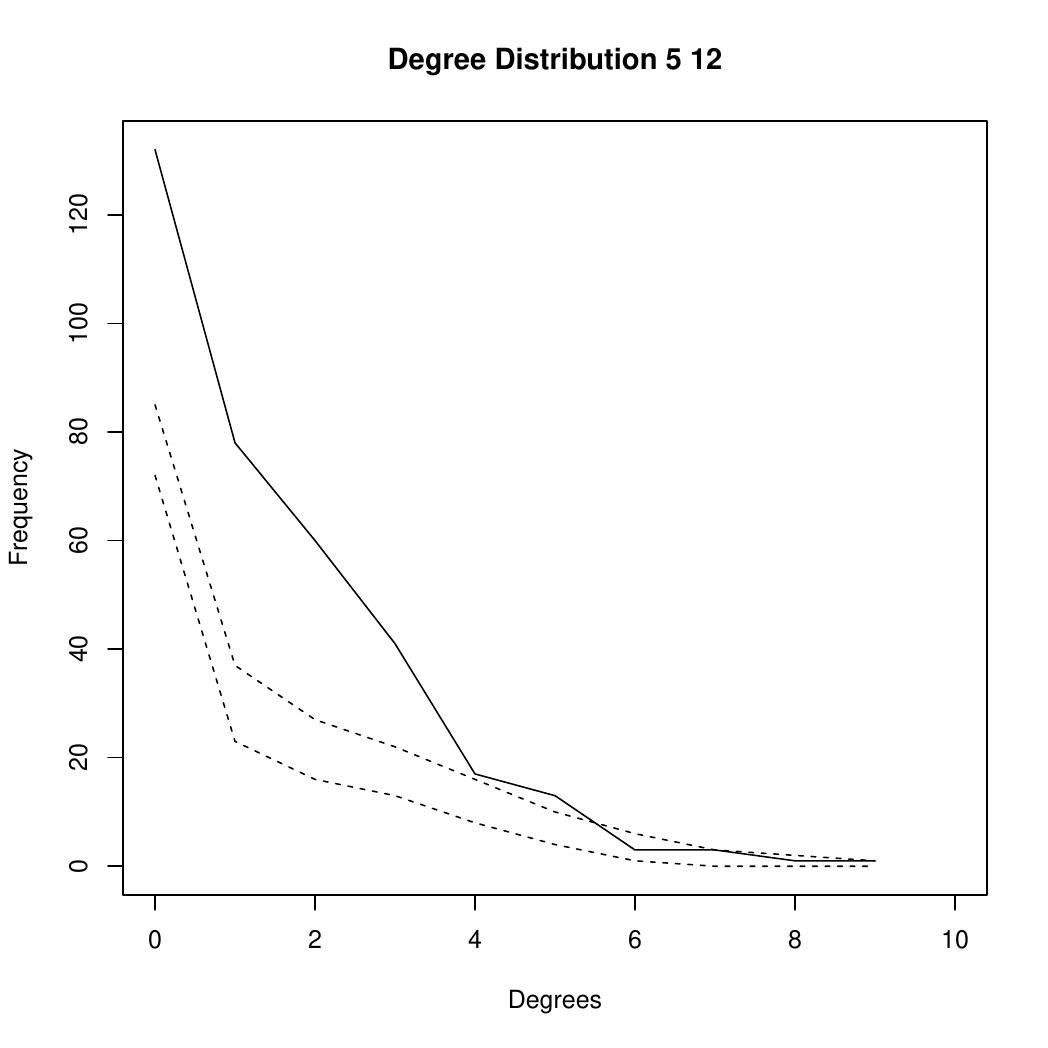}
\caption{Only edges with tour frequency between five to twelve}
\label{fig:degrees184_512}
\end{subfigure}%
\hspace*{1cm}
\begin{subfigure}{0.35\textwidth}
\centering
\includegraphics[width=\linewidth]{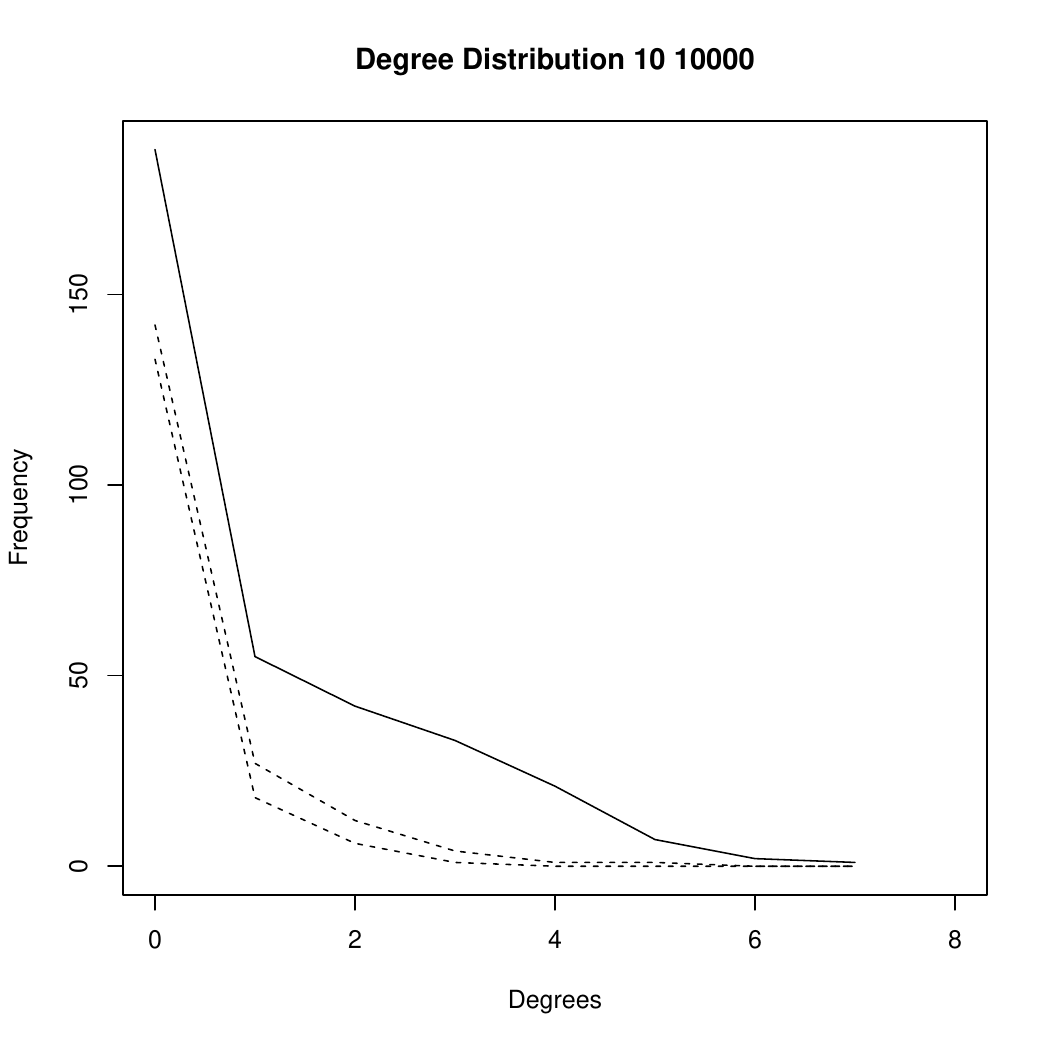}
\caption{Only edges with tour frequency more than ten}
\label{fig:degrees184_10inf}
\end{subfigure}
\caption{Degree distributions of the graphs, which arise by taking different tour frequencies into account (see individual caption) from simulations for 10th July 2015. Dotted lines show 10\% and 90\% quantiles of simulations and solid line shows true distributions.}
\label{fig:degrees184}
\end{figure}

\begin{figure}
\centering
\begin{subfigure}{0.35\textwidth}
\centering
\includegraphics[width=\linewidth]{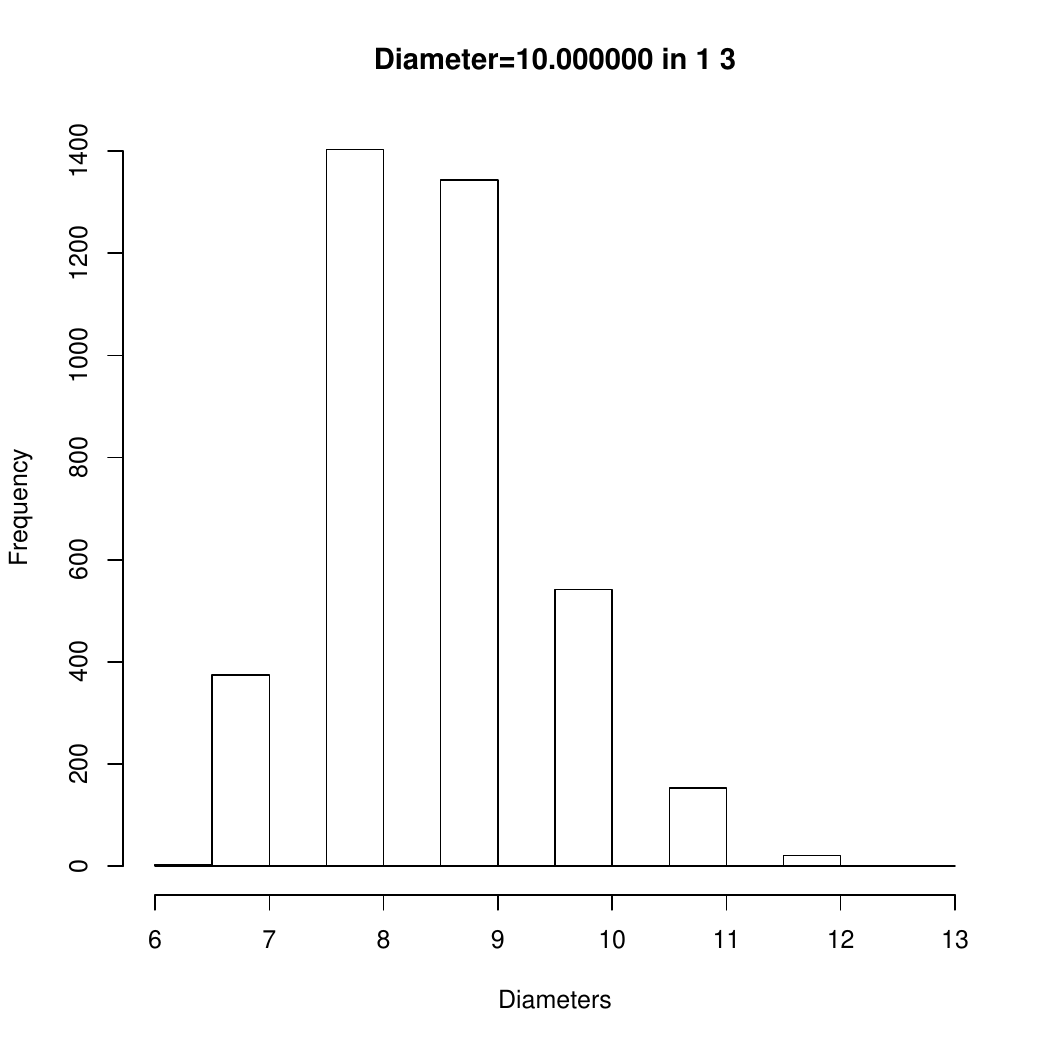}
\caption{Only edges with tour frequency between one to three}
\label{fig:diameters184_13}
\end{subfigure}%
\hspace*{1cm}
\begin{subfigure}{0.35\textwidth}
\centering
\includegraphics[width=\linewidth]{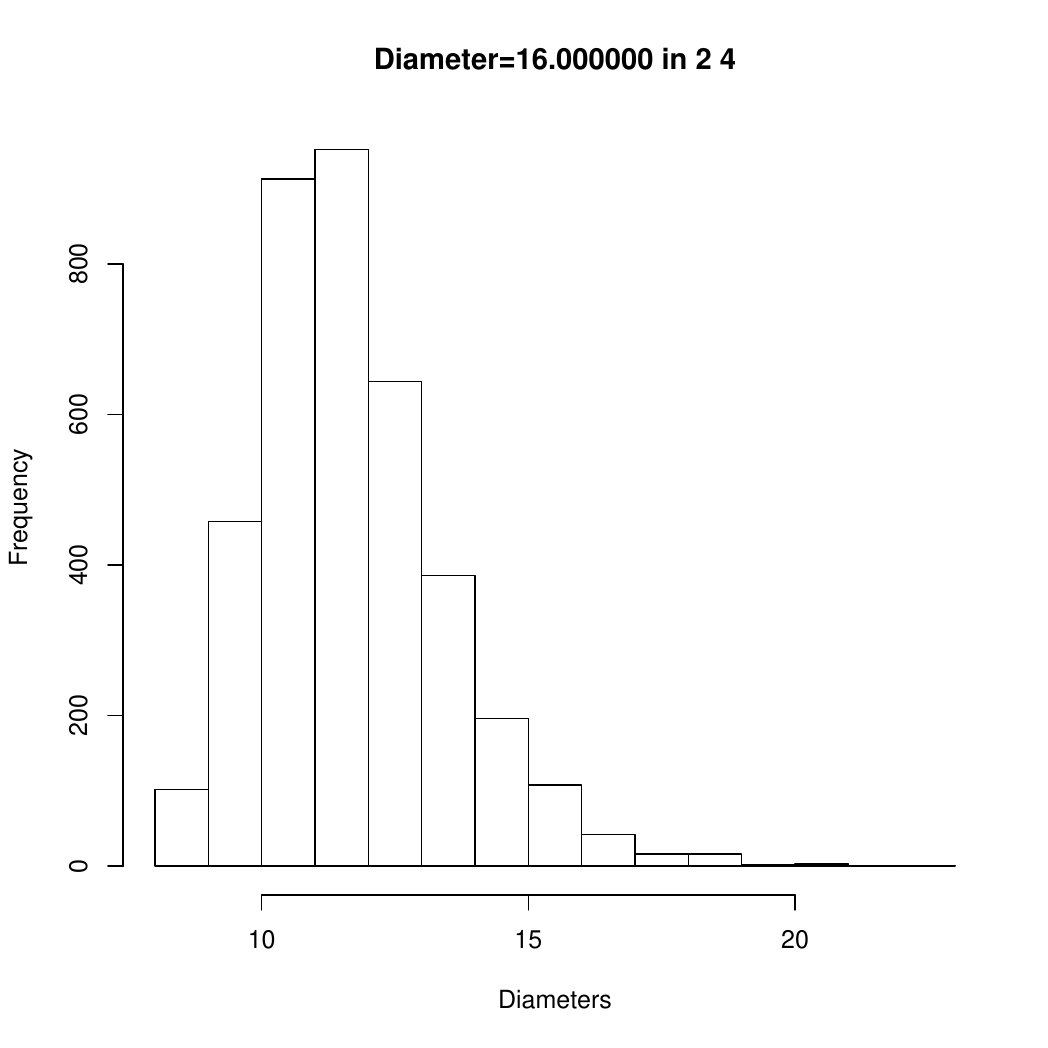}
\caption{Only edges with tour frequency between two to four}
\label{fig:diameters184_24}
\end{subfigure}

\begin{subfigure}{0.35\textwidth}
\centering
\includegraphics[width=\linewidth]{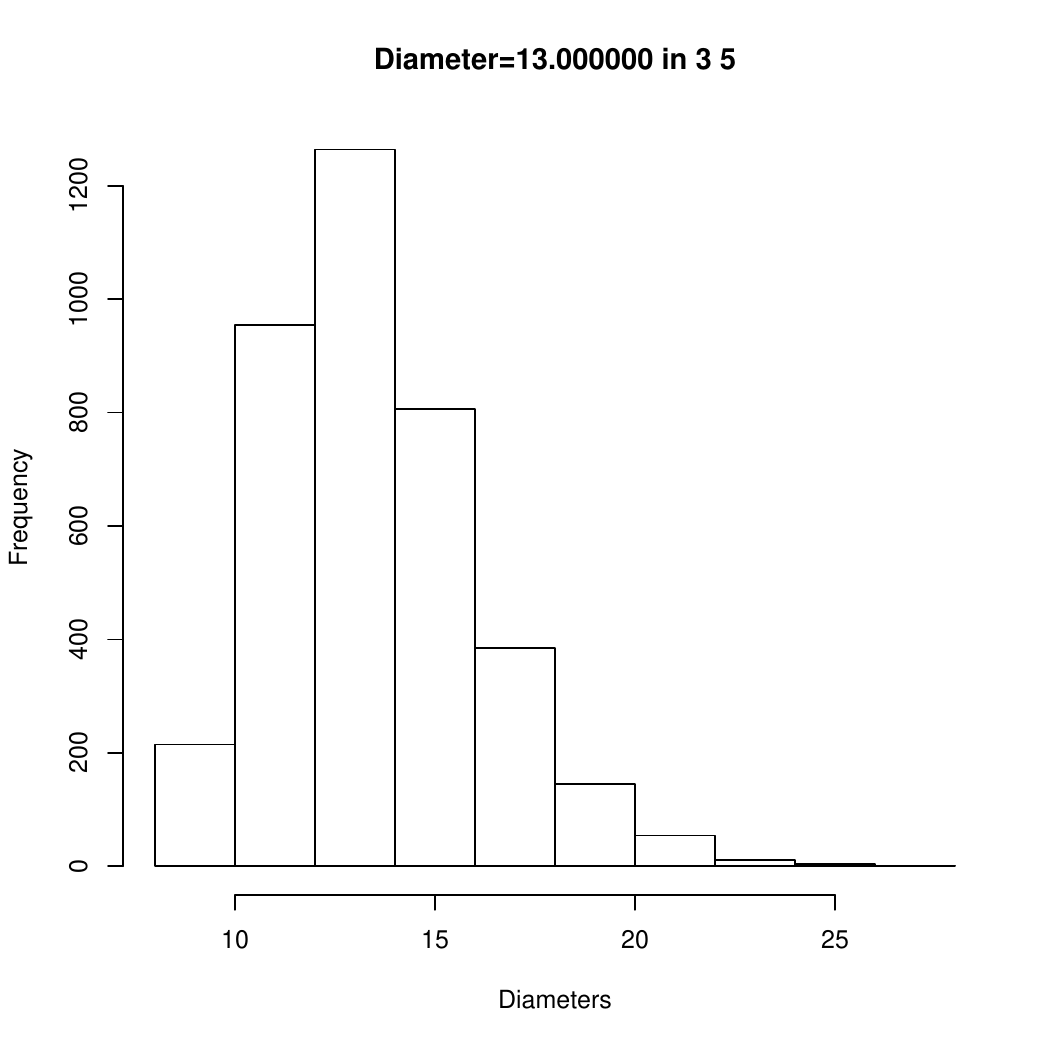}
\caption{Only edges with tour frequency between three to five}
\label{fig:diameters184_35}
\end{subfigure}%
\hspace*{1cm}
\begin{subfigure}{0.35\textwidth}
\centering
\includegraphics[width=\linewidth]{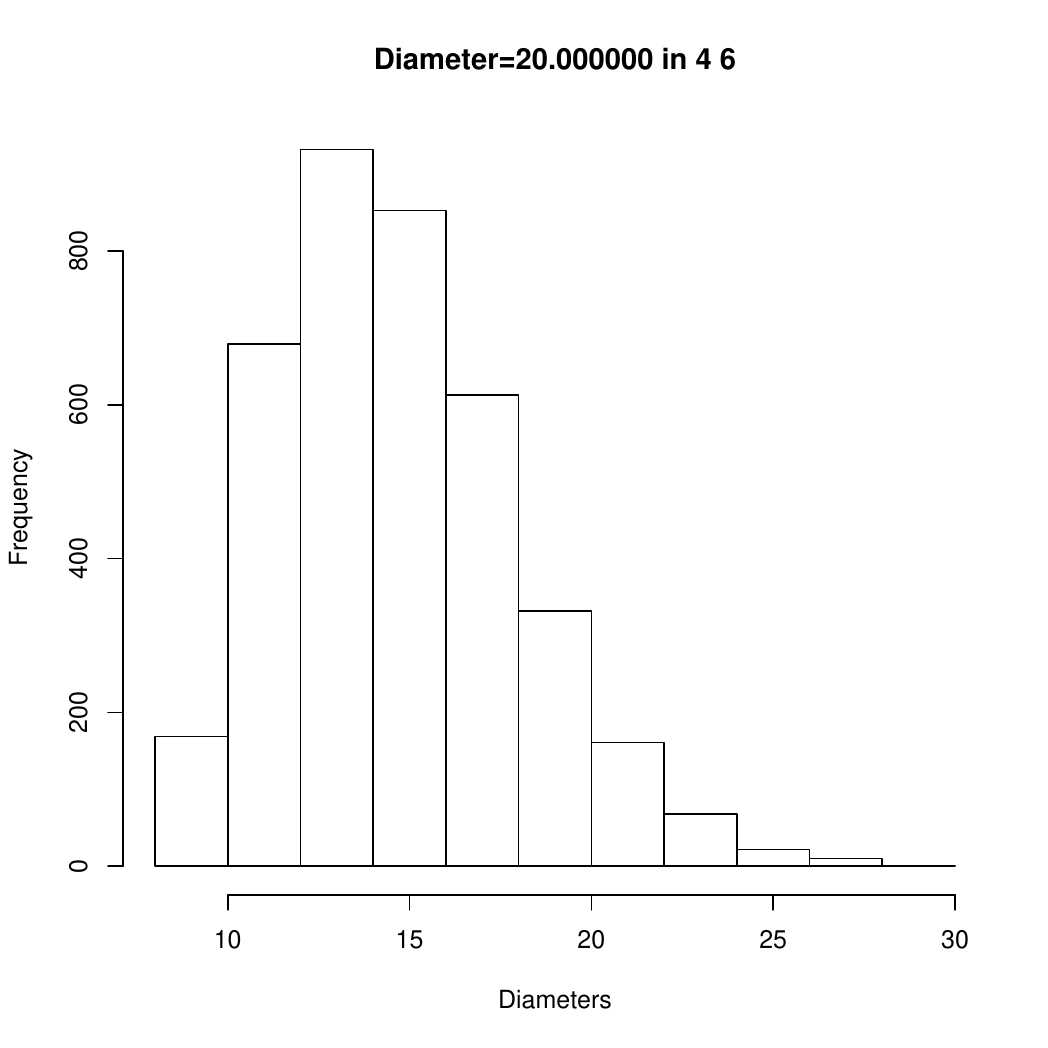}
\caption{Only edges with tour frequency between four to six}
\label{fig:diameters184_46}
\end{subfigure}

\begin{subfigure}{0.35\textwidth}
\centering
\includegraphics[width=\linewidth]{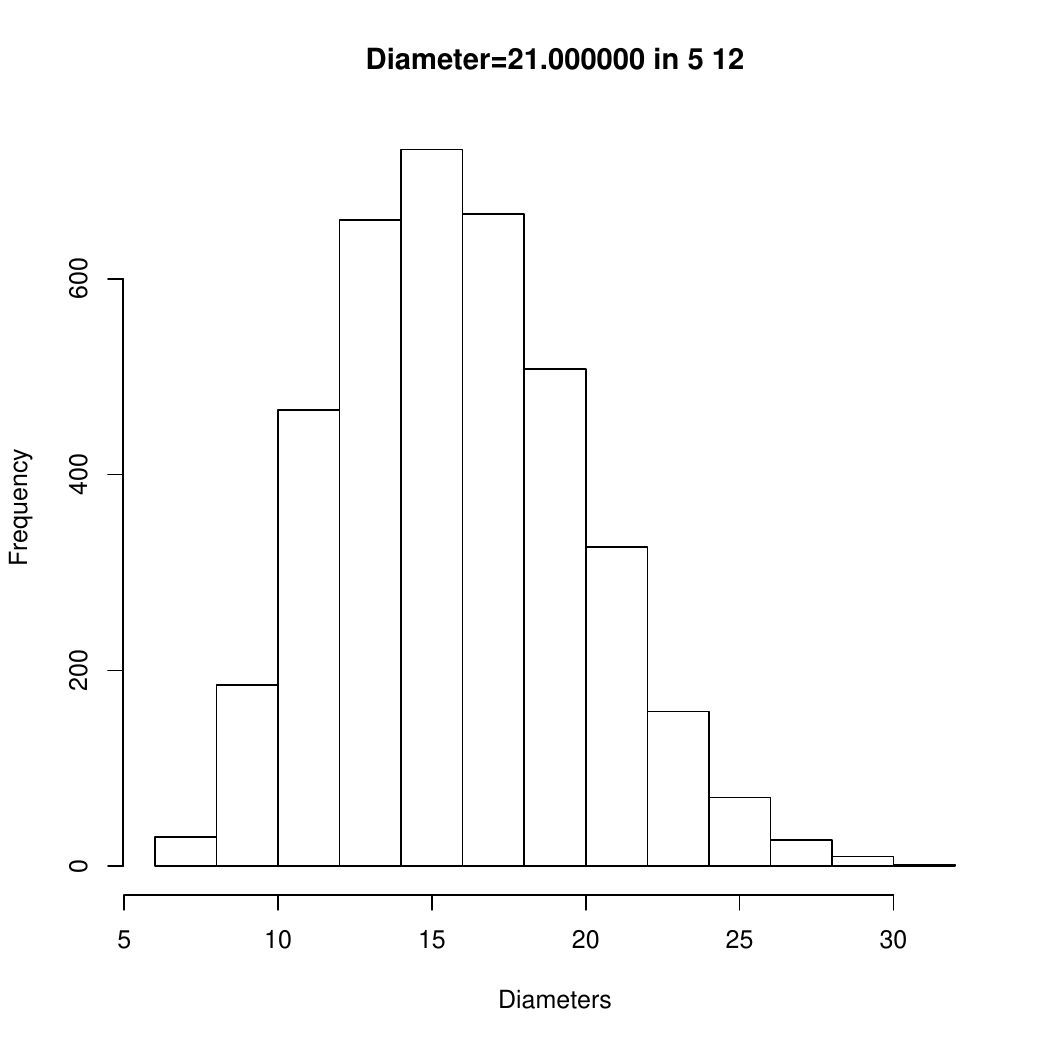}
\caption{Only edges with tour frequency between five to twelve}
\label{fig:diameters184_512}
\end{subfigure}%
\hspace*{1cm}
\begin{subfigure}{0.35\textwidth}
\centering
\includegraphics[width=\linewidth]{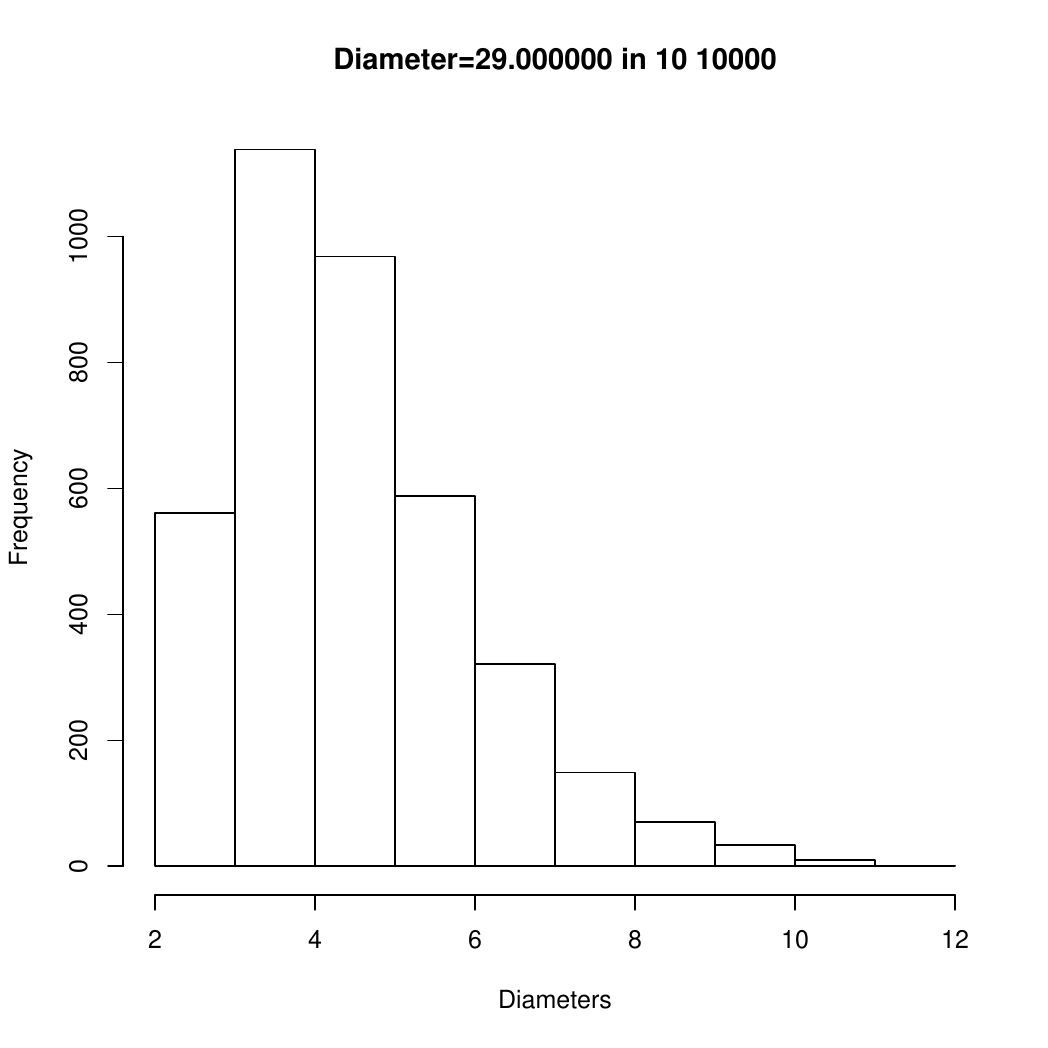}
\caption{Only edges with tour frequency more than ten}
\label{fig:diameters184_10inf}
\end{subfigure}
\caption{Histograms of diameters of the graphs which arise by taking different edges into account (see individual caption) from simulations for 10th July 2015. In the title of the plot the observed value is shown.}
\label{fig:diameters184}
\end{figure}

\begin{figure}
\centering
\begin{subfigure}{0.35\textwidth}
\centering
\includegraphics[width=\linewidth]{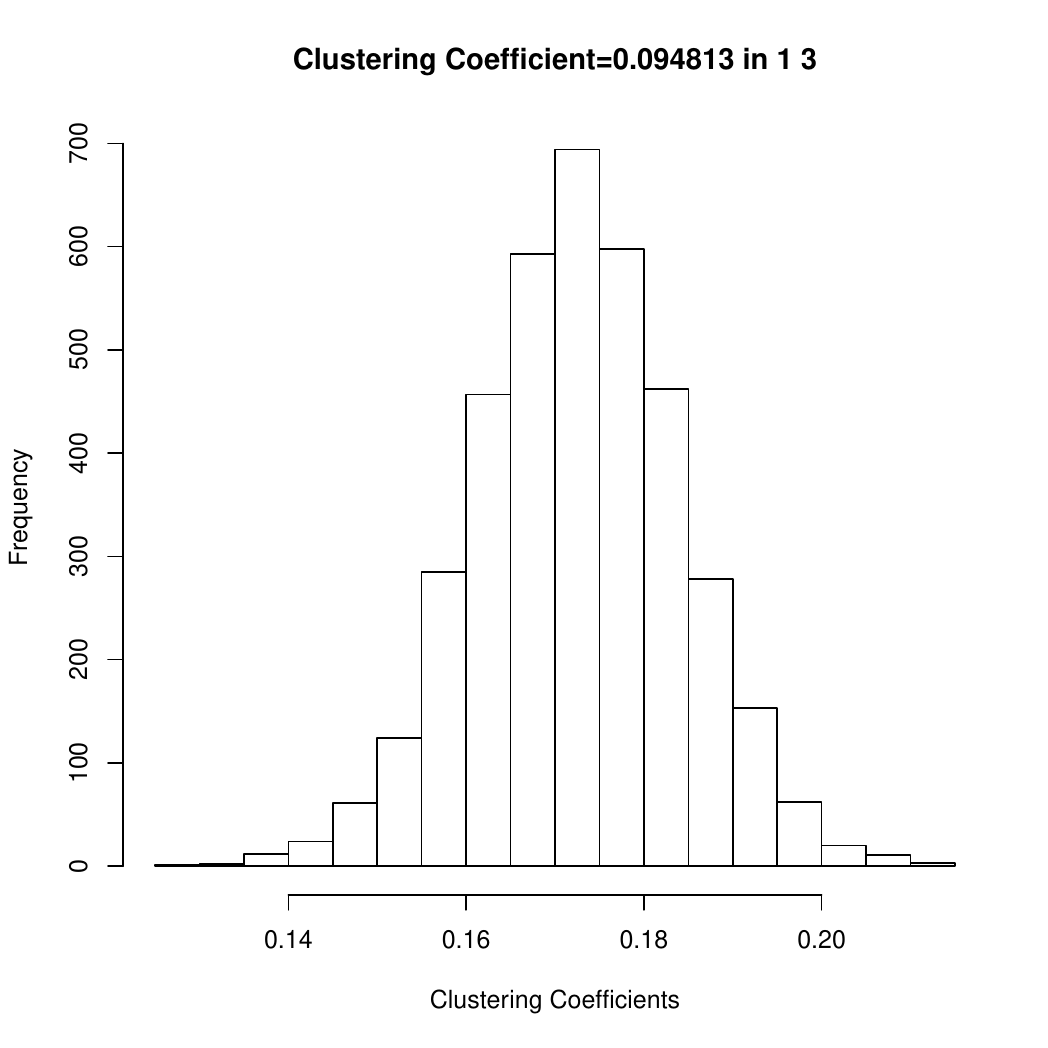}
\caption{Only edges with tour frequency between one to three}
\label{fig:cluster184_13}
\end{subfigure}%
\hspace*{1cm}
\begin{subfigure}{0.35\textwidth}
\centering
\includegraphics[width=\linewidth]{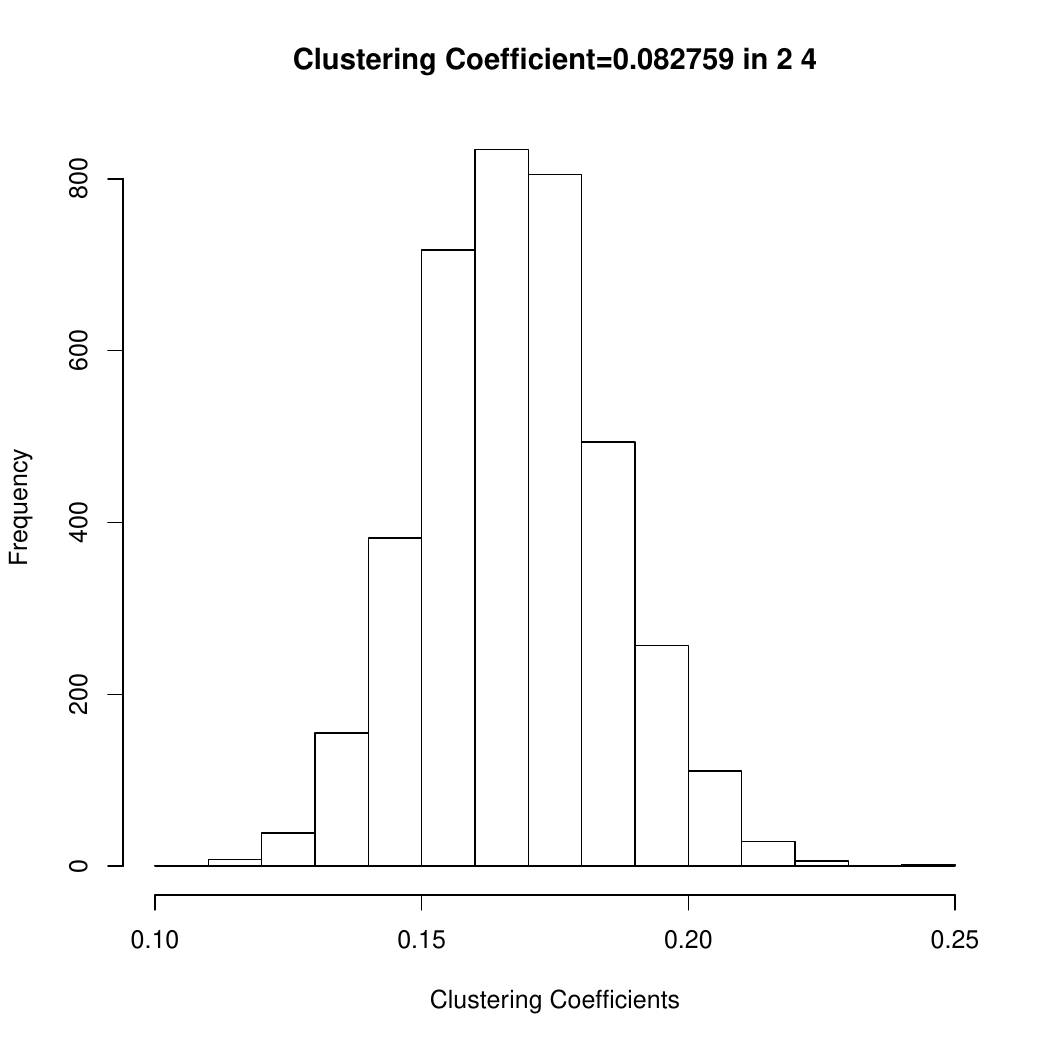}
\caption{Only edges with tour frequency between two to four}
\label{fig:cluster184_24}
\end{subfigure}

\begin{subfigure}{0.35\textwidth}
\centering
\includegraphics[width=\linewidth]{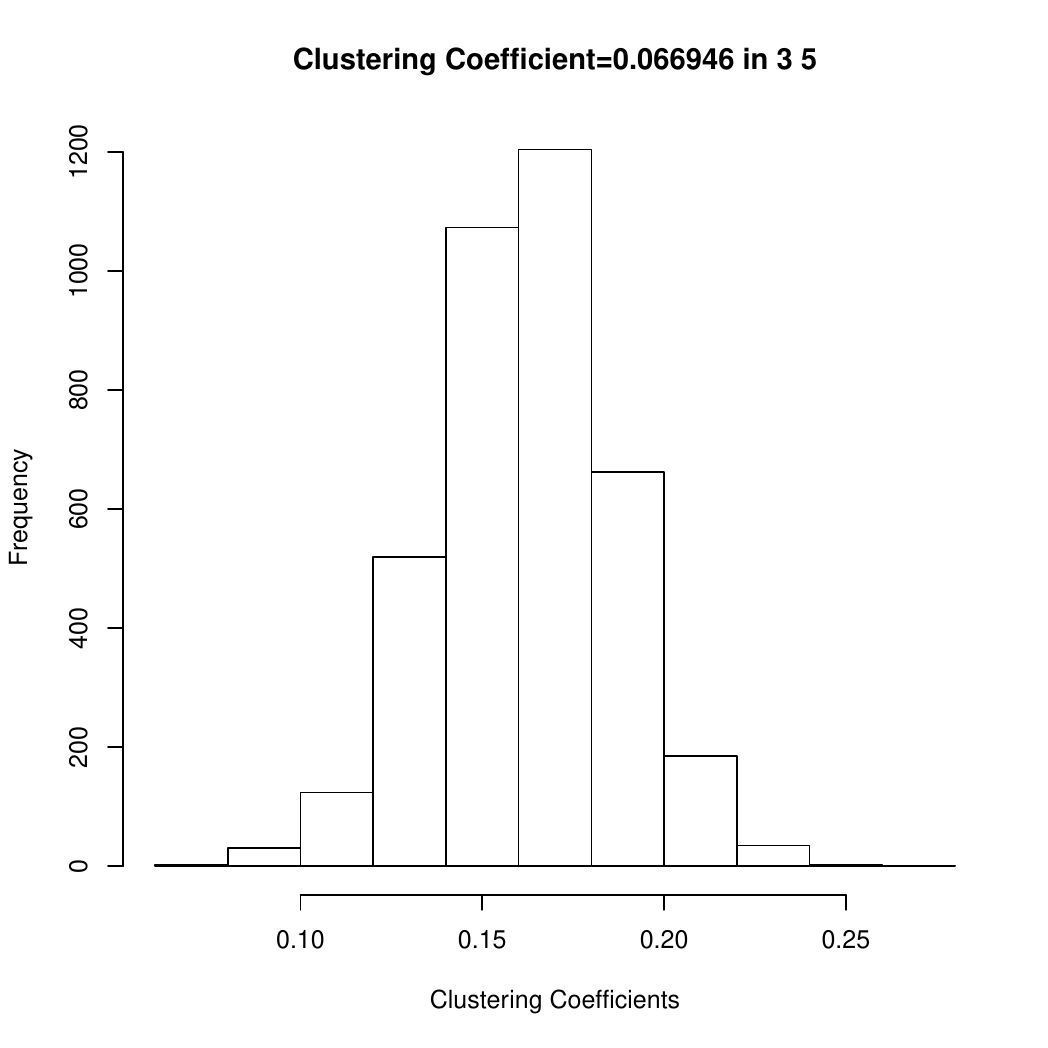}
\caption{Only edges with tour frequency between three to five}
\label{fig:cluster184_35}
\end{subfigure}%
\hspace*{1cm}
\begin{subfigure}{0.35\textwidth}
\centering
\includegraphics[width=\linewidth]{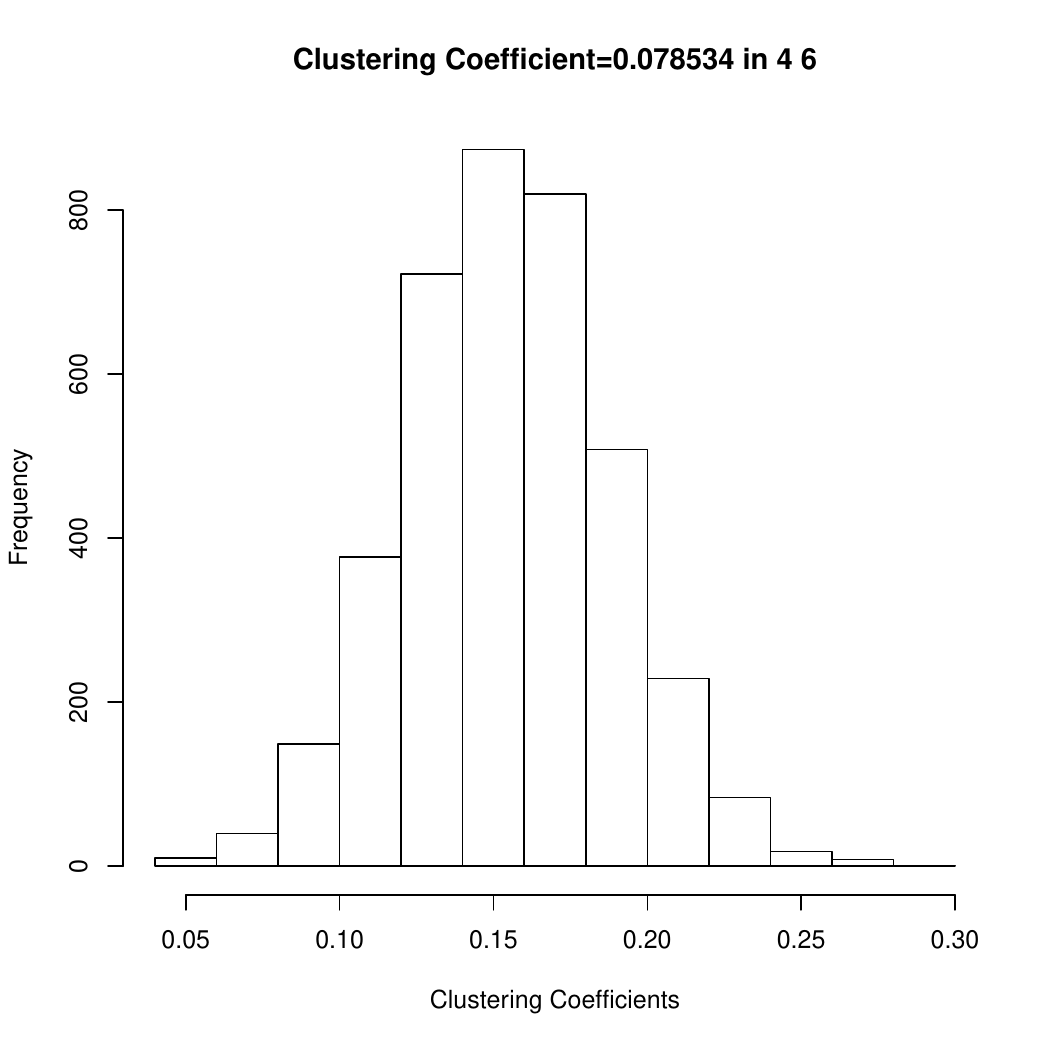}
\caption{Only edges with tour frequency between four to six}
\label{fig:cluster184_46}
\end{subfigure}

\begin{subfigure}{0.35\textwidth}
\centering
\includegraphics[width=\linewidth]{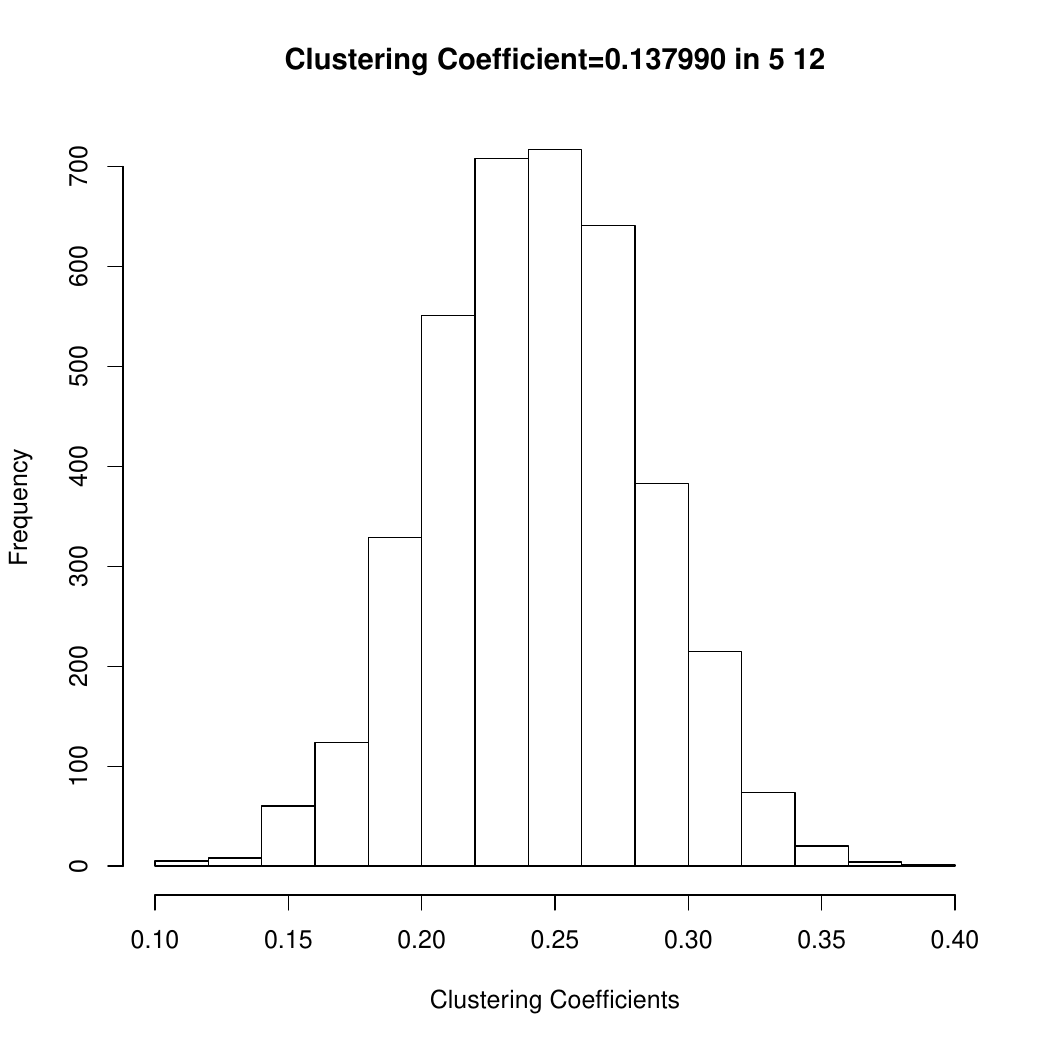}
\caption{Only edges with tour frequency between five to twelve}
\label{fig:cluster184_512}
\end{subfigure}%
\hspace*{1cm}
\begin{subfigure}{0.35\textwidth}
\centering
\includegraphics[width=\linewidth]{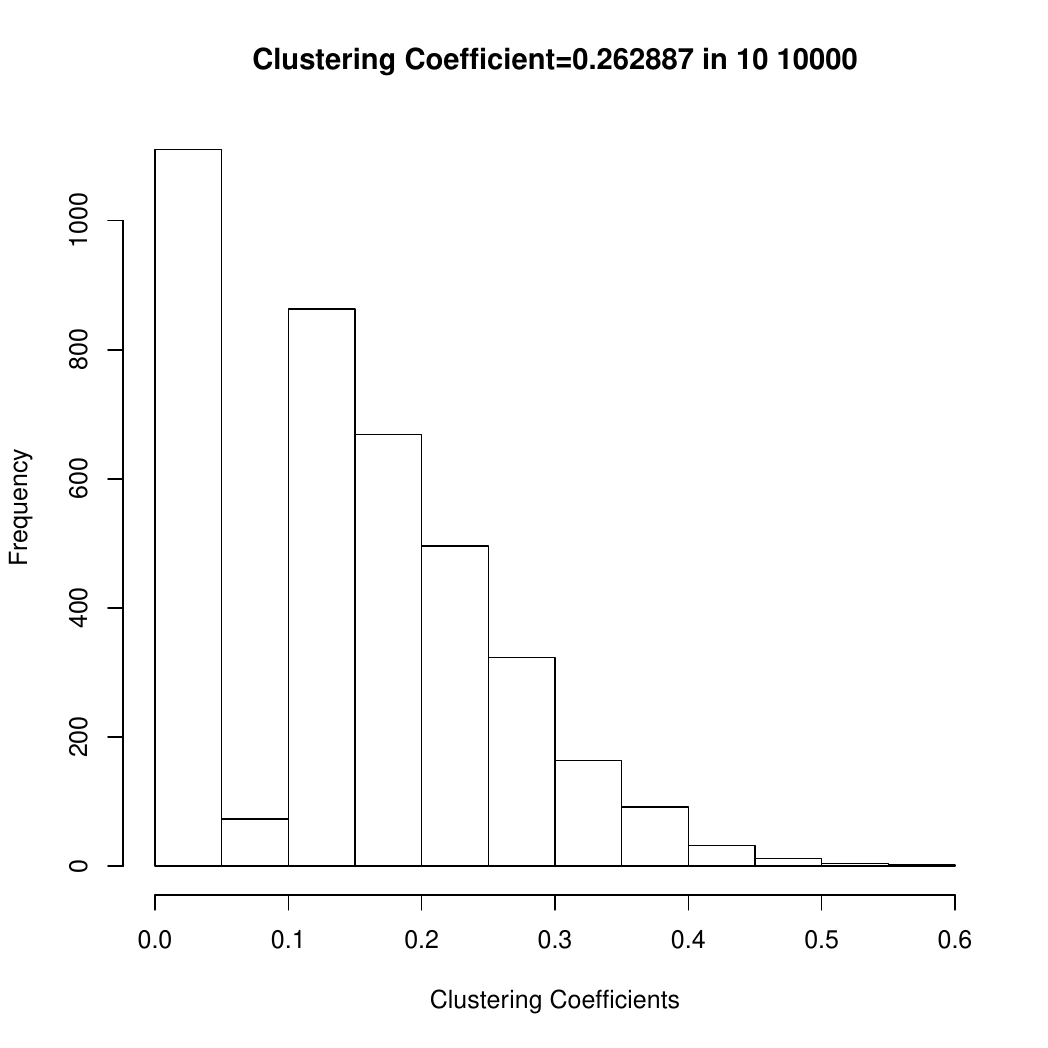}
\caption{Only edges with tour frequency more than ten}
\label{fig:cluster184_10inf}
\end{subfigure}
\caption{Histograms of clustering coefficients of the graphs which arise by taking different edges into account (see individual caption) from simulations for 10th July 2015. In the title of the plot the observed value is shown.}
\label{fig:cluster184}
\end{figure}

\begin{figure}
\centering
\begin{subfigure}{\textwidth}
\centering
\includegraphics[width=0.7\linewidth]{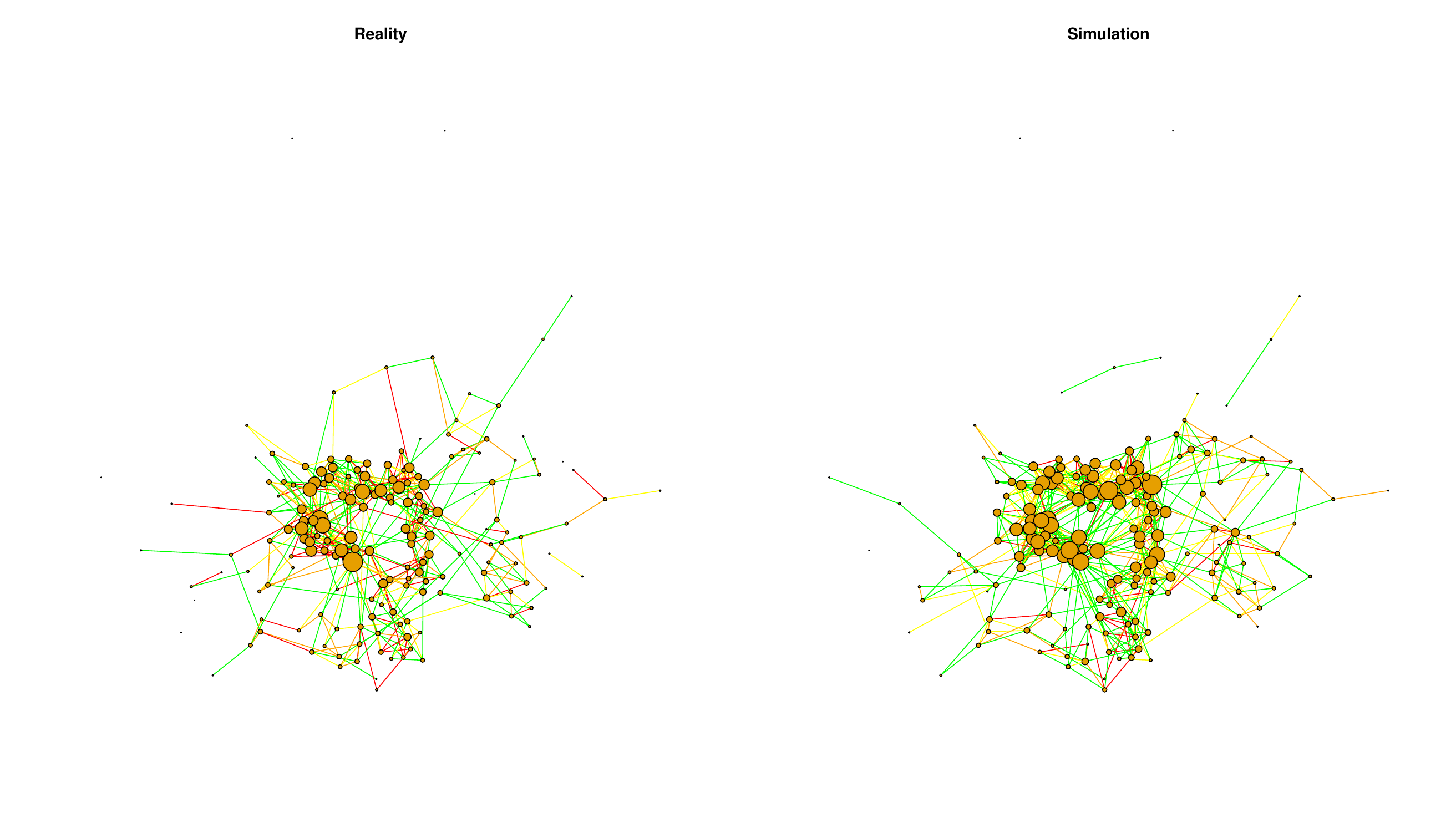}
\caption{12th December 2012}
\label{fig:graphDec}
\end{subfigure}
\hspace*{1cm}
\begin{subfigure}{\textwidth}
\centering
\includegraphics[width=0.7\linewidth]{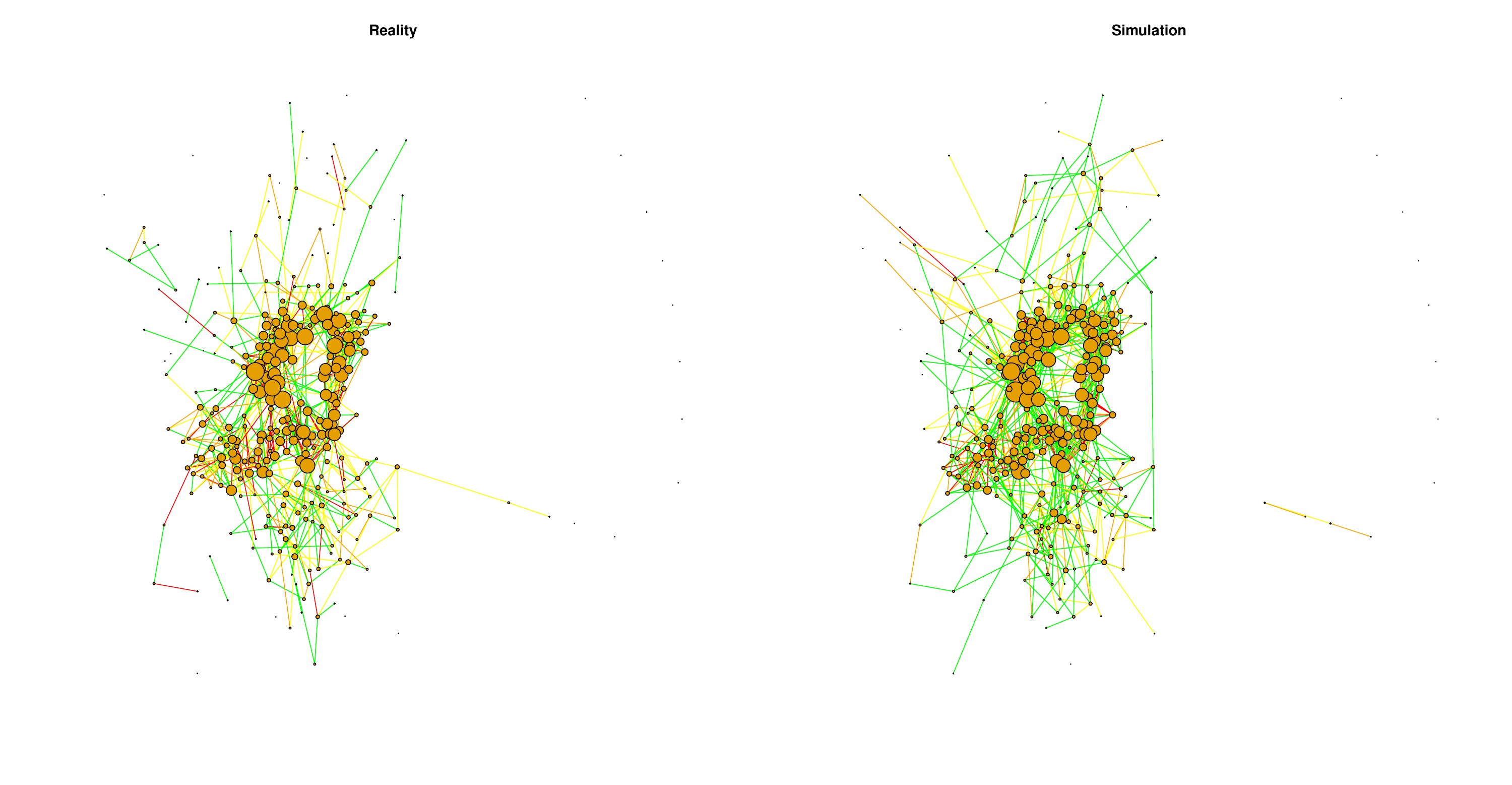}
\caption{18th April 2014}
\label{fig:graphApril}
\end{subfigure}
\hspace*{1cm}
\begin{subfigure}{\textwidth}
\centering
\includegraphics[width=0.7\linewidth]{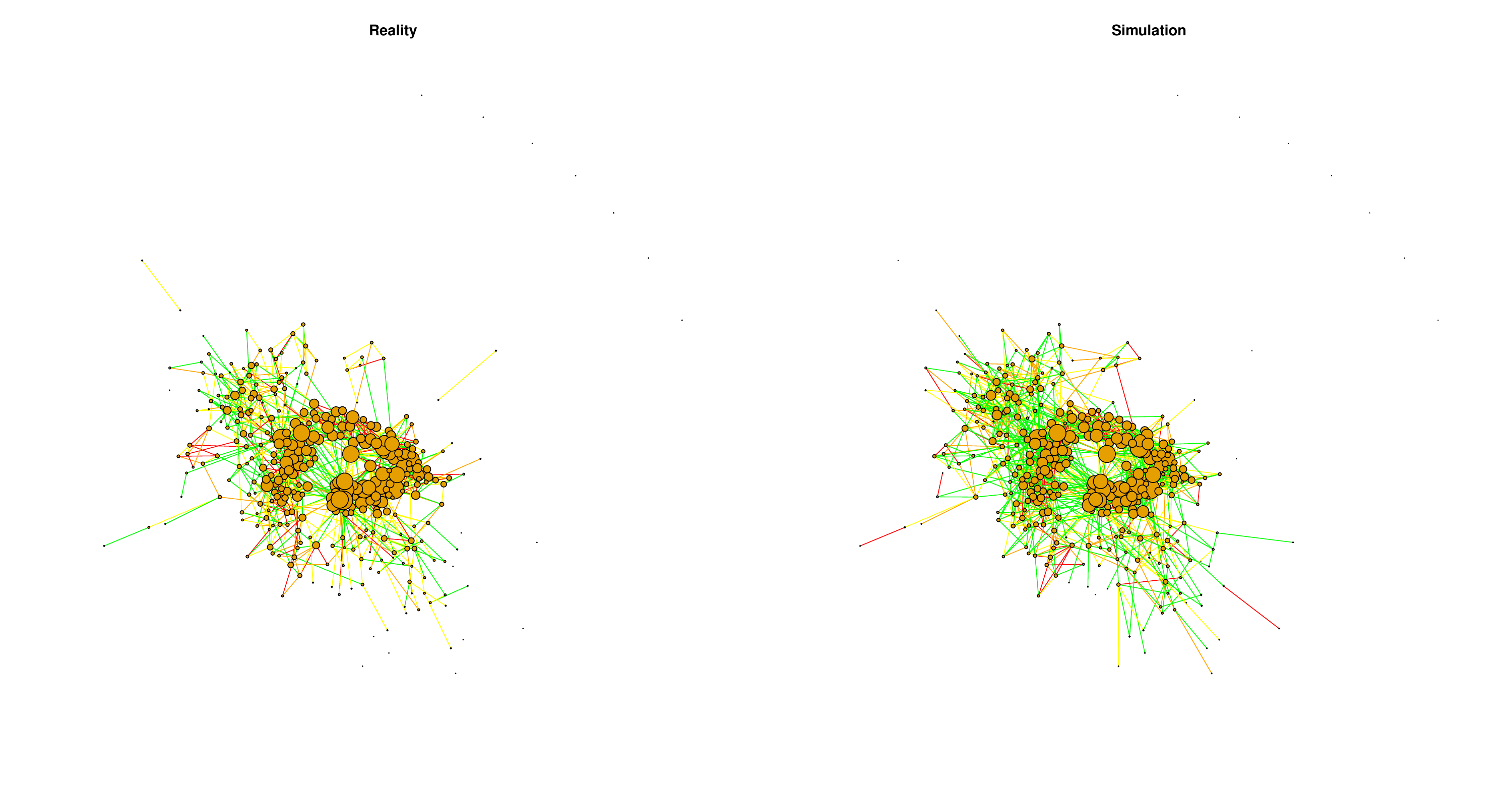}
\caption{10th July 2015}
\label{fig:graphJuly}
\end{subfigure}
\caption{Compares one simulated graph with the true observation.}
\label{fig:graphs}
\end{figure}

\subsection{Bandwidth choice}
Under our assumptions that the covariates stay constant over the day, it makes sense to consider only integral bandwidth lengths (for us one day has length one). In order to choose the bandwidth, we apply a one-sided cross validation (cf.\ \cite{HartYi1998,MMNS11}) approach which was shortly motivated in Section  \ref{data-analysis} and which we now describe in detail.

Let $K$ and $L$ be kernels fulfilling the assumptions in the paper and denote by $\hat{\theta}_K(t_0)$ and $\hat{\theta}_L(t_0)$ the maximum likelihood estimators using $K$ and $L$ respectively. Then, by Theorem \ref{thm:asymptotic_normality}, we get that asymptotically the bias and the variance of the estimators can be written as
\begin{eqnarray*}
bias(\hat{\theta}_K)&=& h^2\int_{-1}^1K(u)u^2du\cdot C_1 \\
var(\hat{\theta}_L)&=&\frac{1}{l_nh}\int_{-1}^1K(u)^2du\cdot C_2
\end{eqnarray*}
where the constants $C_1$ and $C_2$ depend on the true parameter curve $\theta_0$ and the time $t_0$ but not on the kernel. Hence, the corresponding expressions for $\hat{\theta}_L(t_0)$ can be found, just by replacing every $K$ with an $L$. The decomposition of the asymptotic mean squared error in squared bias plus variance yields the following asymptotically optimal bandwidths $h_K$ and $h_L$ , minimizing the asymptotic mean squared error:
$$h_K:=\Bigg(\frac{1}{l_n}\cdot\frac{\int_{-1}^1K(u)^2du}{\left[\int_{-1}^1K(u)u^2du\right]^2}\cdot\frac{C_1}{4C_2}\Bigg)^{\frac{1}{5}}.$$
Again, the corresponding expression for $h_L$ can be found by replacing every $K$ by $L$. So the following formula, known from kernel estimation, holds also true in our setting
\begin{equation}
\label{eq:bandwidth_transform}
h_K=\Bigg(\frac{\int_{-1}^1K(u)^2du}{\left[\int_{-1}^1K(u)u^2du\right]^2}\cdot\frac{\left[\int_{-1}^1L(u)u^2du\right]^2}{\int_{-1}^1L(u)^2du}\Bigg)^{\frac{1}{5}}h_L.
\end{equation}
This means that knowledge of the bandwidth minimizing the mean squared error for kernel $L,$ implies knowledge of the bandwidth minimizing the mean squared error using kernel $K$. Ultimately, we use a triangular kernel $K(u)=(1+u)\Ind_{[-1,0)}(u)+(1-u)\Ind_{[0,1]}(u)$. In order to find the bandwidth $h_K$ for this kernel, we want to apply cross-validation. As proposed in \cite{HartYi1998} one-sided cross validation is an attractive method for the case of time series data. One-sided here means that we apply cross validation to a kernel $L$ which is only supported on the past $[-1,0]$. In order to avoid a bias, we use the one-sided kernel together with local linear approximation. This following heuristic derivations motivates this choice.

Firstly, in our regular maximum likelihood setting, we maximize, over $\mu\in\Theta,$ the expression
\begin{eqnarray*}
&&\sum_{0<t\leq T}\frac{1}{h}K\left(\frac{t-t_0}{h}\right)\sum_{(i,j)\in L_n}\Delta N_{n,ij}(t)\mu^TX_{n,ij}(t) \\
&&\quad\quad\quad\quad\quad-\int_0^T\sum_{(i,j)\in L_n}\frac{1}{h}K\left(\frac{t-t_0}{h}\right)C_{n,ij}(t)e^{\mu^TX_{n,ij}(t)}dt \\
&\approx&\sum_{0<t\leq T}\frac{1}{h}K\left(\frac{t-t_0}{h}\right)\sum_{(i,j)\in L_n}\Delta N_{n,ij}(t)\mu^TX_{n,ij}(t) \\
&&-\int_0^T\sum_{(i,j)\in L_n}\frac{1}{h}K\left(\frac{t-t_0}{h}\right)C_{n,ij}(t)e^{\theta_0(t_0)^TX_{n,ij}(t)} \\
&&\quad\quad\quad\times\left(1+(\mu-\theta_0(t_0))^TX_{n,ij}(t)+\frac{1}{2}\left[(\mu-\theta_0(t_0))^TX_{n,ij}(t)\right]^2\right)dt.
\end{eqnarray*}
Deriving this expression with respect to $\mu$, setting the derivative equal to zero, and rearranging terms, yields (to save space we use here a fraction, although the denominator is a matrix)
\begin{eqnarray*}
&&\hat{\theta}_K(t_0)-\theta_0(t_0) \\
&\approx&\tfrac{\sum_{(i,j)\in L_n}\frac{1}{h}K\left(\frac{t-t_0}{h}\right)\Delta N_{n,ij}(t)X_{n,ij}(t)-\int_0^T\frac{1}{h}K\left(\frac{t-t_0}{h}\right)C_{n,ij}(t)X_{n,ij}e^{\theta_0(t_0)^TX_{n,ij}(t)}dt}{\sum_{(i,j)\in L_m}\int_0^T\frac{1}{h}K\left(\frac{t-t_0}{h}\right)X_{n,ij}(t)X_{n,ij}(t)^Te^{\theta_0(t_0)^TX_{n,ij}(t)}dt}.\end{eqnarray*}
Using the notation $y_1:=\IE(X_{n,ij}(t_0)e^{\theta_0(t_0)^TX_{n,ij}(t_0)}|C_{n,ij}(t_0)=1)\cdot\IP(C_{n,ij}(t_0)=1)$ and $y_2:=\IE(X_{n,ij}(t_0)X_{n,ij}(t_0)^Te^{\theta_0(t_0)^TX_{n,ij}(t_0)}|C_{n,ij}(t_0)=1)\cdot\IP(C_{n,ij}(t_0)=1),$ we obtain the approximation 
\begin{equation}
\label{eq:approx_kernel_estimator}
\hat{\theta}_K(t_0)-\theta_0(t_0)\approx\frac{\sum_{(i,j)\in L_n}\sum_{0<t\leq T}\frac{1}{h}K\left(\frac{t-t_0}{h}\right)\Delta N_{n,ij}(t)X_{n,ij}(t)-y_1}{y_2}.
\end{equation}
Now define the local linear estimator $\hat{\theta}_{LC,K}(t_0),$ with respect to a kernel $K,$ as the value of $\mu_0$ maximizing the following expression over $(\mu_0,\mu_1)\in\Theta^2:$

\begin{eqnarray*}
&&\sum_{0<t\leq T} \frac{1}{h}K\left(\frac{t-t_0}{h}\right)\sum_{(i,j)\in L_n}\Delta N_{n,ij}(t)[\mu_0+\mu_1(t-t_0)]^TX_{n,ij}(t) \\
&&\quad\quad\quad\quad-\int_0^T\sum_{(i,j)\in L_n}\frac{1}{h}K\left(\frac{t-t_0}{h}\right)e^{[\mu_0+\mu_1(t-t_0)]^TX_{n,ij}(t)}dt.
\end{eqnarray*}
Using the same approximations as in the usual kernel estimation setting, and deriving the resulting approximate likelihood, we obtain
\begin{align*}
&\hat{\theta}_{LC,K}-\theta_0(t_0) \\
\approx&\frac{\sum_{(i,j)\in L_n}\sum_{0<t\leq T}\frac{1}{h}K\left(\frac{t-t_0}{h}\right)\frac{M_2-\frac{t-t_0}{h}M_1}{M_2-M_1^2}\Delta N_{n,ij}(t)X_{n,ij}(t)-y_1}{y_2},
\end{align*}
where $M_k:=\int_{-1}^1u^kK(u)du$. The previous computations were just a heuristic. But nevertheless, the similarity between the previous display and \eqref{eq:approx_kernel_estimator} suggests that the local linear estimator $\hat{\theta}_{LC,K}$ using the kernel $K$ is actually just a regular kernel estimator $\hat{\theta}_L$ with kernel function
\begin{equation}
\label{eq:new_kernel}
L(u)=K(u)\frac{M_2-uM_1}{M_2-M_1^2}.
\end{equation}
This aligns with results about kernel estimation, as, for example, stated in \cite{MMNS11}. It can be easily computed that this new kernel is of order one, i.e., $\int uL(u)du=0$, even though the original kernel was not. Hence, knowledge of the optimal bandwidth for the local linear estimator using the kernel $K$ implies knowledge of the optimal bandwidth for any other order one kernel by means of  \eqref{eq:bandwidth_transform}. Taking the same route as in \cite{HartYi1998}, the selector for the bandwidth $\hat{h}_K$ for the triangular kernel $K$ is the following: Let $K^*(u):=2K(u)\Ind_{[-1,0]}(u)$ denote the one sided version of $K.$
\begin{enumerate}
\item Find a bandwidth $\hat{h}_L$ for the local linear estimator $\hat{\theta}_{LC,K^*}$ based on the kernel $K^*$ via cross validation (since we use a one-sided kernel, this step is also called one-sided cross-validation. We will make it more precise later).
\item Compute $\hat{h}$ by using \eqref{eq:bandwidth_transform} with $L$ defined as in \eqref{eq:new_kernel} but with $K$ replaced by $K^*$.
\end{enumerate}

For the one-sided cross-validation in step 1, we minimize, in our bike share data analysis, the following function in $h$
\begin{equation}
\label{eq:cross-validation}
\frac{1}{k_T}\sum_{k=0}^{k_T}\frac{1}{|L(k)|}\sum_{(i,j)\in L(k)}\frac{\left|e^{\hat{\theta}_{LC,K^*}^{(-k)}(k)^TX_{i,j}(k)}-X_{i,j}(k)\right|^2}{e^{\hat{\theta}_{LC,K^*}^{(-k)}(k)^TX_{i,j}(k)}},
\end{equation}
where $k_T$ was the number of weeks (recall that we assume the covariates to remain constant over a day, and that we only consider Fridays), i.e., $k$ refers to the $k$-th Friday in the dataset. $L(k)$ refers to the set of pairs $(i,j)$ between which there was a bike tour on Friday $k$, $X_{i,j}(k)$ is the true number of bike tours observed between $i$ and $j$ on Friday $k$. Finally, $\hat{\theta}_{LC,K^*}^{(-k)}(k)$ is the local linear estimator with respect to the kernel $K^*$ based on all but the $k$-th Friday. Since $K^*$ is left-sided, this really means the estimator is based on Fridays $0,...,k-1,$ and hence the term one-sided cross-validation. The intensities are the theoretical values of the expectation of the number of bike tours if the model is correct. So we compute the squared difference with the true number of bike rides and divide by the estimated intensity, where we only take the non-censored edges into account.

In Section  \ref{data-analysis}, we had displayed results for different bandwidths $h$ of \eqref{eq:cross-validation}  in Figure \ref{fig:bandwidth_selection}. The prediction error of the fit decreases, until the bandwidth is equal to 23. Afterwards the prediction error stays roughly the same and starts to increase when the bandwidth reaches a full year (52 weeks). This may be explained  by a periodicity with a period of approximately one year. If one uses 23 as minimal value we get as asymptotic optimal bandwidth 12 which is approximately 23 divided by $\rho$. Here, following Step 2 of the above described procedure, we use that $\rho$ is approximately equal to 1.82 for triangular kernels.
\end{document}